\documentclass[12pt,twoside,a4paper]{report}

\usepackage[T1]{fontenc}
\usepackage[DM,newLogo,final]{uaThesis}
\def\ThesisYear{2013}
\usepackage{hyperref}
\usepackage{a4wide}
\usepackage{amsmath,amssymb,amsthm}
\usepackage{bm,dsfont}
\usepackage{makeidx}
\usepackage{slashbox,makecell}
\usepackage{graphicx,subfigure}
\usepackage{cite}
\usepackage{fancyhdr}
\usepackage{epstopdf}

\newtheorem{theorem}{Theorem}
\newtheorem{corollary}[theorem]{Corollary}
\newtheorem{lemma}[theorem]{Lemma}
\newtheorem{proposition}[theorem]{Proposition}
\newtheorem{definition}[theorem]{Definition}
\newtheorem{example}[theorem]{Example}
\newtheorem{remark}[theorem]{Remark}

\def\CP{\clearpage{\thispagestyle{empty}\cleardoublepage}}

\newcommand{\N}{\mathbb{N}}

\newcommand{\Chapter}[2]{\chapter[#1]{#1\\[1ex]\Large#2}}

\makeindex
\def\a{\alpha}
\def\ba{{\bm \alpha}}
\def\b{\beta}

\def\t{\tau}
\def\x{{\bf x}}
\def\u{{\bf u}}
\def\l{{\bm\lambda}}
\def\LD{{_0D_t^\a}}
\def\LDa{{_aD_t^\a}}
\def\LDai{{_aD_{t_i}^\a}}
\def\LDaik{{_aD_{t_{i+k}}^\a}}
\def\LDC{{_a^CD_t^\a}}
\def\RDC{{_t^CD_b^\a}}
\def\RD{{_tD_b^\a}}
\def\LDz{{_0D_t^{0.5}}}
\def\LI{{_aI_t^\a}}

\def\RI{{_tI_b^\a}}
\def\LD{{_aD_t^\a}}
\def\RD{{_tD_b^\a}}
\def\RDone{{_tD_1^\a}}
\def\LIz{{_0I_t^{0.5}}}

\def\LHI{{_a\mathcal{I}_t^\a}}
\def\RHI{{_t\mathcal{I}_b^\a}}
\def\LHD{{_a\mathcal{D}_t^\a}}
\def\RHD{{_t\mathcal{D}_b^\a}}
\def\LHIHz{{_1\mathcal{I}_t^{0.5}}}
\def\LHDHz{{_1\mathcal{D}_t^{0.5}}}

\def\RDone{{_tD_1^\a}}
\def\GLa{{^{GL}_{\phantom{1}a}D_t^\a}}
\def\GLb{{^{GL}_{\phantom{1}t}D_b^\a}}
\def\sGLa{{^{sGL}_{\phantom{1}a}D_t^\a}}
\def\w{\left(\omega_k^\a\right)}
\def\wh{\left(\omega_k^{0.5}\right)}

\def\RIT{{_tI_T^{1-\a}}}
\def\LCD{{^C_aD_t^\a}}
\def\RDT{{_tD_T^\a}}
\makeatletter
\renewcommand\part{%
  \if@openright
    \cleardoublepage
  \else
    \clearpage
  \fi
  \thispagestyle{empty}%
  \if@twocolumn
    \onecolumn
    \@tempswatrue
  \else
    \@tempswafalse
  \fi
  \null\vfil
  \secdef\@part\@spart}
\makeatother

\makeindex
\graphicspath{{Images/}}

\begin{document}
\fancyhf{}

\fancyhead[RO]{\small\slshape\nouppercase\rightmark}
\fancyhead[LE]{\small\slshape\nouppercase\leftmark} \fancyfoot[C]{\thepage}

\TitlePage
  \HEADER{\BAR\FIG{\includegraphics[height=60mm]{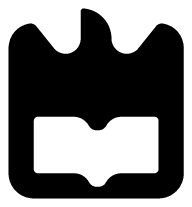}}}
         {\ThesisYear}
    \TITLE{Shakoor Pooseh}
          {M\'{e}todos Computacionais no C\'{a}lculo das Varia\c{c}\~{o}es e Controlo \'{O}ptimo Fraccionais
          \vspace{5mm}\newline
          Computational Methods in the Fractional Calculus of Variations and Optimal Control}
          \EndTitlePage
\titlepage\ 
\endtitlepage

\TitlePage
  \HEADER{}{\ThesisYear}
  \TITLE{Shakoor Pooseh}
          {M\'{e}todos Computacionais no C\'{a}lculo das Varia\c{c}\~{o}es e Controlo \'{O}ptimo Fraccionais
          \vspace{5mm}\newline
          Computational Methods in the Fractional Calculus of Variations and Optimal Control}
\vspace*{15mm}
\TEXT{}
{Tese de Doutoramento apresentada \`{a} Universidade de Aveiro para cumprimento 
dos requisitos necess\'{a}rios \`{a} obten\c{c}\~{a}o do grau de Doutor em Matem\'{a}tica, 
Programa Doutoral em Matem\'{a}tica e Aplica\c{c}\~{o}es, PDMA 2009--2013, da Universidade de Aveiro 
e Universidade do Minho realizada sob a orienta\c{c}\~{a}o cient\'{\i}fica do 
Prof. Doutor Delfim Fernando Marado Torres, Professor Associado com Agrega\c{c}\~{a}o 
do Departamento de Matem\'{a}tica da Universidade de Aveiro e do Prof. Doutor Ricardo Miguel Moreira de Almeida, 
Professor Auxiliar do Departamento de Matem\'{a}tica da Universidade de Aveiro.}
\EndTitlePage
\titlepage\ \endtitlepage

\TitlePage
  \HEADER{}{\ThesisYear}
  \TITLE{Shakoor Pooseh}
          {M\'{e}todos Computacionais no C\'{a}lculo das Varia\c{c}\~{o}es e Controlo \'{O}ptimo Fraccionais
          \vspace{5mm}\newline
          Computational Methods in the Fractional Calculus of Variations and Optimal Control}
  \vspace*{15mm}
  \TEXT{}
{Ph.D. thesis submitted to the University of Aveiro in fulfilment of the requirements 
for the degree of Doctor of Philosophy in Mathematics, Doctoral Programme in Mathematics 
and Applications 2009--2013, of the University of Aveiro and University of Minho, 
under the supervision of Professor Delfim Fernando Marado Torres, Associate Professor 
with Habilitation and tenure of the Department of Mathematics of University of Aveiro 
and Professor Ricardo Miguel Moreira de Almeida, Assistant Professor 
of the Department of Mathematics of University of Aveiro.}
\EndTitlePage
\titlepage\ \endtitlepage

\TitlePage
  \vspace*{55mm}
  \TEXT{\textbf{o j\'uri\newline}}
       {}
  \TEXT{Presidente}
       {\textbf{Prof. Doutor Artur Manuel Soares da Silva}\newline {\small
        Professor Catedr\'{a}tico da Universidade de Aveiro}}
  \vspace*{5mm}

  \TEXT{}
       {\textbf{Prof. Doutor St\'{e}phane Louis Clain}\newline {\small
        Professor Associado com Agrega\c{c}\~{a}o da Escola de Ci\^{e}ncias da Universidade do Minho}}
  \vspace*{5mm}

  \TEXT{}
       {\textbf{Prof. Doutor Manuel Duarte Ortigueira}\newline {\small
        Professor Associado com Agrega\c{c}\~{a}o da Faculdade de Ci\^{e}ncias e Tecnologia
        da Universidade Nova de Lisboa}}
  \vspace*{5mm}

  \TEXT{}
       {\textbf{Prof. Doutor Delfim Fernando Marado Torres}\newline {\small
        Professor Associado com Agrega\c c\~ao da Universidade de Aveiro (Orientador)}}
  \vspace*{5mm}

  \TEXT{}
       {\textbf{Prof. Doutora Erc\'{\i}lia Cristina Costa e Sousa}\newline {\small
        Professora Auxiliar da Faculdade de Ci\^{e}ncias e Tecnologia da Universidade de Coimbra}}
  \vspace*{5mm}

  \TEXT{}
       {\textbf{Prof. Doutora Maria Lu\'{\i}sa Ribeiro dos Santos Morgado}\newline {\small
        Professora Auxiliar da Faculdade de Ci\^{e}ncias e Tecnologia da Universidade de
        Tr\'{a}s-os-Montes e Alto Douro}}
  \vspace*{5mm}

  \TEXT{}
       {\textbf{Prof. Doutor Ricardo Miguel Moreira de Almeida}\newline {\small
        Professor Auxiliar da Universidade de Aveiro (Coorientador)}}
\EndTitlePage
\titlepage\ \endtitlepage
\TitlePage
  \vspace*{55mm}
  \TEXT{\textbf{agradecimentos}}
       {
       Esta tese de doutoramento \'{e} o resultado da colabora\c{c}\~{a}o de muitas pessoas. Primeiro de tudo,
       Delfim F. M. Torres, o meu orientador, que me ajudou muito ao longo dos \'{u}ltimos anos,
       proporcionando uma din\^{a}mica e amig\'{a}vel atmosfera, prop\'{\i}cia \`{a} investiga\c{c}\~{a}o. Foi tamb\'{e}m
       grande sorte da minha parte ter Ricardo Almeida como co-orientador, agindo n\~{a}o s\'{o} nesse
       qualidade, mas tamb\'{e}m como um amigo e colega. Devo a estas duas pessoas muito de aluno
       para orientador.\\
       A actividade cient\'{\i}fica s\'{o} \'{e} poss\'{\i}vel por um efectivo apoio financeiro, que a Funda\c{c}\~{a}o
       Portuguesa para a Ci\^{e}ncia e a Tecnologia (FCT), me forneceu atrav\'{e}s da bolsa de doutoramento
       SFRH/BD/33761/2009, no \^{a}mbito do Programa Doutoral em Matem\'{a}tica e Aplica\c{c}\~{o}es (PDMA) das
       Universidades de Aveiro e Minho. Al\'{e}m do apoio financeiro da FCT, ter sido membro do Grupo
       de Teoria Matem\'{a}tica dos Sistemas e Controlo do Centro de Investiga\c{c}\~{a}o e Desenvolvimento em
       Matem\'{a}tica e Aplica\c{c}\~{o}es (CIDMA) teve um papel fundamental, que aqui real\c{c}o.\\
       A todos os que tiveram um efeito sobre minha vida de estudante de doutoramento, professores,
       funcion\'{a}rios e amigos, gostaria de expressar os meus sinceros agradecimentos. \`{A} minha fam\'{\i}lia
       e esposa, fundamentais como suporte mental e moral, compreens\~{a}o e toler\^{a}ncia, pelo muito que
       me ajudaram e por suportarem algum tipo de v\'{\i}cio em trabalho e ego\'{\i}smo.
       }
\vfil
\vspace{2cm}
\TEXT{}{\begin{center}
\includegraphics[scale=0.1]{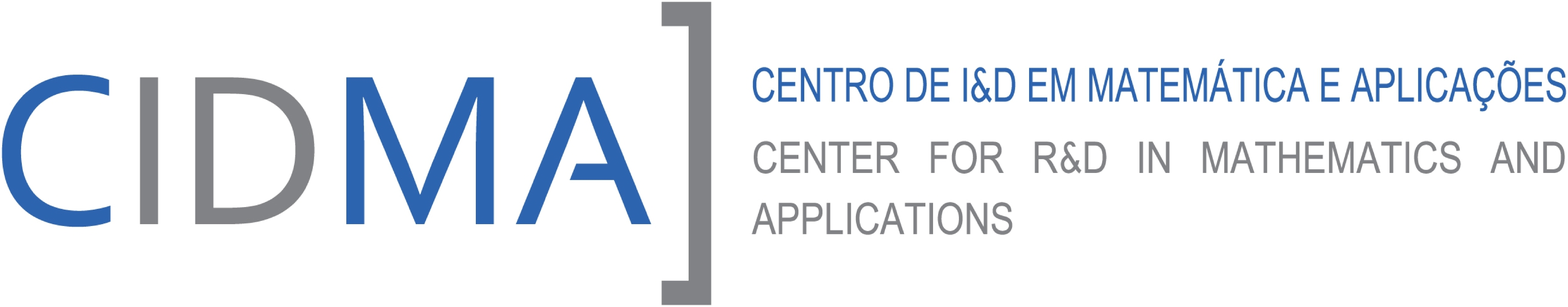}\\[20pt]
\includegraphics[scale=0.2]{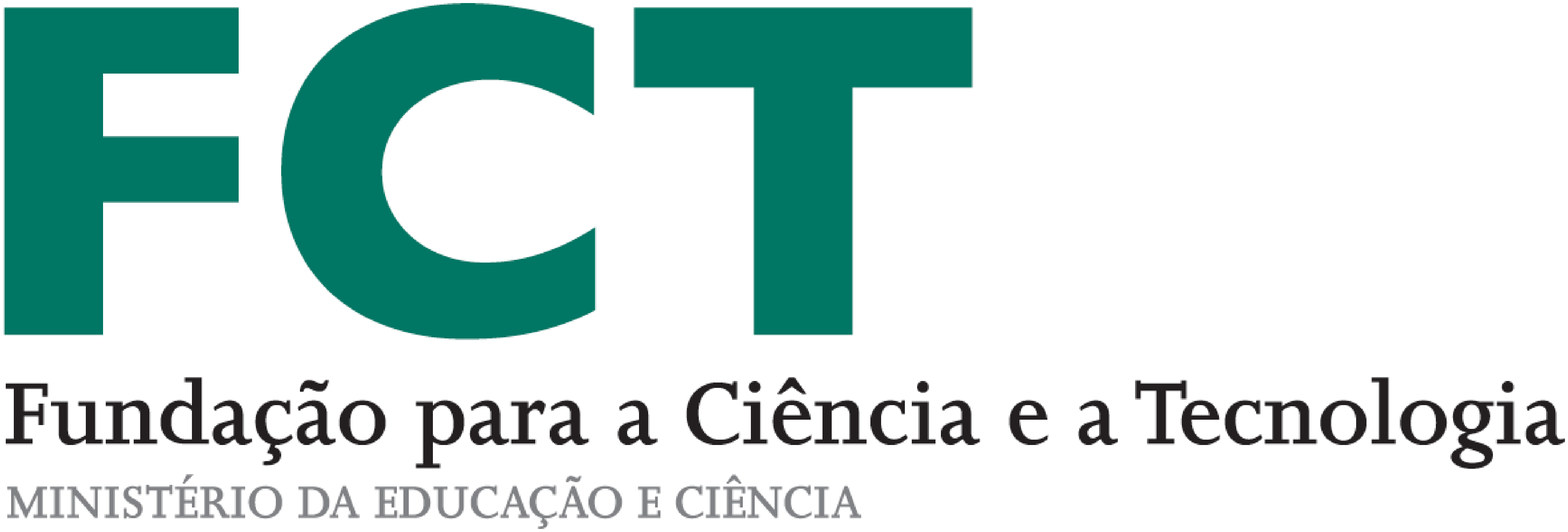}
\end{center}}

\EndTitlePage
\titlepage\ \endtitlepage

\TitlePage
  \vspace*{55mm}
  \TEXT{\textbf{acknowledgements}}
       {
       This thesis is the result of a collaboration of many people. First of all, Delfim F. M. Torres,
       my supervisor, that helped me a lot through these years by providing a dynamic, yet friendly,
       atmosphere for research. A great luck of mine is also having Ricardo Almeida, acting not only
       as my co-advisor, but also as a friend and colleague. I owe these two people much than a student
       to supervisors.\\
       A scientific activity is only possible by a good financial support, which the Portuguese
       Foundation for Science and Technology (FCT)-- {\it Funda\c{c}\~{a}o para a Ci\^{e}ncia e a
       Tecnologia}-- provided me through the Ph.D. fellowship SFRH/BD/33761/2009, within
       \emph{Doctoral Program in Mathematics and Applications} (PDMA) of Universities of Aveiro and
       Minho. Besides financial support, a good research team has a crucial role, which is well
       provided by CIDMA, {\it Center for Research and Development in Mathematics and Applications},
       that I deeply appreciate.\\
       Together with all who had an effect on my studentship life, from teachers and staff to friends,
       I would like to express grateful thanks to my family and my wife whom mental and moral supports,
       understanding, and tolerance helped me a lot; let me to be some kind of workaholic,
       ignorant and selfish.
       }
\vfil
\vspace{2cm}
\TEXT{}{\begin{center}
\includegraphics[scale=0.1]{CIDMA.eps}\\[20pt]
\includegraphics[scale=0.2]{FCT.eps}
\end{center}}

\EndTitlePage
\titlepage\ \endtitlepage

\TitlePage
  \vspace*{55mm}
\TEXT{\textbf{palavras-chave}}
       {Optimiza\c{c}\~{a}o e controlo, c\'{a}lculo fraccional, c\'{a}lculo das varia\c{c}\~{o}es fraccional,
       controlo \'{o}ptimo fraccional, condi\c{c}\~{o}es necess\'{a}rios de optimalidade, m\'{e}todos directos,
       m\'{e}todos indirectos, aproxima\c{c}\~{a}o num\'{e}rica, estima\c{c}\~{a}o de erros, equa\c{c}\~{o}es
       diferenciais fraccionais. }
\TEXT{}\\
\TEXT{\textbf{resumo}}
       {O c\'{a}lculo das varia\c{c}\~{o}es e controlo \'{o}ptimo fraccionais s\~{a}o generaliza\c{c}\~{o}es
        das correspondentes teorias cl\'{a}ssicas, que permitem formula\c{c}\~{o}es e modelar
        problemas com derivadas e integrais de ordem arbitr\'{a}ria. Devido \`{a} car\^{e}ncia
        de m\'{e}todos anal\'{\i}ticos para resolver tais problemas fraccionais, t\'{e}cnicas
        num\'{e}ricas s\~{a}o desenvolvidas. Nesta tese, investigamos a aproxima\c{c}\~{a}o
        de operadores fraccionais recorrendo a s\'{e}ries de derivadas de ordem
        inteira e diferen\c{c}as finitas generalizadas. Obtemos majorantes para o erro
        das aproxima\c{c}\~{o}es propostas e estudamos a sua efici\^{e}ncia. M\'{e}todos directos
        e indirectos para a resolu\c{c}\~{a}o de problemas variacionais fraccionais s\~{a}o
        estudados em detalhe. Discutimos tamb\'{e}m condi\c{c}\~{o}es de optimalidade para
        diferentes tipos de problemas variacionais, sem e com restri\c{c}\~{o}es, e para
        problemas de controlo \'{o}ptimo fraccionais. As t\'{e}cnicas num\'{e}ricas introduzidas
        s\~{a}o ilustradas recorrendo a exemplos.}
  \TEXT{}
       {}
\EndTitlePage
\titlepage\ \endtitlepage

\TitlePage
  \vspace*{55mm}
\TEXT{\textbf{keywords}}
       {Optimization and control, fractional calculus, fractional calculus of variations,
        fractional optimal control, fractional necessary optimality conditions, direct methods,
        indirect methods, numerical approximation, error estimation, fractional differential equations.
       }
\TEXT{}\\
 \TEXT{\textbf{abstract}}
       {The fractional calculus of variations and fractional optimal control are generalizations of the
       corresponding classical theories, that allow problem modeling and formulations with arbitrary
       order derivatives and integrals. Because of the lack of analytic methods to solve such fractional
       problems, numerical techniques are developed. Here, we mainly investigate the approximation of
       fractional operators by means of series of integer-order derivatives and generalized finite
       differences. We give upper bounds for the error of proposed approximations and study their
       efficiency.  Direct and indirect methods in solving fractional variational problems are studied
       in detail. Furthermore, optimality conditions are discussed for different types of unconstrained
       and constrained variational problems and for fractional optimal control problems. The introduced
       numerical methods are employed to solve some illustrative examples.
\newline\newline
\textbf{2010 Mathematics Subject Classification: 26A33, 34A08, 65D20, 33F05,
                                                 49K15, 49M25, 49M99.}
}  \EndTitlePage
\titlepage\ \endtitlepage

\pagenumbering{roman}
\tableofcontents
\CP
\cleardoublepage

\pagenumbering{arabic}
\chapter*{Introduction}\markboth{\MakeUppercase{Introduction}}{}\addcontentsline{toc}{part}{Introduction}

This thesis is devoted to the study of numerical methods in the calculus of variations 
and optimal control in the presence of fractional derivatives and/or integrals. 
A fractional problem of the calculus of variations and optimal control consists 
in the study of an optimization problem in which, the objective functional 
or constraints depend on derivatives and integrals of arbitrary, real or complex, orders. 
This is a generalization of the classical theory, where derivatives and integrals can only 
appear in integer orders. Throughout this thesis we will call the problems in the calculus 
of variations and optimal control, variational problems. If at least one fractional term exists 
in the formulation, it is called a fractional variational problem.

The theory started in 1996 with the works of Riewe, in order to better describe 
non-conservative systems in mechanics \cite{Riewe,Riewe1}. The subject is now under 
strong development due to its many applications in physics and engineering, providing 
more accurate models of physical phenomena (see, \textrm{e.g.}, 
\cite{Kai,Das,Machado1,VibString,AlD,MyID:179,El-Nabulsi1,El-Nabulsi2,MyID:191,%
gastao:delfim,gasta1,Dorota,MyID:203,Ortigueira}).

In order to provide a better understanding, the classical theory of the calculus 
of variations and optimal control is discussed briefly in the beginning of this 
thesis in Chapter~\ref{Prelaminaries}. Major concepts and notions are presented; 
key features are pointed out and some solution methods are detailed. There are 
two major approaches in the classical theory of calculus of variations to solve problems.
In one hand, using Euler--Lagrange necessary optimality conditions, we can reduce
a variational problem to the study of a differential equation. Hereafter,
one can use either analytical or numerical methods to solve the differential
equation and reach the solution of the original problem (see, e.g., \cite{Kirk}).
This approach is referred as indirect methods in the literature.

On the other hand, we can tackle the functional itself, directly. Direct methods are used
to find the extremizer of a functional in two ways: Euler's finite differences
and Ritz methods. In the Ritz method, we either restrict admissible functions
to all possible linear combinations
$$
x_n(t)=\sum_{i=1}^{n}\a_i\phi_i(t),
$$
with constant coefficients $\a_i$ and a set of known basis functions $\phi_i$,
or we approximate the admissible functions with such combinations. Using $x_n$
and its derivatives whenever needed, one can transform the functional
to a multivariate function of unknown coefficients $\a_i$.
By finite differences, however, we consider the admissible functions
not on the class of arbitrary curves, but only on polygonal curves
made upon a given grid on the time horizon. Using an appropriate
discrete approximation of the Lagrangian, and substituting the integral with a sum,
and the derivatives by appropriate approximations,
we can transform the main problem to the optimization of a function of several parameters:
the values of the unknown function on mesh points (see, e.g., \cite{Elsgolts}).

A historical review of fractional calculus comes next in Chapter \ref{FracCalculus}. 
In general terms, the field that allows us to define integrals and derivatives
of arbitrary real or complex order is called fractional calculus 
and can be seen as a generalization of ordinary calculus.
A fractional derivative of order $\a>0$, when $\a=n$ is an integer, coincides with
the classical derivative of order $n\in\mathbb N$, while
a fractional integral is an \textit{n-fold} integral. The origin 
of fractional calculus goes back to the end of the seventeenth century,
though the main contributions have been made during the last few decades \cite{Sun,Machado}.
Namely it has been proven to be a useful tool in engineering and optimal control problems 
(see, \textrm{e.g.}, \cite{Ref01,Ref02,Ref03,Efe,Li,Shen}). Furthermore, during the last three decades,
several numerical methods have been developed in the field of fractional calculus.
Some of their advantages, disadvantages, and improvements, are given in \cite{sug:r3}.

There are several different definitions of fractional derivatives
in the literature, such as Riemann--Liouville, Gr\"{u}nwald--Letnikov, Caputo, etc. 
They posses different properties: each one of those definitions has its 
own advantages and disadvantages. Under certain conditions, however,
they are equivalent and can be used interchangeably. 
The Riemann--Liouville and Caputo are the most common
for fractional derivatives, and for fractional integrals the usual 
one is the Riemann--Liouville definition.

After some introductory arguments of classical theories for variational problems 
and fractional calculus, the next step is providing the framework that is required 
to include fractional terms in variational problems and is shown 
in Chapter \ref{FracVariational}. In this framework,
the fractional calculus of variations and optimal control are research areas under
strong current development. For the state of the art,
we refer the reader to the recent book \cite{book:frac},
for models and numerical methods we refer to \cite{book:Baleanu}.

A fractional variational problem consists in finding the extremizer
of a functional that depends on fractional derivatives and/or integrals
subject to some boundary conditions and possibly some extra constraints.
As a simple example one can consider the following minimization problem:
\begin{equation}
\label{MainProblem}
\begin{gathered}
J[x(\cdot)]=\int_a^b L(t, x(t), \LDa x(t))dt \longrightarrow \min,\\
x(a)=x_a, \quad x(b)=x_b,
\end{gathered}
\end{equation}
that depends on the left Riemann--Liouville derivative,
$\LDa$. Although this has been a common formulation for a fractional variational problem, 
the consistency of fractional operators and the initial conditions is questioned by many authors. 
For further readings we refer to \cite{OrtigueiraIC,OrtigueiraSiLS,Trigeassou} and references therein.

An Euler--Lagrange equation for this problem has been derived first in \cite{Riewe,Riewe1} 
(see also \cite{AgrawalForm}). A generalization of the problem to include fractional integrals, 
the transversality conditions and many other aspects can be found in the literature of recent years.  
See \cite{AlD,Atan,book:frac} and references therein. Indirect methods for fractional variational 
problems have a vast background in the literature and can be considered a well studied subject: see
\cite{AgrawalForm,APTIndInt,Atan,Gastao,Jelicic,Klimek,MyID:227,Riewe1}
and references therein that study different variants of the problem and discuss
a bunch of possibilities in the presence of fractional terms, Euler--Lagrange
equations and boundary conditions. With respect to results on fractional variational calculus 
via Caputo operators, we refer the reader to 
\cite{AGRA1,Ankara:Ric,Almeida,Gastao,Malinowska,MyID:163,MyID:207} and references therein.

Direct methods, however, to the best of our knowledge,
have got less interest and are not well studied. A brief introduction
of using finite differences has been made in \cite{Riewe}, which can be regarded
as a predecessor to what we call here an Euler-like direct method.
A generalization of Leitmann's direct method can be found in \cite{AlD},
while \cite{Lotfi} discusses the Ritz direct method for optimal control problems
that can easily be reduced to a problem of the calculus of variations.

It is well-known that for most problems involving fractional operators,
such as fractional differential equations or fractional variational problems,
one cannot provide methods to compute the exact solutions analytically.
Therefore, numerical methods
are being developed to provide tools for solving such problems.
Using the Gr\"{u}nwald--Letnikov approach, it is convenient to approximate
the fractional differentiation operator, $D^\a$, by generalized finite differences.
In \cite{Podlubny} some problems have been solved by this approximation.
In \cite{Kai1} a predictor-corrector method is presented that converts an initial value problem
into an equivalent Volterra integral equation,
while \cite{Kumar} shows the use of numerical methods
to solve such integral equations. A good survey on numerical methods
for fractional differential equations can be found in \cite{Ford}.

A numerical scheme to solve fractional Lagrange problems
has been presented in \cite{AgrawalNum}. The method is based on approximating
the problem to a set of algebraic equations using some basis functions. 
See Chapter \ref{Survey} for details.
A more general approach can be found in \cite{Tricaud} that uses
the Oustaloup recursive approximation of the fractional derivative,
and reduces the problem to an integer-order (classical) optimal control problem.
A similar approach is presented in \cite{Jelicic}, using an expansion
formula for the left Riemann--Liouville fractional derivative
developed in \cite{Atan0,Atan2}, to establish a new scheme to solve fractional differential equations.

The scheme is based on an expansion formula
for the Riemann--Liouville fractional derivative.
Here we introduce a generalized version of this expansion, in Chapter~\ref{AppFracDer}, 
that results in an approximation, for left Riemann--Liouville derivative, of the form
\begin{equation}
\label{expanMomIntro}
\LDa x(t)\approx A(t-a)^{-\a}x(t)+B(t-a)^{1-\a}\dot{x}(t)
-\sum_{p=2}^N C(\a,p)(t-a)^{1-p-\a}V_p(t),
\end{equation}
with
\begin{equation*}
\left\{
\begin{array}{l}
\dot{V}_p(t)=(1-p)(t-a)^{p-2}x(t)\\
V_p(a)=0,
\end{array}
\right.
\end{equation*}
where $p=2,\ldots,N$, and the coefficients $A=A(\a,N)$, $B(\a,N)$ and $C(\a,p)$ 
are real numbers depending on $\a$ and $N$. The number $N$ is the order of approximation.
Together with a different expansion formula that has been used
to approximate the fractional Euler--Lagrange equation in \cite{Atan},
we perform an investigation of the advantages and disadvantages
of approximating fractional derivatives by these expansions.
The approximations transform fractional derivatives into finite sums
containing only derivatives of integer order \cite{PATFracDer}.

We show the efficiency of such approximations to evaluate fractional derivatives
of a given function in closed form. Moreover, we discuss the possibility
of evaluating fractional derivatives of discrete tabular data.
The application to fractional differential equations
is also developed through some concrete examples.

The same ideas are extended to fractional integrals in Chapter~\ref{AppFracInt}.
Fractional integrals appear in many different contexts,
\textrm{e.g.}, when dealing with fractional variational problems
or fractional optimal control \cite{VibString,APTIndInt,gasta1,Malinowska,Dorota}.
Here we obtain a simple and effective approximation for
fractional integrals. We obtain decomposition formulas for the left
and right fractional integrals of functions of class $C^n$  \cite{PATFracInt}.

In this PhD thesis we also consider the Hadamard
fractional integral and fractional derivative \cite{PATHad}. Although the definitions
go back to the works of Hadamard in 1892 \cite{Hadamard}, this type of operators
are not yet well studied and much exists to be done. For related works
on Hadamard fractional operators, see \cite{Butzer,Butzer2,Katugampola,Kilbas2,Kilbas3,Qian}.

An error analysis is given for each approximation whenever needed. These approximations are studied 
throughout some concrete examples. In each case we try to analyze problems
for which the analytic solution is available, so we can compare the exact and the approximate solutions. 
This approach gives us the ability of measuring the accuracy of each method. 
To this end, we need to measure how close we get to the exact solutions. 
We can use the $2$-norm for instance, and define
an error function $E[x(\cdot),\tilde{x}(\cdot)]$ by
\begin{equation}\label{ErrorDef}
E=\|x(\cdot)-\tilde{x}(\cdot)\|_2
=\left(\int_a^b [x(t)-\tilde{x}(t)]^2dt\right)^{\frac{1}{2}},
\end{equation}
where $x(\cdot)$ is defined on a certain interval $[a,b]$.

Before getting into the usage of these approximations for fractional variational problems, 
we introduce an Euler-like discrete method, and a discretization of the first variation 
to solve such problems in Chapter \ref{Direct}. The finite differences approximation 
for integer-order derivatives is  generalized to derivatives of arbitrary order
and gives rise to the Gr\"{u}nwald--Letnikov fractional derivative.
Given a grid on $[a,b]$ as $a=t_0,t_1,\ldots,t_n=b$, where $t_i=t_0+ih$ for some $h>0$, 
we approximate the left Riemann--Liouville  derivative as
\begin{equation*}
\LDai x(t_i)\simeq \frac{1}{h^\a} \sum_{k=0}^{i}\w x(t_i-kh),
\end{equation*}
where $\w=(-1)^k\binom{\a}{k}=\frac{\Gamma(k-\a)}{\Gamma(-\a)\Gamma(k+1)}$.
The method follows the same procedure as in the classical theory. Discretizing 
the functional by a quadrature rule, integer-order derivatives by finite differences  
and substituting fractional terms by corresponding generalized finite differences, 
results in a system of algebraic equations. Finally, one gets approximate values 
of state and control functions on a set of discrete points \cite{PATDisDir}.

A different direct approach for classical problems has been introduced in \cite{Gregory1,Gregory2}. 
It uses the fact that the first variation of a functional must vanish along an extremizer. That is,
if $x$ is an extremizer of a given variational functional $J$, the first variation of $J$
evaluated at $x$, along any variation $\eta$, must vanish. This means that
\begin{equation*}
J'[x,\eta]=\int_a^b\left[\frac{\partial L}{\partial x}(t,x(t),\LDa x(t))\eta(t)
+\frac{\partial L}{\partial \LDa x}(t,x(t),\LDa x(t))\LDa \eta(t)\right]\,dt=0.
\end{equation*}
With a discretization on time horizon and a quadrature for this integral, we obtain a system
of algebraic equations. The solution to this system gives an approximation 
to the original problem \cite{PATDisVar}.

Considering indirect methods in Chapter \ref{Indirect}, we transform the fractional 
variational problem into an integer-order problem. The job is done by substituting 
the fractional term by the corresponding approximation in which only integer-order 
derivatives exist. The resulting classic problem, which is considered as the 
approximated problem, can be solved by any available method in the literature. 
If we substitute the approximation \eqref{expanMomIntro} for the fractional term 
in \eqref{MainProblem}, the outcome is an integer-order constrained variational problem
\begin{eqnarray*}
J[x(\cdot)]&\approx &\int_a^b L\left(t, x(t),\frac{Ax(t)}{(t-a)^{\a}}
+\frac{B\dot{x}(t)}{(t-a)^{\a-1}}-\sum_{p=2}^N \frac{C(\a,p)V_p(t)}{(t-a)^{p+\a-1}}\right)dt\\
&=&\int_a^b L'\left(t, x(t), \dot{x}(t), V_2(t), \ldots,V_N(t)\right)dt\longrightarrow \min\\
\end{eqnarray*}
subject to
$$
\left\{
           \begin{array}{l}
           \dot{V}_p(t)=(1-p)(t-a)^{p-2}x(t)\\
           V_p(a)=0,
           \end{array}
           \right.
$$
with $p=2,\ldots,N$.
Once we have a tool to transform a fractional variational problem into an integer-order one, 
we can go further to study more complicated problems. As a first candidate, we study 
fractional optimal control problems with free final time in Chapter~\ref{FreeTime}. 
The problem is stated in the following way:
\begin{equation*}
J[x,u,T]=\int_a^T L(t,x(t),u(t))\,dt+\phi(T,x(T))\longrightarrow \min,
\end{equation*}
subject to the control system
\begin{equation*}
M \dot{x}(t) + N\;\LCD x(t) = f(t,x(t),u(t)),
\end{equation*}
and the initial boundary condition
\begin{equation*}
x(a)=x_a,
\end{equation*}
with $(M,N)\not=(0,0)$, and $x_a$ a fixed real number.
Our goal is to generalize previous works
on fractional optimal control problems
by considering the end time $T$ free and the dynamic control system
involving integer and fractional order derivatives.
First, we deduce necessary optimality conditions for this new problem with free end-point. 
Although this could be the beginning of the solution procedure, the lack of techniques 
to solve fractional differential equations prevent further progress. Another approach consists
in using the approximation methods mentioned above,
thereby converting the original problem into
a classical optimal control problem that
can be solved by standard computational techniques \cite{PATFree}.

In the 18th century, Euler considered the problem of optimizing functionals
depending not only on some unknown function $x$ and some derivative of $x$,
but also on an antiderivative of $x$ (see \cite{fraser}). Similar problems have been
recently investigated in \cite{Gregory}, where Lagrangians containing
higher-order derivatives and optimal control problems are considered.
More generally, it has been shown that the results of \cite{Gregory}
hold on an arbitrary time scale \cite{Nat}.
Here, in Chapter \ref{IndInt}, we study such problems within the framework of fractional calculus.
Minimize the cost functional
\begin{equation*}
J[x]=\int_a^b L(t,x(t),{^C_aD_t^\alpha}x(t),{_aI_x^\beta}x(t),z(t))dx,
\end{equation*}
where the variable $z$ is defined by
$$
z(t)=\int_a^t l(\t,x(\t),{^C_aD_\t^\alpha}x(\t),{_aI_\t^\beta}x(\t))d\t,
$$
subject to the boundary conditions
\begin{equation*}
x(a)=x_a \quad \mbox{and} \quad x(b)=x_b.
\end{equation*}
Our main contribution is an extension of the results presented
in \cite{AGRA1,Gregory} by considering Lagrangians containing an antiderivative,
that in turn depends on the unknown function, a fractional integral, and a Caputo
fractional derivative.

Transversality conditions are studied, where the variational functional $J$ depends also on the terminal 
time $T$, $J[x,T]$. We also consider isoperimetric problems with integral constraints of the same type. 
Fractional problems with holonomic constraints are considered and the situation when the Lagrangian 
depends on higher order Caputo derivatives is studied. Other aspects such as the Hamiltonian formalism, 
sufficient conditions of optimality under suitable convexity assumptions on the Lagrangian, 
and numerical results with illustrative examples are described in detail \cite{APTIndInt}.

\CP
\part{Synthesis}
\CP
\chapter{The calculus of variations and optimal control}
\label{Prelaminaries}

In this part we review the basic concepts that have essential role in the understanding 
of the second and main part of this dissertation. Starting with the notion of the calculus 
of variations, and without going into details, we recall the optimal control theory as well 
and point out its variational approach together with main concepts, definitions, and some 
important results from the classical theory. A brief historical introduction to the fractional 
calculus is given afterwards. At the same time, we introduce the theoretical framework 
of the whole work, fixing notations and nomenclature. At the end, the calculus of variations 
and optimal control problems involving fractional operators are discussed as fractional variational problems.


\section{The calculus of variations}

Many authors trace the origins of the calculus of variations back to the ancient times, 
the time of Dido, Queen of Carthage. \index{Queen Dido}Dido's problem had an intellectual 
nature. The question is to lie as much land as possible within a bull's hide. 
Queen Dido intelligently cut the hide into thin strips and no one knows if she encircled 
the land using the line she made off the strips. As it is well-known nowadays, thanks 
to the modern calculus of variations, the solution to Dido's problem is a circle \cite{Kirk}. 
Aristotle (384--322 B.C) expresses a common belief in his \emph{Physics} that nature follows 
the easiest path that requires the least amount of effort. This is the main idea behind 
many challenges to solve real-world problems \cite{Berdichevsky}.


\subsection{From light beams to the Brachistochrone problem}

Fermat believed that ``{\it nature operates by means and ways that are easiest and fastest}'' \cite{Goldstine}. 
Studying the analysis of refractions, he used Galileo's reasoning on falling objects and claimed that in this case nature 
does not take the shortest path, but the one which has the least traverse time. Although the solution to this problem 
does not use variational methods, it has an important role in the solution 
of the most critical problem and the birth of the calculus of variations.

Newton also considered the problem of motion in a resisting medium, which is indeed 
a shape optimization problem. This problem is a well-known and well-studied example 
in the theory of the calculus of variations and optimal control nowadays \cite{Plakhov,Silva,Gouveia}. 
Nevertheless, the original problem, posed by Newton, was solved by only using calculus.

In 1796-1797, John Bernoulli challenged the mathematical world 
to solve a problem that he called the Brachistochrone \index{Brachistochrone}problem:
\begin{quote}\it
If in a vertical plane two points A and B are given, then it is required to
specify the orbit AMB of the movable point M, along which it, starting
from A, and under the influence of its own weight, arrives at B in the
shortest possible time. So that those who are keen of such matters will
be tempted to solve this problem, is it good to know that it is not, as
it may seem, purely speculative and without practical use. Rather it
even appears, and this may be hard to believe, that it is very useful also
for other branches of science than mechanics. In order to avoid a hasty
conclusion, it should be remarked that the straight line is certainly the
line of shortest distance between A and B, but it is not the one which is
traveled in the shortest time. However, the curve AMB, which I shall
disclose if by the end of this year nobody else has found it, is very well
known among geometers \cite{Sussmann}.
\end{quote}
It is not a big surprise that several responses came to this challenge. 
It was the time of some of the most famous mathematical minds. Solutions 
from John and Jakob Bernoulli were published in May 1797 together 
with contributions by Tschrinhaus and l'Hopital and a note from Leibniz. 
Newton also published a solution without a proof. Later on, other variants 
of this problem have been discussed by James Bernoulli.


\subsection{Contemporary mathematical formulation}

Having a rich history, mostly dealing with physical problems, the calculus of variations 
is nowadays an outstanding field with a strong mathematical formulation. Roughly speaking, 
the calculus of variations is the optimization of functionals.
\begin{definition}[Functional]
\label{FuncDef}
A functional \index{Functional!} $J[\cdot]$ is a rule of correspondence, 
from a vector space into its underlying scalar field, that assigns 
to each function $x(\cdot)$ in a certain class $\Omega$ a unique number.
\end{definition}
The domain of a functional, $\Omega$ in Definition~\ref{FuncDef}, 
is a class of functions. Suppose that $x(\cdot)$ is a positive continuous 
function defined on the interval $[a,b]$. The area under 
$x(\cdot)$ can be defined as a functional, i.e.,
$$
J[x]=\int_a^b x(t)dt
$$
is a functional that assigns to each function the area under its curve. 
Just like functions, for each functional, $J[\cdot]$, 
one can define its increment, $\Delta J$.
\begin{definition}[See, e.g., \cite{Kirk}]
Let $x$ be a function and $\delta x$ be its \index{Variation}variation. 
Suppose also that the functional $J$ is defined for $x$ and $x+\delta x$. 
The increment of the functional\index{Functional! Increment of} 
$J$ with respect to $\delta x$ is
$$
\Delta J := J[x+\delta x]-J[x].
$$
\end{definition}
Using the notion of the increment of a functional we define its variation. 
The increment of $J$ can be written as
$$
\Delta J[x,\delta x]= \delta J[x,\delta x]+g(x,\delta x).\parallel \delta x \parallel,
$$
where $\delta J$ is linear in $\delta x$ and
$$
\lim_{\parallel \delta x \parallel\rightarrow 0}g(x,\delta x)=0.
$$
In this case the functional $J$ is said to be differentiable\index{Functional! Differentiable} 
on $x$ and $\delta J$ is its variation evaluated for the function $x$.

Now consider all functions in a class $\Omega$ for which the functional $J$  is defined. 
A function $x^*$ is a relative extremum of $J$ if its increment 
has the same sign for functions sufficiently close to $x^*$, i.e.,
$$
\exists \epsilon >0\, \forall x\in\Omega\,:\, \| x-x^*\| < \epsilon 
\Rightarrow J(x)-J(x^*)\geq 0 \,\vee\, J(x)-J(x^*)\leq 0.
$$
Note that for a relative minimum the increment is non-negative 
and non-positive for the relative maximum.

In this point, \textit{the fundamental theorem of the calculus of variations}\index{Fundamental! theorem} 
is used as a necessary condition to find a relative extreme point.
\begin{theorem}[See, e.g., \cite{Kirk}]
Let $J[x(\cdot)]$ be a differentiable functional defined in $\Omega$. Assume also that the members 
of $\Omega$ are not constrained by any boundaries. Then the variation of $J$, 
for all admissible variations of $x$, vanishes on an extremizer $x^*$.
\end{theorem}

Many problems in the calculus of variations are included in a general problem 
of optimizing a definite integral of the form
\begin{equation}
\label{Funct}
J[x(\cdot)]=\int_a^b L(t, x(t),  \dot{x}(t))dt,
\end{equation}
within a certain class, e.g., the class of continuously differentiable functions. 
In this formulation, the function $L$ is called the \index{Lagrangian}Lagrangian 
and supposed to be twice continuously differentiable. The points $a$ and $b$ 
are called boundaries, or the initial and terminal points, respectively. 
The optimization is usually interpreted as a minimization or a maximization. 
Since these two processes are related, that is, $\max G=-\min -G $, 
in a theoretical context we usually discuss the minimization problem.

The problem is to find a function $x(\cdot)$ with certain properties that gives a minimum 
value to the functional $J$. The function is usually assumed to pass through prescribed points, 
say $x(a)=x_a$ and/or $x(b)=x_b$. These are called the \index{Boundary conditions}boundary conditions. 
Depending on the boundary conditions a variational problem\index{Variational problem!} can be classified as:
\begin{description}
\item [Fixed end points:] the conditions at both end points are given,
$$x(a)=x_a,\qquad x(b)=x_b. 
$$
\item [Free terminal point:] the value of the function 
at the initial point is fixed and it is free at the terminal point,
$$
x(a)=x_a. 
$$
\item [Free initial point:] the value of the function 
at the terminal point is fixed and it is free at the initial point,
$$
x(b)=x_b. 
$$
\item [Free end points:]  both end points are free.
\item [Variable end points:] one point and/or the other 
is required to be on a certain set, e.g., a prescribed curve.
\end{description}
Sometimes the function $x(\cdot)$ is required to satisfy some constraints.\index{Isoperimetric problem}
Isoperimetric problems are a class of constrained variational problems 
for which the unknown function is needed to satisfy an integral of the form
$$
\int_a^b G(t, x(t),  \dot{x}(t))dt=K
$$
in which $K\in \mathbb{R}$ has a fixed given value.

A variational problem\index{Variational problem!} can also be subjected to a dynamic constraint. 
In this setting, the objective is to find an optimizer $x(\cdot)$ for the functional $J$ 
such that an ordinary differential equation is fulfilled, i.e.,
$$
\dot{x}(t)=f(t,x(t)),\quad t\in [a,b].
$$


\subsection{Solution methods}
\label{ClassisDir}

The aforementioned mathematical formulation allows us to derive optimality conditions 
for a large class of problems. The Euler--Lagrange necessary optimality condition 
is the key feature of the calculus of variations. This condition was introduced 
first by Euler in around 1744. Euler used a geometrical insight and 
finite differences\index{Finite! differences} approximations of derivatives 
to derive his necessary condition. Later, on 1755, Lagrange ended at the same result 
using analysis alone. Indeed Lagrange's work was the reason that Euler called 
this field the calculus of variations \cite{Goldstine}.


\subsubsection{Euler--Lagrange equation}
\index{Euler--Lagrange equation}

Let $x(\cdot)$ be a scalar function in $C^2[a,b]$, i.e., it has a continuous first 
and second derivatives on the fixed interval $[a,b]$. Suppose that the Lagrangian $L$ 
in \eqref{Funct} has continuous first and second partial derivatives with respect 
to all of its arguments. To find the extremizers of $J$ one can use the fundamental 
theorem of the calculus of variations: the first variation of the 
functional\index{Variation! of a functional} must vanish on the extremizer. 
By the increment of a functional we have
\begin{eqnarray*}
\Delta J&=& J[x+\delta x]-J[x]\\
        &=&\int_a^b L(t, x+\delta x,  
        \dot{x}+\delta \dot{x})dt-\int_a^b L(t, x,  \dot{x})dt.
\end{eqnarray*}
The first integrand is expanded in a Taylor series and the terms up to the first order 
in $\delta x$ and $\delta \dot{x}$ are kept. 
Finally, combining the integrals, gives the variation $\delta J$ as
$$
\delta J[x, \delta x]=\int_a^b \left( 
\left[ \frac{\partial L}{\partial x}(t,x,\dot{x}) \right]\delta x
+\left[ \frac{\partial L}{\partial \dot{x}}(t,x,\dot{x}) \right]\delta \dot{x}\right)dt.
$$
One can now \index{Integration by parts!}integrate by parts 
the term containing $\delta \dot{x}$ to obtain
$$
\delta J[x, \delta x]=\left[ \frac{\partial L}{\partial \dot{x}}(t,x,\dot{x}) \right]
\delta x\bigg|_a^b+\int_a^b \left( \left[ \frac{\partial L}{\partial x}(t,x,\dot{x}) \right]
-\frac{d}{dt}\left[ \frac{\partial L}{\partial \dot{x}}(t,x,\dot{x}) \right]\right)\delta x\,dt.
$$
Depending on how the boundary conditions are specified, we have different necessary conditions. 
In the very simple form when the problem is in the fixed end-points form, $\delta x(a)=\delta x(b)=0$, 
the terms outside the integral vanish. For the first variation to be vanished one has
$$
\int_a^b \left(  \frac{\partial L}{\partial x}(t,x,\dot{x})
-\frac{d}{dt}\left[ \frac{\partial L}{\partial \dot{x}}(t,x,\dot{x}) \right]\right)\delta x\,dt=0.
$$
According to the \textit{fundamental lemma of the calculus of variations}\index{Fundamental! lemma} 
(see, e.g., \cite{Brunt}), if a function $h(\cdot)$ is continuous and
$$
\int_a^b h(t)\eta(t)dt=0,
$$
for every function $\eta(\cdot)$ that is continuous in the interval $[a,b]$, then $h$ must be 
zero everywhere in the interval $[a,b]$. Therefore, the Euler--Lagrange necessary optimality 
condition\index{Euler--Lagrange equation}, that is an ordinary differential equation, reads to
$$
 \frac{\partial L}{\partial x}(t,x,\dot{x})
-\frac{d}{dt}\left[ \frac{\partial L}{\partial \dot{x}}(t,x,\dot{x}) \right]=0,
$$
when the boundary conditions are given at both end-points. For free end-point problems 
the so-called transversality conditions are added to the Euler--Lagrange equation 
(see, e.g., \cite{MalinowskaNatur}).
\begin{definition}
Solutions to the Euler--Lagrange equation are called extremals for $J$ defined by \eqref{Funct}.
\end{definition}

The necessary condition for optimality can also be derived using the classical method of perturbing 
the extremal and using the Gateaux derivative. The Gateaux differential 
or Gateaux derivative\index{Gateaux derivative} is a generalization of the concept of directional derivative:\\
$$
dF(x;\eta)=\lim_{\epsilon\rightarrow 0}\frac{F(x+\epsilon\eta)-F(x)}{\epsilon}
=\frac{d}{d\epsilon}F(x+\epsilon\eta)\Big|_{\epsilon=0}.
$$

Let $x^*(\cdot)\in C^2[a,b]$ be the extremal and $\eta(\cdot)\in C^2[a,b]$ 
be such that $\eta(a)=\eta(b)=0$. Then for sufficiently small values of $\epsilon$, 
form the family of curves $x^*(\cdot)+\epsilon\eta(\cdot)$. All of these curves reside 
in a neighborhood of $x^*$ and are admissible functions, i.e., they are in the class 
$\Omega$ and satisfy the boundary conditions. We now construct the function
\begin{equation}
\label{JEps}
j(\epsilon)=\int_a^b L(t, x^*(t)+\epsilon\eta(t),  \dot{x}^*(t)
+ \epsilon\dot{\eta}(t))dt, \quad -\delta<\epsilon <\delta.
\end{equation}
Due to the construction of the function $j(\epsilon)$,  the extremum is achieved 
for $\epsilon=0$. Therefore, it is necessary that the first derivative 
of $j(\epsilon)$ vanishes for $\epsilon=0$, i.e.,
$$
\frac{dj(\epsilon)}{d\epsilon}\bigg|_{\epsilon=0}=0.
$$
Differentiating \eqref{JEps} with respect to $\epsilon$, we get
$$
\frac{dj(\epsilon)}{d\epsilon}
= \int_a^b \left(\left[\frac{\partial L}{\partial x}(t,x^*+\epsilon\eta,\dot{x}^*+ \epsilon\dot{\eta}) \right]\eta
+\left[ \frac{\partial L}{\partial \dot{x}}(t,x^*+\epsilon\eta,\dot{x}^*+ \epsilon\dot{\eta}) \right]\dot{\eta}\right)dt.
$$
Setting $\epsilon=0$, we arrive at the formula
$$
\int_a^b \left(\left[\frac{\partial L}{\partial x}(t,x^*,\dot{x}^*) \right]\eta
+\left[ \frac{\partial L}{\partial \dot{x}}(t,x^*,\dot{x}^*) \right]\dot{\eta}\right)dt,
$$
which gives the Euler--Lagrange condition after making an integration by parts and applying the fundamental lemma.

The solution to the Euler--Lagrange equation, if exists, is an extremal for the variational problem. 
Except for simple problems, it is very difficult to solve such differential equations in a closed form. 
Therefore, numerical methods are employed for most practical purposes.


\subsubsection{Numerical methods}

A variational problem can be solved numerically in two different ways: by indirect 
or direct methods\index{Direct method}. Constructing the Euler--Lagrange equation 
and solving the resulting differential equation is known to be the indirect method\index{Indirect method}.

There are two main classes of direct methods. On one hand, we specify a discretization scheme 
by choosing a set of mesh points on the horizon of interest, say $a=t_0,t_1,\ldots,t_n=b$ 
for $[a,b]$. Then we use some approximations for derivatives in terms of the unknown function 
values at $t_i$ and using an appropriate quadrature, the problem is transformed to a finite 
dimensional optimization. This method is known as Euler's method\index{Euler method} in the literature. 
Regarding Figure~\ref{EulerMethod}, the solid line is the function that we are looking for, 
nevertheless, the method gives the polygonal dashed line as an approximate solution.
\begin{figure}[!htp]
\begin{center}
\setlength\fboxrule{0pt}
\fbox{\includegraphics[scale=1]{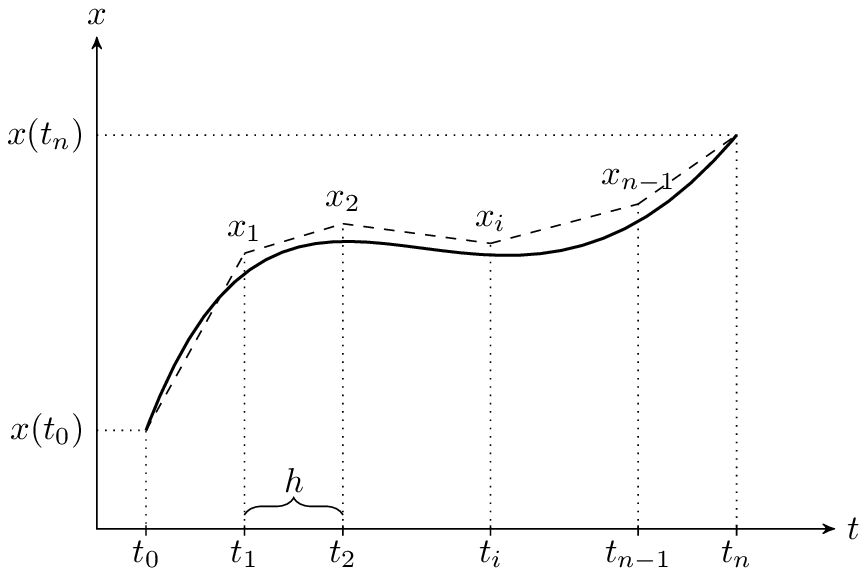}}
\end{center}
\caption{Euler's finite differences\index{Finite! differences} method.}\label{EulerMethod}
\end{figure}

On the other hand, there is the Ritz method\index{Ritz method}, that has an extension 
to functionals of several independent variables which is called Kantorovich's method\index{Kantorovich method}. 
We assume that the admissible  functions can be expanded in some kind of series, e.g. power or Fourier's series, of the form
$$
x(t)=\sum_{k=0}^\infty a_k \phi_k(t).
$$
Using a finite number of terms in the sum as an approximation, and some sort of quadrature again, 
the original problem can be transformed to an equivalent optimization problem for $a_k$, $k=0,1,\ldots,n$.


\section{Optimal control theory}

Optimal control theory is a well-studied subject. Many papers and textbooks present the field very well, 
see \cite{Kirk,Bryson,Pontryagin}. Nevertheless, we introduce some basic concepts without going into details. 
Our main purpose is to review the variational approach to optimal control theory and clarify its connection 
to the calculus of variations. This provides a background for our later investigations on fractional 
variational problems. The formulation is presented for vector functions, $\x=(x_1,\ldots,x_n)$, 
to emphasize the possibility of such functions. This is also valid, and is easy to adapt, for the calculus of variations.


\subsection{Mathematical formulation}

Mathematically speaking, the notion of control is highly connected to dynamical systems. 
A dynamical system is usually formulated using a system of ordinary or partial differential equations. 
In this thesis, dealing only with ordinary derivatives, we consider the dynamics as
$$
 \left\{
 \begin{array}{l}
 \dot{\x}= f(t,\x),\\
 \x(t_0)=\x_0,
 \end{array}\right.
$$
where $\x=(x_1,\ldots,x_n)$, the state of the system, is a vector function, 
$t_0\in \mathbb{R}$, $\x_0\in \mathbb{R}^n$ and $f:\mathbb{R}^{n+1}\to \mathbb{R}^{n}$ are given.

In order to affect the behavior of a system, e.g., a real-life physical system used in technology, 
one can introduce control parameters to the system. 
A controlled system also can be described by a system of ODEs,
$$
 \left\{
 \begin{array}{l}
 \dot{\x}= f(t,\x,\u),\\
 \x(t_0)=\x_0,
 \end{array}\right.
$$
in which $\u\in\Omega\subseteq\mathbb{R}^m$ is the control parameter or variable. 
The control parameters can also be time-varying, i.e., $\u=\u(t)$. In this case 
$f:\mathbb{R}^{n+m+1}\to \mathbb{R}^{n}$ is supposed to be continuous with respect 
to all of its arguments and continuously differentiable with respect to $\x=(x_1,\ldots,x_n)$.

In an optimal control problem\index{Optimal! control}, the main objective is to determine the control 
parameters in a way that certain optimality criteria are fulfilled. In this thesis we consider 
problems in which a functional of the form
$$
J[\x(\cdot),\u(\cdot)]=\int_a^b L(t, \x(t),\u(t)) dt
$$
should be optimized. Therefore, a typical optimal control problem is formulated as
\begin{align*}
 &J[\x(\cdot),\u(\cdot)]=\int_a^b L(t, \x(t),\u(t)) dt\longrightarrow \min\\
& s.t. \quad \left\{
 \begin{array}{l}
 \dot{\x}(t)= f(t,\x(t),\u(t))\\
 \x(t_0)=\x_0,
 \end{array}\right.
\end{align*}
where the state $\x$ and the control $\u$ are assumed to be unbounded.
This formulation can also be considered as a framework for both optimal control and the calculus of variations.  
Let $\dot{x}(t)=u(t)$. Then the optimization of \eqref{Funct} becomes
\begin{align*}
&J[x(\cdot)]=\int_a^b L(t, x(t),u(t)) dt\longrightarrow \min\\
& s.t. \quad\left\{
\begin{array}{l}
\dot{x}(t)= u(t)\\
x(t_0)=x_0,
\end{array}\right.
\end{align*}
that is an optimal control problem. On one hand, we can apply aforementioned direct methods\index{Direct method}. 
On the other hand, indirect methods\index{Indirect method} consist in using Lagrange multipliers in a variational 
approach to obtain the Euler--Lagrange equations. The dynamics is considered as a constraint for a variational 
problem\index{Variational problem!} and is added into the functional. The so-called augmented functional\index{Functional! Augmented} 
is then achieved, that is, the functional
\begin{equation*}
 J_a[\x(\cdot),\u(\cdot)]=\int_a^b \left[L(t, \x(t),\u(t))-\l(t)^T(\dot{\x}(t)-f(t,\x(t),\u(t)))\right]\, dt
\end{equation*}
is treated subject to the boundary conditions.


\subsection{Necessary optimality conditions}

Although the Euler--Lagrange  equations are derived by usual ways, 
e.g., Section~\ref{ClassisDir}, it is common and useful to define 
the Hamiltonian\index{Hamiltonian} function by
$$
H(t,\x,\u,\l)=L(t,\x,\u)+\l^T[f(t,\x,\u)].
$$
Then the necessary optimality conditions read as
\begin{equation*}
\left\{
\begin{array}{l}
\displaystyle\dot{\x}(t)= \frac{\partial H}{\partial\l}(t,\x(t),\u(t),\l(t)),\\[8pt]
\displaystyle\dot{\l}(t)= -\frac{\partial H}{\partial\x}(t,\x(t),\u(t),\l(t)),\\[8pt]
\displaystyle\hspace{5mm}0= \frac{\partial H}{\partial\u}(t,\x(t),\u(t),\l(t)).
\end{array}\right.
\end{equation*}
It is possible to consider a function $\phi(b,\x(b))$ in the objective functional, which makes the cost 
functional dependent on the time and state variables at the terminal point. This can be treated easily 
by some more calculations. Also one can discuss different end-points conditions 
in the same way as we did for the calculus of variations.


\subsection{Pontryagin's minimum principle}

Roughly speaking, unbounded control is an essential assumption to use variational methods freely 
and to obtain the resulting necessary optimality conditions. In contrast, if there is a bound on control, 
$\delta u$ can no more vary freely. Therefore, the fact that $\delta J$ must vanish on a extremal 
is of no use. Nevertheless, special variations can be defined and used to prove that 
for $u^*$ to be an extremal, it is necessary that
$$
H(t,x^*,u^*+\delta u,\lambda^*)\geq H(t,x^*,u^*,\lambda^*),
$$
for all admissible $\delta u$ \cite{Pontryagin}. That is, an optimal control $u^*$ 
is a global minimizer of the Hamiltonian for a control system. This condition 
is known as Pontryagin's minimum principle\index{Minimum Principle}. It is worthwhile 
to note that the condition that the partial derivative of the Hamiltonian with respect 
to control $u$ must vanish on an optimal control is a necessary condition for the minimum principle:
$$
\frac{\partial H}{\partial u}(t,x^*,u^*,\lambda^*)=0.
$$

\CP
\chapter{Fractional Calculus}
\label{FracCalculus}

In the early ages of modern differential calculus, right after the introduction of $\frac{d}{dt}$ 
for the first derivative, in a letter dated 1695, l'Hopital asked Leibniz the meaning 
of $\frac{d^\frac{1}{2}}{dt^\frac{1}{2}}$, the derivative of order $\frac{1}{2}$ \cite{Miller}. 
The appearance of $\frac{1}{2}$ as a fraction gave the name fractional calculus to the study 
of derivatives, and integrals, of any order, real or complex.

There are several different approaches and definitions in fractional calculus for derivatives 
and integrals of arbitrary order. Here we give a historical progress of the theory of fractional 
calculus that includes all we need throughout this thesis. We mostly follow the notation used 
in the books \cite{Kilbas,Samko}. Before getting into the details of the theory, we briefly outline 
the definitions of some special functions that are used in the definitions of fractional derivatives 
and integrals, or appear in some manipulation, e.g., solving fractional differential and integral equations.


\section{Special functions}

Although there are many special functions that appear in fractional calculus, in this thesis only 
a few of them are encountered. The following definitions are introduced together with some properties.
\begin{definition}[Gamma function]
\index{Gamma function}
The Euler integral of the second kind
$$
\Gamma(z)=\int_0^\infty t^{z-1}e^{-t}dt, \qquad \Re(z)>0,
$$
is called the gamma function.
\end{definition}
The gamma function has an important property, $\Gamma(z+1)=z\Gamma(z)$ and hence $\Gamma(z)=z!$ 
for $z\in\N$, that allows us to extend the notion of factorial to real numbers.
For further properties of this special function we refer the reader to \cite{Andrews}.

\begin{definition}[Mittag--Leffler function]\index{Mittag--Leffler function}
Let $\a > 0$. The function $E_\a$ defined by
$$
E_\a(z)=\sum_{j=0}^\infty \frac{z^j}{\Gamma(\a j+1)},
$$
whenever the series converges, is called the one parameter Mittag--Leffler function.
The two-parameter Mittag--Leffler function with parameters $\a, \beta >0$ is defined by
$$
E_{\a,\beta}(z)=\sum_{j=0}^\infty \frac{z^j}{\Gamma(\a j+\beta)}.
$$
\end{definition}
The Mittag--Leffler function is a generalization of exponential series 
and coincides with the series expansion of $e^z$ for $\a=1$.


\section{A historical review}

Attempting to answer the question of l'Hopital, Leibniz tried to explain 
the possibility of the derivative of order $\frac{1}{2}$. He also quoted that 
``{\it this will lead to a paradox with very useful consequences}''. During 
the next century the question was raised again by Euler (1738), expressing 
an interest to the calculation of fractional order derivatives.

The nineteenth century has witnessed much effort in the field. In 1812, Laplace 
discussed non-integer derivatives of some functions that are representable by integrals. 
Later, in 1819, Lacriox generalized  $\frac{d^n}{dt^n}t^n$ to $\frac{d^\frac{1}{2}}{dt^\frac{1}{2}}t$. 
The first challenge of making a definition for arbitrary order derivatives comes from Fourier in 1822, with
$$
\frac{d^\a}{dt^\a}x(t)=\frac{1}{2\pi}\int_{-\infty}^\infty x(\t)d\t
\int_{-\infty}^\infty p^{\a}\cos\left[p(t-\t)+\frac{1}{2}\a\pi\right]dp.
$$
He derived this definition from the integral representation of a function $x(\cdot)$. 
An important step was taken by Abel in 1823. Solving the \textit{Tautochrone} problem, 
he worked with integral equations of the form
$$
\int_0^t (t-\t)^{-\a}x(\t)d\t=k.
$$
Apart from a multiplicative factor, the left hand side of this equation resembles the modern 
definitions of fractional derivatives. Almost ten years later the first definitions 
of fractional operators appeared in the works of Liouville (1832), and has been contributed 
by many other mathematicians like Peacock and Kelland (1839), and  Gregory (1841). Finally, 
starting from 1847, Riemann dedicated some works on fractional integrals that led to the 
introduction of Riemann--Liouville fractional derivatives and integrals by Sonin in 1869.

\begin{definition}[Riemann--Liouville fractional integral]
\index{Riemann--Liouville fractional integral! Left}
\index{Riemann--Liouville fractional integral! Right}
Let $x(\cdot)$ be an integrable function in $[a,b]$ and $\a>0$.
\begin{itemize}
\item The left Riemann--Liouville fractional integral of order $\alpha$ is given by
$$
\LI x(t)=\frac{1}{\Gamma(\alpha)}
\int_a^t (t-\tau)^{\alpha-1}x(\tau)d\tau,
\quad t\in [a,b].
$$
\item The right Riemann--Liouville fractional integral
of order $\alpha$ is given by
$$
\RI x(t)=\frac{1}{\Gamma(\alpha)}\int_t^b (\tau-t)^{\a-1}x(\tau)d\tau,
\quad t\in [a,b].
$$
\end{itemize}
\end{definition}

\begin{definition}[Riemann--Liouville fractional derivative]
\index{Riemann--Liouville fractional derivative! Left}
\index{Riemann--Liouville fractional derivative! Right}
Let $x(\cdot)$ be an absolutely continuous function in $[a,b]$, $\a>0$, and $n=[\a]+1$.
\begin{itemize}
\item The left Riemann--Liouville fractional derivative of order $\alpha$ is given by
\begin{equation*}
\LD x(t)=\frac{1}{\Gamma(n-\a)}\left(\frac{d}{dt}\right)^n
\int_a^t (t-\t)^{n-1-\a}x(\t)d\t, \hspace{1cm}t\in [a,b].
\end{equation*}
\item
The right Riemann--Liouville fractional derivative of order $\a$ is given by
$$
\RD x(t)=\frac{1}{\Gamma(n-\a)}\left (-\frac{d}{dt}\right)^n
\int_t^b (\t-t)^{n-1-\a}x(\t)d\t, \hspace{1cm}t\in [a,b].
$$
\end{itemize}
\end{definition}
These definitions are easily derived from generalizing the Cauchy's 
$n$-fold integral formula. Substituting $n$ by $\a$ in
\begin{eqnarray*}
I^n x(t)&=& \int_0^t\int_0^{t_{n-1}}\ldots\int_0^{t_{1}}x(t_0)dt_0dt_1\ldots dt_{n-1}\\
        &=& \frac{1}{(n-1)!}\int_0^t(t-\t)^{n-1}x(\t)d\t,
\end{eqnarray*}
and using the gamma function, $\Gamma(n)=(n-1)!$, leads to
\begin{eqnarray*}
I^\a x(t)= \frac{1}{\Gamma(\a)}\int_0^t(t-\t)^{\a-1}x(\t)d\t.
\end{eqnarray*}
For the derivative, one has $D^\a x(t)=D^n I^{n-\a} x(t)$.

The next important definition is a generalization of the definition of higher 
order derivatives and appeared in the works of Gr\"{u}nwald (1867) and Letnikov (1868).

In classical theory, given a derivative of certain order $x^{(n)}$, 
there is a finite difference\index{Finite! differences} approximation of the form
\begin{equation*}
x^{(n)}(t)=\lim_{h\rightarrow 0^+} \frac{1}{h^n} \sum_{k=0}^n(-1)^k\binom{n}{k}x(t-kh),
\end{equation*}
where $\binom{n}{k}$  is the binomial coefficient\index{Binomial! coefficient}, that is,
\begin{equation*}
\binom{n}{k}=\frac{n(n-1)(n-2)\cdots (n-k+1)}{k!},\quad n,k \in \mathbb{N}.
\end{equation*}
The Gr\"{u}nwald--Letnikov definition of fractional derivative 
is a generalization of this formula to derivatives of arbitrary order.

\begin{definition}[Gr\"{u}nwald--Letnikov derivative]
\index{Gr\"{u}nwald--Letnikov fractional derivative! Left}
\index{Gr\"{u}nwald--Letnikov fractional derivative! Right}
Let $0<\a<1$ and $\binom{\a}{k}$ be the generalization 
of binomial coefficients to real numbers, that is,
$$
\binom{\a}{k}=\frac{\Gamma(\a+1)}{\Gamma(k+1)\Gamma(\a-k+1)},
$$
where $k$ and $\a$ can be any integer, real or complex, 
except that $\a\notin\{-1,-2,-3,\ldots\}$.
\begin{itemize}
\item The left Gr\"{u}nwald--Letnikov fractional derivative is defined as
\begin{equation}\label{LGLdef}
\GLa x(t)=\lim_{h\rightarrow 0^+} \frac{1}{h^\a} \sum_{k=0}^\infty(-1)^k\binom{\a}{k}x(t-kh).
\end{equation}
\item The right Gr\"{u}nwald--Letnikov derivative is
\begin{equation}\label{RGLdef}
\GLb x(t)=\lim_{h\rightarrow 0^+} \frac{1}{h^\a} \sum_{k=0}^\infty(-1)^k\binom{\a}{k}x(t+kh).
\end{equation}
\end{itemize}
\end{definition}

The series in \eqref{LGLdef} and \eqref{RGLdef}, the Gr\"{u}nwald--Letnikov definitions, 
converge absolutely and uniformly if $x(\cdot)$ is bounded. The infinite sums, 
backward differences for left and forward differences for right derivatives 
in the Gr\"{u}nwald--Letnikov definitions of fractional derivatives, reveal 
that the arbitrary order derivative of a function at a time $t$ depends 
on all values of that function in $(-\infty,t]$ and $[t,\infty)$ for left 
and right derivatives, respectively. This is due to the non-local property 
of fractional derivatives.
\begin{remark}
Equations \eqref{LGLdef} and \eqref{RGLdef} need to be consistent in closed 
time intervals and we need the values of $x(t)$ outside the interval $[a,b]$. 
To overcome this difficulty, we can take
\begin{equation}
x^*(t)=\left\{
\begin{array}{ll}
x(t)& t\in [a,b],\\\nonumber
0 & t\notin [a,b].\nonumber
\end{array}\right.\nonumber
\end{equation}
Then we assume $\GLa x(t)=\GLa x^*(t)$ and $\GLb x(t)=\GLb x^*(t)$ for $t\in [a,b]$.
\end{remark}

These definitions coincide with the definitions of Riemann--Liouville derivatives.
\begin{proposition}[See \cite{Podlubny}]
Let $0< \a<n$, $n\in \mathbb{N}$ and $x(\cdot)\in C^{n-1}[a,b]$. Suppose that $x^{(n)}(\cdot)$ 
is integrable on $[a,b]$. Then for every $\a$ the  Riemann--Liouville derivative exists 
and coincides with the Gr\"{u}nwald--Letnikov derivative:
\begin{eqnarray}
\LDa x(t)&=&\sum_{i=0}^{n-1}\frac{x^{(i)}(a)(t-a)^{i-\a}}{\Gamma(1+i-\a)}+
\frac{1}{\Gamma(n-\a)}\int_a^t(t-\t)^{n-1-\a}x^{(n)}(\t)d\t\nonumber\\
&=&\GLa x(t).\nonumber
\end{eqnarray}
\end{proposition}

Another type of fractional operators, that is investigated in this thesis, 
is the Hadamard type operators introduced in 1892.

\begin{definition}[Hadamard fractional integral]
\index{Hadamard fractional integral! Left}
\index{Hadamard fractional integral! Right}
Let $a,b$ be two real numbers with $0<a<b$ and $x:[a,b]\to\mathbb R$.
\begin{itemize}
\item The left Hadamard fractional integral of order $\a>0$ is defined by
$$
\LHI x(t)=\frac{1}{\Gamma(\alpha)}\int_a^t 
\left(\ln\frac{t}{\tau}\right)^{\alpha-1}\frac{x(\tau)}{\tau}d\tau, \hspace{1cm}t\in (a,b).
$$
\item The right Hadamard fractional integral of order $\a>0$ is defined by
$$
\RHI x(t)=\frac{1}{\Gamma(\alpha)}\int_t^b 
\left(\ln\frac{\tau}{t}\right)^{\a-1}\frac{x(\tau)}{\tau}d\tau, \hspace{1cm}t\in (a,b).
$$
\end{itemize}
\end{definition}

When $\a=m$ is an integer, these fractional integrals are \textit{m-fold} integrals:
$$
{_a\mathcal{I}_t^m} x(t)=\int_a^t \frac{d\tau_1}{\tau_1}\int_a^{\tau_1} 
\frac{d\tau_2}{\tau_2}\ldots \int_a^{\tau_{m-1}} \frac{x(\tau_m)}{\tau_m}d\tau_m
$$
and
$$
{_t\mathcal{I}_b^m}x(t)=\int_t^b \frac{d\tau_1}{\tau_1}\int_{\tau_1}^b 
\frac{d\tau_2}{\tau_2}\ldots \int_{\tau_{m-1}}^b \frac{x(\tau_m)}{\tau_m}d\tau_m.
$$

\begin{definition}[Hadamard fractional derivative] 
\index{Hadamard fractional derivative! Left}
\index{Hadamard fractional derivative! Right}
For fractional derivatives, we also consider the left and right derivatives. 
For $\a>0$ and $n=[\a]+1$.
\begin{itemize}
\item The left Hadamard fractional derivative of order $\a$  is defined by
$$
\LHD x(t)=\left(t\frac{d}{dt}\right)^n\frac{1}{\Gamma(n-\alpha)}\int_a^t 
\left(\ln\frac{t}{\tau}\right)^{n-\alpha-1}\frac{x(\tau)}{\tau}d\tau, \hspace{1cm}t\in (a,b).
$$
\item The right Hadamard fractional derivative of order $\a$  is defined by
$$
\RHD x(t)=\left(-t\frac{d}{dt}\right)^n\frac{1}{\Gamma(n-\alpha)}\int_t^b 
\left(\ln\frac{\tau}{t}\right)^{n-\a-1}\frac{x(\tau)}{\tau}d\tau, \hspace{1cm}t\in (a,b).
$$
\end{itemize}
\end{definition}
When $\a=m$ is an integer, we have
$$
{_a\mathcal{D}_t^m} x(t)=\left(t\frac{d}{dt}\right)^m x(t)
\mbox{ and } {_t\mathcal{D}_b^m}x(t)=\left(-t\frac{d}{dt}\right)^m x(t).
$$

Finally, we recall another definition, the Caputo derivatives, that are believed 
to be more applicable in practical fields such as engineering and physics.
In spite of the success of Riemann--Liouville approach in theory, some difficulties 
arise in practice where initial conditions need to be treated for instance in fractional 
differential equations. Such conditions for Riemann--Liouville case have no clear 
physical interpretations \cite{Podlubny}. The following definition was proposed 
by Caputo in 1967. Caputo's fractional derivatives are, however, 
related to Riemann--Liouville definitions.

\begin{definition}[Caputo's fractional derivatives]
\index{Caputo fractional derivative! Left}
\index{Caputo fractional derivative! Right}
Let $x(\cdot)\in AC[a,b]$ and $\a>0$ with $n=[\a]+1$.
\begin{itemize}
\item The left Caputo fractional derivative of order $\a$ is given by
\begin{equation*}
\LDC x(t)=\frac{1}{\Gamma(n-\a)}\int_a^t 
(t-\t)^{n-1-\alpha}x^{(n)}(\t)d\t, \hspace{1cm}t\in [a,b].
\end{equation*}
\item
The right Caputo fractional derivative of order $\a$ is given by
$$
\RDC x(t)=\frac{(-1)^n}{\Gamma(n-\a)} 
\int_t^b (\t-t)^{n-1-\a}x^{(n)}(\t)d\t, \hspace{1cm}t\in [a,b].
$$
\end{itemize}
\end{definition}

These fractional integrals and derivatives define a rich calculus.
For details see the books \cite{Kilbas,Miller,Samko}. Here we just recall
some useful properties for our purposes.


\section{The relation between Riemann--Liouville and Caputo derivatives}

For $\a>0$ and $n=[\a]+1$, the Riemann--Liouville and Caputo 
derivatives are related by the following formulas:
\begin{equation*}
\LD x(t)={_a^CD_t^\a}x(t) 
+\sum_{k=0}^{n-1}\frac{x^{(k)}(a)}{\Gamma (k+1-\a)}(t-a)^{k-\a}
\end{equation*}
and
\begin{equation*}
\RD x(t)={_t^CD_b^\a}x(t) 
+\sum_{k=0}^{n-1}\frac{x^{(k)}(b)}{\Gamma (k+1-\a)}(b-t)^{k-\a}.
\end{equation*}
In some cases the two derivatives coincide,
$$
_aD_t^\alpha x={_a^CD_t^\alpha} x, 
\hspace{1cm}when~~x^{(k)}(a)=0,\quad k=0,\ldots,n-1,
$$
$$
_tD_b^\beta x={_t^CD_b^\beta} x, 
\hspace{1cm}when~~x^{(k)}(b)=0,\quad k=0,\ldots,n-1.
$$


\section{Integration by parts}
\index{Integration by parts! Fractional}

Formulas of integration by parts have an important role in the proof 
of Euler--Lagrange necessary optimality conditions.

\begin{lemma}[cf. \cite{Kilbas}]
Let $\alpha > 0,~p,~q~\geq1$ and $\frac{1}{p}+\frac{1}{q}\leq 1
+\alpha~$($p\neq 1~and~q\neq 1$ in the case where $\frac{1}{p}+\frac{1}{q}= 1+\alpha$).

\noindent (i) If $\varphi \in L_p(a,b)~and~\psi\in L_q(a,b)$, then
$$
\int_a^b \varphi (t){_aI_t^\alpha}\psi(t)dt
=\int_a^b \psi (t){_tI_b^\alpha}\varphi(t)dt.
$$
(ii) If $g \in {_tI_b^\alpha}(L_p)~and~f\in {_aI_t^\alpha}(L_q)$, then
$$
\int_a^b g(t)\LD f(t)dt
=\int_a^b f(t){_tD_b^\alpha}g(t)dt,
$$
where the space of functions ${_tI_b^\alpha}(L_p)$ and ${_aI_t^\alpha}(L_q)$ 
are defined for $\alpha > 0$ and $1\leq p\leq\infty$ by
$$
{_aI_t^\alpha}(L_p):=\{ f: f={_aI_t^\alpha}\varphi,~~ \varphi\in L_p(a,b)\}
$$
and
$$
{_tI_b^\alpha}(L_p):=\{ f: f={_tI_b^\alpha}\varphi,~~ \varphi\in L_p(a,b)\}.
$$
\end{lemma}

For Caputo fractional derivatives,
$$
\displaystyle\int_{a}^{b} g(t)\cdot {_a^C D_t^\alpha}f(t)dt
=\displaystyle\int_a^b f(t)\cdot {_t D_b^\alpha} g(t)dt+\sum_{j=0}^{n-1}
\left[{_tD_b^{\alpha+j-n}}g(t) \cdot f^{(n-1-j)}(t)\right]_a^b
$$
(see, \textrm{e.g.}, \cite[Eq. (16)]{MR2345467}).
In particular, for $\alpha\in(0,1)$ one has
\begin{equation}
\label{eq:frac:IBP}
\int_{a}^{b}g(t)\cdot {_a^C D_t^\alpha}f(t)dt
=\int_a^b f(t)\cdot {_t D_b^\alpha} g(t)dt
+\left[{_tI_b^{1-\alpha}}g(t) \cdot f(t)\right]_a^b.
\end{equation}
When $\alpha \rightarrow 1$, ${_a^C D_t^\alpha} = \frac{d}{dt}$,
${_t D_b^\alpha} = - \frac{d}{dt}$, ${_tI_b^{1-\alpha}}$
is the identity operator, and \eqref{eq:frac:IBP} gives
the classical formula of integration by parts.


\CP
\chapter{Fractional variational problems}
\label{FracVariational}

A fractional problem of the calculus of variations and optimal control consists 
in the study of an optimization problem, in which the objective functional 
or constraints depend on derivatives and/or integrals of arbitrary, 
real or complex, orders. This is a generalization of the classical theory, 
where derivatives and integrals can only appear in integer orders.


\section{Fractional calculus of variations and optimal control}

Many generalizations of the classical calculus of variations and optimal control 
have been made, to extend the theory to the field of fractional variational 
and fractional optimal control. A simple fractional variational 
problem\index{Variational problem! Fractional} consists in finding 
a function $x(\cdot)$ that minimizes the functional
\begin{equation}
\label{Functional}
J[x(\cdot)]=\int_a^b L(t, x(t), \LD x(t))dt,
\end{equation}
where $\LD$ is the left Riemann--Liouville fractional derivative\index{Riemann--Liouville fractional derivative! Left}. 
Typically, some boundary conditions are prescribed as $x(a)=x_a$ and/or $x(b)=x_b$. Classical techniques have been adopted 
to solve such problems. The Euler--Lagrange equation for a Lagrangian of the form $L(t, x(t), \LD x(t),\RD x(t))$ 
has been derived in \cite{AgrawalForm}. Many variants of necessary conditions of optimality have been studied. 
A generalization of the problem to include fractional integrals, i.e., $L=L(t,{_aI_t^{1-\a}}x(t),\LD x(t))$, 
the transversality conditions of fractional variational problems and many other aspects can be found in the literature 
of recent years. See \cite{Almeida1,AlD,Atan,Riewe,Riewe1} and references therein. Furthermore, it has been shown that 
a variational problem with fractional derivatives can be reduced to a classical problem using an approximation 
of the Riemann--Liouville fractional derivatives in terms of a finite sum, 
where only derivatives of integer order are present \cite{Atan}.

On the other hand, fractional optimal control problems\index{Fractional! optimal control} usually appear in the form of
$$
J[x(\cdot)]=\int_a^b L(t, x(t), u(t))dt\longrightarrow \min\nonumber\\
$$
subject to
$$
 \left\{
\begin{array}{l}
\LD x(t)=f(t,x(t),u(t))\\
x(a)=x_a,~x(b)=x_b,
\end{array}
\right.
$$
where an optimal control $u(\cdot)$ together with an optimal trajectory $x(\cdot)$ are required 
to follow a fractional dynamics\index{Fractional! dynamics} and, at the same time, optimize 
an objective functional. Again, classical techniques are generalized to derive necessary optimality conditions. 
Euler--Lagrange equations have been introduced, e.g., in \cite{AgrawalNum}. A Hamiltonian formalism 
for fractional optimal control problems can be found in \cite{Ozlem} that exactly follows 
the same procedure of the regular optimal control theory, i.e., those with only integer-order derivatives.


\section{A general formulation}

The appearance of fractional terms of different types, derivatives and integrals, 
and the fact that there are several definitions for such operators, makes it difficult 
to present a typical problem to represent all possibilities. Nevertheless, 
one can consider the optimization of functionals of the form
\begin{equation}\label{GenForm}
J[\x(\cdot)]=\int_a^b L(t, \x(t), D^\ba \x(t))dt
\end{equation}
that depends on the fractional derivative, $D^\ba$, in which $\x(\cdot)=(x_1(\cdot),\ldots,x_n(\cdot))$ 
is a vector function, $\ba=(\a_1,\ldots,\a_n)$ and $\a_i$, $i=1,\ldots,n$ are arbitrary real numbers. 
The problem can be with or without boundary conditions. Many settings of fractional variational 
and optimal control problems can be transformed into the optimization of \eqref{GenForm}. 
Constraints that usually appear in the calculus of variations and are always present in optimal control 
problems can be included in the functional using Lagrange multipliers\index{Lagrange multipliers}. 
More precisely, in presence of dynamic constraints as fractional differential equations, we assume 
that it is possible to transform such equations to a vector fractional 
differential equation\index{Fractional! differential equation} of the form
\begin{equation*}
D^\ba \x(t)=f(t, \x(t)).
\end{equation*}
In this stage, we introduce a new variable $\l=(\lambda_1,\lambda_2,\ldots,\lambda_n)$ 
and consider the optimization of
\begin{equation*}
J[\x(\cdot)]=\int_a^b \left[L(t, \x(t), D^\ba \x(t))-\l(t)D^\ba \x(t)+\l(t)f(t, \x(t))\right]dt
\end{equation*}

When the problem depends on fractional integrals, $I^\a$, a new variable can be defined 
as $z(t)=I^\a x(t)$. Recall that $D^\a I^\a x=x$, see \cite{Kilbas}. The equation
\begin{equation*}
D^\a z(t)=D^\a I^\a x(t)=x(t),
\end{equation*}
can be regarded as an extra constraint to be added to the original problem. However, 
problems containing fractional integrals can be treated directly to avoid the complexity 
of adding an extra variable to the original problem. 
Interested readers are addressed to \cite{AlD,PATFracInt}.

Throughout this thesis, by a fractional variational problem\index{Variational problem! Fractional}, 
we mainly consider the following one variable problem with given boundary conditions:
$$
J[x(\cdot)]=\int_a^b L(t, x(t), D^\a x(t))dt\longrightarrow \min
$$
subject to
$$
 \left\{
\begin{array}{l}
x(a)=x_a\\x(b)=x_b.
\end{array}
\right.
$$
In this setting, $D^\a$ can de replaced by any fractional operator that is available in literature, 
say, Riemann--Liouville, Caputo, Gr\"{u}nwald--Letnikov, Hadamard and so forth. The inclusion of a constraint 
is done by Lagrange multipliers. The transition from this problem to the general one, 
equation \eqref{GenForm}, is straightforward and is not discussed here.


\section{Fractional Euler--Lagrange equations}

Many generalizations to the classical calculus of variations have been made in recent years, 
to extend the theory to the field of fractional variational problems. 
As an example, consider the following minimizing problem:
\begin{eqnarray*}
J[x(\cdot)]&=&\int_a^b L(t, x(t), \LD x(t))dt\longrightarrow \min\nonumber\\
&s.t.&~x(a)=x_a,~x(b)=x_b,
\end{eqnarray*}
where $x(\cdot)\in AC[a,b]$ and $L$ is a smooth function of $t$.

Using the classical methods we can obtain the following theorem as the necessary 
optimality condition for the fractional calculus of variations.
\begin{theorem}[cf. \cite{AgrawalForm}]
\label{agr}
Let $J[x(\cdot)]$ be a functional of the form
\begin{eqnarray}\label{obj}
J[x(\cdot)]=\int_a^b L(t, x(t), \LD x(t))dt,
\end{eqnarray}
defined on the set of functions $x(\cdot)$ which have continuous left and right 
Riemann--Liouville derivatives of order $\alpha$ in $[a, b]$, and satisfy the boundary 
conditions $x(a)=x_a$ and $x(b)=x_b$. A necessary condition for $J[x(\cdot)]$ to have 
an extremum for a function $x(\cdot)$ is that $x(\cdot)$ satisfy the following Euler--Lagrange 
equation\index{Euler--Lagrange equation! Fractional}:
$$
\frac{\partial L}{\partial x}+{_tD_b^\alpha}\left(\frac{\partial L}{\partial \LD x}\right)=0.
$$
\end{theorem}

\begin{proof}
Assume that $x^*(\cdot)$ is the optimal solution\index{Optimal! solution}. 
Let $\epsilon\in \mathbb{R}$ and define a family of functions
$$
x(t)=x^*(t)+\epsilon \eta (t)
$$
which satisfy the boundary conditions\index{Boundary conditions}. 
So one should have $\eta (a)=\eta (b)=0$.\\
Since $_aD_t^\alpha$ is a linear operator, it follows that
$$
\LD x(t)=\LD x^*(t)+\epsilon\; \LD\eta (t).
$$
Substituting in \eqref{obj} we find that for each $\eta(\cdot)$
$$
j(\epsilon)=\int_a^bL(t, x^*(t)+\epsilon\eta(t), \LD x^*(t)+\epsilon\;\LD\eta(t))dt
$$
is a function of $\epsilon$ only. Note that $j(\epsilon)$ has an extremum at $\epsilon=0$. 
Differentiating with respect to $\epsilon$ (the Gateaux derivative)\index{Gateaux derivative} we conclude that
$$
\frac{dj}{d\epsilon}\Big|_{\epsilon=0}=\int_a^b \left(\frac{\partial L}{\partial x}\eta
+\frac{\partial L}{\partial \LD x}\LD\eta\right)dt.
$$
The above equation is also called the variation of $J[x(\cdot)]$ along 
$\eta(\cdot)$\index{Variation! of a functional}. For $j(\epsilon)$ 
to have an extremum it is necessary that $\frac{dj}{d\epsilon}\Big|_{\epsilon=0}=0$, 
and this should be true for any admissible $\eta(\cdot)$. Thus,
$$
\int_a^b \left(\frac{\partial L}{\partial x}\eta
+\frac{\partial L}{\partial \LD x}\LD\eta\right)dt=0
$$
for all $\eta(\cdot)$ admissible. Using the formula of integration 
by parts\index{Integration by parts! Fractional} 
on the second and third terms one has
$$
\int_a^b \left[\frac{\partial L}{\partial x}
+{_tD_b^\alpha}\left(\frac{\partial L}{\partial \LD x}\right)\right]\eta\, dt=0
$$
for all $\eta(\cdot)$ admissible. The result follows immediately 
by the fundamental lemma\index{Fundamental! lemma} of the calculus of variations, 
since $\eta$ is arbitrary and $L$ is continuous.
\end{proof}

Generalizing Theorem~\ref{agr} for the case when $L$ depends on several functions, 
i.e., $\x(t)=(x_1(t),\ldots,x_n(t))$ or it includes derivatives of different orders, i.e.,
$$
D^\a \x(t)=(D^{\a_1}x_1(t),\ldots,D^{\a_n}x_n(t)),
$$
is straightforward.


\section{Solution methods}

There are two main approaches to solve variational, including optimal control, problems.\index{Variational problem!} 
On one hand, there are the direct methods.\index{Direct method} In a branch of direct methods, 
the problem is discretized on the interested time interval using discrete values of the unknown function, 
finite differences\index{Finite! differences} for derivatives and finally a quadrature rule for the integral. 
This procedure transforms the variational problem, a dynamic optimization problem, to a static multi-variable 
optimization problem. Better accuracies are achieved by refining the underlying mesh size. 
Another class of direct methods uses function approximation through a linear combination of the elements 
of a certain basis, e.g., power series. The problem is then transformed to the determination 
of the unknown coefficients. To get better results in this sense, is the matter 
of using more adequate or higher order function approximations.

On the other hand, there are the indirect methods.\index{Indirect method} Those transform 
a variational problem to an equivalent differential equation by applying some necessary 
optimality conditions. Euler--Lagrange equations and Pontryagin's minimum principle\index{Minimum Principle} 
are used in this context to make the transformation process. Once we solve the resulting differential equation, 
an extremum for the original problem is reached. Therefore, to reach better results using indirect methods, 
one has to employ powerful integrators. It is worth, however, to mention here that numerical methods 
are usually used to solve practical problems.

These two classes of methods have been generalized to cover fractional problems. 
That is the essential subject of this PhD thesis.

\CP
\Chapter{State of the art}{A short survey on the numerical methods 
for solving fractional variational problems}
\label{Survey}

As it is mentioned earlier, the fractional calculus of variations started with the works 
of Riewe, \cite{Riewe,Riewe1}, in the last years of 1990s. Later, the notion of fractional 
optimal control appeared in the works of Agrawal \cite{AgrawalNum} and Frederico 
and Torres \cite{gasta1}. It is not surprising that the numerical achievements 
in these fields is at an early stage. In this chapter we shall review some recent papers 
which can be classified in direct or indirect methods.

The first effort to solve a fractional optimal control problem numerically 
was made in 2004 by Agrawal \cite{AgrawalNum}. The problem under consideration consists 
in finding an optimal control $u(\cdot)$, which minimizes the functional
$$
J[x,u]=\int_0^1 F(t,x,u)dt,
$$
while it is assumed to satisfy a given dynamic constraint of the form
$$
\LD x(t)=G(t,x,u)
$$
subject to the boundary condition
$$
x(0)=x_0.
$$
The Euler--Lagrange equation can be derived by using a Lagrange multiplier\index{Lagrange multipliers}, 
$\lambda(\cdot)$ \cite{gasta1}. The necessary optimality condition reads to
$$
\left\{\begin{array}{l}
\LD x(t)=G(t,x,u)\\[5pt]
\displaystyle\RDone \lambda(t)=\frac{\partial F}{\partial x}+\lambda \frac{\partial G}{\partial x}\\[8pt]
\displaystyle\hspace{12mm}0=\frac{\partial F}{\partial u}+\lambda \frac{\partial G}{\partial u}
\end{array}\right.,\qquad \left\{\begin{array}{l}
x(0)=x_0\\
\lambda(1)=0.
\end{array}\right.
$$
The paper \cite{AgrawalNum} uses a Ritz method\index{Ritz method} by approximating $x(\cdot)$ 
and $\lambda(\cdot)$ using shifted Legendre polynomials\index{Shifted Legendre polynomials}, i.e.,
$$
x(t)\approx\sum_{j=1}^m c_j P_j(t), \qquad \lambda(t)\approx\sum_{j=1}^m c_j P_j(t).
$$
The shifted Legendre polynomials are explicitly given by
$$
P_n(t)=(-1)^n \sum_{k=0}^n \binom{n}{k}\binom{n+k}{k}(-x)^k.
$$

One can use the orthogonality of Legendre polynomials and the fact that their fractional derivatives 
are available in closed forms. This method, after some calculus operations and simplifications, 
leads to a system of $2m+2$ equations in $2m+2$ unknowns. Approximate solutions to the problem 
then is achieved in terms of linear combinations of the shifted Legendre polynomials.

The same idea has been tried later by several authors. This is done by either using different 
approximations in terms of other basis functions or a different class of variational problems, 
say in the problem formulation or in the fractional term that appears.

Approximating $x(\cdot)$, $u(\cdot)$  and $\lambda(\cdot)$ by multiwavelets is an example 
of a new version of this method. In \cite{Lotfi} the Caputo fractional derivative is used 
in the constraint and another functional is considered. Other aspects like some properties 
of Legendre polynomials and the convergence also are covered in this work.

Another slightly different approach is the use of the so-called multiwavelet collocation 
that has been introduced in \cite{Yousefi}. The method is based on the approximations
\begin{align*}
&x(t)\approx\sum_{i=0}^{2^k-1}\sum_{j=0}^{M}(t-a) cx_{ij}\psi_{ij}(t)+x_0,\\
&u(t)\approx\sum_{i=0}^{2^k-1}\sum_{j=0}^{M} cu_{ij}\psi_{ij}(t),\\
&\lambda(t)\approx\sum_{i=0}^{2^k-1}\sum_{j=0}^{M}(t-a) c\lambda_{ij}\psi_{ij}(t),
\end{align*}
where $t\in[a,b]$ and
$$
\psi_{nm}=\sqrt{2m+1}\frac{2^{k/2}}{\sqrt{b-a}}P_m\left( \frac{2^{k}(t-a)}{{b-a}}-n\right),
\quad \frac{n(t-a)}{{2^k}}+a\leq t<\frac{(n+1)(t-a)}{{2^k}}+a,
$$
with the shifted Legendre polynomials $P_m$. The collocation points $p_i$, 
$1\leq i\leq 2^k(M+1)$, are the roots of  Chebyshev polynomials of degree $2^k(M+1)$. 
The resulting system of algebraic equations is solved to obtain the approximate solutions. 
Although the paper \cite{Yousefi} discusses the general case when $x$ and $u$ are vector functions, 
for the sake of simplicity we outlined it here in one dimension.

A finite element\index{Finite! element} method has been developed in \cite{AgrawalFE}. 
The functional to be minimized has a special form of
\begin{align*}
J[x(\cdot)]&=\int_a^b L(t,x,\LD x) dt\\
&=\int_a^b \left[ \frac{1}{2}A_1(t)(\LD x)^2+A_2(t)(\LD x)x
+\frac{1}{2}A_3(t) x^2+A_4(t)\LD x+A_5(t)x\right]dt.
\end{align*}
The boundary conditions at both end-points are given. In this method, 
the time interval $[a,b]$ is devided into $N$ equally spaced subintervals. 
Let $t_j=a+jh$ where $h=\frac{b-a}{N}$ and $j=0,\ldots,N$. Then the functional is given by
$$
J[x(\cdot)]=\sum_{j=1}^N\int_{t_{j-1}}^{t_j}L(t,x(t),\LD x(t)) dt.
$$
Now one can approximate $x(\cdot)$ over subintervals 
by ``shape'' functions, e.g., splines, as
$$
x(t)=N_j(t)x_{ej},\qquad t\in[t_{j-1},t_j],~j=1,\ldots,N,
$$
and
$$
\LD x(t)=N_j(t)(\LD x)_{ej},\qquad t\in[t_{j-1},t_j],~j=2,\ldots,N,
$$
where $N_j$ is the shape function at the corresponding subinterval, 
and $x_{ej}$ and $(\LD x)_{ej}$ are the nodal values of the unknown function 
and its fractional derivatives. The fractional derivative at each point is also 
approximated using Gr\"{u}nwald--Letnikov definition as an approximation which 
is discussed in Chapter~\ref{Direct}. The remaining process is straightforward.

Another work that is worth to pay attention is the use of a modified Gr\"{u}nwald--Letnikov 
approximation for left and right derivatives to discretize the Euler--Lagrange 
equation \cite{Ozlem}. The approximations are carried out at the central points 
of a certain discretization of the time horizon. Namely, for $a=t_0<t_1<\ldots<t_n=b$,
$$
\LD x(t_{i-1/2})\approx \frac{1}{h^\a} \sum_{k=0}^{i}\w x_{i-j}, \qquad i=1,\ldots,n,
$$
and
$$
\RDone \lambda(t_{i+1/2})\approx \frac{1}{h^\a} \sum_{k=0}^{n-i}\w 
\lambda_{i+j}, \qquad i=n-1,\ldots,0,
$$
where $\w=(-1)^k\binom{\a}{k}=\frac{\Gamma(k-\a)}{\Gamma(-\a)\Gamma(k+1)}$ 
and $x(t_{i-1/2})=(x_{i-1}+x_i)/2$. Solving a system of $2n$ algebraic equations 
in $2n$ unknowns gives the approximate values of the unknown function on mesh points.

Numerical methods, nowadays, are easily implemented on computers, making packages 
and tools to solve problems. Many problems in this thesis have been solved, 
e.g., in MATLAB$^\circledR$, using some predefined routines and solvers. 
The implemented methods are far from being an outstanding and a multipurpose solver. 
They have been designed for special problems and for a relevant problem they may 
need significant modifications.  The only work, to the best of our knowledge, 
directed in the adaptation of the existing toolboxes is \cite{Tricaud}. 
This work uses Oustaloup's approximation formula for fractional derivatives 
and transforms a fractional optimal control problem into a problem in which 
only derivatives of integer order are present. Being a classical problem, 
it can be solved by \textbf{RIOTS-95}, a MATLAB$^\circledR$\index{MATLAB} 
toolbox for optimal control problems\footnote{\url{http://www.schwartz-home.com/RIOTS/}}. 
The problem is to find a control that minimizes the functional
$$
J[u]=G(x(a),x(b))+\int_a^b L(t,x,u)dt
$$
subject to the dynamic control system
$$
\LD x(t)=f(t,x,u),
$$
and the initial condition $x(a)=x_a$. The control may be bounded, $u_{min}\leq u(t)\leq u_{max}$. 
Also other constraints on the boundaries and/or state-control inequality constraints 
may be present. The idea is to use the approximation
$$
\LD x(t)\approx \left\{\begin{array}{l}
\dot{z}=Az+Bu\\
x=Cz+Du,
\end{array}\right.
$$
and transform the problem to the minimization of
$$
J[u]=G(Cz(a)+Du(a),Cz(b)+Du(b))+\int_a^b L(t,Cz+Du,u)dt
$$
such that
$$
\dot{z}(t)=Az+B(f((t,Cz+Du,u)),
$$
and the initial condition
$$
z(a)=\frac{x_a\omega}{C\omega},
$$
where $\omega=[1~~0~~\cdots~~0]^T$. The resulting setting is appropriate as an input for \textbf{RIOTS-95}.

Another approach to benefit the methods and tools of the classical 
theory has been introduced in \cite{Jelicic}. The work is based on an 
approximation formula from \cite{Atan2}, that is improved and discussed 
in a very detailed way throughout our work. 
The control problem to be solved is the following:
$$
J[u]=\int_0^1 L(t,x,u)dt \longrightarrow \min
$$
subject to
$$
\left\{\begin{array}{l}
\dot{x}(t)+k\left(\LD x(t)\right)=f(t,x,u)\\
x(0)=x_0.\end{array}\right.
$$
Using the approximation
\begin{equation*}
\LD x(t)\approx At^{-\a}x(t)-\sum_{p=2}^N C(\a,p)t^{1-p-\a}V_p(t),
\end{equation*}
the problem is transformed into a classic integer-order problem,
$$
J[u]=\int_0^1 L(t,x,u)dt \longrightarrow \min
$$
subject to
$$
\left\{\begin{array}{l}
\dot{x}(t)+k\left(At^{-\a}x(t)-\sum_{p=2}^N C(\a,p)t^{1-p-\a}V_p(t)\right)=f(t,x,u)\\
\dot{V}_p(t)=(1-p)(t-a)^{p-2}x(t)\\
V_p(a)=0, \qquad p=2,\ldots,N\\
x(0)=x_0.\end{array}\right.
$$

\CP
\part{Original Work}
\CP
\chapter{Approximating fractional derivatives}
\label{AppFracDer}

This section is devoted to two approximations for the Riemann--Liouville, Caputo 
and Hadamard derivatives that are referred as fractional operators afterwards. 
We introduce the expansions of fractional operators in terms of infinite sums 
involving only integer-order derivatives. These expansions are then used 
to approximate fractional operators in problems like fractional differential equations, 
fractional calculus of variations, fractional optimal control, etc. In this way, 
one can transform such problems into classical problems. Hereafter, a suitable method, 
that can be found in the classical literature, is employed to find an approximate 
solution for the original fractional problem. Here we focus mainly on the left derivatives 
and the details of extracting corresponding expansions for right derivatives 
are given whenever it is needed to apply new techniques.


\section{Riemann--Liouville derivative}

\subsection{Approximation by a sum of integer-order derivatives}
\index{Expansion formula}

Recall the definition of the left Riemann--Liouville 
derivative\index{Riemann--Liouville fractional derivative! Left} for $\a\in(0,1)$:
\begin{equation}
\label{LeftD}
\LDa x(t)=\frac{1}{\Gamma(1-\a)}\frac{d}{dt}\int_a^t (t-\t)^{-\a}x(\t)d\t.
\end{equation}
The following theorem holds for any function $x(\cdot)$ that is analytic in an interval 
$(c,d)\supset[a,b]$. See \cite{Atan} for a more detailed discussion 
and \cite{Samko}, for a different proof.

\begin{theorem}
\label{ThmIntExp}
Let $(c,d)$, $-\infty<c<d<+\infty$, be an open interval in $\mathbb R$, 
and $[a,b]\subset(c,d)$ be such that for each $t\in[a,b]$ the closed ball 
$B_{b-a}(t)$, with center at $t$ and radius $b-a$, lies in $(c,d)$. 
If $x(\cdot)$ is analytic in $(c,d)$, then
\begin{equation}
\label{ExpIntInf}
\LDa x(t)=\sum_{k=0}^{\infty}\frac{(-1)^{k-1}\a x^{(k)}(t)}{k!(k-\a)
\Gamma(1-\a)}(t-a)^{k-\a}.
\end{equation}
\end{theorem}

\begin{proof}
Since $x(t)$ is analytic in $(c,d)$ and $B_{b-a}(t)\subset(c,d)$ for any 
$\t\in(a,t)$ with $t\in(a,b)$, the Taylor expansion of $x(\t)$ at $t$ 
is a convergent power series, \textrm{i.e.},
$$
x(\t)=x(t-(t-\t))=\sum_{k=0}^{\infty}\frac{(-1)^k x^{(k)}(t)}{k!}(t-\t)^k,
$$
and then by \eqref{LeftD}
\begin{equation}
\label{ExpIntErr}
\LDa x(t)=\frac{1}{\Gamma(1-\a)}\frac{d}{dt}\int_a^t \left((t-\t)^{-\a}
\sum_{k=0}^{\infty}\frac{(-1)^k x^{(k)}(t)}{k!}(t-\t)^k\right)d\t.
\end{equation}
Since $(t-\t)^{k-\a} x^{(k)}(t)$ is analytic, we can interchange integration with summation, so
\begin{eqnarray*}
\LDa x(t)&=& \frac{1}{\Gamma(1-\a)}\frac{d}{dt}\left(\sum_{k=0}^{\infty}
\frac{(-1)^k  x^{(k)}(t)}{k!}\int_a^t (t-\t)^{k-\a}d\t\right)\\
&=& \frac{1}{\Gamma(1-\a)}\frac{d}{dt}\sum_{k=0}^{\infty}\left(
\frac{(-1)^k x^{(k)}(t)}{k!(k+1-\a)}(t-a)^{k+1-\a}\right)\\
&=&\frac{1}{\Gamma(1-\a)}\sum_{k=0}^{\infty}\left(\frac{(-1)^k 
x^{(k+1)}(t)}{k!(k+1-\a)}(t-a)^{k+1-\a}+\frac{(-1)^k x^{(k)}(t)}{k!}(t-a)^{k-\a}\right)\\
&=&\frac{x(t)}{\Gamma(1-\a)}(t-a)^{-\a}\\
&&+\frac{1}{\Gamma(1-\a)}\sum_{k=1}^\infty \left(\frac{(-1)^{k-1}}{(k-\a)(k-1)!}
+\frac{(-1)^k}{k!}\right)x^{(k)}(t)(t-a)^{k-\a}.
\end{eqnarray*}
Observe that
\begin{eqnarray*}
\frac{(-1)^{k-1}}{(k-\a)(k-1)!}+\frac{(-1)^k}{k!}
&=&\frac{k(-1)^{k-1}+k(-1)^k-\a(-1)^k}{(k-\a)k!}\\
&=&\frac{(-1)^{k-1}\a}{(k-\a)k!},
\end{eqnarray*}
since for any $k=0,1,2,\ldots$ we have $k(-1)^{k-1}+k(-1)^{k}=0$. 
Therefore, the expansion formula is reached as required.
\end{proof}

For numerical purposes, a finite number of terms in \eqref{ExpIntInf} is used and one has
\begin{equation}\label{expanInt}
\LDa x(t)\approx\sum_{k=0}^{N}\frac{(-1)^{k-1}\a x^{(k)}(t)}{k!(k-\a)\Gamma(1-\a)}(t-a)^{k-\a}.
\end{equation}

\begin{remark}
With the same assumptions of Theorem~\ref{ThmIntExp}, 
we can expand $x(\t)$ at $t$, where $\t\in(t,b)$,
$$
x(\t)=x(t+(\t-t))=\sum_{k=0}^{\infty}\frac{x^{(k)}(t)}{k!}(\t-t)^k,
$$
and get the following approximation for the right Riemann--Liouville 
derivative:\index{Riemann--Liouville fractional derivative! Right}
\begin{equation*}
\RD x(t)\approx\sum_{k=0}^{N}\frac{-\a x^{(k)}(t)}{k!(k-\a)\Gamma(1-\a)}(b-t)^{k-\a}.
\end{equation*}
\end{remark}

A proof for this expansion is available at \cite{Samko} that uses a similar relation 
for fractional integrals. The proof discussed here, however, 
allows us to extract an error term for this expansion easily.


\subsection{Approximation using moments of a function}
\label{SecAppDer}

By {\it moments}\index{Moment of a function} of a function we have no physical 
or distributive senses in mind. The name comes from the fact that, 
during expansion, the terms of the form
\begin{equation}
\label{defVp}
V_p(t):=V_p(x(t))=(1-p)\int_a^t (\t-a)^{p-2}x(\t) d\t, \quad p\in\mathbb{N}, \t\geq a,
\end{equation}
appear to resemble the formulas of central moments (\textrm{cf.} \cite{Atan2}). 
We assume that $V_p(x(\cdot))$, $p\in \mathbb{N}$, denote the $(p-2)$th 
moment of a function $x(\cdot)\in AC^2[a,b]$.

The following lemma, that is given here without a proof, is the key relation 
to extract an expansion formula\index{Expansion formula} for Riemann--Liouville derivatives.

\begin{lemma}[\textrm{cf.} Lemma 2.12 of \cite{Kai}]
Let $x(\cdot)\in AC[a,b]$ and $0 < \a < 1$. Then the left Riemann--Liouville 
fractional derivative\index{Riemann--Liouville fractional derivative! Left} 
$\LDa x(\cdot)$ exists almost everywhere in $[a,b]$.
Moreover, $\LDa x(\cdot)\in L_p[a,b]$ for $1\leq p<\frac{1}{\a}$ and
\begin{equation}
\label{DecThm}
\LDa x(t)=\frac{1}{\Gamma(1-\a)}\left[ \frac{x(a)}{(t-a)^\a}
+\int_a^t (t-\tau)^{-\a}\dot{x}(\tau)d\tau \right], \qquad t\in (a,b).
\end{equation}
The same argument is valid for the right Riemann--Liouville derivative 
and\index{Riemann--Liouville fractional derivative! Right}
\begin{equation*}
\RD x(t)=\frac{1}{\Gamma(1-\a)}\left[ \frac{x(b)}{(b-t)^\a}
-\int_t^b (\tau-t)^{-\a}\dot{x}(\tau)d\tau \right], \qquad t\in (a,b).
\end{equation*}
\end{lemma}

\begin{theorem}[cf. \cite{Atan2}]\label{ThmAtan}
Let $x(\cdot)\in AC[a,b]$ and $0 < \a < 1$. 
Then the left Riemann--Liouville derivative can be expanded as
\begin{equation}
\label{expanMomInf}
\LDa x(t)=A(\a)(t-a)^{-\a}x(t)+B(\a)(t-a)^{1-\a}\dot{x}(t)
-\sum_{p=2}^{\infty}C(\a,p)(t-a)^{1-p-\a}V_p(t),
\end{equation}
where $V_p(t)$ is defined by \eqref{defVp} and
\begin{eqnarray}\label{Cp}
A(\a)&=&\frac{1}{\Gamma(1-\a)}\left(1+\sum_{p=2}^{\infty}
\frac{\Gamma(p-1+\a)}{\Gamma(\a)(p-1)!}\right),\nonumber\\
B(\a)&=&\frac{1}{\Gamma(2-\a)}\left(1+\sum_{p=1}^{\infty}
\frac{\Gamma(p-1+\a)}{\Gamma(\a-1)p!}\right),\nonumber\\
C(\a,p)&=&\frac{1}{\Gamma(2-\a)\Gamma(\a-1)}\frac{\Gamma(p-1+\a)}{(p-1)!}.
\end{eqnarray}
\end{theorem}

\begin{remark}
The proof of Theorem~\ref{ThmAtan} is done by  T.M. Atanackovi\'c
and B. Stankovi\'c \cite{Atan2} but, unfortunately, has a small mistake:
the coefficient $A(\a)$, where we have an infinite sum, is not well defined
since the series diverges.

For a correct formulation and proof see 
our Theorem~\ref{thm:oft} and Remark~\ref{RemarkAtan}.
\end{remark}

The moments\index{Moment of a function} $V_p(t)$, $p=2,3,\ldots$, 
are regarded as the solutions to the following system of differential equations:
\begin{equation}
\label{sysVp}
\left\{
\begin{array}{l}
\dot{V}_p(t)=(1-p)(t-a)^{p-2}x(t)\\
V_p(a)=0, \qquad p=2,3,\ldots .
\end{array}
\right.
\end{equation}

As before, a numerical approximation is achieved by taking only a finite number 
of terms in the series \eqref{expanMomInf}. We approximate the fractional derivative as
\begin{equation}
\label{expanMom}
\LDa x(t)\approx A(t-a)^{-\a}x(t)+B(t-a)^{1-\a}\dot{x}(t)
-\sum_{p=2}^N C(\a,p)(t-a)^{1-p-\a}V_p(t),
\end{equation}
where $A=A(\a,N)$ and $A=B(\a,N)$ are given by
\begin{eqnarray}
A(\a,N)&=&\frac{1}{\Gamma(1-\a)}\left(1+\sum_{p=2}^N
\frac{\Gamma(p-1+\a)}{\Gamma(\a)(p-1)!}\right),\label{A}\\
B(\a,N)&=&\frac{1}{\Gamma(2-\a)}\left(1+\sum_{p=1}^N
\frac{\Gamma(p-1+\a)}{\Gamma(\a-1)p!}\right)\label{B}.
\end{eqnarray}

\begin{remark}
The expansion \eqref{expanMomInf} has been proposed in \cite{Djordjevic} 
and an interesting, yet misleading, simplification has been made in \cite{Atan2}, 
which uses the fact that the infinite series $\sum_{p=1}^{\infty}
\frac{\Gamma(p-1+\a)}{\Gamma(\a-1)p!}$ tends to $-1$ and concludes that $B(\a)=0$ and thus
\begin{equation}
\label{expanAtan}
\LD x(t)\approx A(\a,N)t^{-\a}x(t)-\sum_{p=2}^N C(\a,p)t^{1-p-\a}V_p(t).
\end{equation}
In practice, however, we only use a finite number of terms in the series. Therefore
\begin{equation*}
1+\sum_{p=1}^N\frac{\Gamma(p-1+\a)}{\Gamma(\a-1)p!}\neq 0,
\end{equation*}
and we keep here the approximation in the form of equation \eqref{expanMom} 
\cite{PATFracDer}. To be more precise, the values of $B(\a,N)$ for different choices 
of $N$ and $\a$ are given in Table~\ref{tab}. It shows that even for a large $N$, 
when $\a$ tends to one, $B(\a,N)$ cannot be ignored. In Figure~\ref{BaN}, 
we plot $B(\a,N)$ as a function of $N$ for different values of $\a$.
\begin{table}[!htp]
\center
\begin{tabular}{|c|c c c c c c c|}
\hline
$N$         &    4   &    7   &   15   &   30   &   70   &  120   &  170   \\
\hline
$B(0.1,N)$  & 0.0310 & 0.0188 & 0.0095 & 0.0051 & 0.0024 & 0.0015 & 0.0011 \\
$B(0.3,N)$  & 0.1357 & 0.0928 & 0.0549 & 0.0339 & 0.0188 & 0.0129 & 0.0101 \\
$B(0.5,N)$  & 0.3085 & 0.2364 & 0.1630 & 0.1157 & 0.0760 & 0.0581 & 0.0488 \\
$B(0.7,N)$  & 0.5519 & 0.4717 & 0.3783 & 0.3083 & 0.2396 & 0.2040 & 0.1838 \\
$B(0.9,N)$  & 0.8470 & 0.8046 & 0.7481 & 0.6990 & 0.6428 & 0.6092 & 0.5884 \\
$B(0.99,N)$ & 0.9849 & 0.9799 & 0.9728 & 0.9662 & 0.9582 & 0.9531 & 0.9498 \\
\hline
\end{tabular}
\caption{$B(\a,N)$ for different values of  $\a$ and $N$.}
\label{tab}
\end{table}

\begin{figure}[!ht]
  \begin{center}
    \subfigure{\includegraphics[scale=0.54]{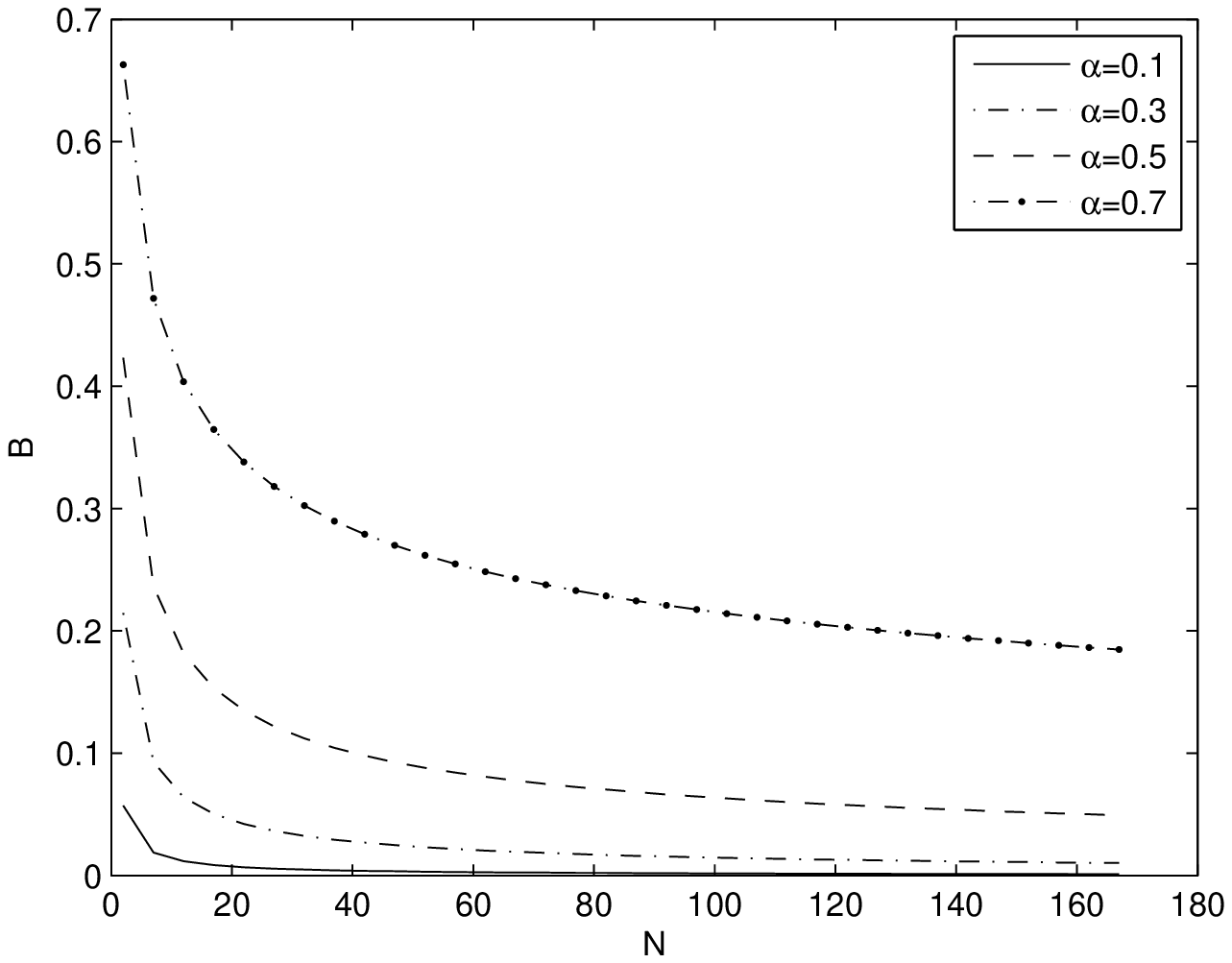}}
    \subfigure{\includegraphics[scale=0.54]{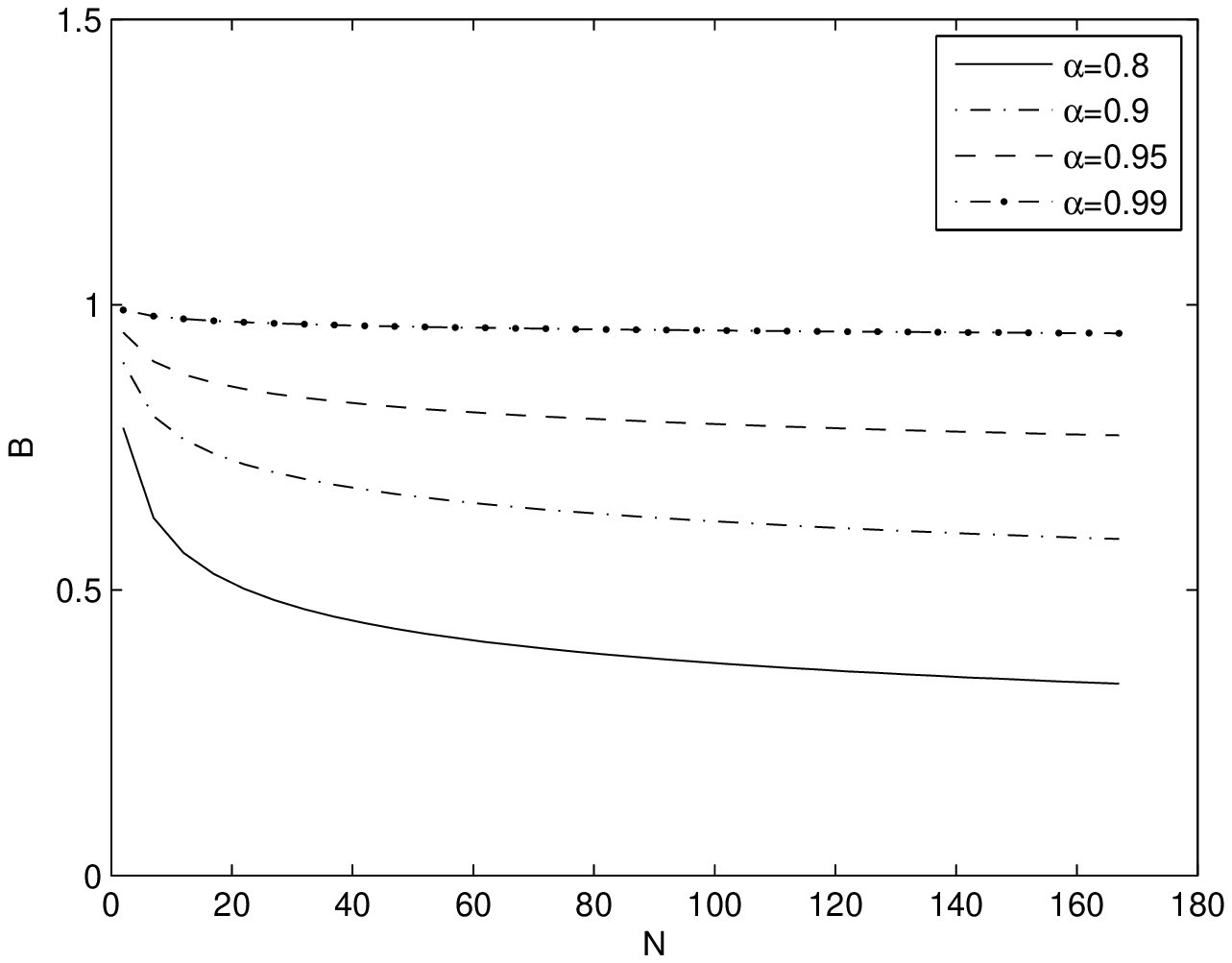}}
  \end{center}
  \caption{$B(\a,N)$ for different values of  $\a$ and $N$.}
  \label{BaN}
\end{figure}
\end{remark}

\begin{remark}
Similar computations give rise to an expansion formula for $\RD$, 
the right Riemann--Liouville fractional derivative:\index{Riemann--Liouville fractional derivative! Right}
\begin{equation}\label{expanMomR}
\RD x(t)\approx A(b-t)^{-\a}x(t)-B(b-t)^{1-\a}\dot{x}(t)-\sum_{p=2}^NC(\a,p)(b-t)^{1-p-\a}W_p(t),
\end{equation}
where
$$
W_p(t)=(1-p)\int_t^b (b-\tau)^{p-2}x(\tau)d\tau.
$$
The coefficients $A=A(\a,N)$ and $B=B(\a,N)$ are the same as \eqref{A} 
and \eqref{B}, respectively, and $C(\a,p)$ is given by \eqref{Cp}.
\end{remark}

\begin{remark}
As stated before, Caputo derivatives are closely related to those of Riemann--Liouville. 
For any function, $x(\cdot)$, and for $\a\in(0,1)$, 
if these two kind of fractional derivatives exist, then we have
$$
\LDC x(t)=\LD x(t)-\frac{x(a)}{\Gamma(1-\a)}(t-a)^{-\a}
$$
and
$$
\RDC x(t)=\RD x(t)-\frac{x(b)}{\Gamma(1-\a)}(b-t)^{-\a}.
$$
Using these relations, we can easily construct approximation formulas for left 
and right Caputo fractional 
derivatives:\index{Caputo fractional derivative! Right}\index{Caputo fractional derivative! Left}
\begin{eqnarray*}
\LDC x(t)&\approx & A(\a,N)(t-a)^{-\a}x(t)+B(\a,N)(t-a)^{1-\a}\dot{x}(t)\\
&&-\sum_{p=2}^N C(\a,p)(t-a)^{1-p-\a}V_p(t)-\frac{x(a)}{\Gamma(1-\a)}(t-a)^{-\a}.
\end{eqnarray*}
\end{remark}
Formula \eqref{expanMomInf} consists of two parts:
an infinite series and two terms including
the first derivative and the function itself.
It can be generalized to contain derivatives of higher-order.

\begin{theorem}
\label{thm:oft}
Fix $n \in \mathbb{N}$ and let $x(\cdot)\in C^n[a,b]$. Then,
\begin{multline}
\label{Gen}
\LDa x(t)=\frac{1}{\Gamma(1-\a)}(t-a)^{-\a}x(t)+\sum_{i=1}^{n-1}
A(\a,i)(t-a)^{i-\a}x^{(i)}(t)\\
+\sum_{p=n}^{\infty}\left[
\frac{-\Gamma(p-n+1+\a)}{\Gamma(-\a)\Gamma(1+\a)(p-n+1)!}(t-a)^{-\a}x(t)
+ B(\a,p)(t-a)^{n-1-p-\a}V_p(t)\right],
\end{multline}
where
\begin{eqnarray*}
A(\a,i)&=&\frac{1}{\Gamma(i+1-\a)}\left[1+\sum_{p
=n-i}^{\infty}\frac{\Gamma(p-n+1+\a)}{\Gamma(\a-i)(p-n+i+1)!}\right],
\quad i = 1,\ldots,n-1,\\
B(\a,p)&=&\frac{\Gamma(p-n+1+\a)}{\Gamma(-\a)\Gamma(1+\a)(p-n+1)!},\\
V_p(t) &=&(p-n+1)\int_a^t (\t-a)^{p-n}x(\t)d\t.
\end{eqnarray*}
\end{theorem}

\begin{proof}
Successive integrating by parts in \eqref{DecThm} gives
\begin{eqnarray*}
\LDa x(t)&=&\frac{x(a)}{\Gamma(1-\a)}(t-a)^{-\a}
+\frac{\dot{x}(a)}{\Gamma(2-\a)}(t-a)^{1-\a}
+\cdots+\frac{x^{(n-1)}(a)}{\Gamma(n-\a)}(t-a)^{n-1-\a}\\
& &+\frac{1}{\Gamma(1-\a)}\int_a^t (t-\tau)^{n-1-\a}x^{(n)}(\tau)d\tau.
\end{eqnarray*}
Using the binomial theorem\index{Binomial! theorem}, we expand the integral term as
$$
\int_a^t (t-\tau)^{n-1-\a}x^{(n)}(\tau)d\tau=(t-a)^{n-1-\a}\sum_{p=0}^{\infty}
\frac{\Gamma(p-n+1+\a)}{\Gamma(1-n+\a)p!(t-a)^p}\int_a^t (\tau-a)^px^{(n)}(\tau)d\tau.
$$
Splitting the sum into $p=0$ and $p=1\ldots\infty$,
and integrating by parts the last integral, we get
\begin{eqnarray*}
\LD x(t)&=&\frac{(t-a)^{-\a}}{\Gamma(1-\a)}x(a)+\cdots
+\frac{(t-a)^{n-2-\a}}{\Gamma(n-1-\a)}x^{(n-2)}(a)\\
& &+\frac{(t-a)^{n-1-\a}}{\Gamma(n-\a)}x^{(n-2)}(t)\left[1
+ \sum_{p=1}^\infty\frac{\Gamma(p-n+1+\a)}{\Gamma(-n+1+\a)p!}\right]\\
& &+ \frac{(t-a)^{n-1-\a}}{\Gamma(n-1-\a)} \sum_{p=1}^\infty\frac{\Gamma(p
-n+1+\a)}{\Gamma(-n+2+\a)(p-1)!(t-a)^p}\int_a^t(\tau-a)^{p-1}x^{(n-1)}(\tau)d\tau.\\
\end{eqnarray*}
The rest of the proof follows a similar routine, \textrm{i.e.},
by splitting the sum into two parts, the first term and the rest,
and integrating by parts the last integral until
$x(\cdot)$ appears in the integrand.
\end{proof}

\begin{remark}\label{RemarkAtan}
The series that appear in $A(\a,i)$ is convergent
for all $ i \in \{1,\ldots,n-1\}$. Fix an $i$ and observe that
$$
\sum_{p=n-i}^{\infty}\frac{\Gamma(p-n+1+\a)}{\Gamma(\a-i)(p-n+i+1)!}
=\sum_{p=1}^{\infty}\frac{\Gamma(p+\a-i)}{\Gamma(\a-i)p!}
={_1F_0}(\a-i,1)-1,
$$
where $_1F_0$ stands for a hypergeometric function\index{Hypergeometric function} \cite{Andrews}.
Since $i>\a$, ${_1F_0}(\a-i,1)$ converges by Theorem~2.1.1 of \cite{Andrews}.

In practice we only use finite sums and for $A(\a,i)$ we can easily compute
the truncation error\index{Error! Truncation}. Although this is a partial error, it gives a good
intuition of why this approximation works well.
Using the fact that ${_1F_0}(a,1)=0$ if $a<0$
(\textrm{cf.} Eq. (2.1.6) in \cite{Andrews}), we have
\begin{equation}\label{TruncAai}
\begin{split}
\frac{1}{\Gamma(i+1-\a)}&\sum_{p=N+1}^{\infty}\frac{\Gamma(p
-n+1+\a)}{\Gamma(\a-i)(p-n+i+1)!}\\
&=\frac{1}{\Gamma(i+1-\a)}\left({_1F_0}(\a-i,1)
- \sum_{p=0}^{N-n+i+1}\frac{\Gamma(p+\a-i)}{\Gamma(\a-i)p!}\right)\\
&=\frac{-1}{\Gamma(i+1-\a)}\sum_{p=0}^{N-n+i+1}\frac{\Gamma(p+\a-i)}{\Gamma(\a-i)p!}.
\end{split}
\end{equation}
In Table~\ref{truncError} we give some values for this error,
with $\a=0.5$ and different values for $i$ and $N-n$.
\begin{table}[!ht]
\center
\begin{tabular}{|l|c|c|c|c|c|}\hline
\diaghead{\theadfont Diag Column }%
{~\\$i$}{$N-n$\\~}&
\thead{0}&\thead{5}&\thead{10}&\thead{15}&\thead{20}\\
 \hline
1 &-0.4231&-0.2364&-0.1819&-0.1533& -0.1350\\
 \hline
2 &0.04702&0.009849&0.004663&0.002838&0.001956\\
 \hline
3 &-0.007052&-0.0006566&-0.0001999&-0.00008963&-0.00004890\\
 \hline
4 &0.001007&0.00004690&0.000009517&0.000003201&0.000001397\\
 \hline
\end{tabular}
\caption{The truncation error \eqref{TruncAai} of $A(\a,i)$ for $\a=0.5$, 
that is, $A(\a,i)-A(\a,i,N)$ with $A(\a,i,N)$ given by \eqref{AaiN}.}
\label{truncError}
\end{table}
\end{remark}

\begin{remark}
Using Euler's reflection formula, one can define $B(\a,p)$
of Theorem~\ref{thm:oft} as
$$
B(\a,p)=\frac{-\sin(\pi\a)\Gamma(p-n+1+\a)}{\pi (p-n+1)!}.
$$
\end{remark}

For numerical purposes, only finite sums are taken
to approximate fractional derivatives.
Therefore, for a fixed $n \in \mathbb{N}$ and $N\geq n$, one has
\begin{equation}\label{ApproxDerGeneralcase}
\LDa x(t)\approx \sum_{i=0}^{n-1} A(\a,i,N)(t-a)^{i-\a}x^{(i)}(t)
+\sum_{p=n}^{N} B(\a,p)(t-a)^{n-1-p-\a}V_p(t),
\end{equation}
where
\begin{equation}\label{AaiN}
A(\a,i,N)=\frac{1}{\Gamma(i+1-\a)}\left[1+\sum_{p=2}^{N}\frac{\Gamma(p
-n+1+\a)}{\Gamma(\a-i)(p-n+i+1)!}\right], \quad i = 0,\ldots,n-1,
\end{equation}
\begin{eqnarray*}
B(\a,p)&=&\frac{\Gamma(p-n+1+\a)}{\Gamma(-\a)\Gamma(1+\a)(p-n+1)!},\nonumber\\
V_p(t) &=&(p-n+1)\int_a^t (\t-a)^{p-n}x(\t)d\t.\nonumber
\end{eqnarray*}

Similarly, we can deduce an expansion formula for the right fractional 
derivative.\index{Riemann--Liouville fractional derivative! Right}

\begin{theorem}
Fix $n\in\mathbb N$ and $x(\cdot)\in C^n[a,b]$. Then,
\begin{multline*}
\RD x(t)=\frac{1}{\Gamma(1-\a)}(b-t)^{-\a} x(t)
+\sum_{i=1}^{n-1}A(\alpha,i)(b-t)^{i-\a} x^{(i)}(t)\\
+\sum_{p=n}^\infty\left[ \frac{-\Gamma(p-n+1+\a)}{\Gamma(-\a)\Gamma(1
+\a)(p-n+1)!}(b-t)^{-\a} x(t)+B(\a,p)(b-t)^{n-1-\a-p}W_p(t)\right],
\end{multline*}
where
\begin{eqnarray*}
A(\a,i)&=&\frac{(-1)^i}{\Gamma(i+1-\a)}\left[1+\sum_{p
=n-i}^\infty\frac{\Gamma(p-n+1+\a)}{\Gamma(-i+\a)(p-n+1+i)!}\right],
\quad i = 1,\ldots,n-1,\\
B(\a,p)&=&\frac{(-1)^n\Gamma(p-n+1+\a)}{\Gamma(-\a)\Gamma(1+\a)(p-n+1)!},\\
W_p(t)&=&(p-n+1)\int_t^b (b-\tau)^{p-n}x(\tau)d\tau.
\end{eqnarray*}
\end{theorem}

\begin{proof}
Analogous to the proof of Theorem~\ref{thm:oft}.
\end{proof}


\subsection{Numerical evaluation of fractional derivatives}

In \cite{Podlubny} a numerical method to evaluate fractional derivatives
is given based on the Gr\"{u}nwald--Letnikov definition of fractional derivatives.
It uses the fact that for a large class of functions, the Riemann--Liouville
and the Gr\"{u}nwald--Letnikov definitions are equivalent.
We claim that the approximations discussed so far provide a good tool
to compute numerically the fractional derivatives of given functions.
For functions whose higher-order derivatives are easily available,
we can freely choose between approximations \eqref{expanInt} or \eqref{expanMom}.
But in the case that difficulties arise in computing higher-order derivatives,
we choose the approximation \eqref{expanMom} that needs only the values
of the first derivative and function itself. Even if the first derivative
is not easily computable, we can use the approximation given by \eqref{expanAtan}
with large values for $N$ and $\a$ not so close to one. As an example, we compute
$\LD x(t)$, with $\a=\frac{1}{2}$, for $x(t)=t^4$ and $x(t)=e^{2t}$.
The exact formulas of the derivatives are derived from
\begin{equation*}
\LDz(t^n)=\frac{\Gamma(n+1)}{\Gamma(n+1-0.5)}t^{n-0.5}
\quad \text{and} \quad \LDz(e^{\lambda t})
=t^{-0.5}E_{1,1-0.5}(\lambda t),
\end{equation*}
where $E_{\a,\beta}$ is the two parameter
Mittag--Leffler\index{Mittag--Leffler function} function \cite{Podlubny}.
Figure~\ref{EvalInt} shows the results using approximation 
\eqref{expanInt} with error $E$ computed by \eqref{ErrorDef}.
As we can see, the third approximations are reasonably accurate for both cases.
Indeed, for $x(t)=t^4$, the approximation with $N=4$ coincides with
the exact solution because the derivatives of order five and more vanish.
\begin{figure}[!ht]
  \begin{center}
    \subfigure[$\LDz(t^4)$]{\includegraphics[scale=0.54]{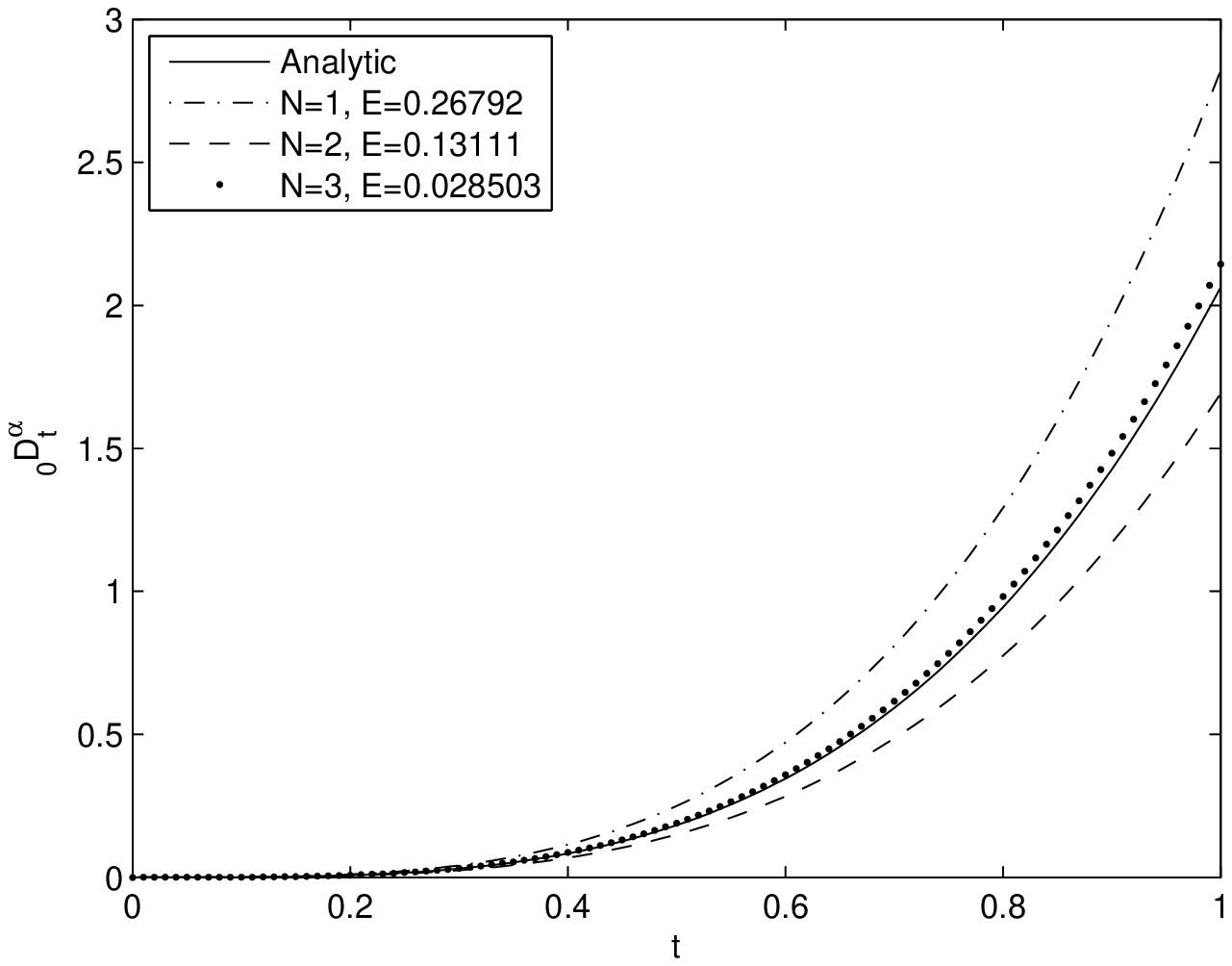}}
    \subfigure[$\LDz(e^{2t})$]{\includegraphics[scale=0.54]{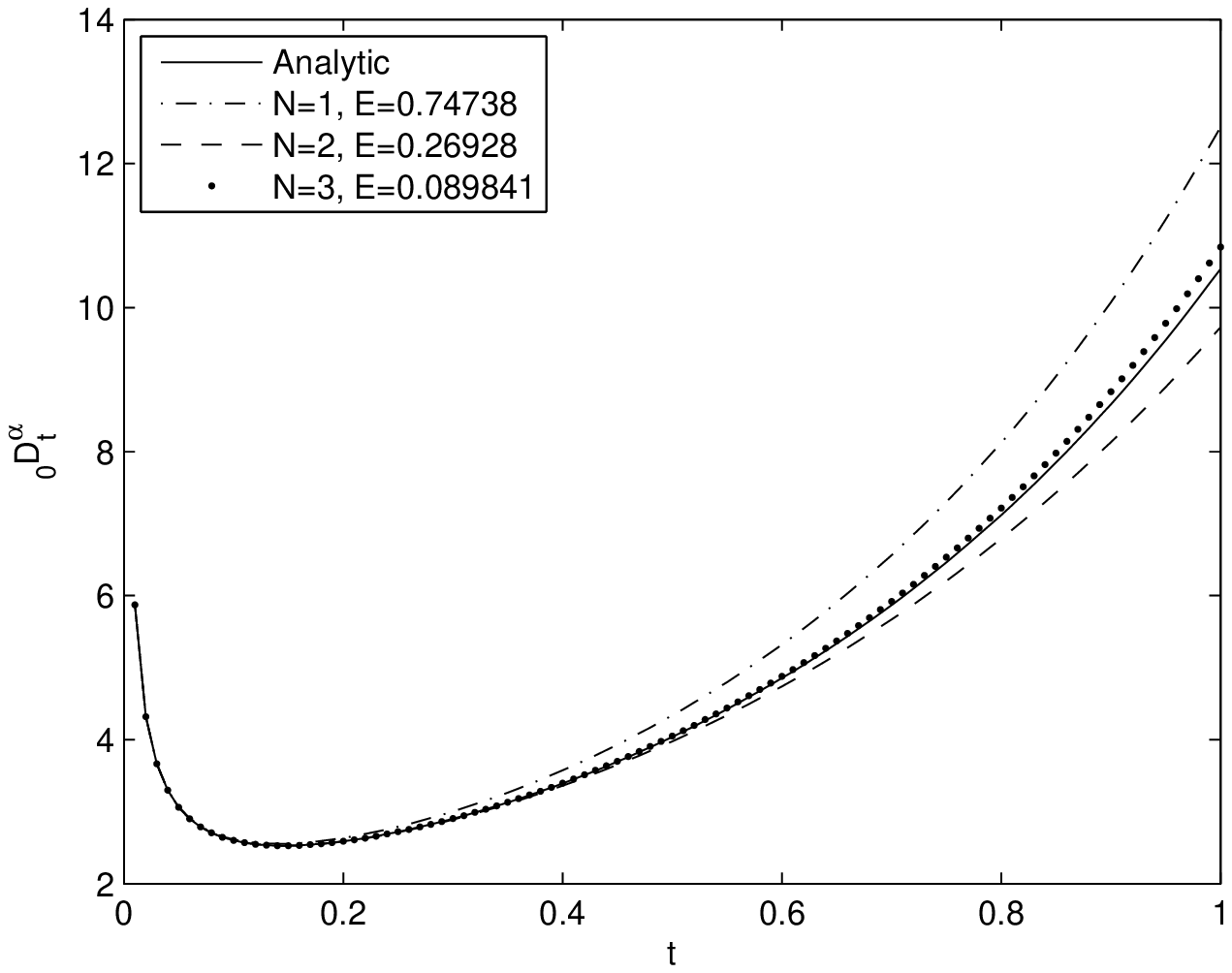}}
  \end{center}
  \caption{Analytic (solid line) {\it versus} numerical approximation \eqref{expanInt}.}
  \label{EvalInt}
\end{figure}
The same computations are carried out using approximation \eqref{expanMom}.
In this case, given a function $x(\cdot)$, we can compute $V_p$ by definition
or integrate the system \eqref{sysVp} analytically or by any numerical integrator.
As it is clear from Figure~\ref{EvalMom},
one can get better results by using larger values of $N$.
\begin{figure}[!ht]
  \begin{center}
    \subfigure[$\LDz(t^4)$]{\includegraphics[scale=0.54]{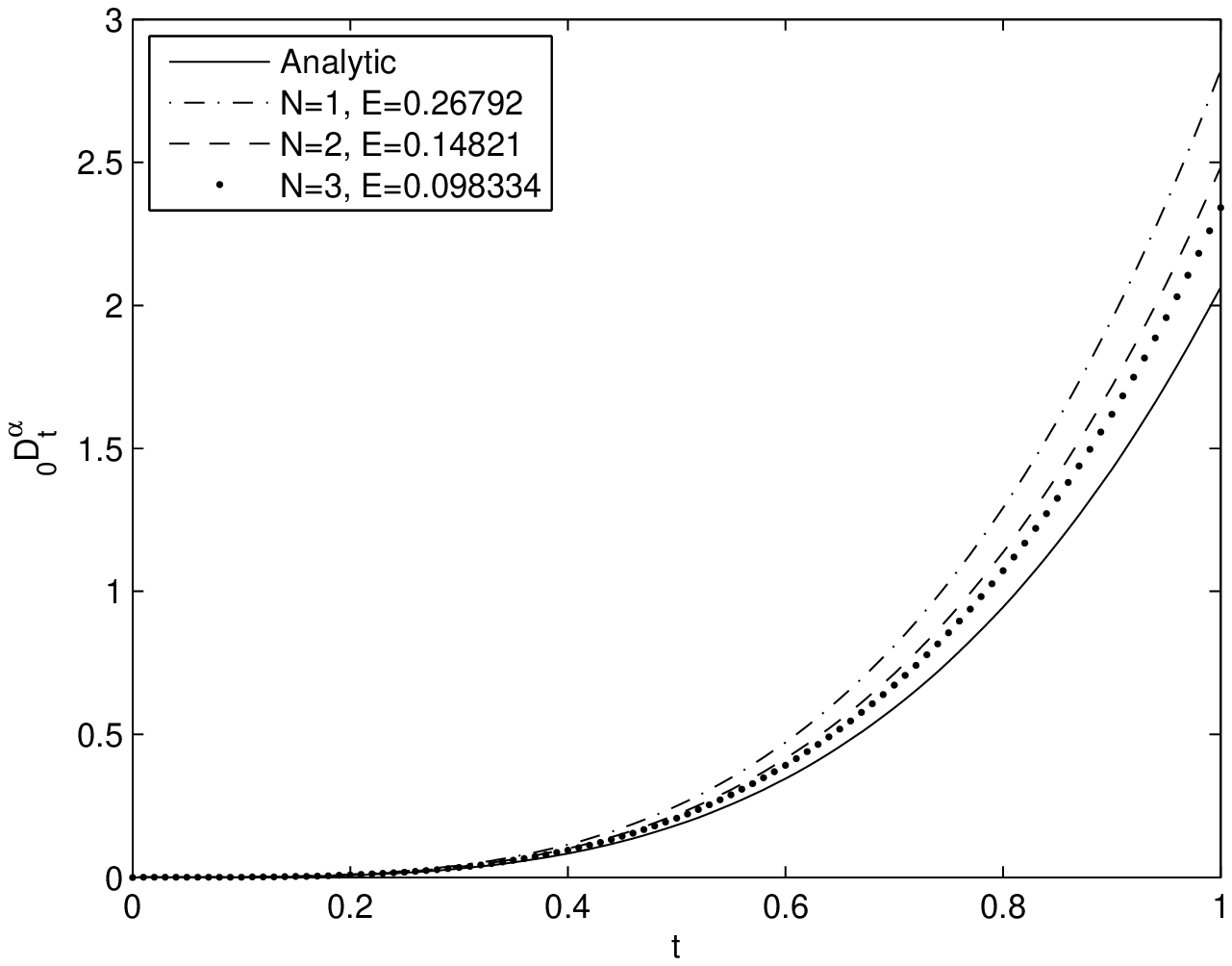}}
    \subfigure[$\LDz(e^{2t})$]{\includegraphics[scale=0.54]{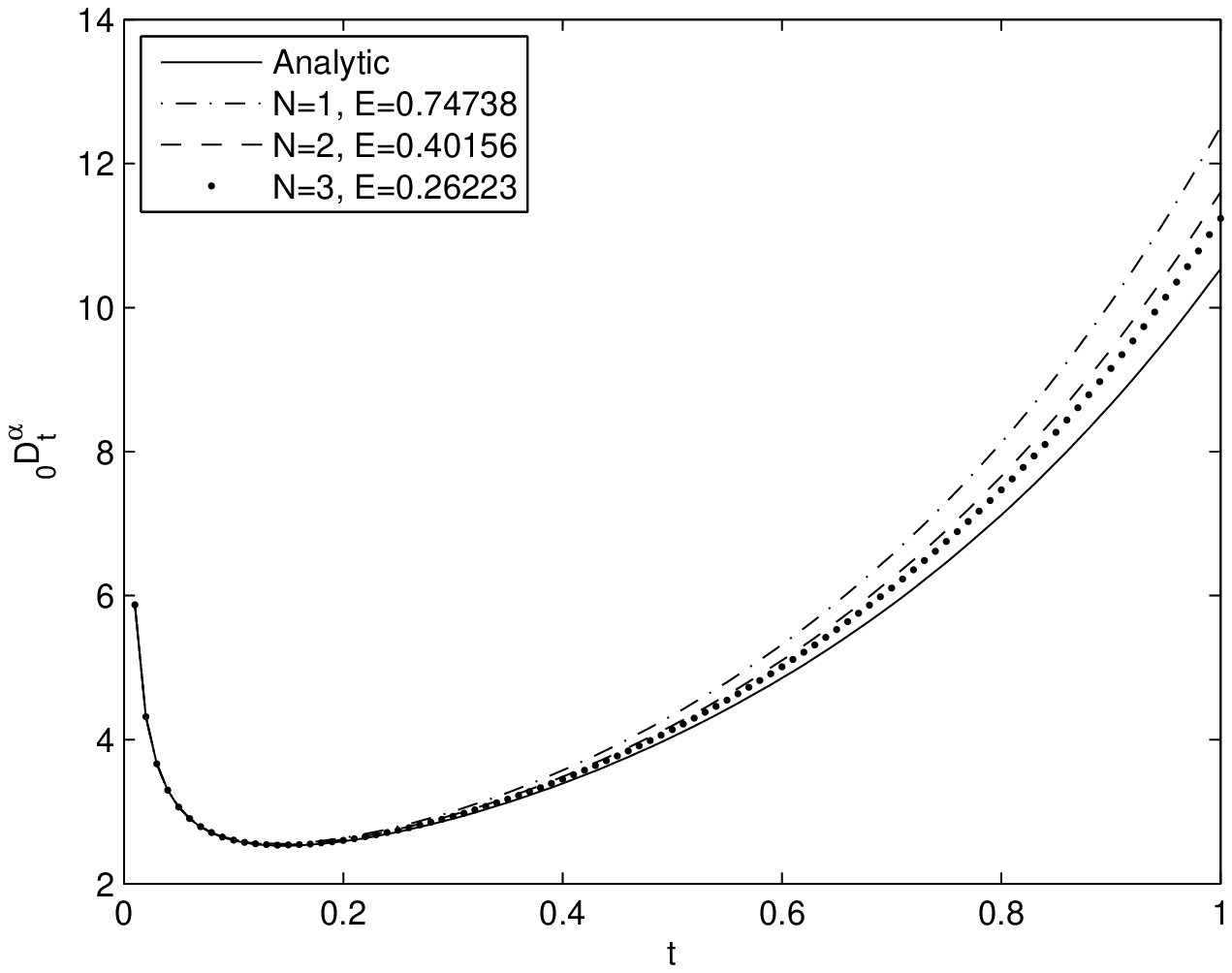}}
  \end{center}
  \caption{Analytic  (solid line) {\it versus} numerical approximation \eqref{expanMom}.}
  \label{EvalMom}
\end{figure}
Comparing Figures~\ref{EvalInt} and \ref{EvalMom}, we find out that the approximation
\eqref{expanInt} shows a faster convergence. Observe that both functions are analytic
and it is easy to compute higher-order derivatives. The approximation \eqref{expanInt}
fails for non-analytic functions as stated in \cite{Atan2}.

\begin{remark}
A closer look to \eqref{expanInt} and \eqref{expanMom} reveals
that in both cases the approximations are not computable
at $a$ and $b$ for the left and right fractional derivatives, respectively.
At these points we assume that it is possible to extend them continuously
to the closed interval $[a,b]$.
\end{remark}

In what follows, we show that by omitting the first derivative
from the expansion, as done in \cite{Atan2}, one may loose
a considerable accuracy in computation. Once again, we compute
the fractional derivatives of $x(t)=t^4$ and $x(t)=e^{2t}$,
but this time we use the approximation given by \eqref{expanAtan}.
Figure~\ref{compareEval} summarizes the results. The expansion
up to the first derivative
gives a more realistic approximation using quite small $N$,
$3$ in this case.
\begin{figure}[!ht]
\begin{center}
\subfigure[$\LDz(t^4)$]{\includegraphics[scale=0.54]{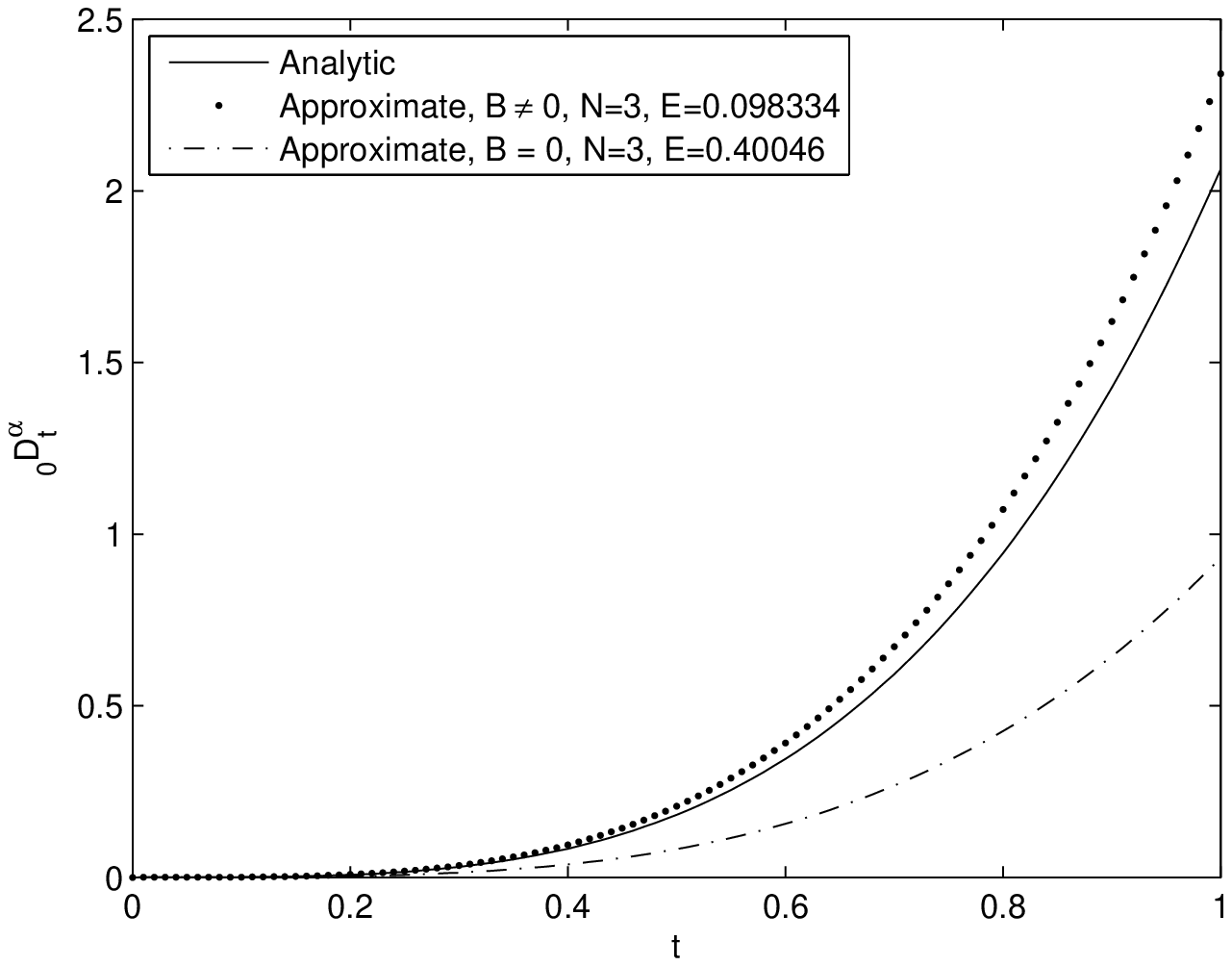}}
\subfigure[$\LDz(e^{2t})$]{\includegraphics[scale=0.54]{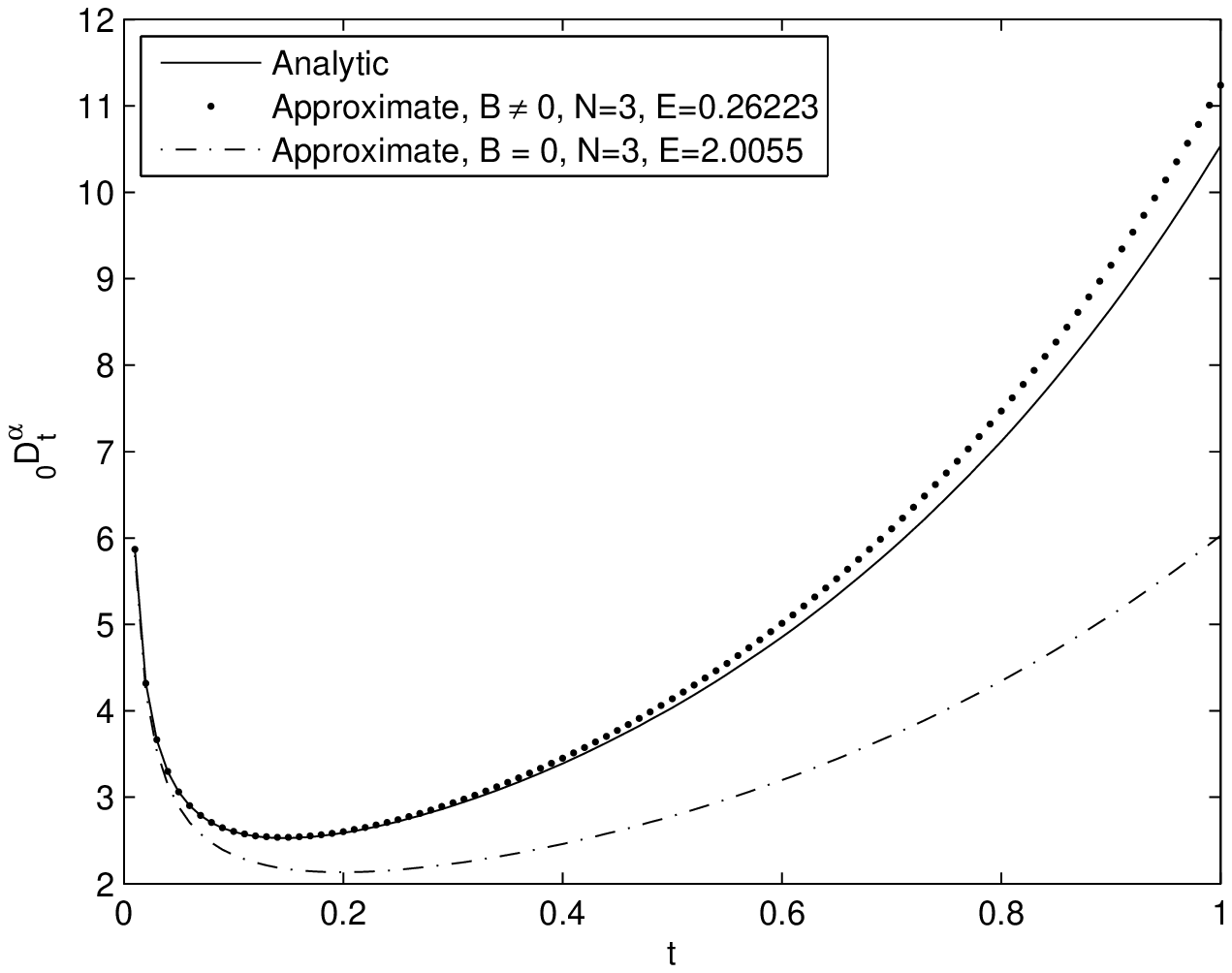}}
\end{center}
\caption{Comparison of approximation \eqref{expanMom} proposed here
and approximation \eqref{expanAtan} of \cite{Atan2}.}
\label{compareEval}
\end{figure}
To show how the appearance of higher-order derivatives in generalization \eqref{Gen}
gives better results, we evaluate fractional derivatives of $x(t)=t^4$ and $x(t)=e^{2t}$
for different values of $n$. We consider $n=1,2,3$, $N=6$ for $x(t)=t^4$
(Figure~\ref{ExpEt}) and $N=4$ for $x(t)=e^{2t}$ (Figure~\ref{ExpExn}).
\begin{figure}[!ht]
\begin{center}
\subfigure[$\LDz(t^4)$]{\label{ExpEt}\includegraphics[scale=0.54]{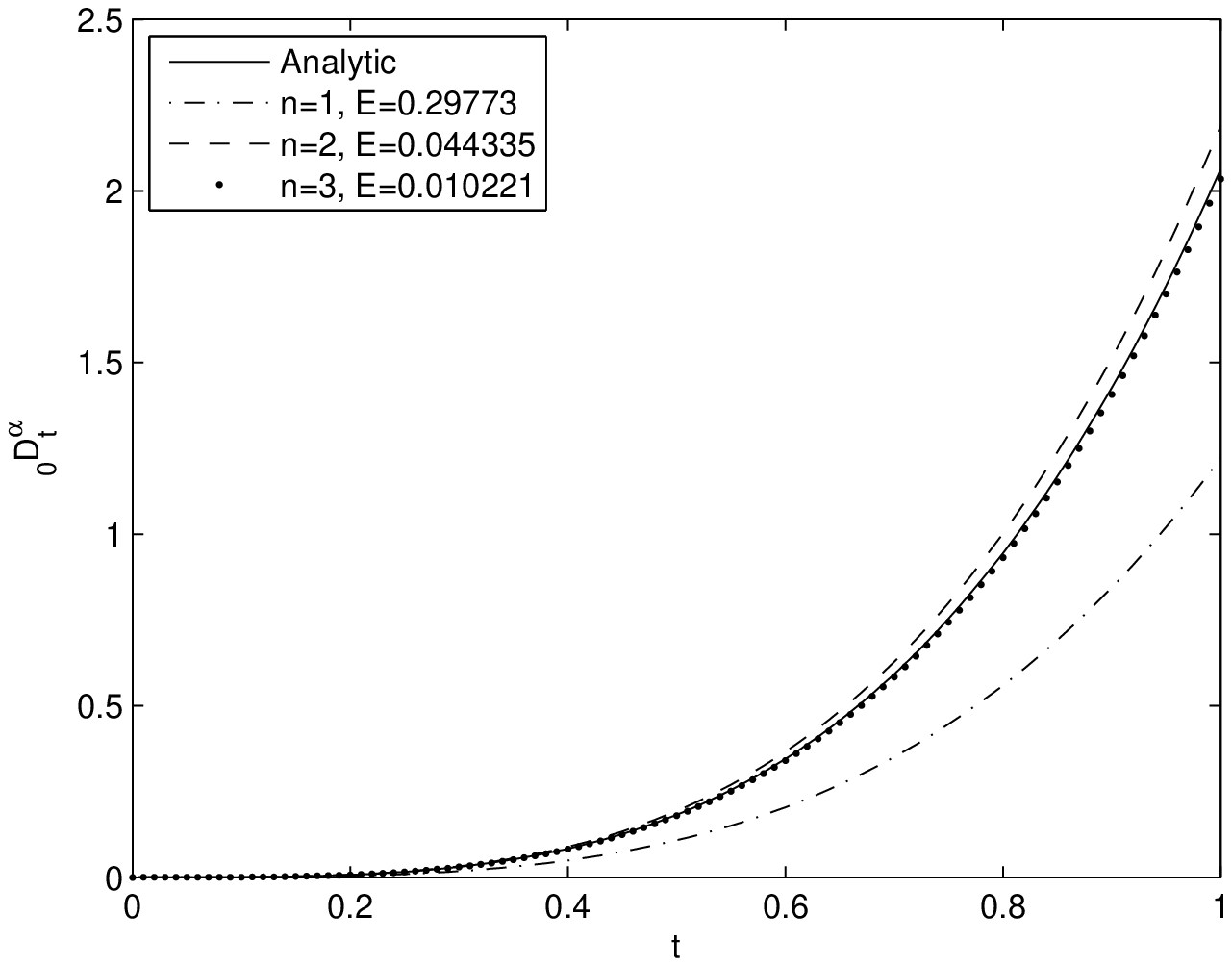}}
\subfigure[$\LDz(e^{2t})$]{\label{ExpExn}\includegraphics[scale=0.54]{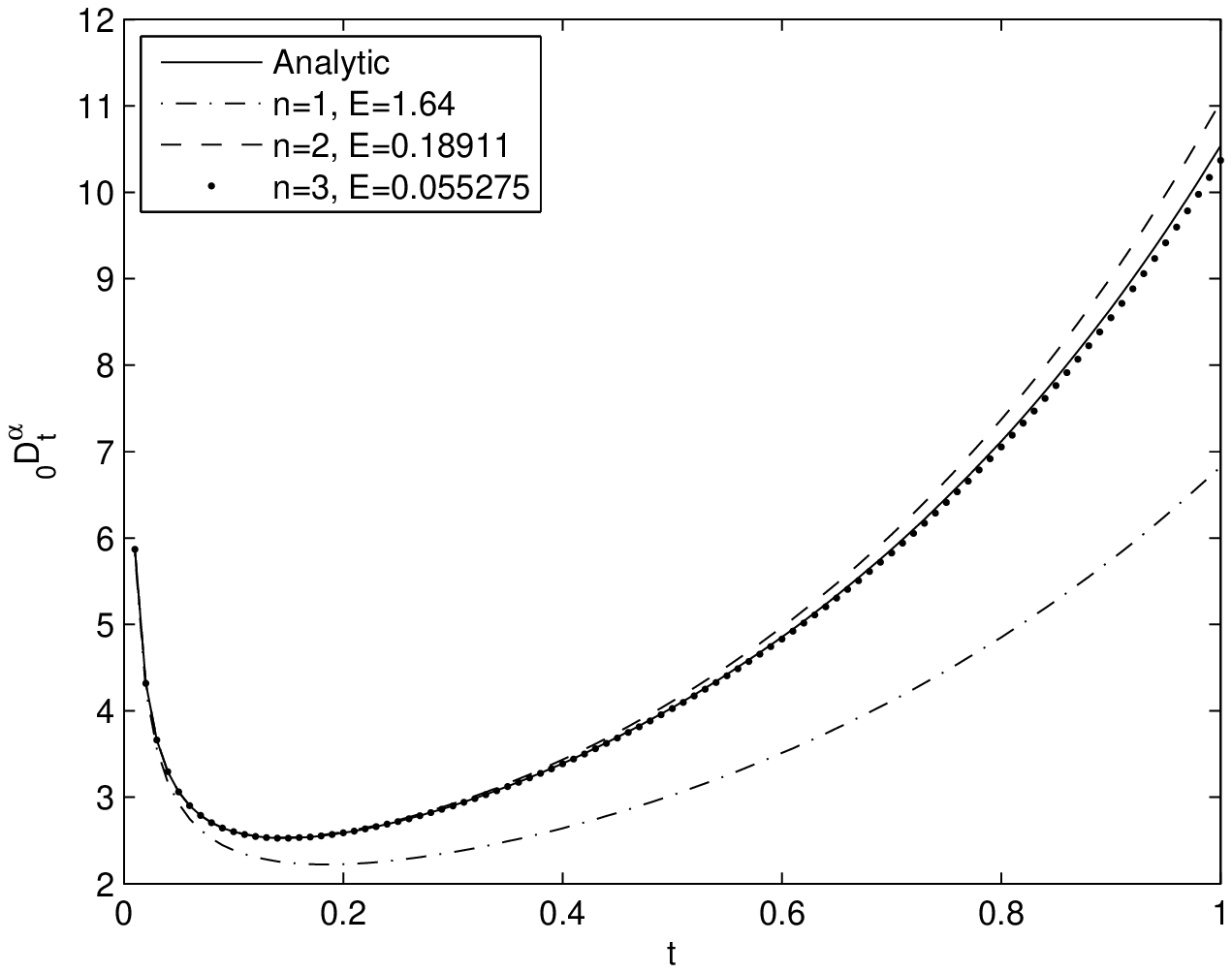}}
\end{center}
\caption{Analytic (solid line) {\it versus} numerical approximation \eqref{Gen}.}
\end{figure}


\subsection{Fractional derivatives of tabular data}

In many situations, the function itself is not accessible in a closed form, but
as a tabular data for discrete values of the independent variable.
Thus, we cannot use the definition to compute the fractional derivative directly.
Approximation \eqref{expanMom} that uses the function and its first derivative
to evaluate the fractional derivative, seems to be a good candidate in those cases.
Suppose that we know the values of $x(t_i)$ on $n+1$ distinct points in a given interval $[a,b]$,
\textrm{i.e.}, for $t_i$, $i=0,1,\ldots,n$, with $t_0=a$ and $t_n=b$.
According to formula \eqref{expanMom}, the value of the fractional derivative
of $x(\cdot)$ at each point $t_i$ is given approximately by
$$
\LDai x(t_i)\approx A(\a,N)(t_i-a)^{-\a}x(t_i)
+B(\a,N)(t_i-a)^{1-\a}\dot{x}(t_i)
-\sum_{p=2}^NC(p,\a)(t_i-a)^{1-p-\a}V_p(t_i).
$$
The values of $x(t_i)$, $i=0,1,\ldots,n$, are given. A good approximation for
$\dot{x}(t_i)$ can be obtained using the forward,
centered, or backward difference approximation
of the first-order derivative \cite{Stoer}. For $V_p(t_i)$ one can either
use the definition and compute the integral numerically,
\textrm{i.e.}, $V_p(t_i)=\int_a^{t_i} (1-p)(\tau-a)^{p-2}x(\tau)d\tau$,
or it is possible to solve \eqref{sysVp} as an initial value problem.
All required computations are straightforward and only need
to be implemented with the desired accuracy. The only thing to take
care is the way of choosing a good order, $N$, in the formula \eqref{expanMom}.
Because no value of $N$, guaranteeing the error
to be smaller than a certain preassigned number, is known a priori,
we start with some prescribed value for $N$ and increase it step by step.
In each step we compare, using an appropriate norm,
the result with the one of previous step.
For instance, one can use the Euclidean norm
$\|(\LDa)^{new}-(\LDa)^{old} \|_2$ and terminate the procedure
when it's value is smaller than a predefined $\epsilon$.
For illustrative purposes, we compute the fractional derivatives
of order $\a=0.5$ for tabular data extracted from $x(t)=t^4$ and $x(t)=e^{2t}$.
The results are given in Figure~\ref{tabular}.
\begin{figure}[!ht]
  \begin{center}
    \subfigure[$\LDz(t^4)$]{\includegraphics[scale=0.54]{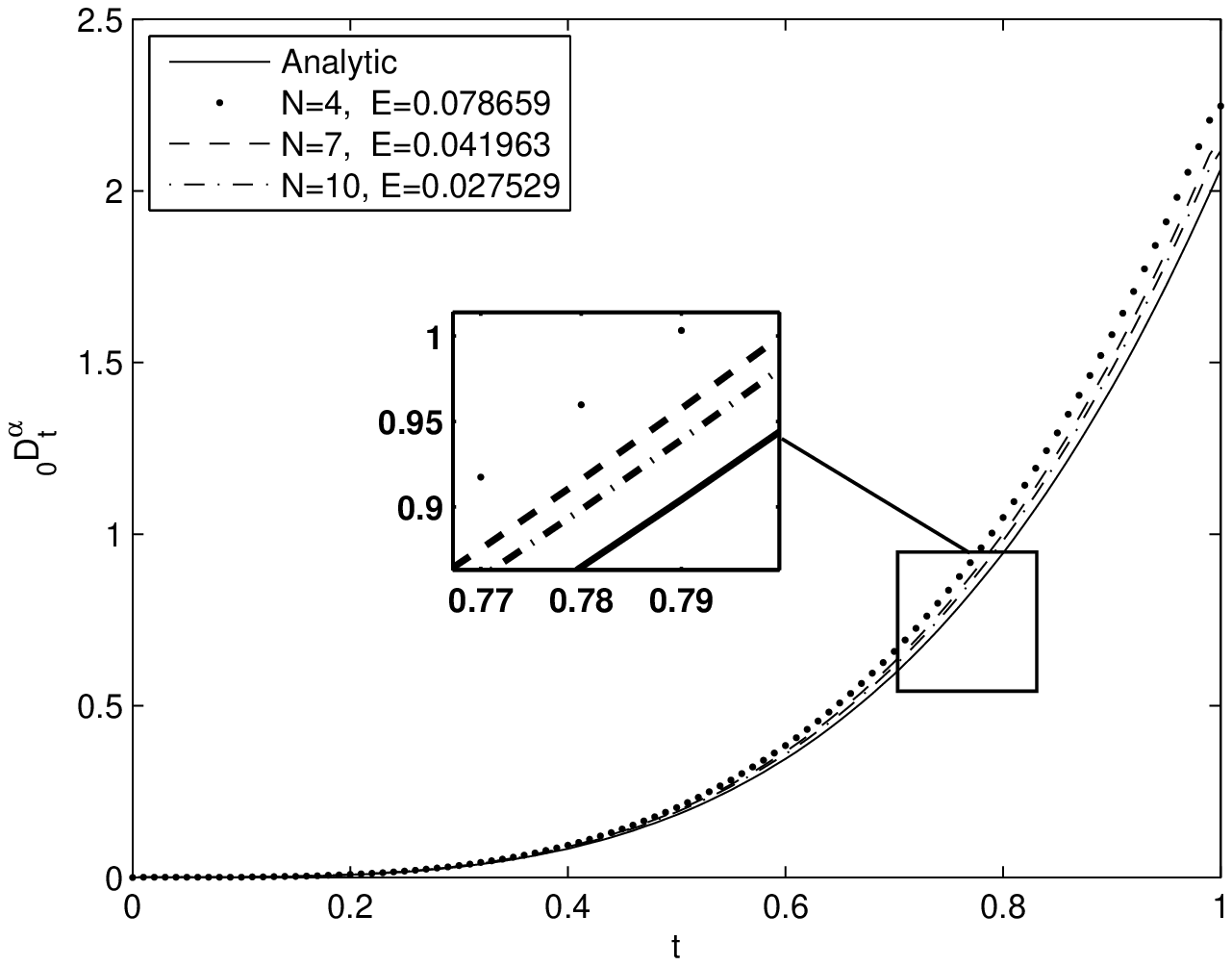}}
    \subfigure[$\LDz(e^{2t})$]{\includegraphics[scale=0.54]{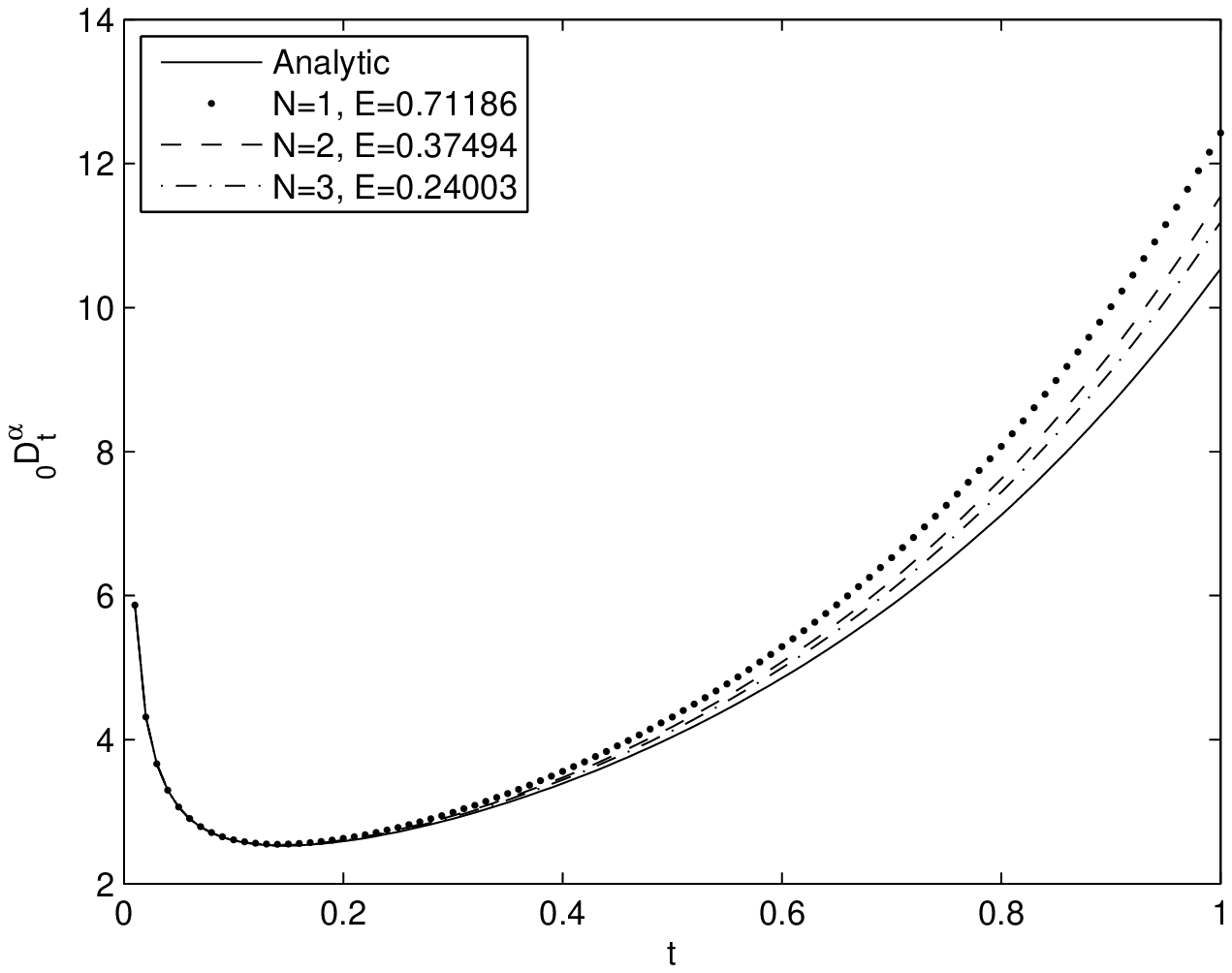}}
  \end{center}
  \caption{Fractional derivatives of tabular data.}
  \label{tabular}
\end{figure}


\subsection{Applications to fractional differential equations}

The classical theory of ordinary differential equations is a well developed field
with many tools available for numerical purposes. Using the approximations
\eqref{expanInt} and \eqref{expanMom}, one can transform a fractional
ordinary differential equation\index{Fractional! differential equation} into a classical ODE.

We should mention here that, using \eqref{expanInt}, derivatives
of higher-order appear in the resulting ODE, while we only have
a limited number of initial or boundary conditions available.
In this case the value of $N$, the order of approximation,
should be equal to the number of given conditions.
If we choose a larger $N$, we will encounter
lack of initial or boundary conditions. This problem is not present
in the case in which we use the approximation \eqref{expanMom},
because the initial values for the auxiliary variables $V_p$, $p=2,3,\ldots$,
are known and we don't need any extra information.

Consider, as an example, the following initial value problem:
\begin{equation}
\label{ex1}
\left\{
\begin{array}{l}
\LDz x(t)+x(t)=t^2+\frac{2}{\Gamma(2.5)}t^{\frac{3}{2}},\\
x(0)=0.
\end{array}
\right.
\end{equation}
We know that $\LDz (t^2)=\frac{2}{\Gamma(2.5)}t^{\frac{3}{2}}$.
Therefore, the analytic solution for system \eqref{ex1} is $x(t)=t^2$.
Because only one initial condition is available, we can only expand
the fractional derivative up to the first derivative
in \eqref{expanInt}. One has
\begin{equation}
\label{ExampOdeInt}
\left\{
\begin{array}{l}
1.5642~t^{-0.5}x(t)+0.5642~t^{0.5}\dot{x}(t)=t^2+1.5045~t^{1.5},\\
x(0)=0.
\end{array}
\right.
\end{equation}
This is a classical initial value problem
and can be easily treated numerically.
The solution is drawn in Figure~\ref{odeInt}.
As expected, the result is not satisfactory.
Let us now use the approximation given by \eqref{expanMom}.
The system in \eqref{ex1} becomes
\begin{equation}
\label{ex2}
\left\{
\begin{array}{l}
A(N)t^{-0.5}x(t)+B(N)t^{0.5}\dot{x}(t)
-\sum_{p=2}^NC(p)t^{0.5-p}V_p+x(t)=t^2+\frac{2}{\Gamma(2.5)}t^{1.5},\\
\dot{V}_p(t)=(1-p)(t-a)^{p-2}x(t), \quad p=2,3,\ldots,N,\\
x(0)=0,\\
V_p(0)=0, \quad p=2,3,\ldots,N.
\end{array}
\right.
\end{equation}
We solve this initial value problem for $N=7$.
The MATLAB$^\circledR$\index{MATLAB} \textsf{ode45} built-in function is used
to integrate system \eqref{ex2}. The solution is given
in Figure~\ref{odeVp} and shows a better approximation
when compared with \eqref{ExampOdeInt}.
\begin{figure}[!ht]
\begin{center}
\subfigure[Exact {\it versus} Approximation
\eqref{expanInt}.]{\label{odeInt}\includegraphics[scale=0.535]{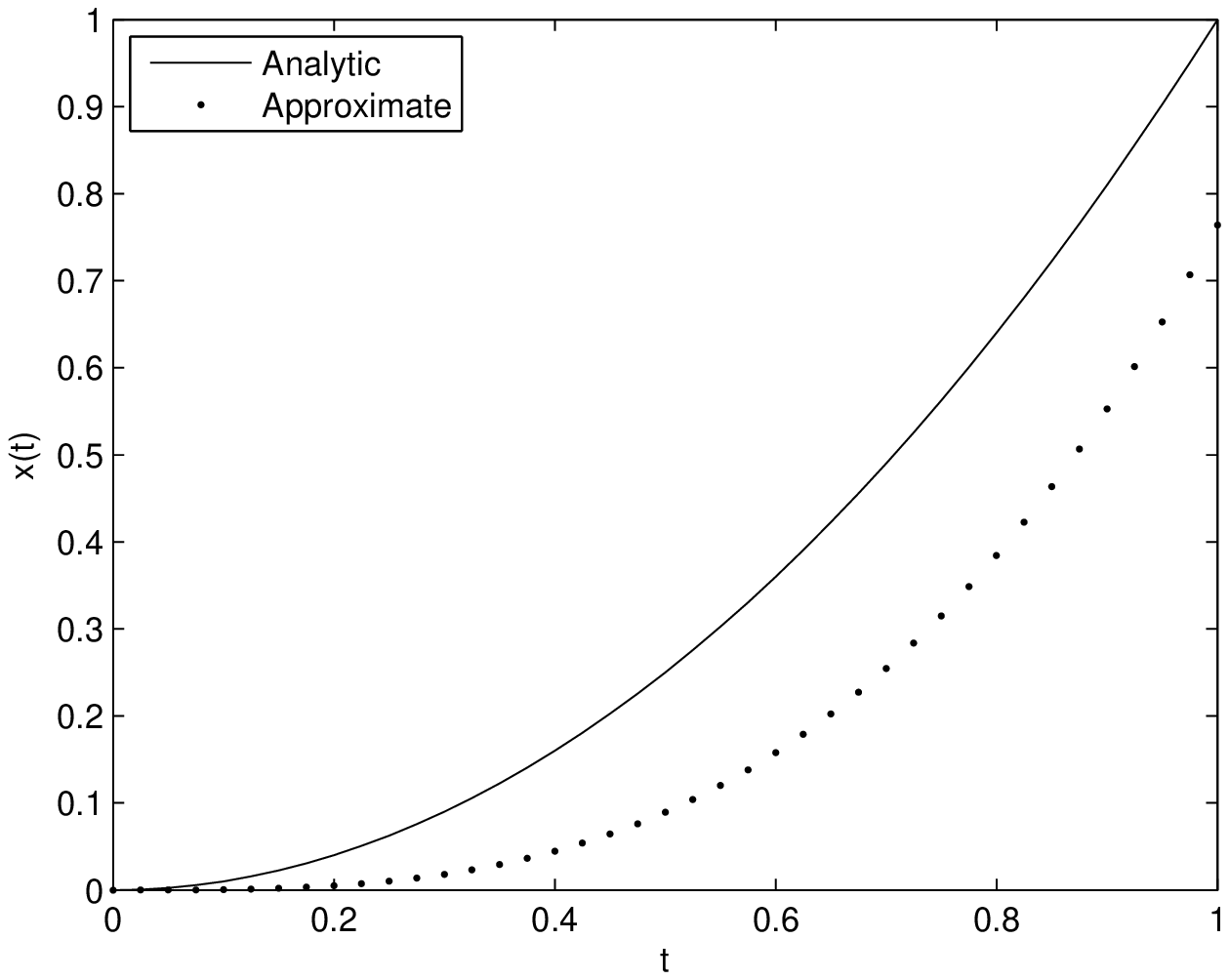}}
\subfigure[Exact {\it versus} Approximation \eqref{expanMom}.]{\label{odeVp}
\includegraphics[scale=0.535]{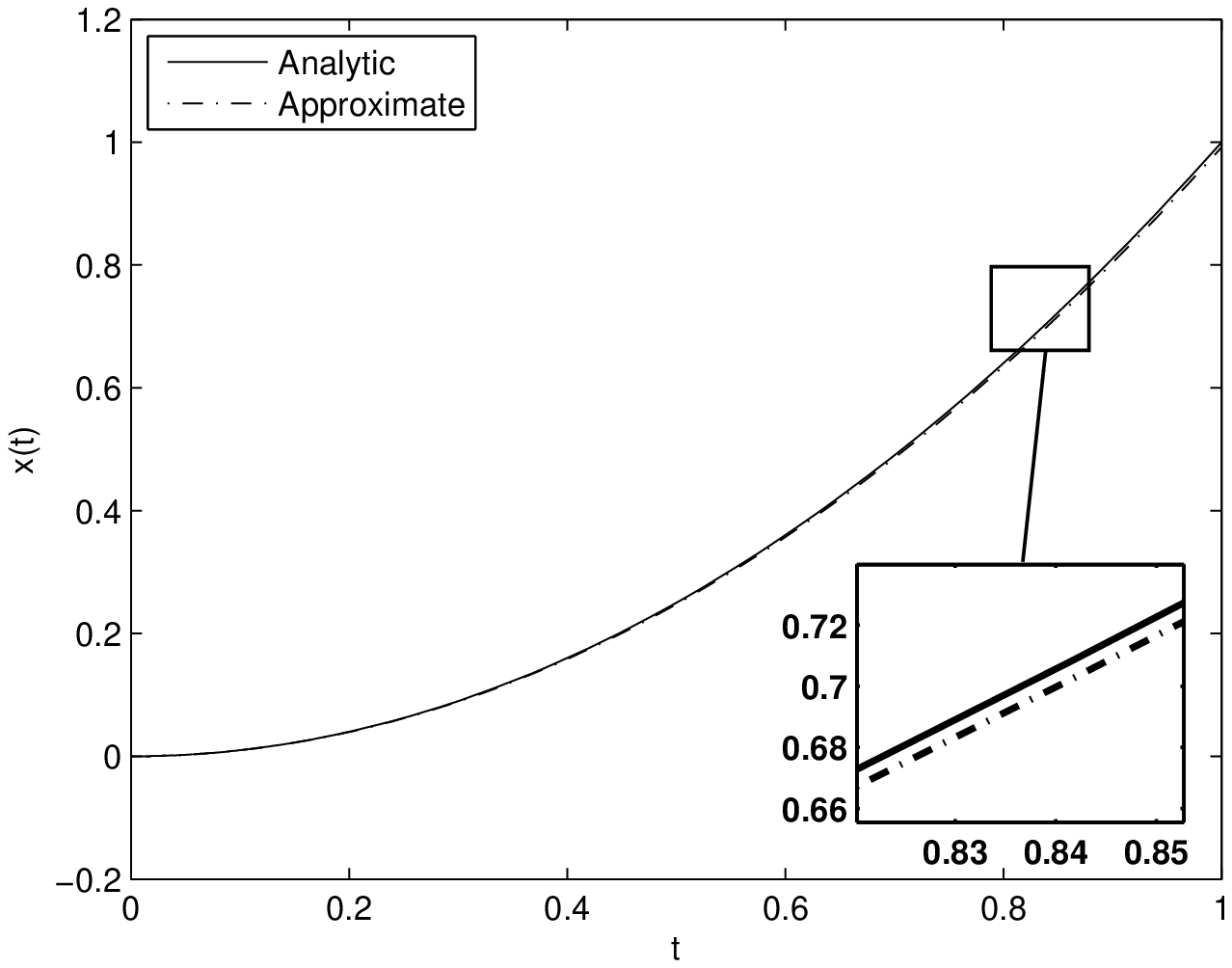}}\\
\end{center}
\caption{Two approximations applied to fractional differential equation \eqref{ex1}.}
\end{figure}

\begin{remark}
To show the difference caused by the appearance of the first derivative
in formula \eqref{expanMom}, we solve the initial value problem \eqref{ex1}
with $B(\a,N)=0$. Since the original fractional differential equation
does not depend on integer-order derivatives of function $x(\cdot)$,
\textrm{i.e.}, it has the form
$$
\LDa x(t)+f(x,t)=0,
$$
by \eqref{expanAtan} the dependence to derivatives of  $x(\cdot)$ vanishes.
In this case one needs to apply the operator $_aD_t^{1-\a}$
to the above equation and obtain
$$
\dot{x}(t)+_aD_t^{1-\a}[f(x,t)]=0.
$$
Nevertheless, we can use \eqref{expanMom} directly without any trouble.
Figure~\ref{cmpAtanMe} shows that at least for a moderate accurate method,
like the MATLAB$^\circledR$ routine \textsf{ode45}, taking $B(\a,N)\neq 0$ into account
gives a better approximation.
\begin{figure}[!ht]
  \begin{center}
    \includegraphics[scale=0.7]{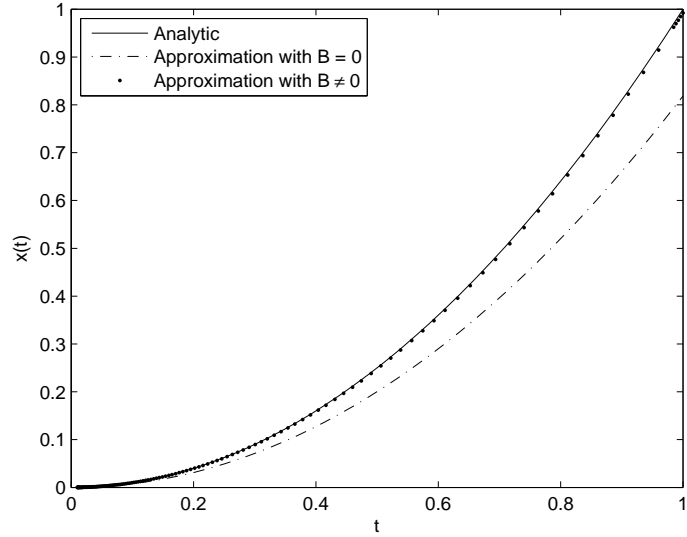}
  \end{center}
  \caption{Comparison of our approach to that of \cite{Atan2}.}
  \label{cmpAtanMe}
\end{figure}
\end{remark}


\section{Hadamard derivatives}

For Hadamard derivatives, the expansions can be obtained 
in a quite similar way and are introduced next \cite{PATHad}.


\subsection{Approximation by a sum of integer-order derivatives}

Assume that a function $x(\cdot)$ admits derivatives of any order, 
then expansion formulas for the Hadamard fractional integrals 
and derivatives of $x$, in terms of its integer-order derivatives, are given in 
\cite[Theorem~17]{Butzer2}:\index{Hadamard fractional derivative! Left}\index{Hadamard fractional derivative! Right}
$$
{_0\mathcal{I}_t^\a}x(t)=\sum_{k=0}^\infty S(-\a,k)t^k x^{(k)}(t)
$$
and
$$
{_0\mathcal{D}_t^\a}x(t)=\sum_{k=0}^\infty S(\a,k)t^k x^{(k)}(t),
$$
where
$$
S(\a,k)=\frac{1}{k!}\sum_{j=1}^k(-1)^{k-j} {k \choose j} j^{\a}
$$
is the Stirling function.

As approximations we truncate infinite sums 
at an appropriate order $N$ and get the following forms:
$$
{_0\mathcal{I}_t^\a}x(t)=\sum_{k=0}^N S(-\a,k)t^k x^{(k)}(t)
$$
and
$$
{_0\mathcal{D}_t^\a}x(t)=\sum_{k=0}^N S(\a,k)t^k x^{(k)}(t).
$$


\subsection{Approximation using moments of a function}

The same idea of expanding Riemann--Liouville derivatives, 
with slightly different techniques, is used to derive expansion formulas 
for left and right Hadamard derivatives\index{Hadamard fractional derivative! Left}. 
The following lemma is the basis for such new relations.

\begin{lemma}
Let $\a\in(0,1)$ and $x(\cdot)$ be an absolutely continuous function on $[a,b]$. 
Then the Hadamard fractional derivatives may be expressed by
\begin{equation}
\label{LHDformula}
\LHD x(t)=\frac{x(a)}{\Gamma(1-\a)}\left(\ln\frac{t}{a}\right)^{-\a}
+\frac{1}{\Gamma(1-\a)}\int_a^t \left(\ln\frac{t}{\t}\right)^{-\a}\dot{x}(\t)d\t
\end{equation}
and
\begin{equation}
\label{RHDformula}
\RHD x(t)=\frac{x(b)}{\Gamma(1-\a)}\left(\ln\frac{b}{t}\right)^{-\a}
-\frac{1}{\Gamma(1-\a)}\int_t^b \left(\ln\frac{\t}{t}\right)^{-\a}\dot{x}(\t)d\t.
\end{equation}
\end{lemma}

A proof of this lemma for an arbitrary $\a>0$ can be found in \cite[Theorem~3.2]{Kilbas2}.

Applying similar techniques as presented in Theorem~\ref{thm:oft} to the formulas 
\eqref{LHDformula} and \eqref{RHDformula} gives the following theorem.

\begin{theorem}
\label{ThmHad32}
Let $n\in\mathbb N$, $0<a<b$ and $x:[a,b]\to\mathbb R$
be a function of class $C^{n+1}$. Then
$$
\LHD x(t)\approx \sum_{i=0}^{n}A_i(\alpha,N)\left(\ln\frac{t}{a}\right)^{i-\a}
x_{i,0}(t)+\sum_{p=n+1}^N B(\a,p)\left(\ln\frac{t}{a}\right)^{n-\a-p}V_p(t)
$$
with
$$
\begin{array}{rl}
A_i(\a,N)&=\displaystyle\frac{1}{\Gamma(i+1-\a)}\left[1
+\sum_{p=n-i+1}^N\frac{\Gamma(p+\a-n)}{\Gamma(\a-i)(p-n+i)!}\right],
\quad i \in\{0,\ldots,n\},\\[12pt]
B(\a,p)&=\displaystyle\frac{\Gamma(p+\a-n)}{\Gamma(-\a)\Gamma(1+\a)(p-n)!},
\quad p \in \{n+1,\ldots\}, \\[12pt]
V_p(t)&=\displaystyle\int_a^t (p-n)\left(\ln\frac{\tau}{a}\right)^{p-n-1}\frac{x(\tau)}{\tau}d\tau,
\quad p \in \{n+1,\ldots\}.
\end{array}
$$
\end{theorem}
\begin{remark}
The right Hadamard fractional derivative\index{Hadamard fractional derivative! Right} 
can be expanded in the same way. This gives the following approximation:
\begin{eqnarray*}
\RHD x(t)&\approx& A(\a,N)\left(\ln\frac{b}{t}\right)^{-\a}x(t)
-B(\a,N)\left(\ln\frac{b}{t}\right)^{1-\a}t\dot{x}(t)\\
&&\quad-\sum_{p=2}^N C(\a,p)\left(\ln\frac{b}{t}\right)^{1-\a-p}W_p(t),
\end{eqnarray*}
with
\begin{equation*}
W_p(t)=(1-p)\int_t^b \left(\ln\frac{b}{\tau}\right)^{p-2}\frac{x(\tau)}{\tau}d\tau.
\end{equation*}
\end{remark}

\begin{remark}
In the particular case $n=1$, one obtains from Theorem~\ref{ThmHad32} that
\begin{eqnarray}\label{HadAprx}
\LHD x(t)&\approx&A(\a,N)\left(\ln\frac{t}{a}\right)^{-\a}x(t)
+B(\a,N)\left(\ln\frac{t}{a}\right)^{1-\a}t\dot{x}(t)\nonumber\\
&&\quad+\sum_{p=2}^N C(\a,p)\left(\ln\frac{t}{a}\right)^{1-\a-p}V_p(t)
\end{eqnarray}
with
\begin{eqnarray*}
A(\a,N)&=&\frac{1}{\Gamma(1-\a)}\left(1+\sum_{p=2}^N\frac{\Gamma(p+\a-1)}{\Gamma(\a)(p-1)!}\right),\\
B(\a,N)&=&\frac{1}{\Gamma(2-\a)}\left(1+\sum_{p=1}^N\frac{\Gamma(p+\a-1)}{\Gamma(\a-1)p!}\right).
\end{eqnarray*}
\end{remark}


\subsection{Examples}

In this section we apply \eqref{HadAprx} to compute fractional derivatives, of order $\a=0.5$, 
for $x(t)=\ln(t)$ and $x(t)=t^4$. The exact Hadamard fractional derivative 
is available for $x(t)=\ln(t)$ and we have
$$
\LHDHz (\ln(t))=\frac{\sqrt{\ln t}}{\Gamma(1.5)}.
$$
For $x(t)=t^4$ only an approximation of Hadamard fractional derivative is found in the literature:
$$
\LHDHz t^4\approx \frac{1}{\Gamma(0.5)\sqrt{\ln t}}
+\frac{0.5908179503}{\Gamma(0.5)}4t^4 \mbox{erf}(3\sqrt{\ln t}),
$$
where $\mbox{erf}(\cdot)$ in the so-called Gauss error function,
$$
\mbox{erf}(t)=\frac{1}{\sqrt{\pi}}\int_0^t e^{-\t^2}\,d\t.
$$
The results of applying \eqref{HadAprx} to evaluate fractional 
derivatives are depicted in Figure~\ref{HadTestFig}.

\begin{figure}[ht!]
\begin{center}
\subfigure[$\LHDHz(\ln t)$]{\includegraphics[scale=0.54]{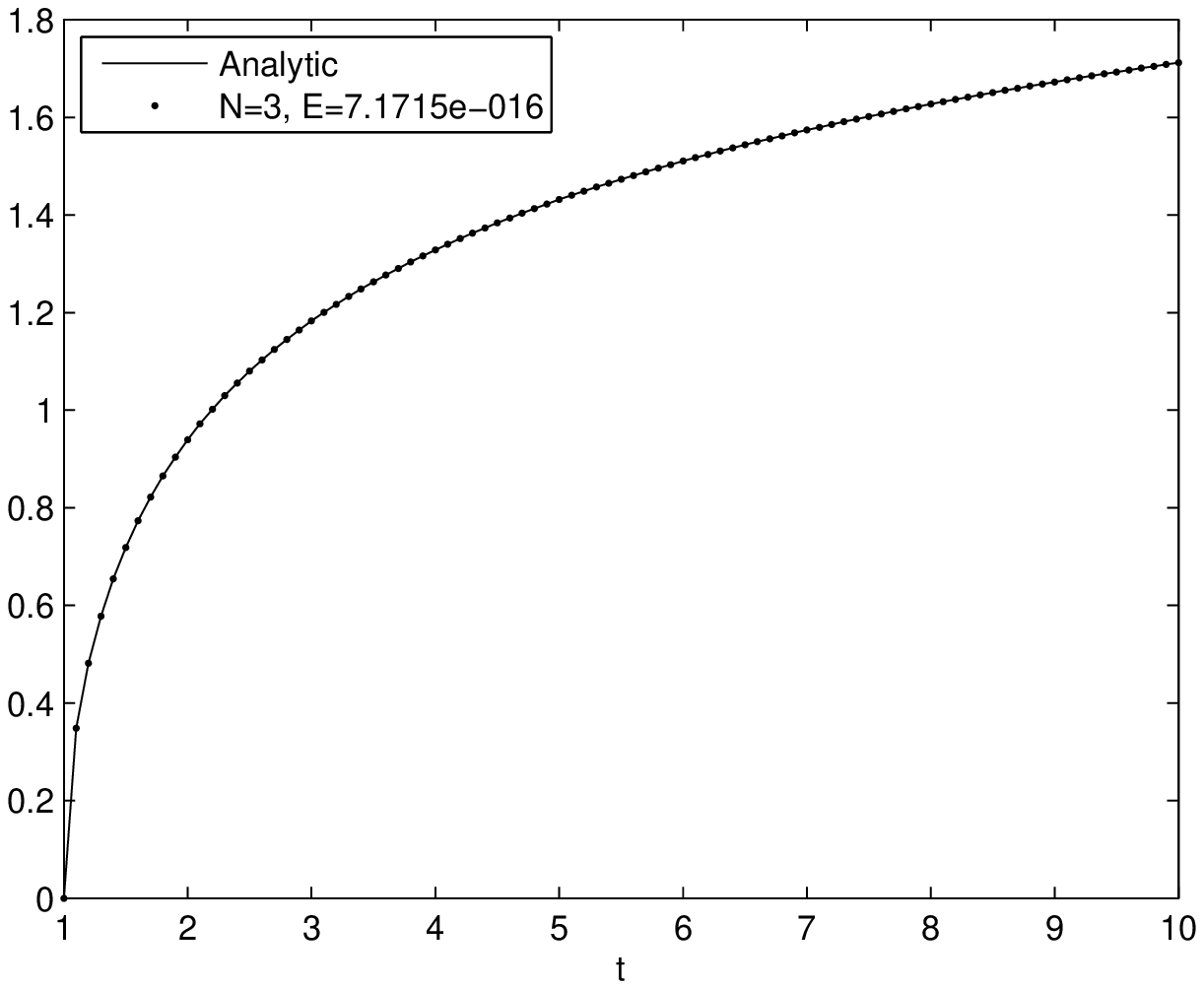}}
\subfigure[$\LHDHz(t^4)$]{\includegraphics[scale=0.54]{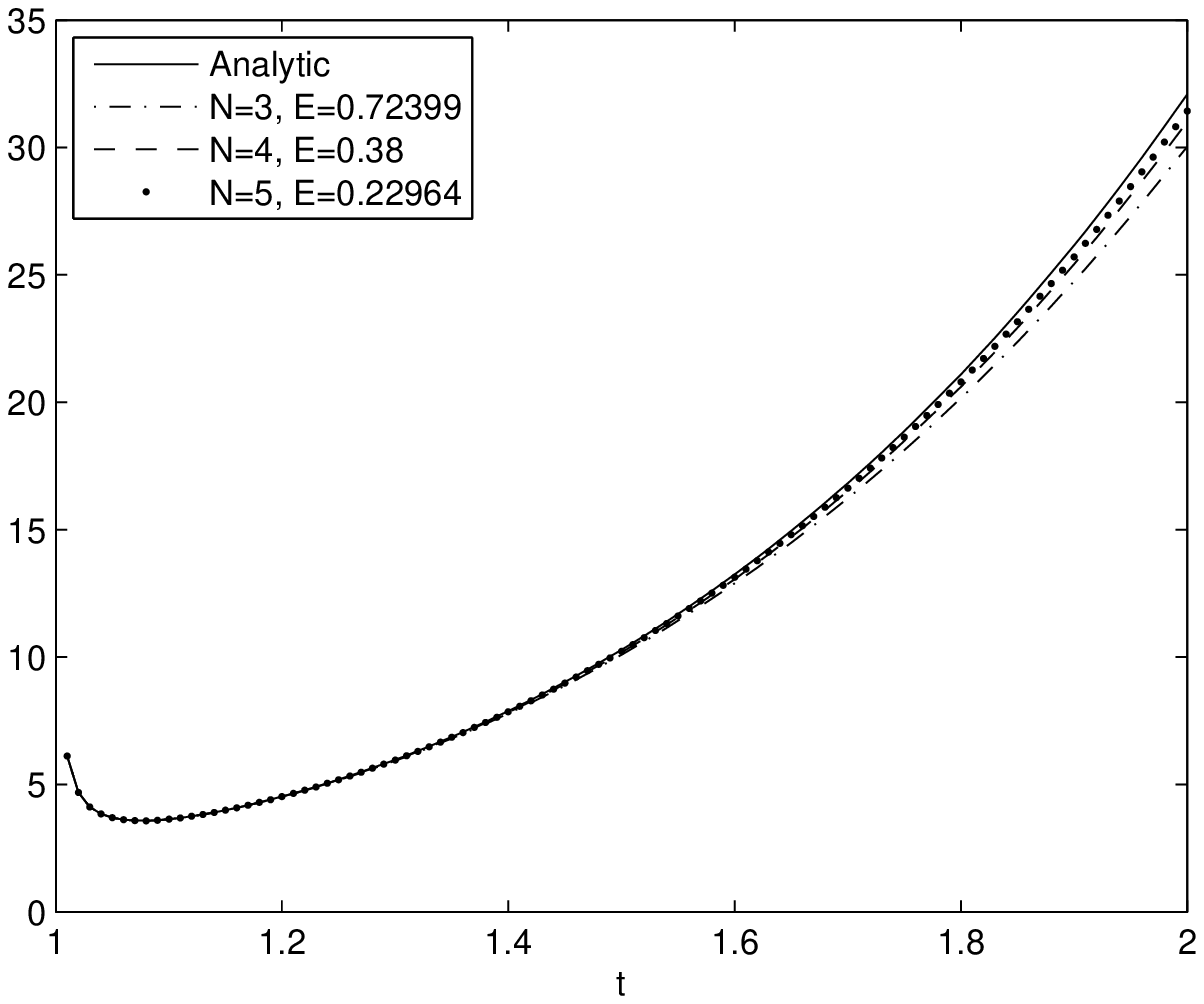}}
\end{center}
\caption{Analytic {\it versus} numerical approximation \eqref{HadAprx}.}
\label{HadTestFig}
\end{figure}
As another example, we consider the following fractional differential 
equation\index{Fractional! differential equation} 
involving a Hadamard fractional derivative:
\begin{equation}
\label{FDE1}
\left\{
\begin{array}{l}
\displaystyle\LHDHz x(t)+x(t)
=\frac{\sqrt{x(t)}}{\Gamma(1.5)}+\ln t\\
x(1)=0.
\end{array}
\right.
\end{equation}
Obviously, $x(t)=\ln t$ is a solution for \eqref{FDE1}.
Since we have only one initial condition, we replace the operator
$\LHDHz (\cdot)$ by the expansion with $n=1$ and thus obtaining
\begin{equation}
\label{eq:na:exp}
\left\{
\begin{array}{l}
\displaystyle\left[1+A_0(\ln t)^{-0.5}\right]x(t)
+A_1(\ln t)^{0.5}t \dot{x}(t)
+\sum_{p=2}^NB(0.5,p)(\ln t)^{0.5-p}V_p(t)
=\frac{\sqrt{x(t)}}{\Gamma(1.5)}+\ln t,\\
\displaystyle\dot{V_p}(t)=(p-1)(\ln t)^{p-2}\frac{x(t)}{t},
\quad p=2,3,\ldots,N,\\
x(1)=0,\\
V_p(1)=0, \quad p=2,3, \ldots, N.
\end{array}\right.
\end{equation}

In Figure~\ref{FDEn=1} we compare the analytical solution
of problem \eqref{FDE1} with the numerical result
for $N=2$ in \eqref{eq:na:exp}.
\begin{figure}
\begin{center}
\includegraphics[scale=.7]{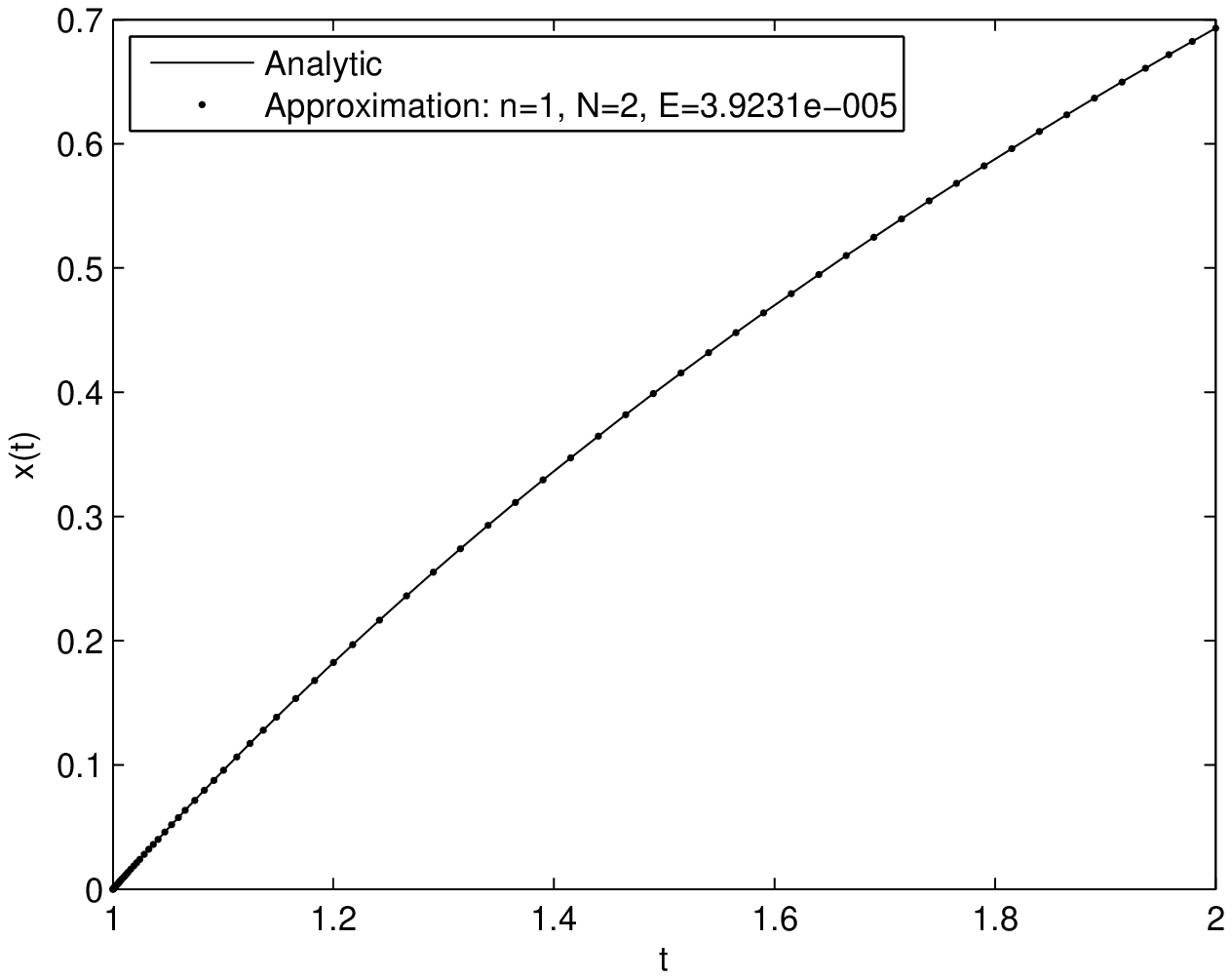}\\
\caption{Analytic {\it versus} numerical approximation
for problem \eqref{FDE1} with one initial condition.}\label{FDEn=1}
\end{center}
\end{figure}


\section{Error analysis}

When we approximate an infinite series by a finite sum, the choice of the order 
of approximation is a key question. Having an estimate knowledge of truncation 
errors\index{Error! Truncation}, one can choose properly up to which order 
the approximations should be made to suit the accuracy requirements. 
In this section we study the errors of the approximations presented so far.

Separation of an error term in \eqref{ExpIntErr} ends in
\begin{eqnarray}
\label{IntErr}
\LDa x(t)&=&\frac{1}{\Gamma(1-\a)}\frac{d}{dt}\int_a^t \left((t-\t)^{-\a}
\sum_{k=0}^{N}\frac{(-1)^k x^{(k)}(t)}{k!}(t-\t)^k\right)d\t\nonumber\\
& &+\frac{1}{\Gamma(1-\a)}\frac{d}{dt}\int_a^t \left((t-\t)^{-\a}
\sum_{k=N+1}^{\infty}\frac{(-1)^k x^{(k)}(t)}{k!}(t-\t)^k\right)d\t.
\end{eqnarray}
The first term in \eqref{IntErr} gives \eqref{expanInt} directly 
and the second term is the error caused by truncation. The next step 
is to give a local upper bound for this error, $E_{tr}(t)$.

The series
$$
\sum_{k=N+1}^{\infty}\frac{(-1)^k x^{(k)}(t)}{k!}(t-\t)^k,
\quad \t\in (a,t), \quad t\in (a,b),
$$
is the remainder of the Taylor expansion of $x(\t)$ and thus bounded 
by $\left|\frac{M}{(N+1)!}(t-\t)^{N+1}\right|$ in which
$$
M=\displaystyle\max_{\t \in [a,t]}|x^{(N+1)}(\t)|.
$$
Then,
$$
E_{tr}(t)\leq \left|\frac{M}{\Gamma(1-\alpha)(N+1)!}\frac{d}{dt}
\int_a^t (t-\t)^{N+1-\alpha}d\t\right|
=\frac{M}{\Gamma(1-\alpha)(N+1)!}(t-a)^{N+1-\alpha}.
$$

In order to estimate a truncation error\index{Error! Truncation} 
for approximation \eqref{expanMom}, the expansion procedure is carried 
out with separation of $N$ terms in binomial expansion as
\begin{eqnarray}
\label{expanError}
\left(1-\frac{\t-a}{t-a}\right)^{1-\a}&=&\sum_{p=0}^{\infty}
\frac{\Gamma(p-1+\a)}{\Gamma(\a-1)p!}\left(\frac{\t-a}{t-a}\right)^p\nonumber\\
&=&\sum_{p=0}^{N}\frac{\Gamma(p-1+\a)}{\Gamma(\a-1)p!}\left(\frac{\t-a}{t-a}\right)^p+R_N(\t),
\end{eqnarray}
where
$$
R_N(\t)=\sum_{p=N+1}^{\infty}\frac{\Gamma(p
-1+\a)}{\Gamma(\a-1)p!}\left(\frac{\t-a}{t-a}\right)^p.
$$
Integration by parts on the right-hand-side of \eqref{DecThm} gives
\begin{equation}\label{exp2}
\LD x(t)=\frac{x(a)}{\Gamma(1-\a)}(t-a)^{-\a}
+\frac{\dot{x}(a)}{\Gamma(2-\a)}(t-a)^{1-\a}+\frac{1}{\Gamma(2-\a)}
\int_a^t (t-\t)^{1-\a}\ddot{x}(\t)d\t.
\end{equation}
Substituting \eqref{expanError} into \eqref{exp2}, we get
\begin{eqnarray*}
\LD x(t)&=&\frac{x(a)}{\Gamma(1-\a)}(t-a)^{-\a}
+\frac{\dot{x}(a)}{\Gamma(2-\a)}(t-a)^{1-\a}\\
&& +\frac{(t-a)^{1-\a}}{\Gamma(2-\a)}\int_a^t
\left(\sum_{p=0}^{N}\frac{\Gamma(p-1+\a)}{\Gamma(\a-1)p!}\left(
\frac{\t-a}{t-a}\right)^p+R_N(\t)\right)\ddot{x}(\t)d\t\\
&=&\frac{x(a)}{\Gamma(1-\a)}(t-a)^{-\a}
+\frac{\dot{x}(a)}{\Gamma(2-\a)}(t-a)^{1-\a}\\
&& +\frac{(t-a)^{1-\a}}{\Gamma(2-\a)}\int_a^t
\left(\sum_{p=0}^{N}\frac{\Gamma(p-1+\a)}{\Gamma(\a-1)p!}
\left(\frac{\t-a}{t-a}\right)^p\right)\ddot{x}(\t)d\t\\
&& +\frac{(t-a)^{1-\a}}{\Gamma(2-\a)}\int_a^t R_N(\t)\ddot{x}(\t)d\t.
\end{eqnarray*}
At this point, we apply the techniques of \cite{Atan2} to the first three terms 
with finite sums. Then, we receive \eqref{expanMom} with an extra term of truncation error:
\begin{equation*}
E_{tr}(t)=\frac{(t-a)^{1-\a}}{\Gamma(2-\a)}\int_a^t R_N(\t)\ddot{x}(\t)d\t.
\end{equation*}
Since $0\leq\frac{\t-a}{t-a}\leq 1$ for $\t\in [a,t]$, one has
\begin{eqnarray*}
|R_N(\t)|&\leq & \sum_{p=N+1}^{\infty}\left| 
\frac{\Gamma(p-1+\a)}{\Gamma(\a-1)p!}\right|
=\sum_{p=N+1}^{\infty}\left|\binom{1-\a}{p}\right|
\leq\sum_{p=N+1}^{\infty}\frac{\mathrm{e}^{(1-\a)^2+1-\a}}{p^{2-\a}} \\
&\leq&\int_{p=N}^{\infty}\frac{\mathrm{e}^{(1-\a)^2+1-\a}}{p^{2-\a}}dp
=\frac{\mathrm{e}^{(1-\a)^2+1-\a}}{(1-\a)N^{1-\a}}.
\end{eqnarray*}
Finally, assuming 
$L_n=\displaystyle\max_{\t \in [a,t]}\left|x^{(n)}(\t)\right|$, 
we conclude that
\begin{equation*}
|E_{tr}(t)|\leq L_2 
\frac{\mathrm{e}^{(1-\a)^2+1-\a}}{\Gamma(2-\a)(1-\a)N^{1-\a}} (t-a)^{2-\a}.
\end{equation*}

In the general case, the error is given by the following result.

\begin{theorem}
If we approximate the left Riemann--Liouville fractional 
derivative\index{Riemann--Liouville fractional derivative! Left} by the finite sum
\eqref{ApproxDerGeneralcase}, then the error $E_{tr}(\cdot)$ is bounded by
\begin{equation}
\label{eq:formula:rv1}
|E_{tr}(t)|\leq L_n
\frac{\mathrm{e}^{(n-1-\a)^2+n-1-\a}}{\Gamma(n-\a)(n-1-\a)N^{n-1-\a}}(t-a)^{n-\a}.
\end{equation}
\end{theorem}
From \eqref{eq:formula:rv1} we see that if
the test function grows very fast or the point $t$ is far from $a$,
then the value of $N$ should also increase in order to have a good approximation.
Clearly, if we increase the value of $n$,
then we need also to increase the value of $N$ to control the error.

\begin{remark}
Following similar techniques, one can extract an error bound 
for the approximations of Hadamard derivatives\index{Hadamard fractional derivative! Left}.  
When we consider finite sums in \eqref{HadAprx}, the error is bounded by
$$
\left| E_{tr}(t)\right|\leq L(t)\frac{e^{(1-\a)^2+1-\a}}{\Gamma(2-\a)(1-\a)
N^{1-\a}}\left(\ln\frac{t}{a}\right)^{1-\a}(t-a),
$$
where
$$
L(t)=\max_{\tau\in[a,t]}|\dot{x}(\tau)+\t\ddot{x}(\t)|.
$$
For the general case, the expansion up to the derivative 
of order $n$, the error is bounded by
$$
\left| E_{tr}(t)\right|\leq L_n(t)\frac{e^{(n-\a)^2
+n-\a}}{\Gamma(n+1-\a)(n-\a)N^{n-\a}}
\left(\ln\frac{t}{a}\right)^{n-\a}(t-a),
$$
where
$$
L_n(t)=\max_{\tau\in[a,t]}|x_{n,1}(\tau)|.
$$
\end{remark}

\CP
\chapter{Approximating fractional integrals}
\label{AppFracInt}

We obtain a new decomposition of the Riemann--Liouville operators
of fractional integration as a series involving derivatives (of integer order).
The new formulas are valid for functions of class $C^n$,
$n \in \mathbb{N}$, and allow us to develop suitable numerical
approximations with known estimations for the error.
The usefulness of the obtained results,
in solving fractional integral equations,
is illustrated \cite{PATFracInt}.


\section{Riemann--Liouville fractional integral}

\subsection{Approximation by a sum of integer-order derivatives}

For analytical functions, we can rewrite the left Riemann--Liouville 
fractional integral\index{Riemann--Liouville fractional integral! Left}
as a series involving integer-order derivatives only.
If $x$ is analytic in $[a,b]$, then
\begin{equation}
\label{analytical}
\LI x(t)=\frac{1}{\Gamma(\a)}
\sum_{k=0}^\infty\frac{(-1)^k(t-a)^{k+\a}}{(k+\a)k!}x^{(k)}(t)
\end{equation}
for all $t\in[a,b]$ (\textrm{cf.} Eq. (3.44) in \cite{Miller}).
From the numerical point of view, one considers finite sums
and the following approximation:
\begin{equation}
\label{analytical2}
\LI x(t)\approx\frac{1}{\Gamma(\a)}\sum_{k=0}^N
\frac{(-1)^k(t-a)^{k+\a}}{(k+\a)k!}x^{(k)}(t).
\end{equation}

One problem with formula \eqref{analytical}
is that in order to have a ``good'' approximation 
we need to take a large value for $n$. In applications,
this approach may not be suitable. 
Here we present a new decomposition formula 
for functions of class $C^n$. The advantage is that 
even for $n=1$, we can achieve an appropriate accuracy.


\subsection{Approximation using moments of a function}
\index{Moment of a function}

Before we give the result in its full extension, we explain the method for $n=3$.
To that purpose, let $x\in C^3[a,b]$. Using integration by parts three times, we deduce that
\begin{eqnarray*}
\LI x(t)&=&\frac{x(a)}{\Gamma(\a+1)}(t-a)^\a
+\frac{\dot{x}(a)}{\Gamma(\a+2)}(t-a)^{\a+1}
+\frac{\ddot{x}(a)}{\Gamma(\a+3)}(t-a)^{\a+2}\\
&&+\frac{1}{\Gamma(\a+3)}\int_a^t (t-\tau)^{\a+2}x^{(3)}(\tau)d\tau.
\end{eqnarray*}
By the binomial formula, we can rewrite the fractional integral as
\begin{eqnarray*}
\LI x(t)&=&\frac{x(a)}{\Gamma(\a+1)}(t-a)^\a
+\frac{\dot{x}(a)}{\Gamma(\a+2)}(t-a)^{\a+1}
+\frac{\ddot{x}(a)}{\Gamma(\a+3)}(t-a)^{\a+2}\\
&&+\frac{(t-a)^{\a+2}}{\Gamma(\a+3)}\sum_{p=0}^\infty
\frac{\Gamma(p-\a-2)}{\Gamma(-\a-2)p!(t-a)^p}
\int_a^t (\tau-a)^px^{(3)}(\tau)d\tau.
\end{eqnarray*}
The rest of the procedure follows the same pattern: decompose the sum into
a first term plus the others, and integrate by parts\index{Integration by parts!}. Then assuming
$$
\begin{array}{ll}
A_0(\a)&=\displaystyle\frac{1}{\Gamma(\a+1)}\left[1+\sum_{p=3}^\infty
\frac{\Gamma(p-\a-2)}{\Gamma(-\a)(p-2)!}\right],\\
A_1(\a)&=\displaystyle\frac{1}{\Gamma(\a+2)}\left[1+\sum_{p=2}^\infty
\frac{\Gamma(p-\a-2)}{\Gamma(-\a-1)(p-1)!}\right],\\
A_2(\a)&=\displaystyle\frac{1}{\Gamma(\a+3)}\left[1+\sum_{p=1}^\infty
\frac{\Gamma(p-\a-2)}{\Gamma(-\a-2)p!}\right],\\
\end{array}
$$
we obtain
\begin{equation*}
\begin{split}
\LI x(t) &= \frac{x(a)}{\Gamma(\a+1)}(t-a)^\a+\frac{\dot{x}(a)}{\Gamma(\a+2)}(t-a)^{\a+1}
+A_2(\a)(t-a)^{\a+2}\ddot{x}(t)\\
&\quad +\frac{(t-a)^{\a+2}}{\Gamma(\a+2)}
\sum_{p=1}^\infty\frac{\Gamma(p-\a-2)}{\Gamma(-\a-1)(p-1)!(t-a)^p}
\int_a^t (\tau-a)^{p-1}\ddot{x}(\tau)d\tau\\
&= \frac{x(a)}{\Gamma(\a+1)}(t-a)^\a
+A_1(\a)(t-a)^{\a+1}\dot{x}(t)+A_2(\a)(t-a)^{\a+2}\ddot{x}(t)\\
&\quad +\frac{(t-a)^{\a+2}}{\Gamma(\a+1)}\sum_{p=2}^\infty
\frac{\Gamma(p-\a-2)}{\Gamma(-\a)(p-2)!(t-a)^p}
\int_a^t (\tau-a)^{p-2}\dot{x}(\tau)d\tau\\
&= A_0(\a)(t-a)^\a x(t)+A_1(\a)(t-a)^{\a+1}\dot{x}(t)+A_2(\a)(t-a)^{\a+2}\ddot{x}(t)\\
&\quad +\frac{(t-a)^{\a+2}}{\Gamma(\a)}\sum_{p=3}^\infty
\frac{\Gamma(p-\a-2)}{\Gamma(-\a+1)(p-3)!(t-a)^p}\int_a^t (\tau-a)^{p-3}x(\tau)d\tau.
\end{split}
\end{equation*}
Therefore, we can expand $\LI x(t)$ as
\begin{multline}
\label{def:n=3}
\LI x(t)=A_0(\alpha)(t-a)^\a x(t)+ A_1(\alpha)(t-a)^{\a+1} \dot{x}(t)
+ A_2(\alpha)(t-a)^{\a+2} \ddot{x}(t)\\
+ \sum_{p=3}^\infty B(\a,p)(t-a)^{\a+2-p}V_p(t),
\end{multline}
where
\begin{equation}\label{def:B3}
B(\a,p)=\frac{\Gamma(p-\a-2)}{\Gamma(\a)\Gamma(1-\a)(p-2)!},
\end{equation}
and
\begin{equation}
\label{def:Vp3}
V_p(t)=\int_a^t (p-2)(\tau-a)^{p-3}x(\tau)d\tau.
\end{equation}

\begin{remark}
Function $V_p$ given by \eqref{def:Vp3} may be defined
as the solution of the differential equation
$$
\left\{
\begin{array}{l}
\dot{V}_p(t)=(p-2)(t-a)^{p-3}x(t)\\
V_p(a)=0,
\end{array}\right.
$$
for $p=3,4,\ldots$
\end{remark}

\begin{remark}
When $\a$ is not an integer, we may use
Euler's reflection formula (\textrm{cf.} \cite{Beals})
$$
\Gamma(\a)\Gamma(1-\a)=\frac{\pi}{\sin(\pi\a)},
$$
to simplify expression $B(\a,p)$ in \eqref{def:B3}.
\end{remark}

Following the same reasoning, we are able to deduce a general formula
of decomposition for fractional integrals, depending on the order
of smoothness of the test function.

\begin{theorem}
\label{TheoDecomp}
Let $n\in\mathbb N$ and $x\in C^n[a,b]$. Then
\begin{equation}
\label{ExpanDecomp}
\LI x(t)=\sum_{i=0}^{n-1}A_i(\alpha)(t-a)^{\a+i} x^{(i)}(t)
+\sum_{p=n}^\infty B(\a,p)(t-a)^{\a+n-1-p}V_p(t),
\end{equation}
where
\begin{equation}
\label{def:B}
\begin{array}{ll}
A_i(\a)&=\displaystyle\frac{1}{\Gamma(\a+i+1)}\left[1+\sum_{p=n-i}^\infty
\frac{\Gamma(p-\a-n+1)}{\Gamma(-\a-i)(p-n+1+i)!}\right],
\quad i = 0, \ldots, n-1,\\
B(\a,p)&=\displaystyle\frac{\Gamma(p-\a-n+1)}{\Gamma(\a)\Gamma(1-\a)(p-n+1)!},
\end{array}
\end{equation}
and
\begin{equation}
\label{def:Vp}
V_p(t)=\int_a^t (p-n+1)(\tau-a)^{p-n}x(\tau)d\tau,
\end{equation}
$p = n, n+1, \ldots$
\end{theorem}

A remark about the convergence of the series in $A_i(\a)$,
for $i\in\{0,\ldots,n-1\}$, is in order. Since
\begin{equation}
\label{eq:conv:ser}
\begin{array}{ll}
\displaystyle \sum_{p=n-i}^\infty\frac{\Gamma(p-\a-n+1)}{\Gamma(-\a-i)(p-n+1+i)!}
&= \displaystyle\sum_{p=0}^\infty\frac{\Gamma(p-\a-i)}{\Gamma(-\a-i) p!}-1\\
&={_1F_0} (-\a-i,1),\\
\end{array}
\end{equation}
where ${_1F_0}$ denotes the hypergeometric function,\index{Hypergeometric function}
and because $\a+i>0$, we conclude that \eqref{eq:conv:ser}
converges absolutely (\textrm{cf.} Theorem 2.1.2 in \cite{Andrews}).
In fact, we may use Eq. (2.1.6) in \cite{Andrews} to conclude that
$$
\sum_{p=n-i}^\infty\frac{\Gamma(p-\a-n+1)}{\Gamma(-\a-i)(p-n+1+i)!}=-1.
$$
Therefore, the first $n$ terms of our decomposition \eqref{ExpanDecomp} vanish.
However, because of numerical reasons, we do not follow this procedure here.
Indeed, only finite sums of these coefficients are to be taken,
and we obtain a better accuracy for the approximation
taking them into account (see Figures~\ref{ExpTkA} and \ref{ExpTk2A}).
More precisely, we consider finite sums up to order $N$,
with $N\geq n$. Thus, our approximation will depend on two parameters:
the order of the derivative $n\in\mathbb N$, and the number
of terms taken in the sum, which is given by $N$.
The left fractional integral is then approximated by
\begin{equation}
\label{Approx:LI}
\LI x(t)\approx \sum_{i=0}^{n-1}A_i(\alpha,N)(t-a)^{\a+i} x^{(i)}(t)
+\sum_{p=n}^N B(\a,p)(t-a)^{\a+n-1-p}V_p(t),
\end{equation}
where
\begin{equation}
\label{Def:A}
A_i(\a,N)=\frac{1}{\Gamma(\a+i+1)}\left[1+\sum_{p=n-i}^N
\frac{\Gamma(p-\a-n+1)}{\Gamma(-\a-i)(p-n+1+i)!}\right],
\end{equation}
and $B(\a,p)$ and $V_p(t)$ are given by \eqref{def:B}
and \eqref{def:Vp}, respectively.

To measure the truncation errors\index{Error! Truncation} made by neglecting
the remaining terms, observe that
\begin{equation}
\label{error:A}
\begin{split}
\frac{1}{\Gamma(\a+i+1)}&\sum_{p=N+1}^\infty\frac{\Gamma(p-\a-n+1)}{\Gamma(-\a-i)(p-n+1+i)!}
=\frac{1}{\Gamma(\a+i+1)}\sum_{p=N-n+2+i}^\infty\frac{\Gamma(p-\a-i)}{\Gamma(-\a-i)p!}\\
&= \frac{1}{\Gamma(\a+i+1)}\left[  {_2F_1} (-\a-i,-,-,1)
- \sum_{p=0}^{N-n+1+i}\frac{\Gamma(p-\a-i)}{\Gamma(-\a-i)p!} \right]\\
&=\frac{-1}{\Gamma(\a+i+1)}\sum_{p=0}^{N-n+i+1}\frac{\Gamma(p-\a-i)}{\Gamma(-\a-i)p!}.
\end{split}
\end{equation}
Similarly,
\begin{equation}
\label{error:B}
\frac{1}{\Gamma(\a)\Gamma(1-\a)}\sum_{p=N+1}^\infty\frac{\Gamma(p-\a-n+1)}{(p-n+1)!}
=\frac{-1}{\Gamma(\a)\Gamma(1-\a)}\sum_{p=0}^{N-n+1}\frac{\Gamma(p-\a)}{p!}.
\end{equation}
In Tables~\ref{tab1} and \ref{tab2} we exemplify some values
for \eqref{error:A} and \eqref{error:B}, respectively, with $\a=0.5$
and for different values of $N$, $n$ and $i$. Observe that the errors
only depend on the values of $N-n$ and $i$ for \eqref{error:A},
and on the value of  $N-n$ for \eqref{error:B}.


\begin{table}[!ht]
\center
\begin{tabular}{|l|c|c|c|c|c|}\hline
\diaghead{\theadfont Diag Column }%
{~\\$i$}{$N-n$\\~}&
\thead{0}&\thead{1}&\thead{2}&\thead{3}&\thead{4}\\
\hline
0 & -0.5642 & -0.4231 & -0.3526 & -0.3085 & -0.2777  \\
 \hline
1 &0.09403 & 0.04702 & 0.02938 & 0.02057 & 0.01543 \\
 \hline
2 & -0.01881 & -0.007052& -0.003526& -0.002057& -0.001322\\
 \hline
3 & 0.003358& 0.001007 & 0.0004198 & 0.0002099 & 0.0001181 \\
 \hline
4 & -0.0005224 & -0.0001306 & -0.00004664& -0.00002041 & -0.00001020\\
 \hline
5 & $7.12 \times10^{-5}$& $1.52\times10^{-5}$& $4.77\times10^{-6}$
  &$ 1.85\times10^{-6}$ & $8.34\times10^{-7}$ \\
\hline
\end{tabular}
\caption{Values of error \eqref{error:A} for $\a=0.5$.}
\label{tab1}
\end{table}


\begin{table}[!ht]
\center
\begin{tabular}{|c|c|c|c|c|c|}
\hline
 $N-n$ & 0 & 1 & 2 & 3 & 4  \\
 \hline
 & 0.5642 & 0.4231 & 0.3526 & 0.3085 & 0.2777  \\
 \hline
\end{tabular}
\caption{Values of error \eqref{error:B} for $\a=0.5$.}
\label{tab2}
\end{table}


Everything done so far is easily adapted
to the right fractional integral.\index{Riemann--Liouville fractional integral! Right}
In fact, one has:

\begin{theorem}
Let $n\in\mathbb N$ and $x\in C^n[a,b]$. Then
$$
\RI x(t)=\sum_{i=0}^{n-1}A_i(\alpha)(b-t)^{\a+i} x^{(i)}(t)
+\sum_{p=n}^\infty B(\a,p)(b-t)^{\a+n-1-p}W_p(t),
$$
where
\begin{equation*}
\begin{split}
A_i(\a)&=\frac{(-1)^i}{\Gamma(\a+i+1)}\left[1+\sum_{p=n-i}^\infty
\frac{\Gamma(p-\a-n+1)}{\Gamma(-\a-i)(p-n+1+i)!}\right],\\
B(\a,p)&=\frac{(-1)^n\Gamma(p-\a-n+1)}{\Gamma(\a)\Gamma(1-\a)(p-n+1)!},\\
W_p(t) &= \int_t^b (p-n+1)(b-\tau)^{p-n}x(\tau)d\tau.
\end{split}
\end{equation*}
\end{theorem}


\subsection{Numerical evaluation of fractional integrals}

In this section we exemplify the proposed
approximation procedure with some examples.
In each step, we evaluate the accuracy of our method, \textrm{i.e.},
the error when substituting $\LI x$ by an approximation $\tilde{\LI}x$.
For that purpose, we take the distance given by
$$
E=\sqrt{\int_a^b \left(\LI x(t)-\tilde{\LI}x(t)\right)^2dt}.
$$
Firstly, consider $x_1(t)=t^3$ and $x_2(t)=t^{10}$ with $t\in[0,1]$. Then
$$
\LIz x_1(t)=\frac{\Gamma(4)}{\Gamma(4.5)}t^{3.5}\
\mbox{ and } \ \LIz x_2(t)=\frac{\Gamma(11)}{\Gamma(11.5)}t^{10.5}
$$
(\textrm{cf.} Property~2.1 in \cite{Kilbas}).
Let us consider Theorem~\ref{TheoDecomp} for $n=3$,
\textrm{i.e.}, expansion \eqref{def:n=3} for different values of step $N$.
For function $x_1$, small values of $N$ are enough
($N=3,4,5$). For $x_2$ we take $N=4,6,8$. In Figures~\ref{ExpTk}
and \ref{ExpTk2} we represent the graphs of the fractional integrals
of $x_1$ and $x_2$ of order $\a=0.5$ together with different
approximations. As expected, when $N$ increases
we obtain a better approximation for each fractional integral.


\begin{figure}[!ht]
  \begin{center}
    \subfigure[$\LIz(t^3)$]{\label{ExpTk}\includegraphics[scale=0.54]{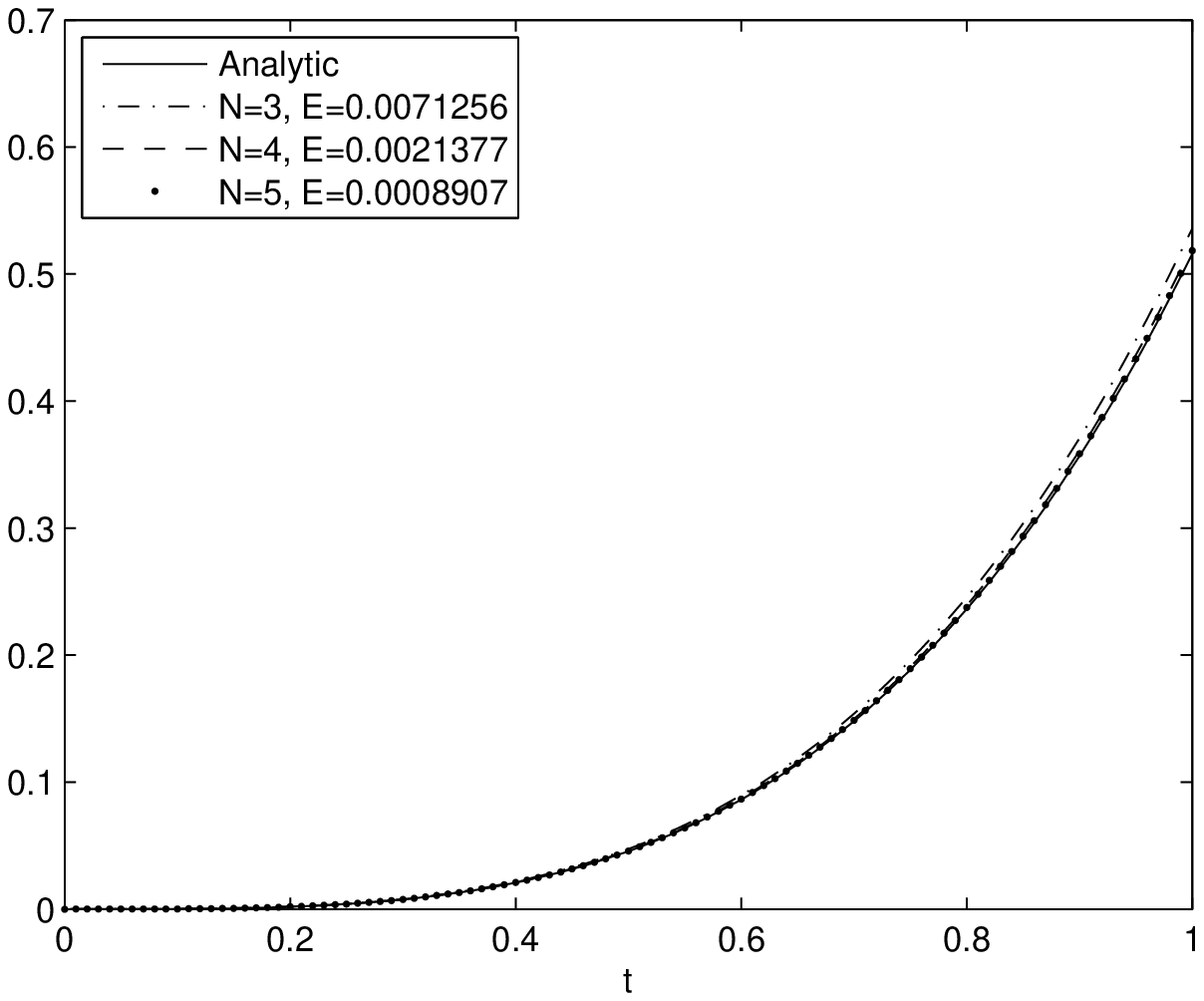}}
    \subfigure[$\LIz(t^{10})$]{\label{ExpTk2}\includegraphics[scale=0.54]{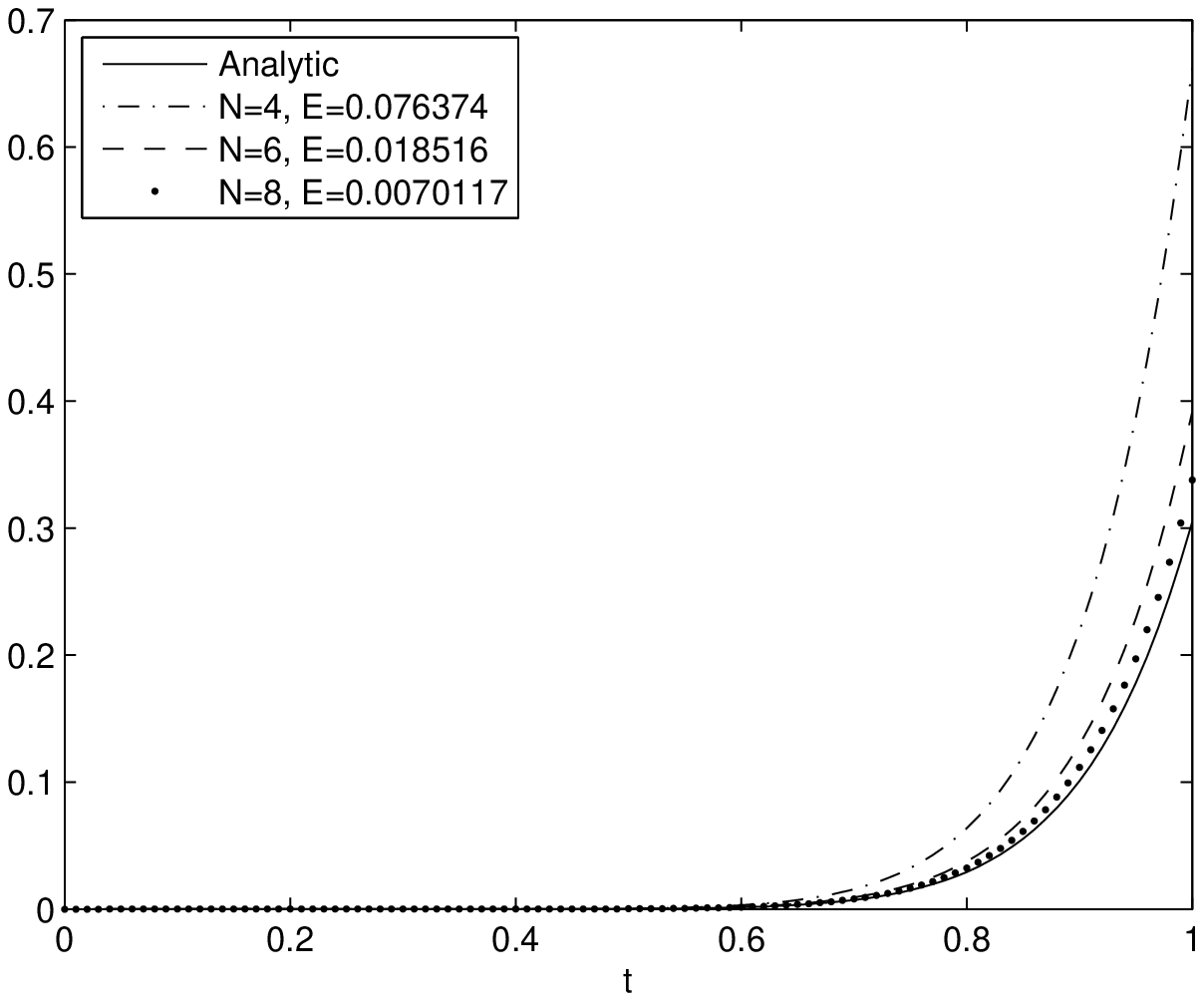}}
  \end{center}
  \caption{Analytic {\it versus} numerical approximation for a fixed $n$.}
\end{figure}


Secondly, we apply our procedure to the transcendental functions
$x_3(t)=e^t$ and $x_4(t)=\sin(t)$. Simple calculations give
$$
\LIz x_3(t)=\sqrt{t}\sum_{k=0}^\infty\frac{t^k}{\Gamma(k+1.5)}
\ \mbox{ and }\  \LIz x_4(t)=\sqrt{t}\sum_{k=0}^\infty
\frac{(-1)^k t^{2k+1}}{\Gamma(2k+2.5)}.
$$
Figures~\ref{ExpEtInt} and \ref{ExpSint} show the numerical results
for each approximation, with $n=3$. We see that for a small value of $N$
one already obtains a good approximation for each function.


\begin{figure}[!ht]
  \begin{center}
    \subfigure[$\LIz(e^t)$]{\label{ExpEtInt}\includegraphics[scale=0.54]{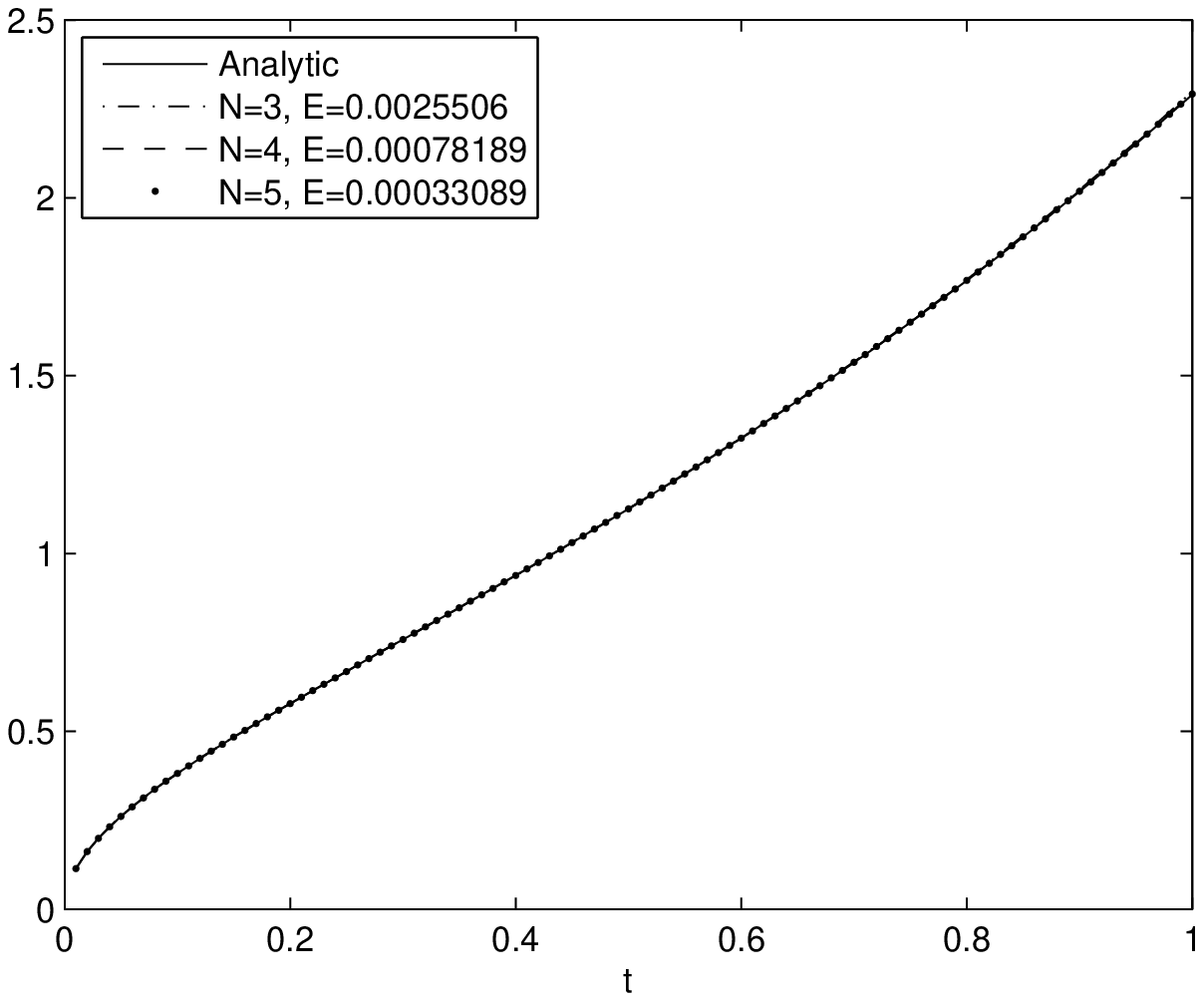}}
    \subfigure[$\LIz(\sin(t))$]{\label{ExpSint}\includegraphics[scale=0.54]{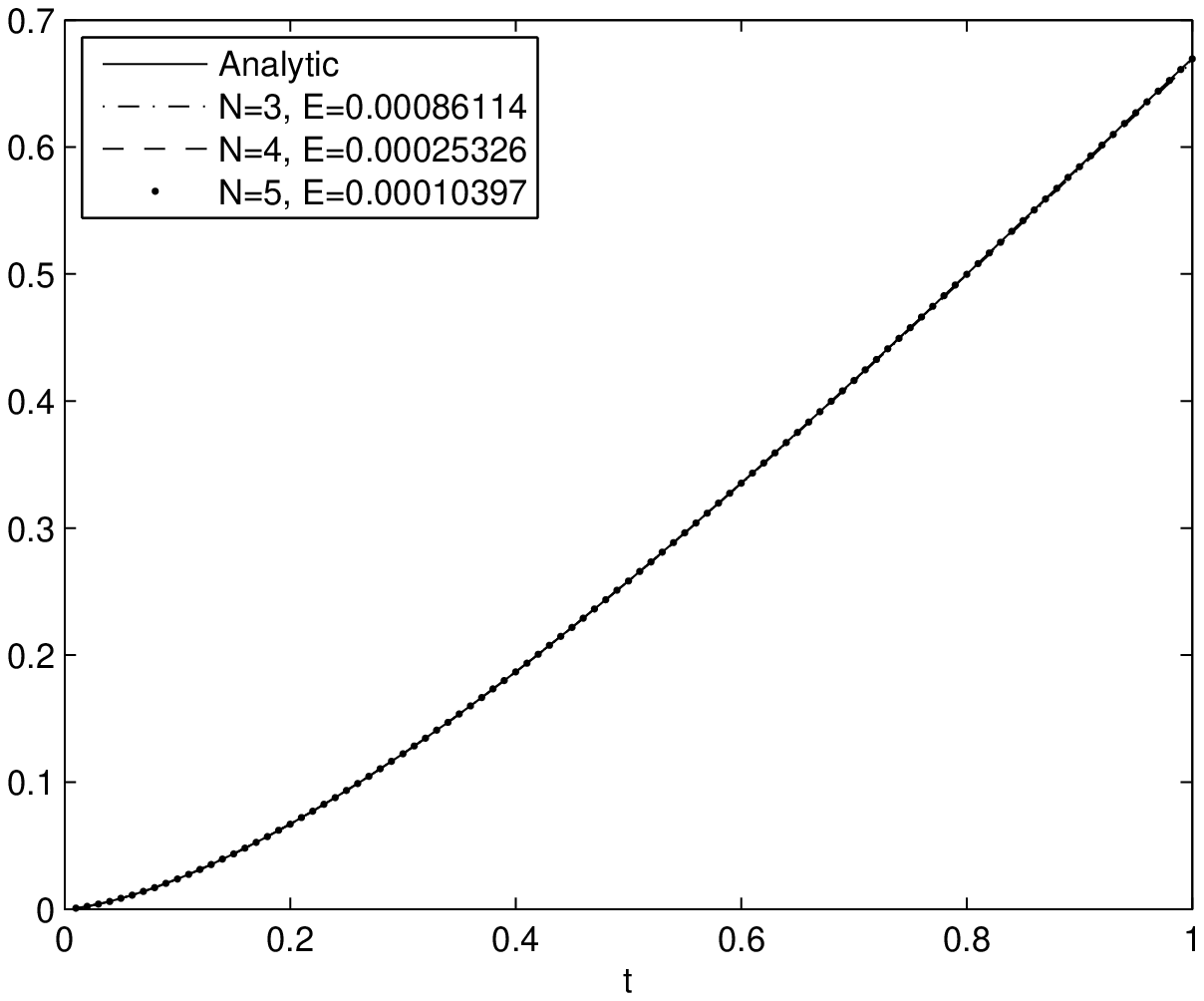}}
  \end{center}
  \caption{Analytic {\it versus} numerical approximation for a fixed $n$.}
\end{figure}


For analytical functions, we may apply the well-known formula \eqref{analytical2}.
In Figure~\ref{Fig:IntExp3} we show the results of approximating
with \eqref{analytical2}, $N=1,2,3$, for functions $x_3(t)$
and $x_4(t)$. We remark that, when we consider expansions up to the second derivative,
\textrm{i.e.}, the cases $n=3$ as in \eqref{def:n=3} and expansion \eqref{analytical2}
with $N=2$, we obtain a better accuracy using our approximation
\eqref{def:n=3} even for a small value of $N$.


\begin{figure}[!ht]
\begin{center}
\subfigure[$\LIz(e^t)$]{\includegraphics[scale=0.54]{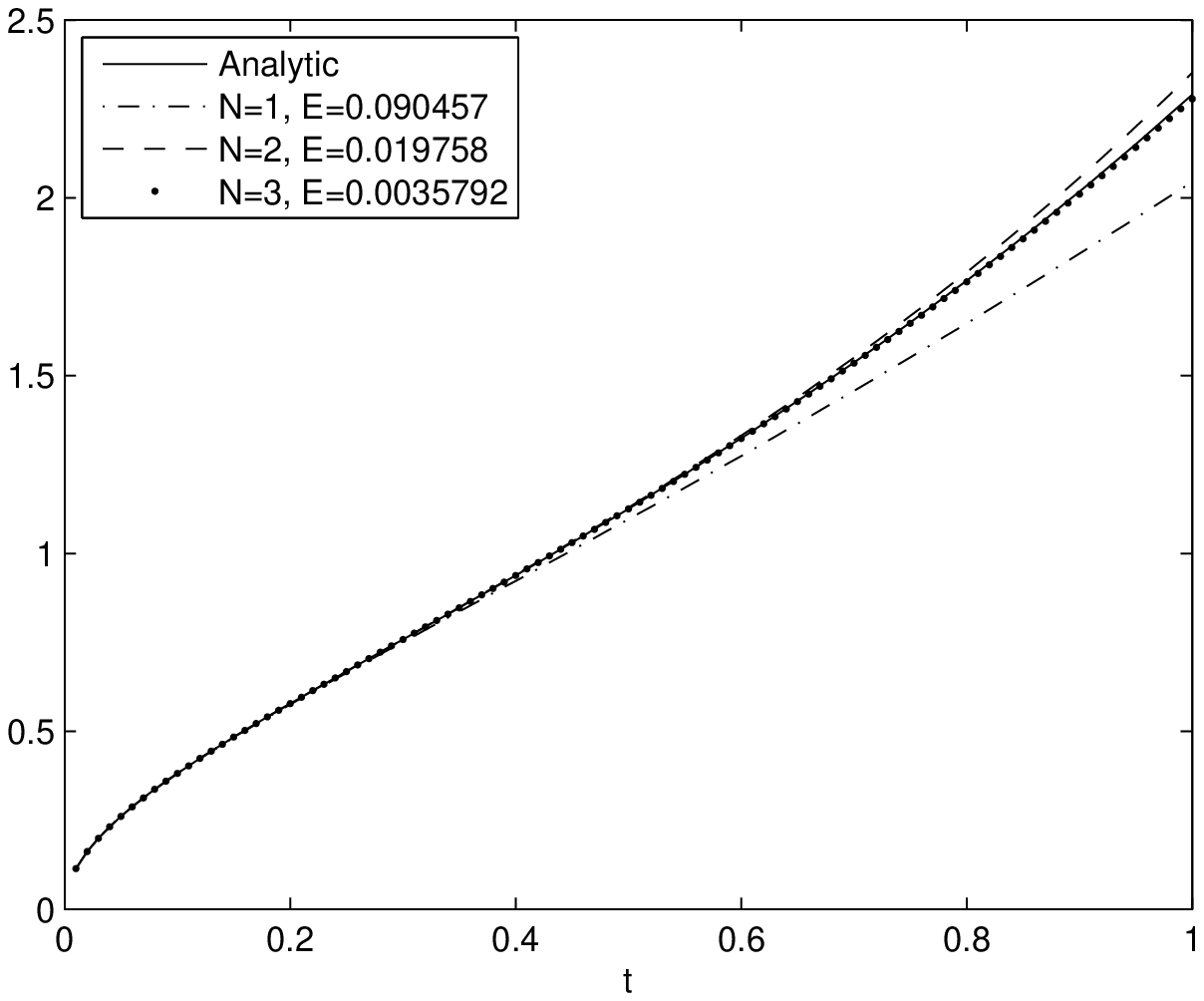}}
\subfigure[$\LIz(\sin(t))$]{\includegraphics[scale=0.54]{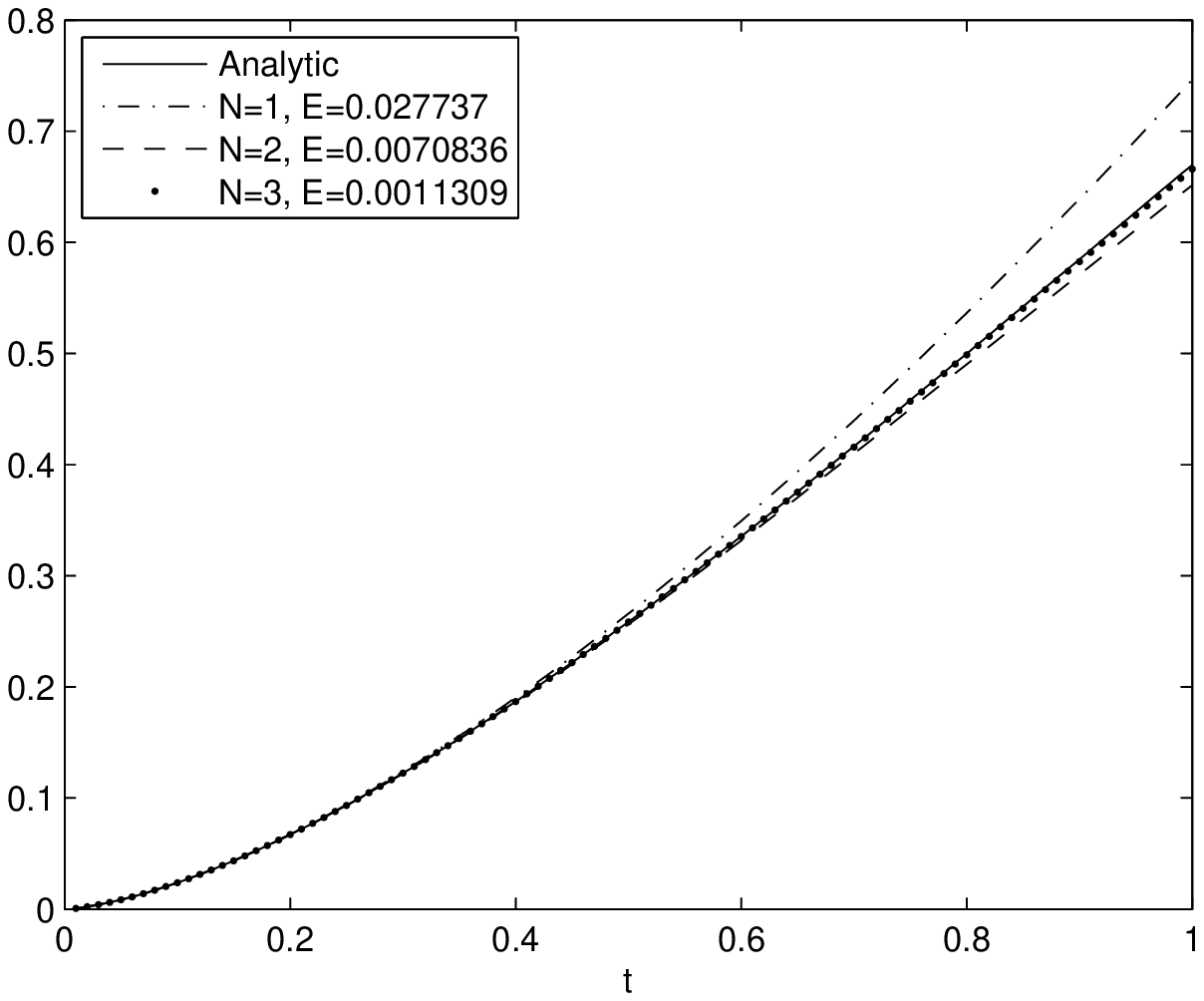}}
\end{center}
\caption{Numerical approximation using \eqref{analytical2} of previous literature.}
\label{Fig:IntExp3}
\end{figure}


Another way to approximate fractional integrals is to fix $N$
and consider several sizes for the decomposition, \textrm{i.e.},
letting $n$ to vary. Let us consider the two test functions $x_1(t)=t^3$
and $x_2(t)=t^{10}$, with $t\in[0,1]$ as before. In both cases
we consider the first three approximations of the fractional integral,
\textrm{i.e.}, for $n=1,2,3$. For the first function we fix
$N=3$, for the second one we choose $N=8$. Figures~\ref{ExpTkn} and \ref{ExpTkn2}
show the numerical results. As expected, for a greater value of $n$ the error decreases.


\begin{figure}[!ht]
\begin{center}
\subfigure[$\LIz(t^3)$]{\label{ExpTkn}\includegraphics[scale=0.54]{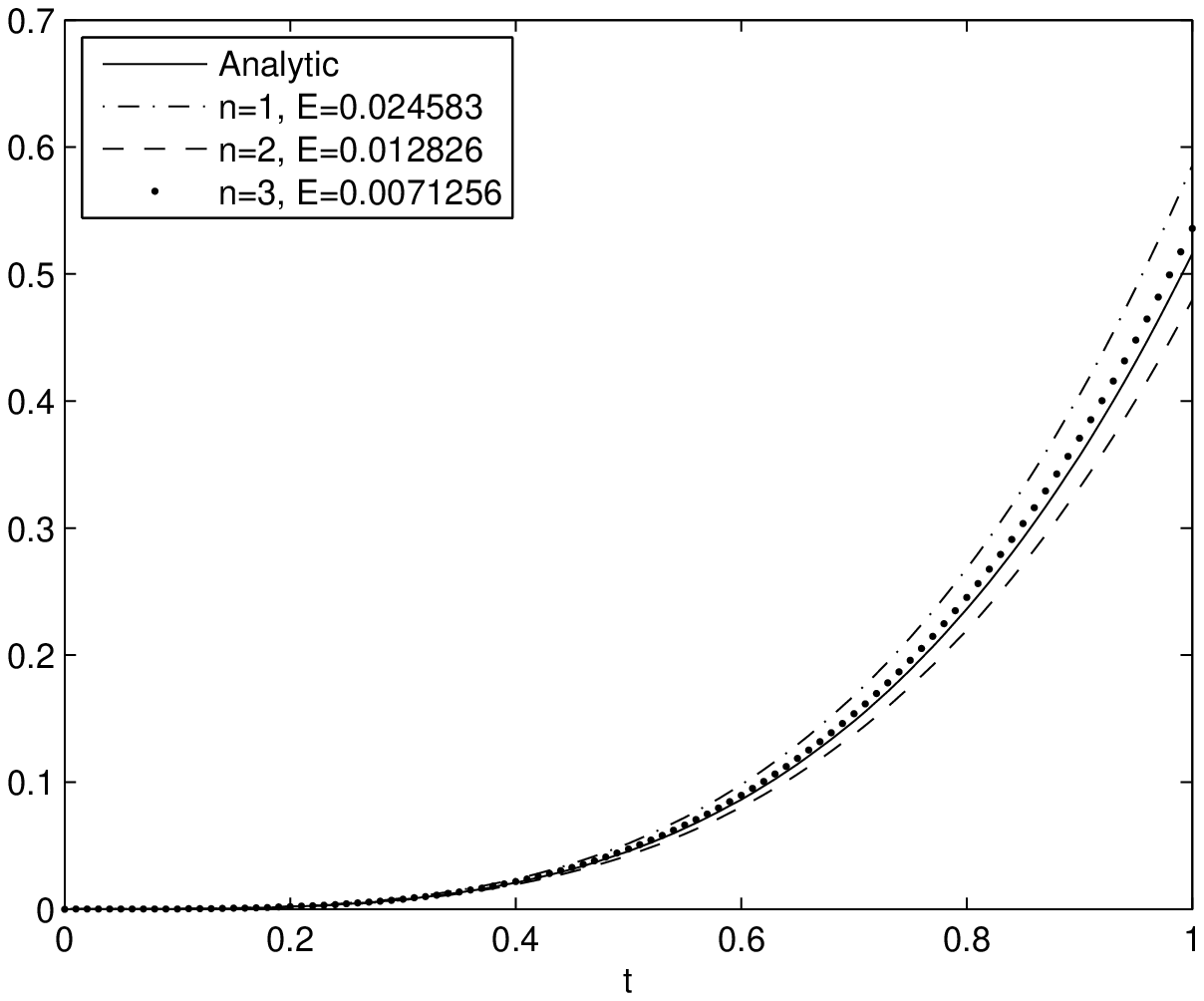}}
\subfigure[$\LIz(t^{10})$]{\label{ExpTkn2}\includegraphics[scale=0.54]{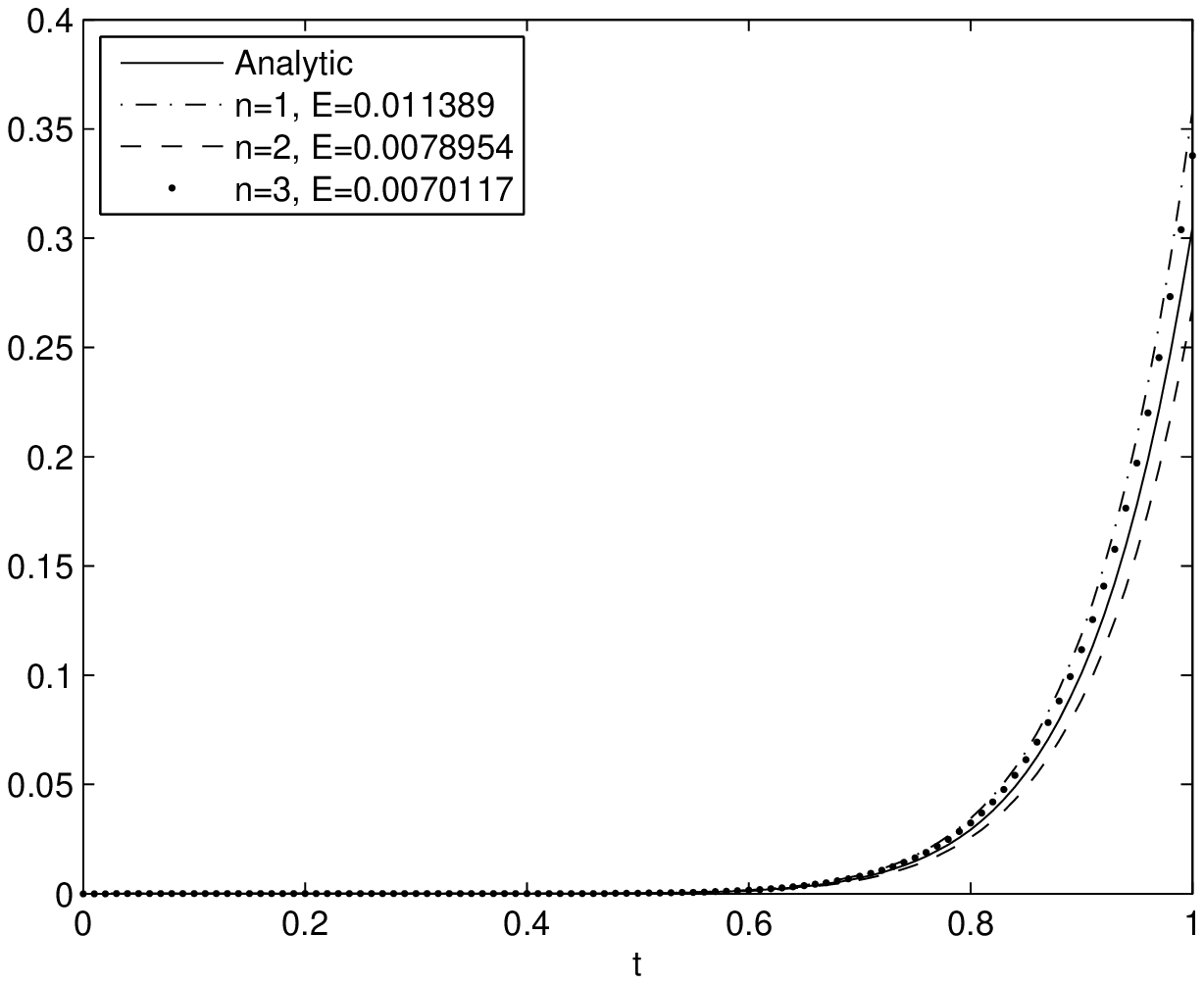}}
\end{center}
\caption{Analytic {\it versus} numerical approximation for a fixed $N$.}
\end{figure}

We mentioned before that although the terms $A_i$ are all equal to zero,
for $i \in \{0,\ldots,n-1\}$, we consider them in the decomposition formula.
Indeed, after we truncate the sum, the error is lower. This is illustrated
in Figures~\ref{ExpTkA} and \ref{ExpTk2A}, where we study the approximations
for $\LIz x_1(t)$ and $\LIz x_2(t)$ with $A_i\not=0$ and $A_i=0$.
\begin{figure}[!ht]
  \begin{center}
    \subfigure[$\LIz(t^3)$]{\label{ExpTkA}\includegraphics[scale=0.54]{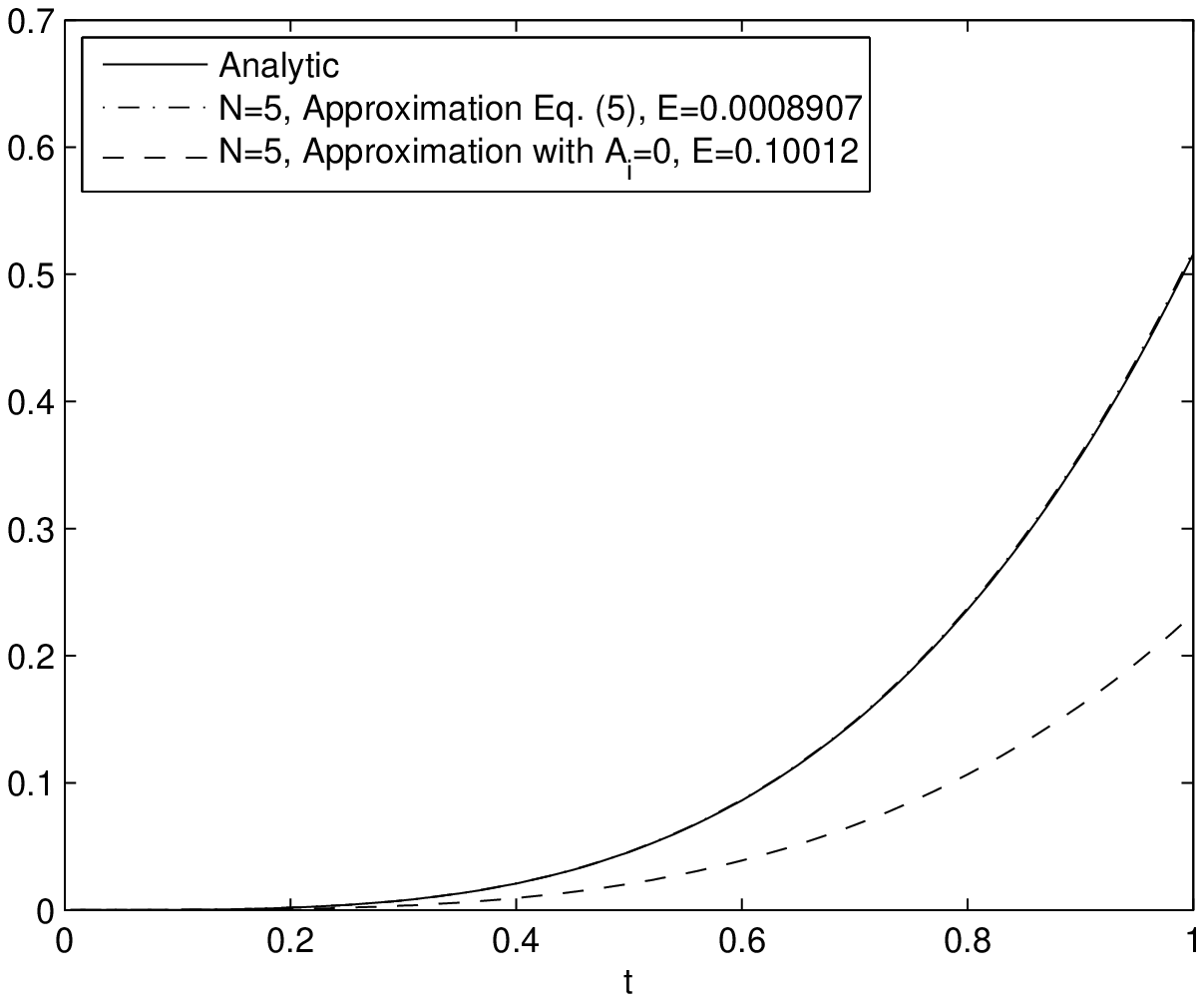}}
    \subfigure[$\LIz(t^{10})$]{\label{ExpTk2A}\includegraphics[scale=0.54]{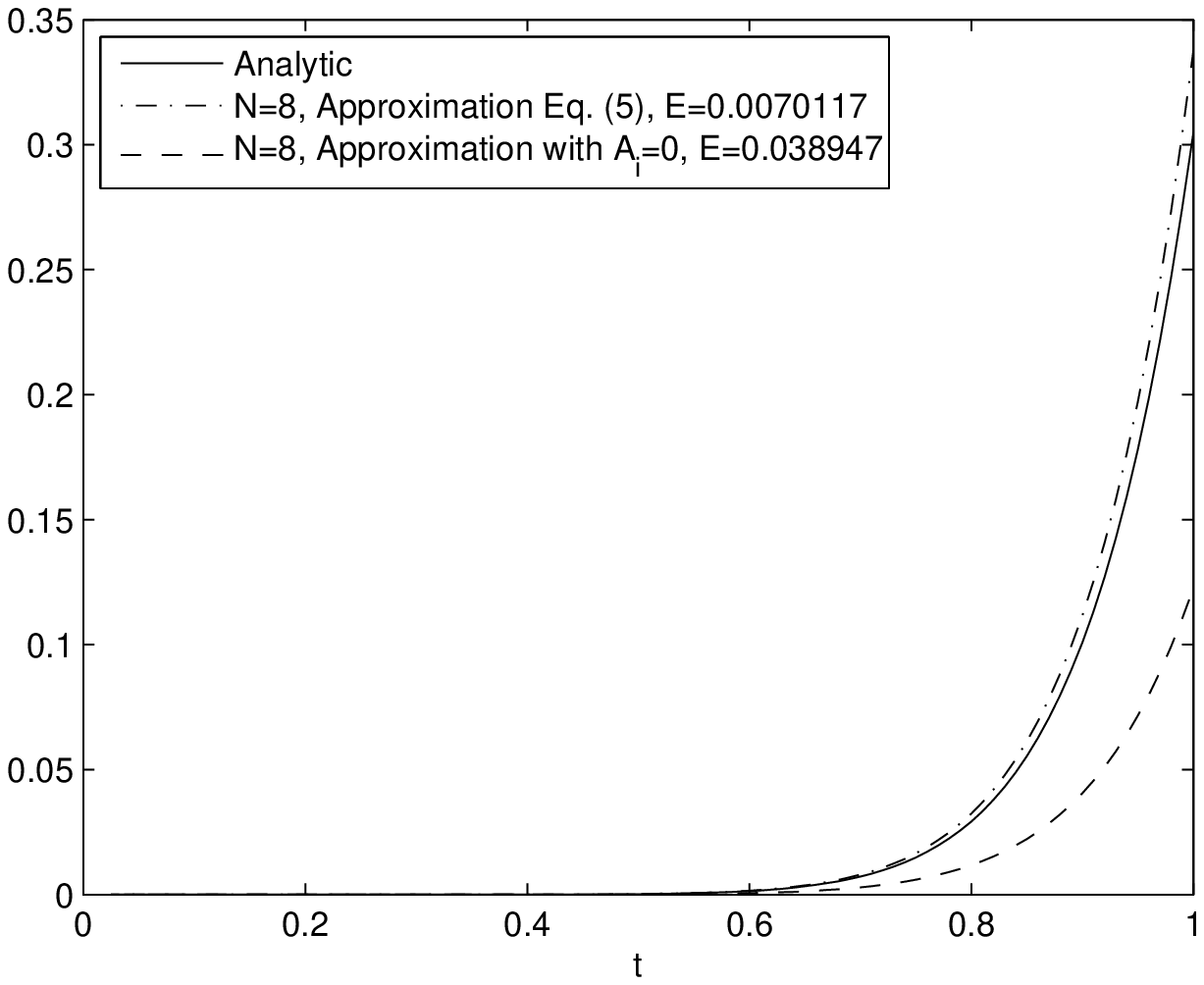}}
  \end{center}
  \caption{Comparison of approximation \eqref{def:n=3} and approximation with $A_i=0$.}
\end{figure}


\subsection{Applications to fractional integral equations}

In this section we show how the proposed approximations can be applied
to solve a fractional integral\index{Fractional! integral equation}
equation (Example~\ref{example:1}) which depends on the left Riemann--Liouville
fractional integral. The main idea is to rewrite the initial problem 
by replacing the fractional integrals by an expansion
of type \eqref{analytical} or \eqref{ExpanDecomp},
and thus getting a problem involving integer-order derivatives,
which can be solved by standard techniques.


\begin{example}
\label{example:1}
Consider the following fractional system:
\begin{equation}
\label{system}
\left\{ \begin{array}{l}
\LIz x(t)=\frac{\Gamma(4.5)}{24}t^4\\
x(0)=0.
\end{array}\right.
\end{equation}
Since $\LIz t^{3.5}=\frac{\Gamma(4.5)}{24}t^\a$, the function
$t\mapsto t^{3.5}$ is a solution to problem \eqref{system}.

To provide a numerical method to solve such type of systems,
we replace the fractional integral by approximations \eqref{analytical2}
and \eqref{Approx:LI}, for a suitable order. We remark that the order
of approximation, $N$ in \eqref{analytical2} and $n$ in \eqref{Approx:LI},
are restricted by the number of given initial or boundary conditions.
Since \eqref{system} has one initial condition, in order to solve it numerically,
we will consider the expansion for the fractional integral up to the first derivative,
\textrm{i.e.}, $N=1$ in \eqref{analytical2} and $n=2$ in \eqref{Approx:LI}.
The order $N$ in \eqref{Approx:LI} can be freely chosen.

Applying approximation \eqref{analytical2}, with $\a=0.5$,
we transform \eqref{system} into the initial value problem
$$
\left\{
\begin{array}{l}
1.1285t^{0.5} x(t)-0.3761t^{1.5} \dot{x}(t)=\frac{\Gamma(4.5)}{24}t^4,\\
x(0)=0,
\end{array}\right.
$$
which is a first order ODE.
The solution is shown in Figure~\ref{IntFI}.
It reveals that the approximation remains close to the exact solution
for a short time and diverges drastically afterwards. Since we have
no extra information, we cannot increase the order of approximation to proceed.

To use expansion \eqref{ExpanDecomp}, we rewrite the problem as
a standard one, depending only on a derivative of first order.
The approximated system that we must solve is
$$
\left\{
\begin{array}{l}
A_0(0.5,N)t^{0.5} x(t)+A_1(0.5,N)t^{1.5} \dot{x}(t)
+\sum_{p=2}^NB(0.5,p)t^{1.5-p}V_p(t)
=\frac{\Gamma(4.5)}{24}t^4,\\
\dot{V}_p(t)=(p-1)t^{p-2}x(t),\quad p=2,3,\ldots,N,\\
x(0)=0,\\
V_p(0)=0, \quad p=2,3,\ldots,N,\\
\end{array}\right.
$$
where $A_0$ and $A_1$ are given as in \eqref{Def:A} and $B$ is given
by Theorem~\ref{TheoDecomp}. Here, by increasing $N$, we get better
approximations to the fractional integral and we expect more accurate
solutions to the original problem \eqref{system}. For $N=2$ and $N=3$
we transform the resulting system of ordinary differential equations
to a second and a third order differential equation, respectively. Finally,
we solve them using the \textsf{Maple} built in function \textsf{dsolve}.
For example, for $N=2$ the second-order equation takes the form
$$
\left\{
\begin{array}{l}
\ddot{V}_2(t)=\frac{6}{t}\dot{V}_2(t)+\frac{6}{t^2}V_2(t)-5.1542t^{2.5}\\
V_2(0)=0\\
\dot{V}_2(0)=x(0)=0,
\end{array}
\right.
$$
and the solution is $x(t)=\dot{V}_2(t)=1.34t^{3.5}$.
In Figure~\ref{momFI} we compare the exact solution with numerical
approximations for two values of $N$.

\begin{figure}[!ht]
\begin{center}
\subfigure[Approximation by \eqref{analytical2}.]{\label{IntFI}\includegraphics[scale=0.54]{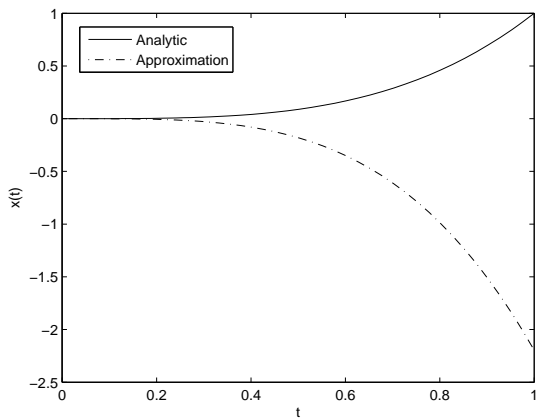}}
\subfigure[Approximation by \eqref{Approx:LI}.]{\label{momFI}\includegraphics[scale=0.54]{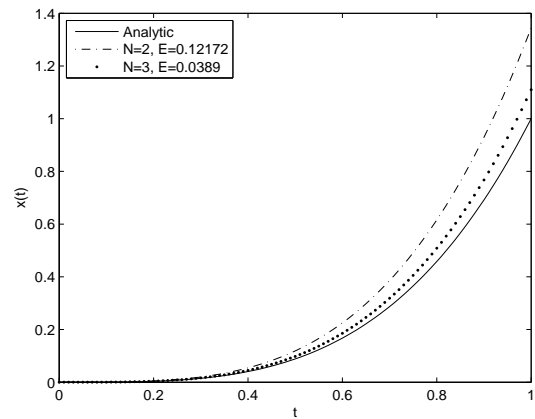}}
\end{center}
\caption{Analytic {\it versus} numerical solution to problem \eqref{system}.}
\end{figure}
\end{example}


\section{Hadamard fractional integrals}

\subsection{Approximation by a sum of integer-order derivatives}

For an arbitrary $\a>0$ we refer the reader to \cite[Theorem~3.2]{Kilbas2}.
If a function $x$ admits derivatives of any order,
then expansion formulas\index{Fractional! integral equation} 
for the Hadamard fractional integrals and derivatives
of $x$, in terms of its integer-order derivatives,
are given in \cite[Theorem~17]{Butzer2}:
$$
{_0\mathcal{I}_t^\a}x(t)=\sum_{k=0}^\infty S(-\a,k)t^k x^{(k)}(t)
$$
and
$$
{_0\mathcal{D}_t^\a}x(t)=\sum_{k=0}^\infty S(\a,k)t^k x^{(k)}(t),
$$
where
$$
S(\a,k)=\frac{1}{k!}\sum_{j=1}^k(-1)^{k-j} {k \choose j} j^{\a}
$$
is the Stirling function.


\subsection{Approximation using moments of a function}
\index{Moment of a function}

In this section we consider the class of differentiable functions
up to order $n+1$, $x\in C^{n+1}[a,b]$, and deduce expansion formulas
for the Hadamard fractional integrals in terms of $x^{(i)}(\cdot)$,
for $i\in \{0,\ldots,n\}$. Before presenting the result in its full extension,
we briefly explain the techniques involved for the particular case $n=2$.
To that purpose, let $x\in C^3[a,b]$.
Integrating by parts three times, we obtain
$$
\begin{array}{ll}
\LHI x(t)&=\displaystyle-\frac{1}{\Gamma(\a)}
\int_a^t-\frac{1}{\tau} \left(\ln\frac{t}{\tau}\right)^{\a-1}x(\tau)d\tau\\
&=\displaystyle\frac{1}{\Gamma(\a+1)} \left(\ln\frac{t}{a}\right)^{\a}x(a)
-\frac{1}{\Gamma(\a+1)}\int_a^t-\frac{1}{\tau}
\left(\ln\frac{t}{\tau}\right)^{\a}\tau \dot{x}(\tau)d\tau\\
&=\displaystyle\frac{1}{\Gamma(\a+1)} \left(\ln\frac{t}{a}\right)^{\a}x(a)
+\frac{1}{\Gamma(\a+2)} \left(\ln\frac{t}{a}\right)^{\a+1}a\dot{x}(a)\\
&\quad\displaystyle -\frac{1}{\Gamma(\a+2)}\int_a^t-\frac{1}{\tau}
\left(\ln\frac{t}{\tau}\right)^{\a+1}(\tau \dot{x}(\tau)+\tau^2\ddot{x}(\tau))d\tau\\
&=\displaystyle\frac{1}{\Gamma(\a+1)} \left(\ln\frac{t}{a}\right)^{\a}x(a)
+\frac{1}{\Gamma(\a+2)} \left(\ln\frac{t}{a}\right)^{\a+1}a\dot{x}(a)\\
&\displaystyle\quad+\frac{1}{\Gamma(\a+3)}
\left(\ln\frac{t}{a}\right)^{\a+2}(a\dot{x}(a) + a^2\ddot{x}(a))\\
&\quad+\displaystyle\frac{1}{\Gamma(\a+3)}
\int_a^t \left(\ln\frac{t}{\tau}\right)^{\a+2}(\dot{x}(\tau)
+3\tau \ddot{x}(\tau)+\tau^2\dddot{x}(\tau))d\tau.
\end{array}
$$
On the other hand, using the binomial theorem\index{Binomial! theorem}, we have
$$
\begin{array}{ll}
\displaystyle\left(\ln\frac{t}{\tau}\right)^{\a+2}
&=\displaystyle\left(\ln\frac{t}{a}\right)^{\a+2}\left(
1-\frac{\ln\frac{\tau}{a}}{\ln\frac{t}{a}}\right)^{\a+2}\\
&=\displaystyle\left(\ln\frac{t}{a}\right)^{\a+2}
\sum_{p=0}^\infty\frac{\Gamma(p-\a-2)}{\Gamma(-\a-2)p!}
\cdot \frac{\left(\ln\frac{\tau}{a}\right)^p}{\left(\ln\frac{t}{a}\right)^p}.
\end{array}
$$
This series converges since $\tau\in[a,t]$ and $\a+2>0$.
Combining these formulas, we get
\begin{eqnarray*}
\LHI x(t)&=&\frac{x(a)}{\Gamma(\a+1)} \left(\ln\frac{t}{a}\right)^{\a}
+\frac{a\dot{x}(a)}{\Gamma(\a+2)} \left(\ln\frac{t}{a}\right)^{\a+1}
+\frac{a\dot{x}(a)+a^2\ddot{x}(a)}{\Gamma(\a+3)} \left(\ln\frac{t}{a}\right)^{\a+2}\\
&&+\frac{1}{\Gamma(\a+3)}\left(\ln\frac{t}{a}\right)^{\a+2}
\sum_{p=0}^\infty\Gamma_0(\a,p,t)
\int_a^t \left(\ln\frac{\tau}{a}\right)^p \left(\dot{x}(\tau)+3\tau \ddot{x}(\tau)
+\tau^2\dddot{x}(\tau)\right)d\tau,
\end{eqnarray*}
where
$$
\Gamma_i(\a,p,t)=\frac{\Gamma(p-\a-2)}{\Gamma(-\a-2+i)(p-i)!\left(\ln\frac{t}{a}\right)^p}.
$$
Now, split the series into the two cases $p=0$ and $p=1\ldots\infty$,
and integrate by parts the second one. We obtain
$$
\begin{array}{ll}
\LHI x(t)&=\displaystyle\frac{1}{\Gamma(\a+1)}
\left(\ln\frac{t}{a}\right)^{\a}x(a)
+\frac{1}{\Gamma(\a+2)} \left(\ln\frac{t}{a}\right)^{\a+1}a\dot{x}(a)\\
&\displaystyle\quad +\frac{1}{\Gamma(\a+3)}
\left(\ln\frac{t}{a}\right)^{\a+2}(t\dot{x}(t)
+t^2\ddot{x}(t))\left[1+\sum_{p=1}^\infty\frac{\Gamma(p-\a-2)}{\Gamma(-\a-2)p!} \right]\\
&\displaystyle\quad +\frac{1}{\Gamma(\a+2)}\left(\ln\frac{t}{a}\right)^{\a+2}
\sum_{p=1}^\infty \Gamma_1(\a,p,t)
\int_a^t \left(\ln\frac{\tau}{a}\right)^{p-1} (\dot{x}(\tau)+\tau \ddot{x}(\tau))d\tau.
\end{array}
$$

Repeating this procedure two more times, we obtain the following:
$$
\begin{array}{ll}
\LHI x(t)&=\displaystyle \frac{1}{\Gamma(\a+1)}
\left(\ln\frac{t}{a}\right)^{\a}x(t)\left[1
+\sum_{p=3}^\infty\frac{\Gamma(p-\a-2)}{\Gamma(-\a)(p-2)!} \right]\\
&\displaystyle\quad +\frac{1}{\Gamma(\a+2)}
\left(\ln\frac{t}{a}\right)^{\a+1}t\dot{x}(t)\left[1
+\sum_{p=2}^\infty\frac{\Gamma(p-\a-2)}{\Gamma(-\a-1)(p-1)!} \right]\\
&\displaystyle\quad +\frac{1}{\Gamma(\a+3)} \left(\ln\frac{t}{a}\right)^{\a+2}(t\dot{x}(t)
+t^2\ddot{x}(t))\left[1+\sum_{p=1}^\infty\frac{\Gamma(p-\a-2)}{\Gamma(-\a-2)p!} \right]\\
&\displaystyle\quad +\frac{1}{\Gamma(\a)}  \left(\ln\frac{t}{a}\right)^{\a+2}
\sum_{p=3}^\infty\frac{\Gamma(p-\a-2)}{\Gamma(-\a+1)(p-3)!\left(\ln\frac{t}{a}\right)^p}
\int_a^t \left(\ln\frac{\tau}{a}\right)^{p-3} \frac{x(\tau)}{\tau}d\tau,
\end{array}
$$
or, in a more concise way,
$$
\begin{array}{ll}
\LHI x(t)&=\displaystyle A_0(\alpha)\left(\ln\frac{t}{a}\right)^{\a} x(t)
+ A_1(\alpha)\left(\ln\frac{t}{a}\right)^{\a+1}t\dot{x}(t)\\
&\quad\displaystyle + A_2(\alpha) \left(\ln\frac{t}{a}\right)^{\a+2}(t\dot{x}(t)
+t^2\ddot{x}(t))+ \sum_{p=3}^\infty B(\a,p)\left(\ln\frac{t}{a}\right)^{\a+2-p}V_p(t),
\end{array}
$$
with
$$
\begin{array}{ll}
A_0(\a)&=\displaystyle\frac{1}{\Gamma(\a+1)}\left[1
+\sum_{p=3}^\infty\frac{\Gamma(p-\a-2)}{\Gamma(-\a)(p-2)!}\right],\\
A_1(\a)&=\displaystyle\frac{1}{\Gamma(\a+2)}\left[1
+\sum_{p=2}^\infty\frac{\Gamma(p-\a-2)}{\Gamma(-\a-1)(p-1)!}\right],\\
A_2(\a)&=\displaystyle\frac{1}{\Gamma(\a+3)}\left[1
+\sum_{p=1}^\infty\frac{\Gamma(p-\a-2)}{\Gamma(-\a-2)p!}\right],
\end{array}
$$
\begin{equation}
\label{Case:B3}
B(\a,p)=\displaystyle\frac{\Gamma(p-\a-2)}{\Gamma(\a)\Gamma(1-\a)(p-2)!},
\end{equation}
and
\begin{equation}
\label{Case:Vp3}
V_p(t)=\int_a^t (p-2)\left(\ln\frac{\tau}{a}\right)^{p-3}\frac{x(\tau)}{\tau}d\tau,
\end{equation}
where we assume the series and the integral $V_p$ to be convergent.

\begin{remark}
\index{Fractional! differential equation}
\index{Fractional! integral equation}
When useful, namely on fractional differential and integral equations,
we can define $V_p$ as in \eqref{Case:Vp3}
by the solution of the system
$$
\left\{
\begin{array}{l}
\dot{V_p}(t)=\displaystyle (p-2)\left(
\ln\frac{t}{a}\right)^{p-3}\frac{x(t)}{t}\\
V_p(a)=0,
\end{array}
\right.
$$
for all $p=3,4,\ldots$
\end{remark}

We now discuss the convergence of the series involved
in the definitions of $A_i(\a)$, for $i \in \{0,1,2\}$.
Simply observe that
$$
\sum_{p=3-i}^\infty\frac{\Gamma(p-\a-2)}{\Gamma(-\a-i)(p-2+i)!}
={_1F_0} (-\a-i,1)-1,
$$
and ${_1F_0}(a,x)$ converges absolutely when $|x|=1$
if $a<0$ (\cite[Theorem~2.1.2]{Andrews}).

For numerical purposes, only finite sums are considered,
and thus the Hadamard left fractional integral\index{Hadamard fractional integral! Left} 
is approximated by the decomposition
\begin{equation}
\begin{array}{ll}
\label{Case:n=3}
\LHI x(t)&\approx\displaystyle A_0(\alpha,N)\left(\ln\frac{t}{a}\right)^{\a} x(t)
+ A_1(\alpha,N)\left(\ln\frac{t}{a}\right)^{\a+1}t\dot{x}(t)\\
&\quad\displaystyle + A_2(\alpha,N) \left(\ln\frac{t}{a}\right)^{\a
+2}(t\dot{x}(t)+t^2\ddot{x}(t))+ \sum_{p=3}^N B(\a,p)\left(
\ln\frac{t}{a}\right)^{\a+2-p}V_p(t),
\end{array}
\end{equation}
with
$$
\begin{array}{ll}
A_0(\a,N)&=\displaystyle\frac{1}{\Gamma(\a+1)}\left[1
+\sum_{p=3}^N\frac{\Gamma(p-\a-2)}{\Gamma(-\a)(p-2)!}\right],\\
A_1(\a,N)&=\displaystyle\frac{1}{\Gamma(\a+2)}\left[1
+\sum_{p=2}^N\frac{\Gamma(p-\a-2)}{\Gamma(-\a-1)(p-1)!}\right],\\
A_2(\a,N)&=\displaystyle\frac{1}{\Gamma(\a+3)}\left[1
+\sum_{p=1}^N\frac{\Gamma(p-\a-2)}{\Gamma(-\a-2)p!}\right],
\end{array}
$$
$B(\a,p)$ and $V_p(t)$ as in \eqref{Case:B3}--\eqref{Case:Vp3}, and $N\geq3$.

Following similar arguments as done for $n=2$,
we can prove the general case with an expansion
up to the derivative of order $n$. First,
we introduce a notation. Given $k\in\mathbb N \cup\{0\}$,
we define the sequences $x_{k,0}(t)$ and $x_{k,1}(t)$
recursively by the formulas
$$
x_{0,0}(t)=x(t) \mbox{ and } x_{k+1,0}(t)
=t\frac{d}{dt} x_{k,0}(t),
\mbox{ for } k\in\mathbb N\cup\{0\},
$$
and
$$
x_{0,1}(t)=\dot{x}(t) \mbox{ and } x_{k+1,1}(t)
=\frac{d}{dt} (t x_{k,1}(t)),
\mbox{ for } k\in\mathbb N\cup\{0\}.
$$

\begin{theorem}
Let $n\in\mathbb N$, $0<a<b$ and $x:[a,b]\to\mathbb R$
be a function of class $C^{n+1}$. Then,
$$
\LHI x(t)=\sum_{i=0}^{n}A_i(\alpha)\left(\ln\frac{t}{a}\right)^{\a+i} x_{i,0}(t)
+\sum_{p=n+1}^\infty B(\a,p)\left(\ln\frac{t}{a}\right)^{\a+n-p}V_p(t)
$$
with
$$
\begin{array}{ll}
A_i(\a)&=\displaystyle\frac{1}{\Gamma(\a+i+1)}\left[1
+\sum_{p=n-i+1}^\infty\frac{\Gamma(p-\a-n)}{\Gamma(-\a-i)(p-n+i)!}\right],\\[12pt]
B(\a,p)&=\displaystyle\frac{\Gamma(p-\a-n)}{\Gamma(\a)\Gamma(1-\a)(p-n)!},\\[12pt]
V_p(t)&=\displaystyle\int_a^t (p-n)\left(\ln\frac{\tau}{a}\right)^{p-n-1}\frac{x(\tau)}{\tau}d\tau.
\end{array}
$$
\end{theorem}

\begin{proof}
Applying integration by parts repeatedly
and the binomial formula, we arrive to
$$
\begin{array}{ll}
\LHI x(t)&=\displaystyle\sum_{i=0}^{n}\frac{1}{\Gamma(\a+i+1)}
\left(\ln\frac{t}{a}\right)^{\a+i}x_{i,0}(a)\\
&\quad\displaystyle+\frac{1}{\Gamma(\a+n+1)}
\left(\ln\frac{t}{a}\right)^{\a+n}
\sum_{p=0}^\infty\frac{\Gamma(p-\a-n)}{\Gamma(-\a-n)p!\left(
\ln\frac{t}{a}\right)^p} \int_a^t \left(\ln\frac{\tau}{a}\right)^p x_{n,1}(\tau)d\tau.
\end{array}
$$
To achieve the expansion formula, we repeat the same procedure as for the case $n=2$:
we split the sum into two parts (the first term plus the remaining)
and integrate by parts the second one. The convergence of the series
$A_i(\a)$ is ensured by the relation
$$
\sum_{p=n-i+1}^\infty\frac{\Gamma(p-\a-n)}{\Gamma(-\a-i)(p-n+i)!}
={_1F_0} (-\a-i,1)-1.
$$
\end{proof}
An estimation for the error bound is given in Section~\ref{ErrorSecHad}.

Similarly to what was done with the left fractional integral,
we can also expand the right Hadamard fractional integral.\index{Hadamard fractional integral! Right}

\begin{theorem}
\label{HRFI}
Let $n\in\mathbb N$, $0<a<b$ and $x:[a,b]\to\mathbb R$
be a function of class $C^{n+1}$. Then,
$$
\RHI x(t)=\sum_{i=0}^{n}A_i(\alpha)\left(\ln\frac{b}{t}\right)^{\a+i} x_{i,0}(t)
+\sum_{p=n+1}^\infty B(\a,p)\left(\ln\frac{b}{t}\right)^{\a+n-p}W_p(t)
$$
with
$$
\begin{array}{ll}
A_i(\a)&=\displaystyle\frac{(-1)^i}{\Gamma(\a+i+1)}\left[1
+\sum_{p=n-i+1}^\infty\frac{\Gamma(p-\a-n)}{\Gamma(-\a-i)(p-n+i)!}\right],\\
B(\a,p)&=\displaystyle\frac{\Gamma(p-\a-n)}{\Gamma(\a)\Gamma(1-\a)(p-n)!},\\
W_p(t)&=\displaystyle\int_t^b (p-n)\left(
\ln\frac{b}{\tau}\right)^{p-n-1}\frac{x(\tau)}{\tau}d\tau.
\end{array}
$$
\end{theorem}

\begin{remark}
Analogously to what was done for the left fractional integral,
one can consider an approximation for the right Hadamard fractional
integral\index{Hadamard fractional integral! Right} by considering
finite sums in the expansion obtained in Theorem~\ref{HRFI}.
\end{remark}


\subsection{Examples}

We obtained approximation formulas for the Hadamard fractional integrals. 
The error caused by such decompositions is given later in Section~\ref{ErrorSecHad}. 
In this section we study several cases, comparing the solution with the approximations.
To gather more information on the accuracy, we evaluate the error using the distance
$$
E=\sqrt{\int_a^b \left(\LHI x(t)-\tilde{\LHI}x(t)\right)^2dt},
$$
where $\tilde{\LHI}x(t)$ is the approximated value.

To begin with, we consider $\a=0.5$ and functions $x_1(t)=\ln t$
and $x_2(t)=1$ with $t\in[1,10]$. Then,
$$
\LHIHz x_1(t)=\frac{\sqrt{\ln^3 t}}{\Gamma(2.5)}
\mbox{ and } \LHIHz x_2(t)
=\frac{\sqrt{\ln t}}{\Gamma(1.5)}
$$
(\textrm{cf.} \cite[Property~2.24]{Kilbas}).
We consider the expansion formula for $n=2$
as in \eqref{Case:n=3} for both cases.
We obtain then the approximations
$$
\LHIHz x_1(t)\approx \left[ A_0(0.5,N)+A_1(0.5,N)
+\sum_{p=3}^N B(0.5,p)\frac{p-2}{p-1}\right]\sqrt{\ln^3 t}
$$
and
$$
\LHIHz x_2(t)\approx \left[ A_0(0.5,N)
+\sum_{p=3}^N B(0.5,p)\right]\sqrt{\ln t}.
$$
The results are exemplified in Figures~\ref{ExpLnt}
and \ref{Exp1}. As can be seen, the value $N=3$ is enough
in order to obtain a good accuracy in the sense of the error function.

\begin{figure}[ht!]
\begin{center}
\subfigure[$\LHIHz(\ln t)$]{\label{ExpLnt}\includegraphics[scale=0.54]{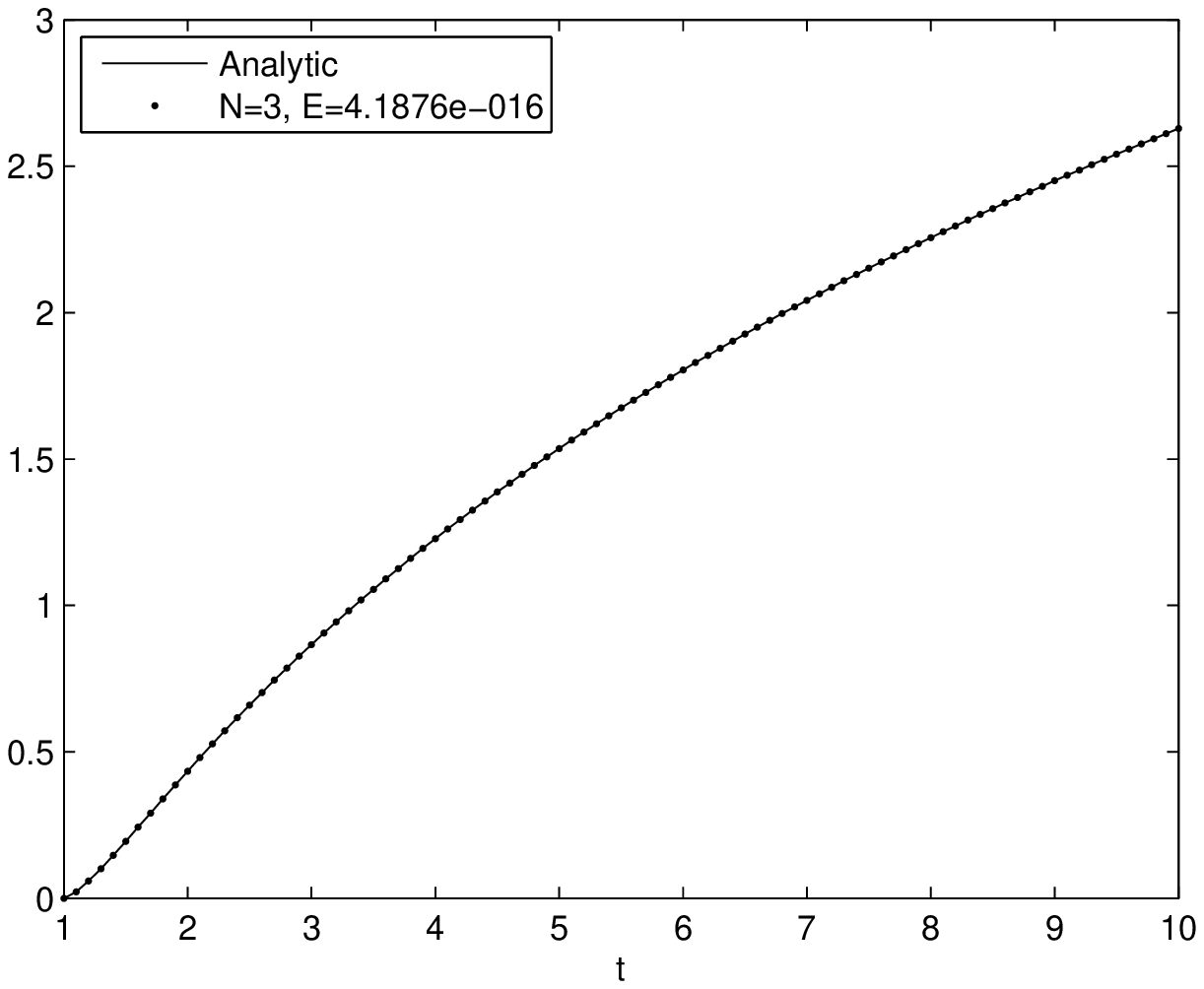}}
\subfigure[$\LHIHz(1)$]{\label{Exp1}\includegraphics[scale=0.54]{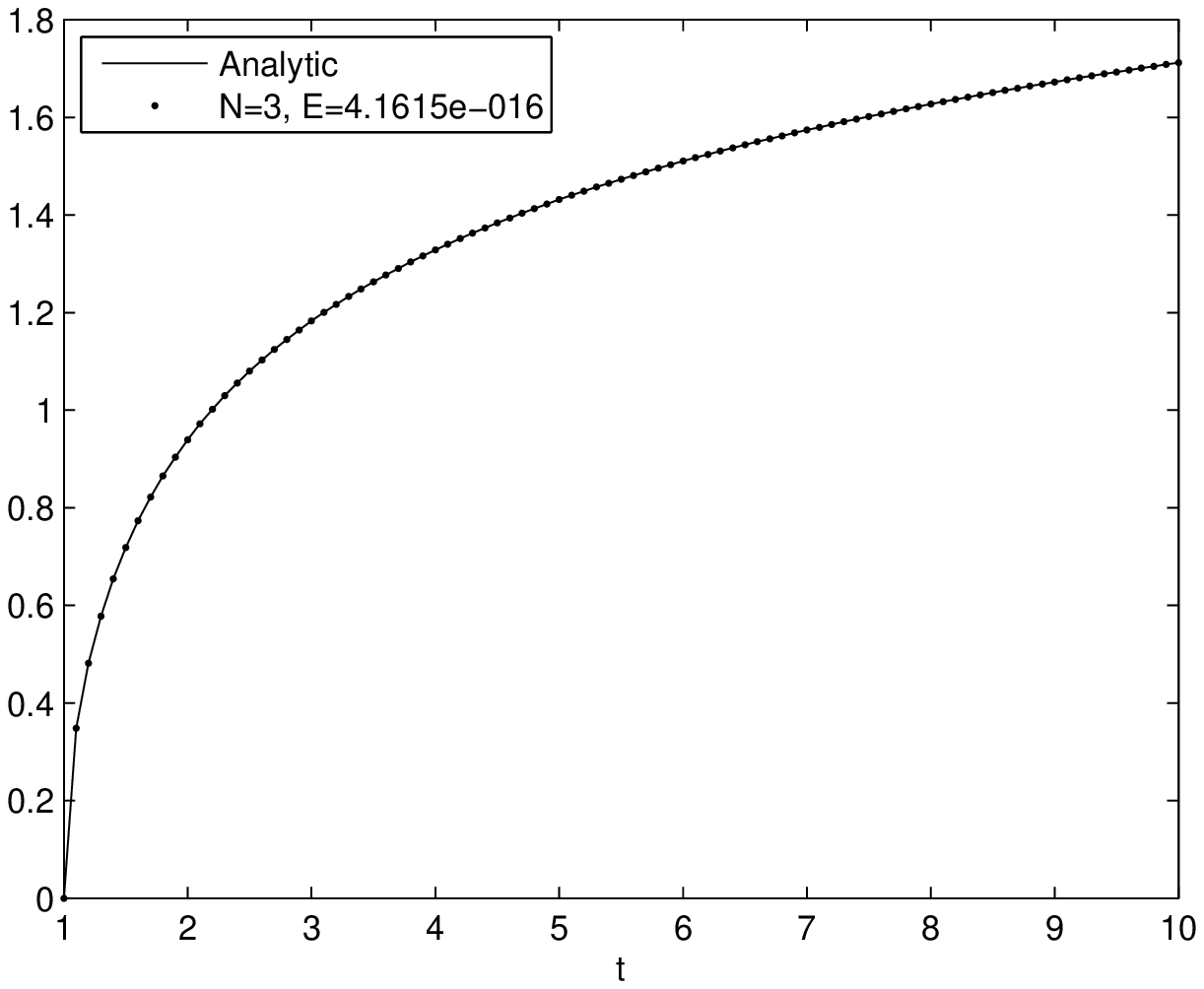}}
\end{center}
\caption{Analytic vs. numerical approximation for $n=2$.}
\end{figure}

We now test the approximation on the power functions $x_3(t)=t^4$
and $x_4(t)=t^9$, with $t\in[1,2]$. Observe first that
$$
\LHIHz(t^k)=\frac{1}{\Gamma(0.5)}\int_1^t
\left(\ln\frac{t}{\tau}\right)^{-0.5}\tau^{k-1}d\tau
=\frac{t^k}{\Gamma(0.5)}\int_0^{\ln t} \xi^{-0.5}e^{-\xi k} d\xi
$$
by the change of variables $\xi=\ln\frac{t}{\tau}$. In our cases,
$$
\LHIHz(t^4)\approx \frac{0.8862269255}{\Gamma(0.5)}t^4 \mbox{erf}(2\sqrt{\ln t})
\mbox { and } \LHIHz(t^9)\approx \frac{0.5908179503}{\Gamma(0.5)}t^9
\mbox{erf}(3\sqrt{\ln t}),
$$
where $\mbox{erf}(\cdot)$ is the error function\index{Error! function}. In Figures~\ref{Expt4}
and \ref{Expt9} we show approximations for several values of $N$.
We mention that, as $N$ increases, the error decreases
and thus we obtain a better approximation.

\begin{figure}[ht!]
\begin{center}
\subfigure[$\LHIHz(t^4)$]{\label{Expt4}\includegraphics[scale=0.54]{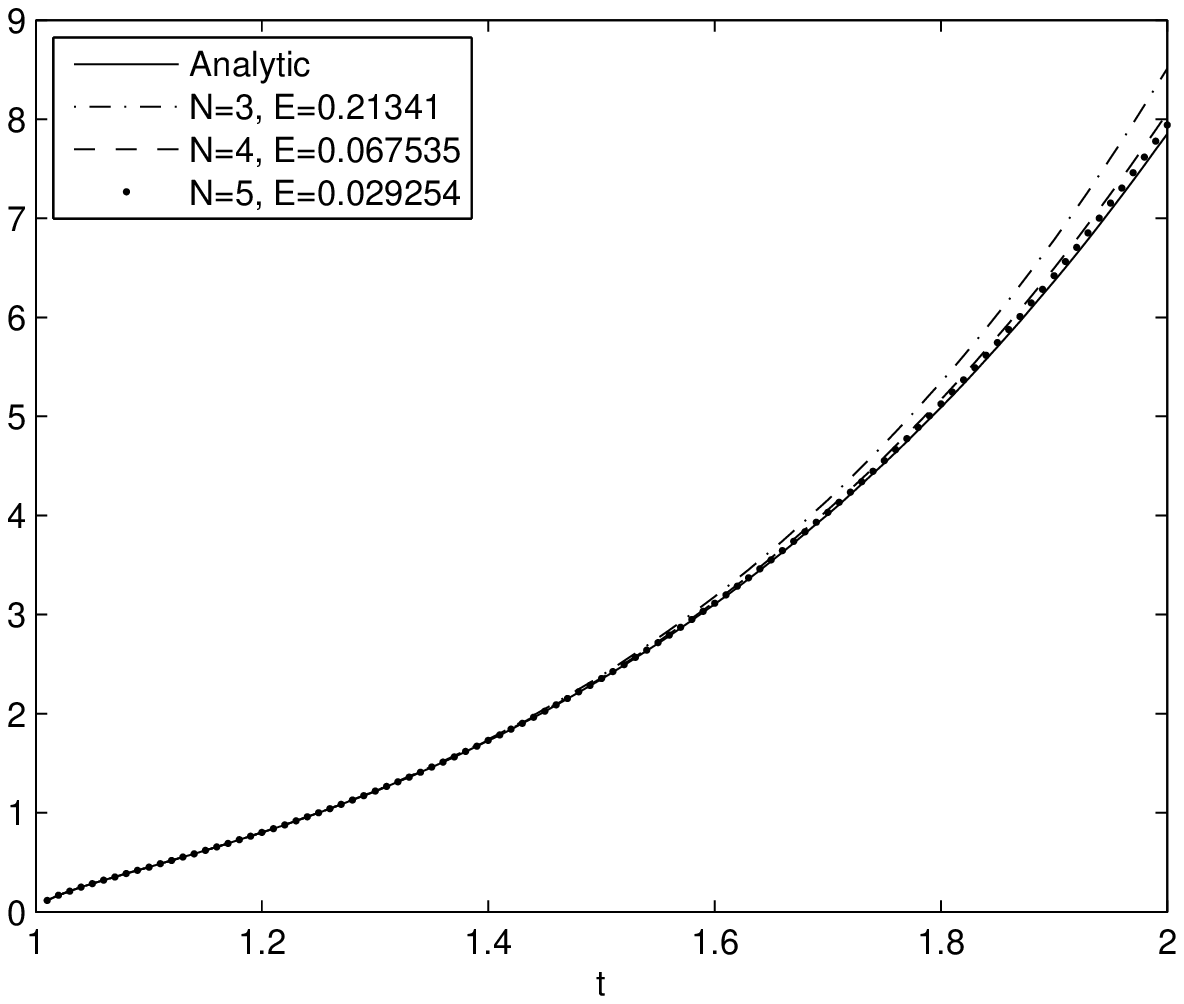}}
\subfigure[$\LHIHz(t^9)$]{\label{Expt9}\includegraphics[scale=0.54]{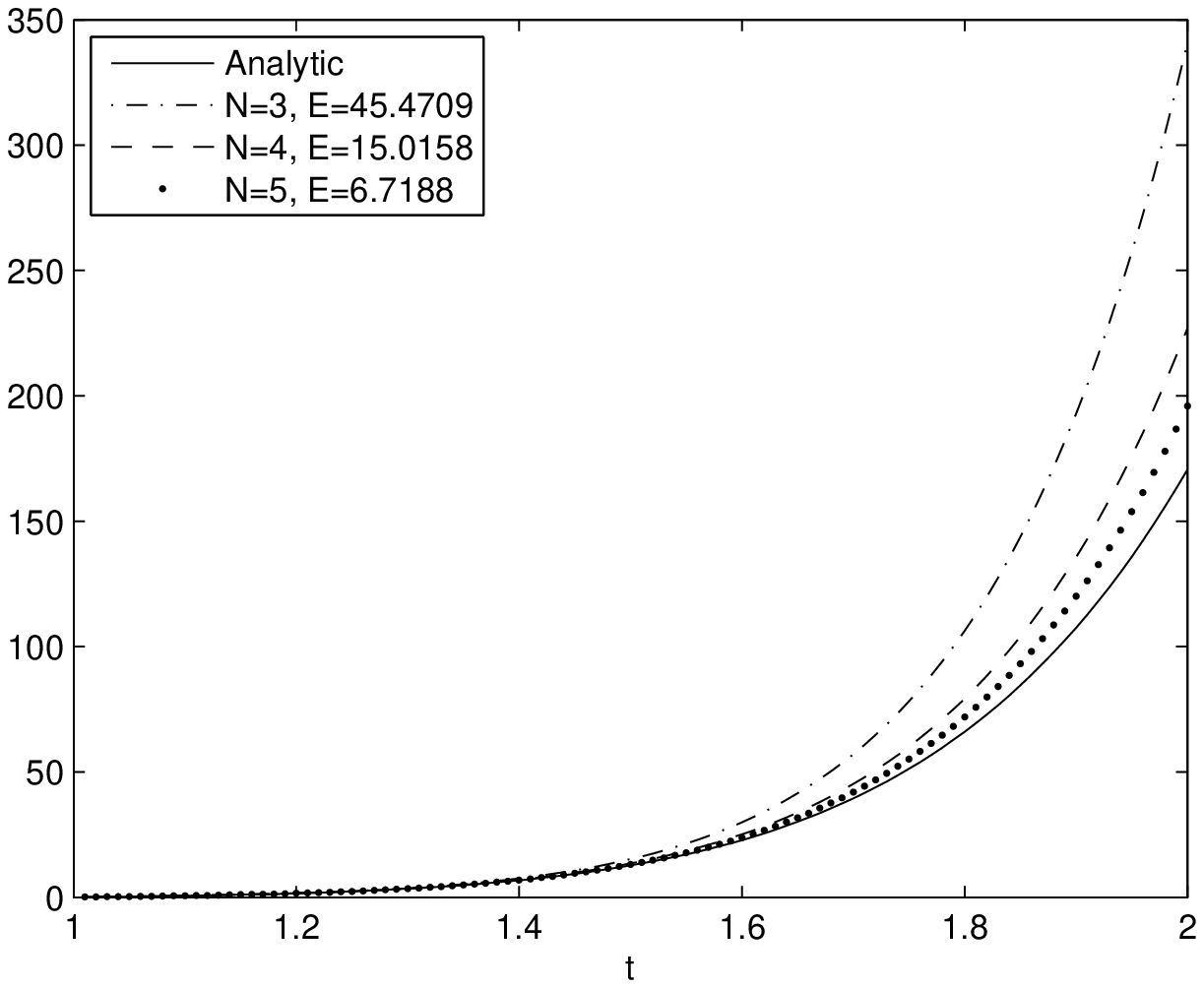}}
\end{center}
\caption{Analytic vs. numerical approximation for $n=2$.}
\end{figure}

Another way to obtain different expansion formulas is to vary $n$.
To exemplify, we choose the previous test functions $x_i$, for $i=1,2,3,4$,
and consider the cases $n=2,3,4$ with $N=5$ fixed. The results are shown
in Figures~\ref{ExpLntNfixed}, \ref{Exp1Nfixed}, \ref{Expt4Nfixed}
and \ref{Expt9Nfixed}. Observe that as $n$ increases, the error may increase.
This can be easily explained by analysis of the error formula, and the values
of the sequence $x_{(k,0)}$ involved. For example, for $x_4$ we have
$x_{(k,0)}(t)=9^kt^9$, for $k=0\ldots,n$. This suggests that, when we increase
the value of $n$ and the function grows fast, in order to obtain a better
accuracy on the method, the value of $N$ should also increase.
\begin{figure}[ht!]
\begin{center}
\subfigure[$\LHIHz(\ln t)$]{\label{ExpLntNfixed}\includegraphics[scale=0.54]{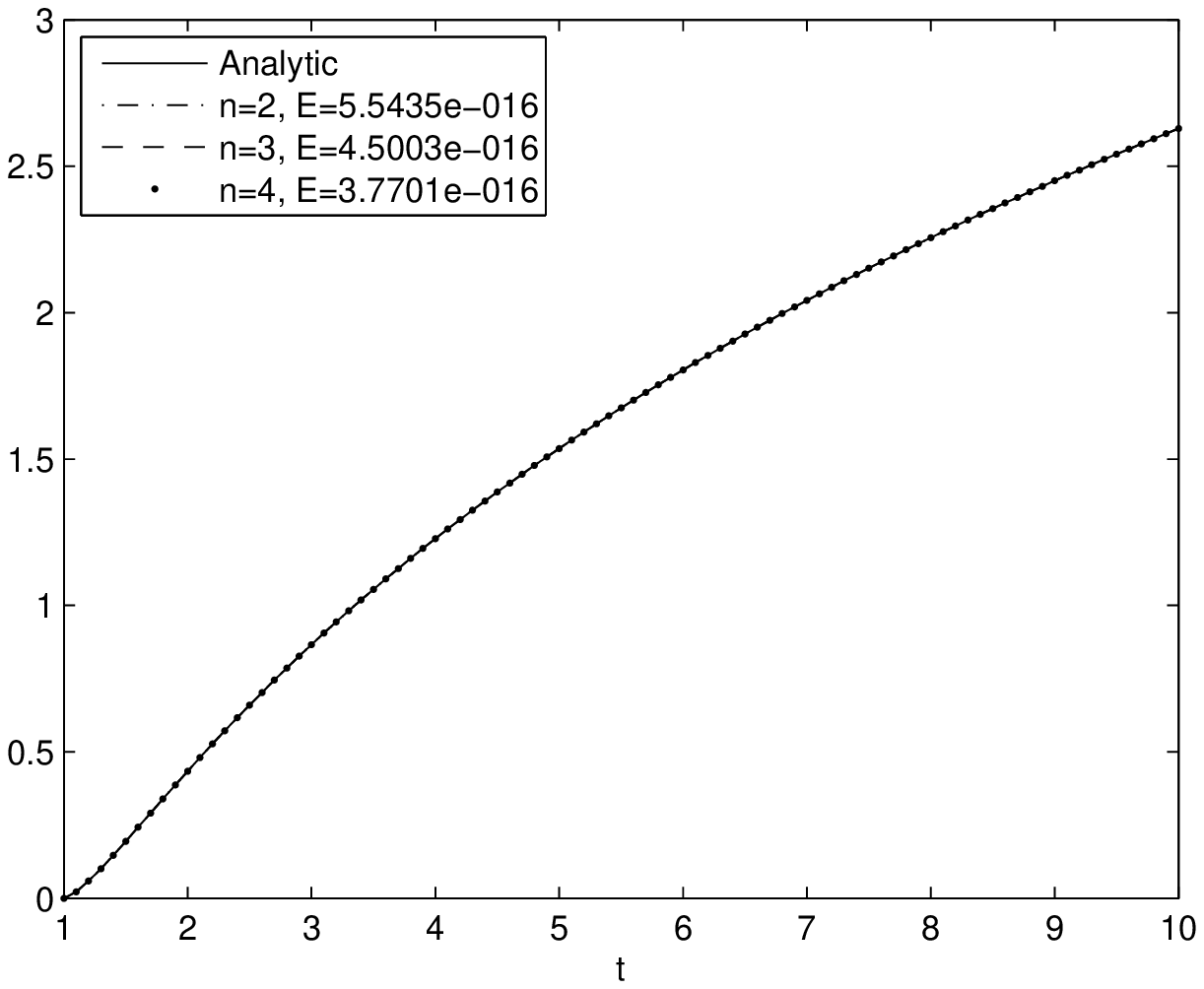}}
\subfigure[$\LHIHz(1)$]{\label{Exp1Nfixed}\includegraphics[scale=0.54]{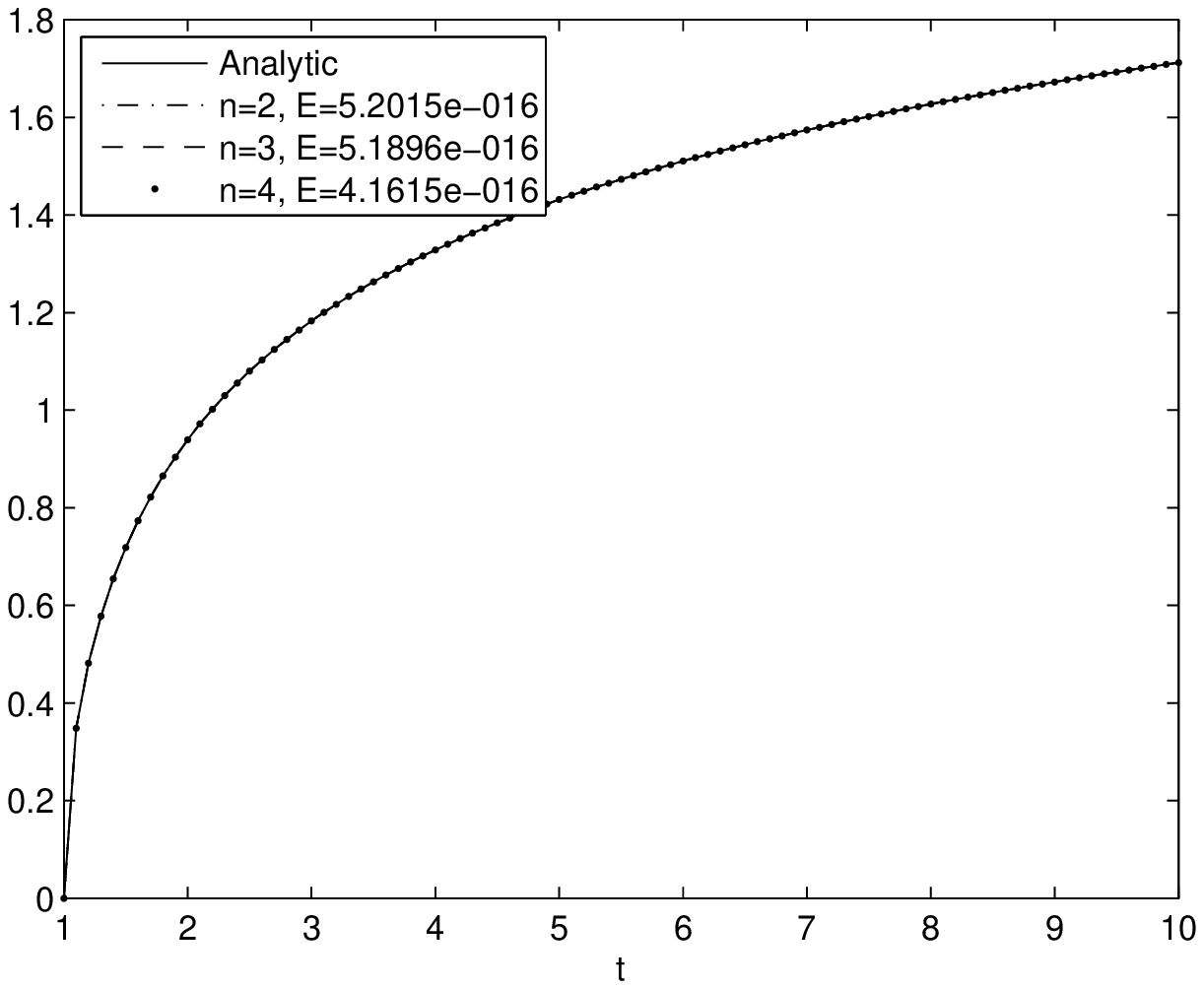}}
\subfigure[$\LHIHz(t^4)$]{\label{Expt4Nfixed}\includegraphics[scale=0.54]{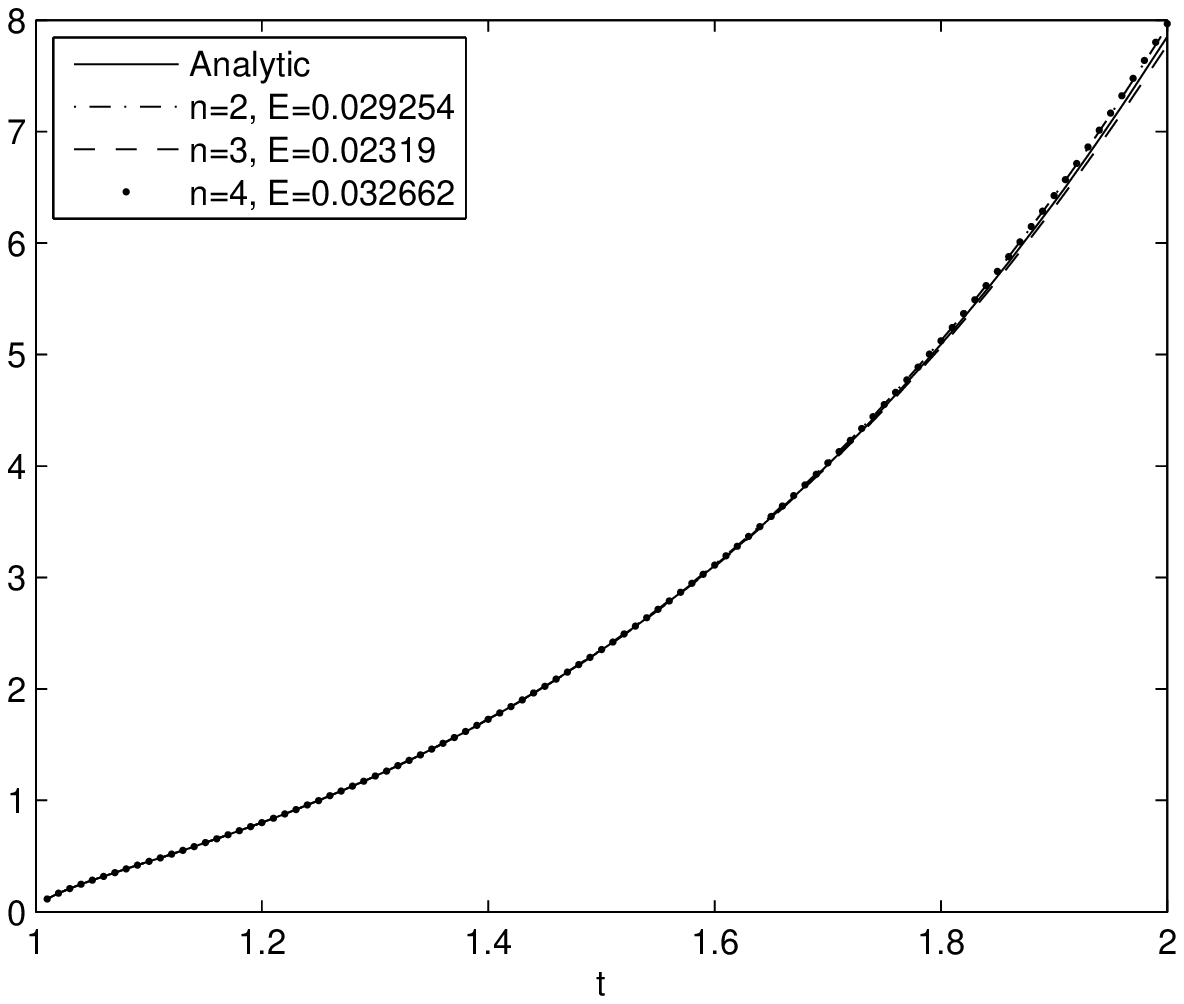}}
\subfigure[$\LHIHz(t^9)$]{\label{Expt9Nfixed}\includegraphics[scale=0.54]{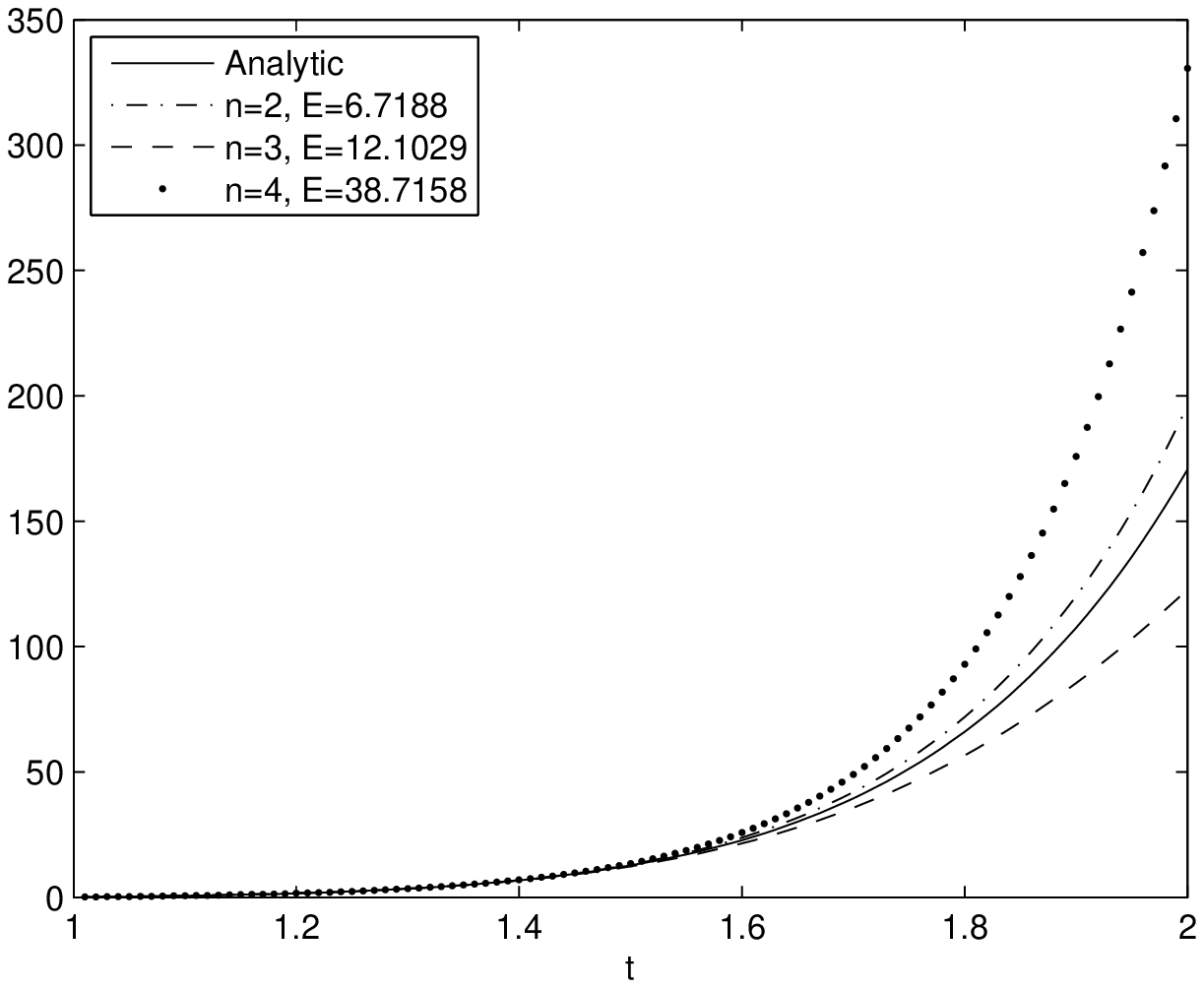}}
\end{center}
\caption{Analytic vs. numerical approximation for $n=2,3,4$ and $N=5$.}
\end{figure}


\section{Error analysis}
\label{ErrorSecHad}

In the previous section we deduced an approximation formula for
the left Riemann--Liouville fractional integral (Eq.~\eqref{Approx:LI}).
The order of magnitude of the coefficients that we ignore during
this procedure is small for the examples that we have chosen
(Tables~\ref{tab1} and \ref{tab2}). The aim of this section
is to obtain an estimation for the error, when considering sums
up to order $N$. We proved that
\begin{multline*}
\LI x(t)=\frac{(t-a)^\a}{\Gamma(\a+1)}x(a)
+\cdots+\frac{(t-a)^{\a+n-1}}{\Gamma(\a+n)}x^{(n-1)}(a)\\
+\frac{(t-a)^{\a+n-1}}{\Gamma(\a+n)}
\int_a^t \left(1-\frac{\tau-a}{t-a}\right)^{\a+n-1}x^{(n)}(\tau)d\tau.
\end{multline*}
Expanding up to order $N$ the binomial, we get
$$
\left(1-\frac{\tau-a}{t-a}\right)^{\a+n-1}
=\sum_{p=0}^N\frac{\Gamma(p-\a-n+1)}{\Gamma(1-\a-n)\,p!}\left(\frac{\tau
-a}{t-a}\right)^p+R_N(\tau),
$$
where
$$
R_N(\tau)=\sum_{p=N+1}^\infty\frac{\Gamma(p-\a-n+1)}{\Gamma(1
-\a-n)\,p!}\left(\frac{\tau-a}{t-a}\right)^p.
$$
Since $\tau\in[a,t]$, we easily deduce an upper bound for $R_N(\tau)$:
\begin{equation*}
\begin{split}
\left|R_N(\tau)\right| &\leq  \sum_{p=N+1}^\infty\left|
\frac{\Gamma(p-\a-n+1)}{\Gamma(1-\a-n)\,p!} \right|
= \sum_{p=N+1}^\infty\left| \binom{\a+n-1}{p} \right|
\leq  \sum_{p=N+1}^\infty \frac{e^{(\a+n-1)^2+\a+n-1}}{p^{\a+n}}\\
&\leq \int_{N}^\infty \frac{e^{(\a+n-1)^2+\a+n-1}}{p^{\a+n}}dp
=\frac{e^{(\a+n-1)^2+\a+n-1}}{(\a+n-1)N^{\a+n-1}}.
\end{split}
\end{equation*}
Thus, we obtain an estimation for the truncation error\index{Error! Truncation} $E_{tr}(\cdot)$:
$$
\left| E_{tr}(t)\right|\leq L_n \frac{(t-a)^{\a+n}
e^{(\a+n-1)^2+\a+n-1}}{\Gamma(\a+n)(\a+n-1)N^{\a+n-1}},
$$
where $\displaystyle L_n=\max_{\tau\in[a,t]}|x^{(n)}(\tau)|$.

We proceed with an estimation for the error on the approximation for the Hadamard fractional integral.
We have proven before that
\begin{eqnarray*}
\LHI x(t)&=&\frac{x(a)}{\Gamma(\a+1)} \left(\ln\frac{t}{a}\right)^{\a}
+\frac{a\dot{x}(a)}{\Gamma(\a+2)} \left(\ln\frac{t}{a}\right)^{\a+1}
+\frac{(a\dot{x}(a)+a^2\ddot{x}(a))}{\Gamma(\a+3)} \left(\ln\frac{t}{a}\right)^{\a+2}\\
&&+\frac{1}{\Gamma(\a+3)}  \left(\ln\frac{t}{a}\right)^{\a+2} \int_a^t
\sum_{p=0}^\infty\Gamma_1(\a,p,t)(\dot{x}(\tau)+3\tau \ddot{x}(\tau)+\tau^2\dddot{x}(\tau))d\tau.
\end{eqnarray*}
When we consider finite sums up to order $N$, the error is given by
$$
\left| E_{tr}(t)\right|=\left|\frac{1}{\Gamma(\a+3)}
\left(\ln\frac{t}{a}\right)^{\a+2}
\int_a^t R_N(\tau) (\dot{x}(\tau)+3\tau \ddot{x}(\tau)
+\tau^2\dddot{x}(\tau))d\tau\right|
$$
with
$$
R_N(\tau)=\sum_{p=N+1}^\infty\frac{\Gamma(p-\a-2)}{\Gamma(-\a-2)p!}\frac{
\left(\ln\frac{\tau}{a}\right)^p}{\left(\ln\frac{t}{a}\right)^p}.
$$
Since $\tau\in[a,t]$, we have
$$
\begin{array}{ll}
|R_N(\tau)| &  \displaystyle \leq \sum_{p=N+1}^\infty\left|
\binom{\a+2}{p} \right|
\leq  \sum_{p=N+1}^\infty \frac{e^{(\a+2)^2+\a+2}}{p^{\a+3}}\\
&\displaystyle\leq \int_{N}^\infty \frac{e^{(\a+2)^2+\a+2}}{p^{\a+3}}dp
=\frac{e^{(\a+2)^2+\a+2}}{(\a+2)N^{\a+2}}.
\end{array}
$$
Therefore,
$$
\left| E_{tr}(t)\right|\leq  \frac{e^{(\a+2)^2+\a+2}}{(\a+2)N^{\a+2}}
\frac{\left(\ln\frac{t}{a}\right)^{\a+2}}{\Gamma(\a+3)}
\left[(t-a)L_1(t)+3(t-a)^2L_2(t)+(t-a)^3L_3(t)\right],
$$
where
$$
L_i(t)=\max_{\tau\in[a,t]}|x^{(i)}(\tau)|,
\quad i \in \{1,2,3\}.
$$
We remark that the error formula tends to zero as $N$ increases. 
Moreover, if we consider the approximation
$$
\LHI x(t)\approx\sum_{i=0}^{n}A_i(\alpha,N)\left(\ln\frac{t}{a}\right)^{\a+i}
x_{i,0}(t)+\sum_{p=n+1}^N B(\a,p)\left(\ln\frac{t}{a}\right)^{\a+n-p}V_p(t)
$$
with $N \geq n+1$ and
$$
A_i(\a,N)=\frac{1}{\Gamma(\a+i+1)}\left[1
+\sum_{p=n-i+1}^N\frac{\Gamma(p-\a-n)}{\Gamma(-\a-i)(p-n+i)!}\right],
$$
then the error is bounded by the expression
$$
\left| E_{tr}(t)\right|\leq L_n(t)\frac{e^{(\a+n)^2+\a+n}}{
\Gamma(\a+n+1)(\a+n)N^{\a+n}} \left(\ln\frac{t}{a}\right)^{\a+n}(t-a),
$$
where
$$
L_n(t)=\max_{\tau\in[a,t]}|x_{n,1}(\tau)|.
$$
For the right Hadamard integral, the error is bounded by
$$
\left| E_{tr}(t)\right|\leq L_n(t)\frac{e^{(\a+n)^2
+\a+n}}{\Gamma(\a+n+1)(\a+n)N^{\a+n}}
\left(\ln\frac{b}{t}\right)^{\a+n}(b-t),
$$
where
$$
L_n(t)=\max_{\tau\in[t,b]}|x_{n,1}(\tau)|.
$$

\CP
\chapter{Direct methods}
\label{Direct}\index{Direct method}

In the presence of fractional operators, the same ideas that were discussed 
in Section~\ref{ClassisDir}, are applied to discretize the problem. 
Many works can be found in the literature that use different types 
of basis functions to establish Ritz-like methods for the fractional 
calculus of variations and optimal control. Nevertheless, 
finite differences have got less interest. A brief introduction
of using finite differences has been made in \cite{Riewe}, which can be regarded
as a predecessor to what we call here an Euler-like direct method.
A generalization of Leitmann's direct method can be found in \cite{AlD},
while \cite{Lotfi} discusses the Ritz direct method for optimal control problems
that can easily be reduced to a problem of the calculus of variations.


\section{Finite differences for fractional derivatives}

Recall the definitions of Gr\"{u}nwald--Letnikov\index{Gr\"{u}nwald--Letnikov fractional derivative! Left}, 
e.g. \eqref{LGLdef}. It exhibits a finite difference\index{Finite! differences} 
nature involving an infinite series. For numerical purposes we need a finite sum in \eqref{LGLdef}. 
Given a grid on $[a,b]$ as $a=t_0,t_1,\ldots,t_n=b$, where $t_i=t_0+ih$ for some $h>0$, 
we approximate the left Riemann--Liouville 
derivative\index{Riemann--Liouville fractional derivative! Left} as
\begin{equation}
\label{GLApprx}
\LDa x(t_i)\approx \frac{1}{h^\a} \sum_{k=0}^{i}\w x(t_i-kh),
\end{equation}
where $\w=(-1)^k\binom{\a}{k}=\frac{\Gamma(k-\a)}{\Gamma(-\a)\Gamma(k+1)}$.
\begin{remark}
Similarly, one can approximate the right Riemann--Liouville  derivative by
\begin{equation}
\label{GLApprxR}
\RD x(t_i)\approx \frac{1}{h^\a} \sum_{k=0}^{n-i}\w x(t_i+kh).
\end{equation}
\end{remark}
\begin{remark}
The Gr\"{u}nwald--Letnikov approximation of Riemann--Liouville 
is a first order approximation \cite{Podlubny}, i.e.,
\begin{equation}
\LDa x(t_i)= \frac{1}{h^\a} \sum_{k=0}^{i}\w x(t_i-kh)+\mathcal{O}(h).\nonumber
\end{equation}
\end{remark}
\begin{remark}
It has been shown that the implicit Euler method solution to a certain fractional 
partial differential equation based on Gr\"{u}nwald--Letnikov approximation to the 
fractional derivative, is unstable \cite{Meerschaert}. Therefore, discretizing 
fractional derivatives,  shifted Gr\"{u}nwald--Letnikov derivatives are used 
and despite the slight difference they exhibit a stable performance at least 
for certain cases. The left shifted Gr\"{u}nwald--Letnikov derivative is defined by
\begin{equation*}
\sGLa x(t_i)\approx \frac{1}{h^\a} \sum_{k=0}^{i}\w x(t_i-(k-1)h).
\end{equation*}
\end{remark}

Other finite difference approximations can be found in the literature. 
Specifically, we refer to \cite{Kai2}, Diethelm's backward finite differences formula 
for Caputo fractional derivative, with $0<\a<2$ and $\a\neq 1$, 
that is an approximation of order $\mathcal{O}(h^{2-\a})$:
$$
\LDC x(t_i)\approx \frac{h^{-\a}}{\Gamma(2-\a)}\sum_{j=0}^i a_{i,j}
\left(x_{i-j}-\sum_{k=0}^{\lfloor\a\rfloor}\frac{(i-j)^kh^k}{k!}x^{(k)}(a)\right),
$$
where
$$
a_{i,j}=\left \{
\begin{array}{ll}
1,& \text{if } i=0,\\
(j+1)^{1-\a}-2j^{1-\a}+(j-1)^{1-\a},&\text{if } 0<j<i,\\
(1-\a)i^{-\a}-i^{1-\a}+(i-1)^{1-\a},&\text{if } j=i.
\end{array}\right.
$$


\section{Euler-like direct method for variational problems}

\subsection{Euler's classic direct method}

Euler's method\index{Euler method} in the classical theory of the calculus 
of variations uses finite difference approximations for derivatives and 
is also referred as the method of finite differences. The basic idea 
of this method is that instead of considering the values of a functional
\begin{equation*}
J[x(\cdot)]=\int_a^b L(t, x(t), \dot{x}(t))dt
\end{equation*}
with boundary conditions $x(a)=x_a$ and $ x(b)=x_b$, on arbitrary admissible curves, 
we only track the values at an $n+1$ grid points, $t_i,~i=0,\ldots,n$, 
of the interested time interval \cite{fda12}. The functional $J[x(\cdot)]$ 
is then transformed into a function $\Psi(x(t_1),x(t_2),\ldots,x(t_{n-1}))$ 
of the values of the unknown function on mesh points. Assuming $h=t_{i}-t_{i-1}$, 
$x(t_i)=x_i$ and $\dot{x}_i\approx \frac{x_{i}-x_{i-1}}{h}$, one has
\begin{eqnarray*}
J[x(\cdot)]&\approx&\Psi(x_1,x_2,\ldots,x_{n-1})
=h\sum_{i=1}^{n}L\left(t_i, x_i,\frac{x_{i}-x_{i-1}}{h}\right),\\
&&x_0=x_a,\quad x_n=x_b.
\end{eqnarray*}
The desired values of $x_i$, $i=1,\ldots,n-1$, are the extremum 
of the multi-variable function $\Psi$ which is the solution to the system
$$
\frac{\partial \Psi}{\partial x_i}=0,\quad i=1,\ldots,n-1.
$$

The fact that only two terms in the sum, $(i-1)$th and $i$th, depend on $x_i$ makes 
it rather easy to find the extremum of $\Psi$ solving a system of algebraic equations.
For each $n$, we obtain a polygonal line which is an approximate solution of the 
original problem. It has been shown that passing to the limit as $h\rightarrow 0$, 
the linear system corresponding to finding the extremum of $\Psi$ is equivalent 
to the Euler--Lagrange equation for problem \cite{Tuckey}.


\subsection{Euler-like direct method}

As mentioned earlier, we consider a simple version of fractional variational problems 
where the fractional term has a Riemann--Liouville derivative on a finite time interval 
$[a,b]$. The boundary conditions are given and we approximate the derivative using 
Gr\"{u}nwald--Letnikov approximation given by \eqref{GLApprx}. In this context, 
we discretize the functional in \eqref{Functional} using a simple quadrature rule 
on the mesh points, $a=t_0,t_1,, \ldots,t_n=b$, with $h=\frac{b-a}{n}$. The goal 
is to find the values $x_1,\ldots,x_{n-1}$ of the unknown function $x(\cdot)$ 
at the points $t_i$, $i=1,\ldots,n-1$. The values of $x_0$ and $x_n$ are given. 
Applying the quadrature rule gives
\begin{eqnarray}
J[x(\cdot)]&=& \sum_{i=1}^{n}\int_{t_{i-1}}^{t_{i}} L(t_i, x_i, \LDai x_i)dt
\approx \sum_{i=1}^{n} h L(t_i, x_i, \LDai x_i),\nonumber
\end{eqnarray}
and by approximating the fractional derivatives at mesh points using \eqref{GLApprx} we have
\begin{equation}
\label{disFuncl}
J[x(\cdot)]\approx \sum_{i=1}^{n}hL\left(t_i, x_i, \frac{1}{h^\a} \sum_{k=0}^{i}\w x_{i-k}\right).
\end{equation}
Hereafter the procedure is the same as in classical case. The right-hand-side of \eqref{disFuncl} 
can be regarded as a function $\Psi$ of $n-1$ unknowns $\mathbf{x}=(x_1,x_2,\ldots,x_{n-1})$,
\begin{equation}
\label{sumFrac}
\Psi(\mathbf{x})=\sum_{i=1}^{n}hL\left(t_i, x_i, \frac{1}{h^\a} \sum_{k=0}^{i}\w x_{i-k}\right).
\end{equation}
To find an extremum for $\Psi$, one has to solve the following system of algebraic equations:
\begin{equation}
\label{AlgSys}
\frac{\partial \Psi}{\partial x_i}=0,\qquad i=1,\ldots,n-1.
\end{equation}
Unlike the classical case, all terms, starting from $i$th term, 
in \eqref{sumFrac} depend on $x_i$ and we have
\begin{eqnarray}
\label{GenSys}
\frac{\partial \Psi}{\partial x_i}
&=&h\frac{\partial L}{\partial x}(t_i,x_i,\LDai x_i)+h\sum_{k=0}^{n-i}\frac{1}{h^\a}\w
\frac{\partial L}{\partial~\LDa x}(t_{i+k},x_{i+k},\LDaik x_{i+k}).
\end{eqnarray}
Equating the right hand side of \eqref{GenSys} with zero one has
\begin{align}
&\frac{\partial L}{\partial x}(t_i,x_i,\LDai x_i)+\frac{1}{h^\a}
\sum_{k=0}^{n-i}\w\frac{\partial L}{\partial~\LDa x}(t_{i+k},x_{i+k},
\LDaik x_{i+k})=0.\nonumber
\end{align}
Passing to the limit and considering the approximation formula for the right 
Riemann--Liouville derivative, equation \eqref{GLApprxR}, 
it is straightforward to verify that:

\begin{theorem}
The Euler-like method for a fractional variational problem of the form \eqref{Functional} 
is equivalent to the fractional Euler--Lagrange equation\index{Euler--Lagrange equation! Fractional}
\begin{equation}
\frac{\partial L}{\partial x}+\RD\frac{\partial L}{\partial~\LDa x}=0,\nonumber
\end{equation}
as the mesh size, $h$, tends to zero.
\end{theorem}

\begin{proof}
Consider a minimizer $(x_1,\ldots,x_{n-1})$ of $\Psi$, a variation\index{Variation}
function $\eta\in C[a,b]$ with $\eta(a)=\eta(b)=0$ and define $\eta_i=\eta(t_i)$, 
for $i=0,\ldots,n$. We remark that $\eta_0=\eta_n=0$ and that 
$(x_1+\epsilon\eta_1,\ldots,x_{n-1}+\epsilon\eta_{n-1})$ is a variation of 
$(x_1,\ldots,x_{n-1})$, with $|\epsilon|<r$, for some fixed $r>0$. Therefore, 
since $(x_1,\ldots,x_{n-1})$ is a minimum for $\Psi$, proceeding with Taylor's expansion, we deduce
\begin{eqnarray*}
0&\leq & \Psi(x_1+\epsilon\eta_1,\ldots,x_{n-1}+\epsilon\eta_{n-1})-\Psi(x_1,\ldots,x_{n-1})\\
&=&\epsilon\sum_{i=1}^n h\left[ \frac{\partial L}{\partial x}[i]\eta_i
+\frac{\partial L}{\partial {_aD_t^{\a}}}[i] \frac{1}{h^\a}
\sum_{k=0}^i(\omega^\a_k) \eta_{i-k} \right]+\mathcal{O}(\epsilon),
\end{eqnarray*}
where
$$
[i]=\left(t_i,x_i,\frac{1}{h^\a} \sum_{k=0}^i(\omega^\a_k) x_{i-k} \right).
$$
Since $\epsilon$ takes any value, it follows that
\begin{equation}
\label{sum1}
\sum_{i=1}^n h\left[ \frac{\partial L}{\partial x}[i]\eta_i
+\frac{\partial L}{\partial {_aD_t^{\a}}}[i] \frac{1}{h^\a}
\sum_{k=0}^i(\omega^\a_k) \eta_{i-k} \right]=0.
\end{equation}
On the other hand, since $\eta_0=0$, reordering the terms of the sum, 
it follows immediately that
$$
\sum_{i=1}^n \frac{\partial L}{\partial {_aD_t^{\a}}}[i]
\sum_{k=0}^i(\omega^\a_k) \eta_{i-k}
= \sum_{i=1}^n \eta_i \sum_{k=0}^{n-i}(\omega^\a_k) 
\frac{\partial L}{\partial {_aD_t^{\a}}}[i+k].
$$
Substituting this relation into equation \eqref{sum1}, we obtain
$$
\sum_{i=1}^n\eta_i h\left[ \frac{\partial L}{\partial x}[i]+\frac{1}{h^\a} 
\sum_{k=0}^{n-i}(\omega^\a_k)\frac{\partial L}{\partial {_aD_t^{\a}}}[i+k] \right]=0.
$$
Since $\eta_i$ is arbitrary, for $i=1,\ldots,n-1$, we deduce that
$$
\frac{\partial L}{\partial x}[i]+\frac{1}{h^\a} 
\sum_{k=0}^{n-i}(\omega^\a_k)\frac{\partial L}{\partial {_aD_t^{\a}}}[i+k]=0, 
\quad \mbox{for } i=1,\ldots,n-1.
$$
Let us study the case when $n$ goes to infinity. 
Let $\overline t \in ]a,b[$ and $i\in\{1,\ldots,n\}$ such that 
$t_{i-1}<\overline t \leq t_i$. First observe that in such case, 
we also have $i\to\infty$ and $n-i\to\infty$. 
In fact, let $i\in\{1,\ldots,n\}$ be such that
$$
a+(i-1)h<\overline t\leq a+ih.
$$
So, $i<(\overline t-a)/h+1$, which implies that
$$
n-i>n\frac{b-\overline t}{b-a}-1.
$$
Then
$$
\lim_{n\to\infty,i\to\infty}t_i=\overline t.
$$
Assume that there exists a function $\overline x\in C[a,b]$ satisfying
$$
\forall \epsilon>0\, \exists N \, 
\forall n\geq N \,: |x_i-\overline x(t_i)|
<\epsilon, \quad \forall i=1,\ldots,n-1.
$$
As $\overline x$ is uniformly continuous, we have
$$
\forall \epsilon>0\, \exists N \, \forall n\geq N \,: 
|x_i-\overline x(\overline t)|<\epsilon, \quad \forall i=1,\ldots,n-1.
$$
By the continuity assumption of $\overline x$, we deduce that
$$
\lim_{n\to\infty,i\to\infty}\frac{1}{h^\a} \sum_{k=0}^{n-i}(\omega^\a_k)
\frac{\partial L}{\partial {_aD_t^{\a}}}[i+k]={_tD^{\a}_b }
\frac{\partial L}{\partial {_aD_t^{\a}}}(\overline t, 
\overline x (\overline t),{_aD_{\overline t}^{\a}} \overline x(\overline t)).
$$
For $n$ sufficiently large (and therefore $i$ also sufficiently large),
$$
\lim_{n\to\infty,i\to\infty}\frac{\partial L}{\partial x}[i]
=\frac{\partial L}{\partial x}(\overline t, 
\overline x (\overline t),{_aD_{\overline t}^{\a}} \overline x(\overline t)).
$$
In conclusion,
\begin{equation}
\label{fracELE}
\frac{\partial L}{\partial x}(\overline t, \overline x (\overline t),{_aD_{\overline t}^{\a}} 
\overline x(\overline t))+{_tD^{\a}_b }\frac{\partial L}{\partial {_aD_t^{\a}}}(\overline t, 
\overline x (\overline t),{_aD_{\overline t}^{\a}} \overline x(\overline t))=0.
\end{equation}
Using the continuity condition, we prove that the fractional Euler--Lagrange equation 
\eqref{fracELE} for all values on the closed interval $a\leq t\leq b$ holds.
\end{proof}


\subsection{Examples}

Now we apply Euler-like direct method to some test problems for which 
the exact solutions are in hand. Although we propose problems on 
to the interval $[0,1]$, moving to arbitrary intervals is a matter 
of more computations. To measure the errors related to approximations, 
different norms can be used. Since a direct method seeks for the function 
values at certain points, we use the maximum norm to determine how close 
we can get to the exact value at that point. Assume that the exact value 
of the function $x(\cdot)$, at the point $t_i$, is $x(t_i)$ 
and it is approximated by $x_i$. The error is defined as
\begin{equation}
\label{ErrorDis}
E=\max\{|x(t_i)-x_i|,~i=1,2,\ldots,n\}.
\end{equation}

\begin{example}
\label{Example1}
Our goal here is to minimize a quadratic Lagrangian\index{Lagrangian} 
on $[0,1]$ with fixed boundary conditions. 
Consider the following minimization problem:
\begin{equation}
\label{Exmp1}
\left\{
\begin{array}{l}
J[x(\cdot)]=\int_0^1 \left(\LDz x(t)
-\frac{2}{\Gamma(2.5)}t^{1.5}\right)^2 dt \rightarrow \min\\
x(0)=0,~x(1)=1.
\end{array}
\right.
\end{equation}
Since the Lagrangian is always positive, 
problem \eqref{Exmp1} attains its minimum when
$$
\LDz x(t)-\frac{2}{\Gamma(2.5)}t^{1.5}=0,
$$
and has the obvious solution of the form $x(t)=t^2$ 
because $\LDz t^2=\frac{2}{\Gamma(2.5)}t^{1.5}$.
\end{example}

To begin with, we approximate the fractional derivative by
\begin{equation}
\LDz x(t_i)\approx \frac{1}{h^{0.5}} \sum_{k=0}^{i}\wh x(t_i-kh)\nonumber
\end{equation}
for a fixed $h>0$. The functional is now transformed into
\begin{equation}
J[x(\cdot)]\approx\int_0^1 \left(\frac{1}{h^{0.5}} 
\sum_{k=0}^{i}\wh x_{i-k}-\frac{2}{\Gamma(2.5)}t^{1.5}\right)^2dt.\nonumber
\end{equation}
Finally, we approximate the integral by a rectangular 
rule and end with the discrete problem
\begin{equation}
\Psi(\mathbf{x})=\sum_{i=1}^{n}h\left(\frac{1}{h^{0.5}} 
\sum_{k=0}^{i}\wh x_{i-k}-\frac{2}{\Gamma(2.5)}t^{1.5}_i\right)^2.\nonumber
\end{equation}

Since the Lagrangian\index{Lagrangian} in this example is quadratic, system \eqref{AlgSys} 
has a linear form and therefore is easy to solve. Other problems may end with a system 
of nonlinear equations. Simple calculations lead to the system
\begin{equation}
\label{Ex1Sys}
\mathbf{A}\mathbf{x}=\mathbf{b},
\end{equation}
in which
\begin{equation*}
\mathbf{A}=\left[\begin{array}{llll}
\sum_{i=0}^{n-1}A_i^2        & \sum_{i=1}^{n-1}A_{i}A_{i-1} & \cdots & \sum_{i=n-2}^{n-1}A_{i}A_{i-(n-2)} \\
\sum_{i=0}^{n-2}A_{i}A_{i+1} & \sum_{i=1}^{n-2}A_i^2        & \cdots & \sum_{i=n-3}^{n-2}A_{i}A_{i-(n-3)} \\
\sum_{i=0}^{n-3}A_{i}A_{i+2} & \sum_{i=1}^{n-3}A_{i}A_{i+1} & \cdots & \sum_{i=n-4}^{n-3}A_{i}A_{i-(n-4)} \\
\vdots                       & \vdots                       & \ddots & \vdots                        \\
\sum_{i=0}^{1}A_{i}A_{i+n-2} & \sum_{i=0}^{1}A_{i}A_{i+n-3} & \cdots & \sum_{i=0}^{1}A_i^2
\end{array}\right]
\end{equation*}
where $A_i=(-1)^{i}h^{1.5}\binom{0.5}{i}$, $\x=(x_1,\cdots,x_{n-1})^T$ 
and $\mathbf{b}=(b_1,\cdots,b_{n-1})^T$ with
\begin{equation*}
b_i=\sum_{k=0}^{n-i}\frac{2h^2A_k}{\Gamma(2.5)}t_{k+i}^{1.5}
-A_{n-i}A_0-\left(\sum_{k=0}^{n-i}A_kA_{k+i}\right).
\end{equation*}
The linear system \eqref{Ex1Sys} is easily solved for different values of $n$. 
As indicated in Figure~\ref{Ex1Fig}, by increasing the value of $n$ we get better solutions.
\begin{figure}[!tp]
\begin{center}
\includegraphics[scale=.8]{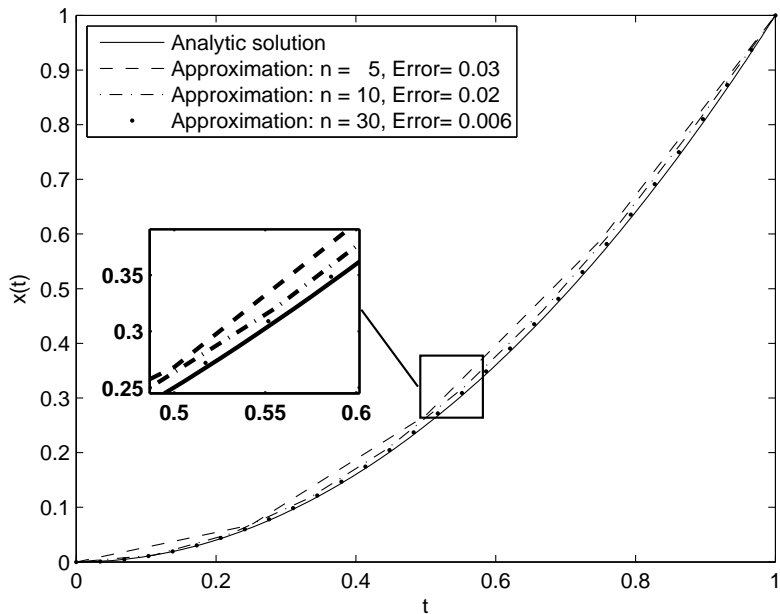}
\end{center}
\caption{Analytic and approximate solutions of Example~\ref{Example1}.}\label{Ex1Fig}
\end{figure}

Let us now move to another example for which the solution 
is obtained by the fractional Euler--Lagrange equation.
\begin{example}
\label{Example2}
Consider the following minimization problem:
\begin{equation}\label{Exmp2}
\left\{
\begin{array}{l}
J[x(\cdot)]=\int_0^1 \left(\LDz x(t)-\dot{x}^2(t)\right) dt \rightarrow \min\\
x(0)=0,~x(1)=1.
\end{array}
\right.
\end{equation}
In this case the only way to get a solution is by use of Euler--Lagrange 
equations\index{Euler--Lagrange equation! Fractional}. The Lagrangian depends 
not only on the fractional derivative, but also on the first order derivative 
of the function. The Euler--Lagrange equation for this setting becomes
\begin{equation*}
\frac{\partial L}{\partial x}+\RD \frac{\partial L}{\partial \LDa}
-\frac{d}{dt}\left(\frac{\partial L}{\partial \dot{x}}\right)=0
\end{equation*}
and, by direct computations, a necessary condition 
for $x(\cdot)$ to be a minimizer of \eqref{Exmp2} is
$$
_tD_1^{\a} 1+2\ddot{x}(t)=0,~\mbox{or}~ \ddot{x}(t)
=\frac{1}{2\Gamma(1-\a)}(1-t)^{-\a}.
$$
Subject to the given boundary conditions, the above second-order 
ordinary differential equation has the solution
\begin{equation}
\label{solEx51}
x(t)=-\frac{1}{2\Gamma(3-\a)}(1-t)^{2-\a}
+\left(1-\frac{1}{2\Gamma(3-\a)}\right)t+\frac{1}{2\Gamma(3-\a)}.
\end{equation}
\end{example}

Discretizing problem \eqref{Exmp2} with the same assumptions 
of Example~\ref{Example1} ends in a linear system of the form
\begin{equation}\label{Ex2Sys}
\left[\begin{array}{ccccccc}
2      & -1    & 0     & 0    & \cdots & 0 & 0\\
-1     & 2     & -1    & 0    & \cdots & 0 & 0\\
0      & -1    & 2     & -1   & \cdots & 0 & 0\\
\vdots & \vdots& \vdots&\vdots& \ddots & \vdots&\vdots \\
0      &  0    & 0     &  0   & \cdots &-1 & 2
\end{array}\right]\left[\begin{array}{c}x_1\\x_2\\x_3\\\vdots\\x_{n-1}\end{array}\right]=
\left[\begin{array}{c}b_1\\b_2\\b_3\\\vdots\\b_{n-1}\end{array}\right],
\end{equation}
where
$$
b_i=\frac{h}{2}\sum_{k=0}^{n-i-1}(-1)^{k}h^{0.5}\binom{0.5}{k},\qquad i=1,2,\cdots,n-2,
$$
and
$$
b_{n-1}=\frac{h}{2}\sum_{k=0}^{1}\left((-1)^{k}h^{0.5}\binom{0.5}{k}\right)+x_n.
$$
System \eqref{Ex2Sys} is linear and can be solved for any $n$ to reach the desired accuracy. 
The analytic solution together with some approximated solutions are shown in Figure~\ref{Ex2Fig}.
\begin{figure}[!tp]
\begin{center}
\includegraphics[scale=.8]{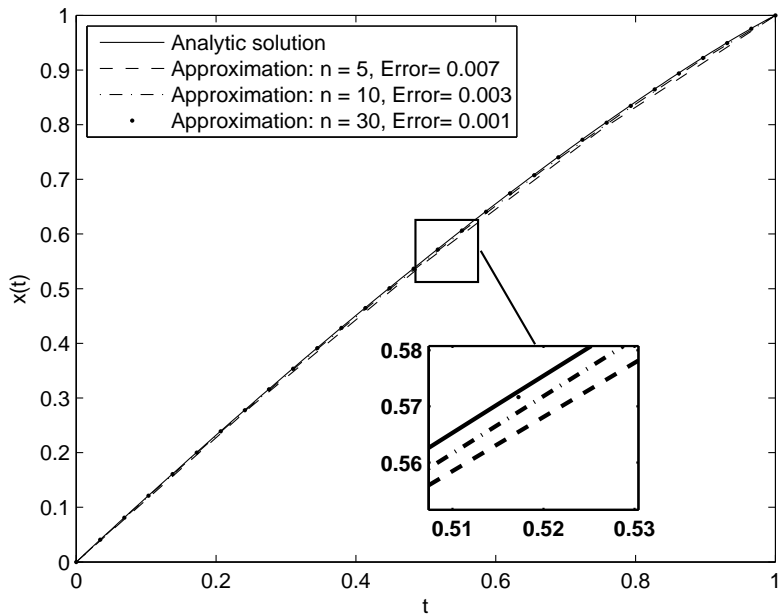}
\end{center}
\caption{Analytic and approximate solutions of Example \ref{Example2}.}\label{Ex2Fig}
\end{figure}

Both examples above end with linear systems and their solvability is simply dependant 
on the matrix of coefficients. Now we try our method on a more complicated problem, 
yet analytically solvable with an oscillating solution.

\begin{example}
\label{Example3}
Let $0<\a<1$ and we are supposed to minimize 
a functional with the following Lagrangian on $[0,1]$:
\begin{equation}
L=\left(\LDz x(t)-\frac{16\Gamma(6)}{\Gamma(5.5)}t^{4.5}
+\frac{20\Gamma(4)}{\Gamma(3.5)}t^{2.5}
-\frac{5}{\Gamma(1.5)}t^{0.5}\right)^4. \nonumber
\end{equation}
This example has an obvious solution too. Since $L$ is positive, 
$\int_0^1 Ldt$ subject to the boundary conditions $x(0)=0$ 
and $x(1)=1$ has a minimizer of the form
\begin{equation}
x(t)=16t^{5}-20t^{3}+5t.\nonumber
\end{equation}
Note that $\LD (t-a)^\nu=\frac{\Gamma(\nu+1)}{\Gamma(\nu+\-\a)}(t-a)^{\nu-\a}$.
\end{example}

The appearance of a fourth power in the Lagrangian, results in a nonlinear system 
as we apply the Euler-like direct method to this problem. For $j=1,\ldots,n-1$, we have
\begin{equation}
\label{Sys3}
\sum_{i=j}^n \left(\omega_{i-j}^{0.5}\right)\left(\frac{1}{h^{0.5}} 
\sum_{k=0}^{i}\wh x_{i-k}-\phi(t_i)\right)^3=0,
\end{equation}
where
$$
\phi(t)=\frac{16\Gamma(6)}{\Gamma(5.5)}t^{4.5}
+\frac{20\Gamma(4)}{\Gamma(3.5)}t^{2.5}-\frac{5}{\Gamma(1.5)}t^{0.5}.
$$
System \eqref{Sys3} is solved for different values of $n$ 
and the results are depicted in Figure~\ref{Ex3Fig}.
\begin{figure}[!hbtp]
\begin{center}
\includegraphics[scale=.8]{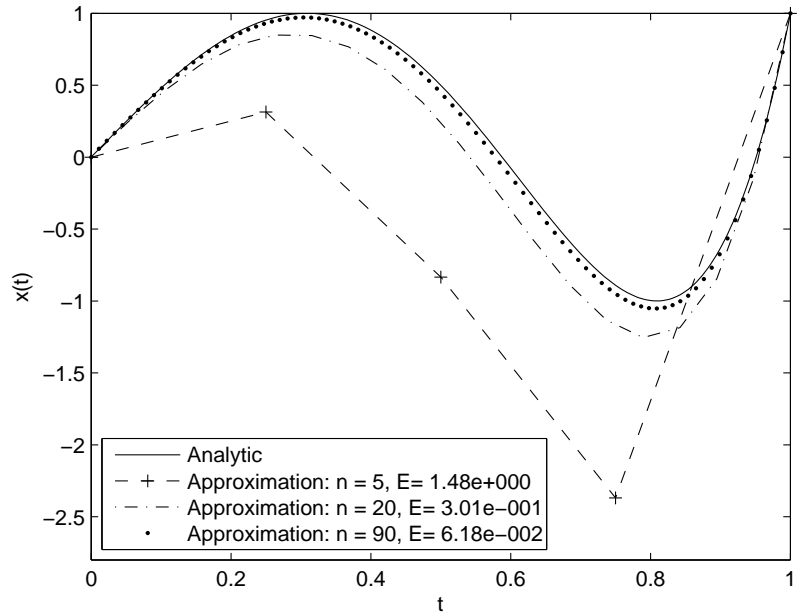}
\end{center}
\caption{Analytic and approximate solutions of Example~\ref{Example3}.}\label{Ex3Fig}
\end{figure}

These examples show that an Euler-like direct method reduces a variational problem
to a system of algebraic equations. When the resulting system is linear,
better solutions are obtained by increasing the number of mesh points
as long as the resulted matrix of coefficients is invertible. The method
is very fast in this case.

The situation is completely different when the problem ends with a nonlinear system. 
Table~\ref{table} summarizes the results regarding the running time and the error.

\begin{table}[!htp]
\center
\begin{tabular}{|c|c|c|c|}
\hline
          &  n &  T   & E \\
\hline
Example 1 &  5  &  $1.9668\times 10^{-4}$ & 0.0264 \\
          &  10 &  $2.8297\times 10^{-4}$ & 0.0158 \\
          &  30 &  $9.8318\times 10^{-4}$ & 0.0065 \\
\hline
Example 2 &  5  &  $2.4053\times 10^{-4}$ & 0.0070 \\
          &  10 &  $3.0209\times 10^{-4}$ & 0.0035 \\
          &  30 &  $7.3457\times 10^{-4}$ & 0.0012 \\
\hline
Example 3 &  5  &  0.0126      & 1.4787 \\
          &  20 &  0.2012      & 0.3006 \\
          &  90 &  26.355      & 0.0618 \\
\hline
\end{tabular}
\caption{Number of mesh points, $n$, with corresponding run time in seconds,
$T$, and error, $E$ \eqref{ErrorDis}.}\label{table}
\end{table}


\section{A discrete time method on the first variation}
\label{sec:2}

The fact that the first variation of a variational functional must vanish along an extremizer
is the base of most effective solution schemes to solve problems of the calculus of variations.
We generalize the method to variational problems involving fractional order derivatives.
First order splines are used as variations, for which fractional derivatives are known.
The Gr\"{u}nwald--Letnikov definition of fractional derivative is used,
because of its intrinsic discrete nature that leads to straightforward approximations \cite{PATDisVar}.

The problem under consideration is stated in the following way: find the extremizers of
\begin{equation}
\label{def:funct}
J[x]=\int_a^bL(t,x(t),\LDa x(t))\,dt
\end{equation}
subject to given boundary conditions $x(a)=x_a$ and $x(b)=x_b$.
Here, $L:[a,b]\times\mathbb R^2\to\mathbb R$ is such that
$\frac{\partial L}{\partial x}$ and $\frac{\partial L}{\partial \LDa x}$
exist and are continuous for all triplets $(t,x(t),\LDa x(t))$.
If $x$ is a solution to the problem and $\eta:[a,b]\to\mathbb R$
is a variation\index{Variation} function, i.e., $\eta(a)=\eta(b)=0$,
then the first variation of $J$ at $x$, with the variation $\eta$,
whatever choice of $\eta$ is taken, must vanish:
\begin{equation}
\label{1variation}
J'[x,\eta]=\int_a^b\left[\frac{\partial L}{\partial x}(t,x(t),\LDa x(t))\eta(t)
+\frac{\partial L}{\partial \LDa x}(t,x(t),\LDa x(t))\LDa \eta(t)\right]\,dt=0.
\end{equation}
Using an integration by parts formula for fractional derivatives and the Dubois--Reymond lemma,
Riewe \cite{Riewe1} proved that if $x$ is an extremizer of \eqref{def:funct}, then
$$
\frac{\partial L}{\partial x}(t,x(t),\LDa x(t))
+{_tD^\a_b}\left(\frac{\partial L}{\partial \LDa x}\right)(t,x(t),\LDa x(t))=0
$$
(see also \cite{AgrawalForm}). This fractional differential 
equation\index{Fractional! differential equation}
is called an Euler--Lagrange equation.
For the state of the art on the subject we refer the reader to the recent book \cite{book:frac}.
Here, instead of solving such Euler--Lagrange equation, we apply a discretization over time
and solve a system of algebraic equations. The procedure has proven
to be a successful tool for classical variational problems \cite{Gregory1,Gregory2}.

The discretization method is the following. Let $n\in\mathbb N$
be a fixed parameter and $h=\frac{b-a}{n}$. If we define
$t_i=a+ih$, $x_i=x(t_i)$, and $\eta_i=\eta(t_i)$ for $i=0,\ldots,n$,
the integral \eqref{1variation} can be approximated by the sum
$$
J'[x,\eta)]\approx h \sum_{i=1}^n\left[\frac{\partial L}{\partial x}(t_i,x(t_i),
{_aD_{t_i}^\a} x(t_i))\eta(t_i)+\frac{\partial L}{\partial \LDa x}(t_i,x(t_i),
{_aD_{t_i}^\a} x(t_i)){_aD_{t_i}^\a} \eta(t_i)\right].
$$
To compute the fractional derivative, we replace it by the sum as in
\eqref{GLApprx}, and to find an approximation for $x$
on mesh points one must solve the equation
\begin{align}\label{VarEquation}
\sum_{i=1}^n&\left[\frac{\partial L}{\partial x}\left(t_i,x_i,\frac{1}{h^\a}
\sum_{k=0}^{i}\w x_{i-k}\right)\eta_i\right.\nonumber\\[5pt]
&\qquad\left.+\frac{\partial L}{\partial \LDa x}\left(t_i,x_i,\frac{1}{h^\a}
\sum_{k=0}^{i}\w x_{i-k}\right)\frac{1}{h^\a} \sum_{k=0}^{i}\w \eta_{i-k}\right]=0.
\end{align}
For different choices of $\eta$, one obtains different equations. Here we use simple variations.
More precisely, we use first order splines as the set of variation\index{Variation} functions:
\begin{equation}
\label{VariationFunction}
\eta_j(t)=\left\{
\begin{array}{ll}
\displaystyle\frac{t-t_{j-1}}{h} & \mbox{ if } t_{j-1}\leq t <t_j,\\
\displaystyle\frac{t_{j+1}-t}{h} & \mbox{ if } t_j \leq t<t_{j+1},\\
0 & \mbox{ otherwise,}
\end{array}\right.
\end{equation}
for $j=1,\ldots,n-1$. We remark that conditions $\eta_j(a)=\eta_j(b)=0$
are fulfilled for all $j$, and that $\eta_j(t_i)=0$ for $i\not=j $ and
$\eta_j(t_j)=1$. The fractional derivative of $\eta_j$ at any point $t_i$
is also computed using approximation \eqref{GLApprx}:
$$
{_aD_{t_i}^\a}\eta_j(t_i)=\left\{
\begin{array}{ll}
\displaystyle\frac{1}{h^\a}(w_{i-j}^\a) & \mbox{ if } j\leq i,\\
0 & \mbox{ otherwise.}
\end{array}\right.
$$
Using $\eta_j$, $j=1,\ldots,n-1$, and equation \eqref{VarEquation}
we establish the following system of $n-1$ algebraic equations
with $n-1$ unknown variables $x_1,\ldots,x_{n-1}$:
\begin{equation}
\label{system1}
\left\{\begin{array}{l}
\displaystyle\frac{\partial L}{\partial x}\{x_1\}
+\frac{1}{h^\a}\sum_{i=1}^{n}\left[
\frac{\partial L}{\partial \LDa x}\{x_i\}(w_{i-1}^\a)\right]=0,\\
\displaystyle\frac{\partial L}{\partial x}\{x_2\}
+\frac{1}{h^\a}\sum_{i=2}^{n}\left[
\frac{\partial L}{\partial \LDa x}\{x_i\}(w_{i-2}^\a)\right]=0,\\
\quad \vdots\\
\displaystyle\frac{\partial L}{\partial x}\{x_{n-1}\}
+\frac{1}{h^\a}\sum_{i=n-1}^{n}\left[
\frac{\partial L}{\partial \LDa x}\{x_i\}(w_{i-n+1}^\a)\right]=0,
\end{array}\right.
\end{equation}
where we define
$$
\{x_i\}=\left(t_i,x_i,\frac{1}{h^\a} \sum_{k=0}^{i}\w x_{i-k}\right).
$$
The solution to \eqref{system1}, if exists, gives an approximation
to the values of the unknown function $x$ on mesh points $t_i$.

We have considered so far the so called fundamental
or basic problem of the fractional calculus of variations \cite{book:frac}.
However, other types of problems can be solved applying similar techniques.
Let us show how to solve numerically the isoperimetric problem, that is,
when in the initial problem the set of admissible functions must satisfy
some integral constraint that involves a fractional derivative.
We state the fractional isoperimetric\index{Isoperimetric problem} problem as follows.

Assume that the set of admissible functions are subject not only
to some prescribed boundary conditions, but to some integral constraint, say
$$
\int_a^bg(t,x(t),\LDa x(t))\,dt=K,
$$
for a fixed $K\in\mathbb R$. As usual, we assume that
$g:[a,b]\times\mathbb R^2\to\mathbb R$ is such that
$\frac{\partial g}{\partial x}$ and $\frac{\partial g}{\partial \LDa x}$
exist and are continuous. The common procedure to solve this problem
follows some simple steps: first we consider the auxiliary function
\begin{equation}
\label{eq:aux:f}
F=\lambda_0 L(t,x(t),\LDa x(t))+\lambda g(t,x(t),\LDa x(t)),
\end{equation}
for some constants $\lambda_0$ and $\lambda$ to be determined later. Next,
it can be proven that $F$ satisfies the fractional Euler--Lagrange equation
and that in case the extremizer does not satisfies the Euler--Lagrange associated to $g$,
then we can take $\lambda_0=1$ (cf. \cite{AlmeidaIso}).
In conclusion, the first variation of $F$ evaluated
along an extremal must vanish, and so we obtain a system similar to \eqref{system1},
replacing $L$ by $F$. Also, from the integral constraint, we obtain another
equation derived by discretization that is used to obtain $\lambda$:
$$
h \sum_{i=1}^n g\left(t_i,x_i,\frac{1}{h^\a} \sum_{k=0}^{i}\w x_{i-k}\right)=K.
$$

We show the usefulness of our approximate method
with three problems of the fractional calculus of variations.


\subsection{Basic fractional variational problems}

\begin{example}
\label{Ex1}
Consider the following variational problem: to minimize the functional
$$
J(x)=\int_0^1\left(\LDz x(t)-\frac{2}{\Gamma(2.5)}t^{1.5}\right)^2\,dt
$$
subject to the boundary conditions $x(0)=0$ and $x(1)=1$.
It is an easy exercise to verify that the solution is the function $x(t)=t^2$.
\end{example}
We apply our method to this problem, for the variation \eqref{VariationFunction}.
The functional $J$ does not depend on $x$ and is quadratic with respect
to the fractional term. Therefore, the first variation is linear.
The resulting algebraic system from \eqref{system1} is also linear and easy to solve:
$$
\left\{\begin{array}{l}
\displaystyle \sum_{i=0}^{n-1}\left(\omega_i^{0.5}\right)^2 x_1
+\sum_{i=1}^{n-1}\left(\omega_i^{0.5}\right)\left(\omega_{i-1}^{0.5}\right) x_2
+\sum_{i=2}^{n-1}\left(\omega_i^{0.5}\right)\left(\omega_{i-2}^{0.5}\right)x_3\\
\quad \displaystyle +\cdots+
\sum_{i=n-2}^{n-1}\left(\omega_i^{0.5}\right)\left(\omega_{i-(n-2)}^{0.5}\right)x_{n-1}
=\frac{2h^{2}}{\Gamma(2.5)}\sum_{i=0}^{n-1}\left(\omega_i^{0.5}\right)(i+1)^{1.5}
-\left(\omega_0^{0.5}\right)\left(\omega_{n-1}^{0.5}\right),\\
\displaystyle \sum_{i=0}^{n-2}\left(\omega_i^{0.5}\right)\left(\omega_{i+1}^{0.5}\right) x_1
+\sum_{i=0}^{n-2}\left(\omega_i^{0.5}\right)^2 x_2
+\sum_{i=1}^{n-2}\left(\omega_i^{0.5}\right)\left(\omega_{i-1}^{0.5}\right) x_3\\
\quad \displaystyle +\cdots +
\sum_{i=n-3}^{n-2}\left(\omega_i^{0.5}\right)\left(\omega_{i-(n-3)}^{0.5}\right)x_{n-1}
=\frac{2h^{2}}{\Gamma(2.5)}\sum_{i=0}^{n-2}\left(\omega_i^{0.5}\right)(i+2)^{1.5}
-\left(\omega_0^{0.5}\right)\left(\omega_{n-2}^{0.5}\right),\\
\qquad \vdots\\
\displaystyle \sum_{i=0}^{1}\left(\omega_i^{0.5}\right)\left(\omega_{i+n-2}^{0.5}\right) x_1
+\sum_{i=0}^{1}\left(\omega_i^{0.5}\right)\left(\omega_{i+n-3}^{0.5}\right) x_2
+\sum_{i=0}^{1}\left(\omega_i^{0.5}\right)\left(\omega_{i+n-4}^{0.5}\right) x_3\\
\quad \displaystyle +\cdots + \sum_{i=0}^{1}\left(\omega_i^{0.5}\right)^2x_{n-1}
=\frac{2h^{2}}{\Gamma(2.5)}\sum_{i=0}^{1}\left(\omega_i^{0.5}\right)(i+n-1)^{1.5}
-\left(\omega_0^{0.5}\right)\left(\omega_{1}^{0.5}\right).
\end{array}\right.
$$
The exact solution together with three numerical approximations,
with different discretization step sizes,
are depicted in Figure~\ref{FigEx1}.
\begin{figure}[!h]
\center
\includegraphics[scale=0.5]{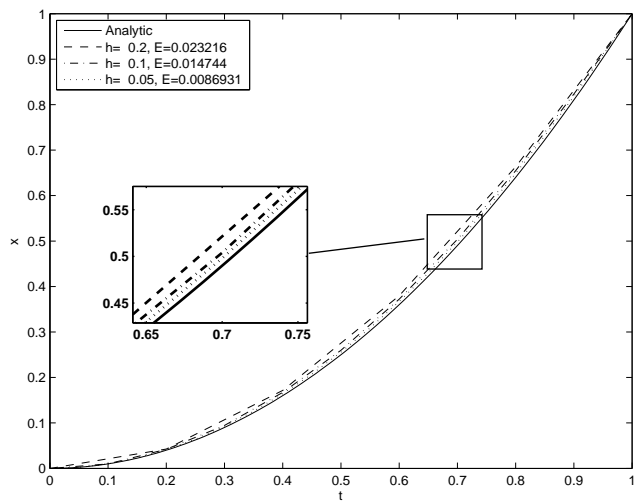}
\caption{Exact solution {\it versus} numerical approximations to Example~\ref{Ex1}.}
\label{FigEx1}
\end{figure}


\begin{example}
\label{Ex2}
Find the minimizer of the functional
$$
J(x)=\int_0^1\left(\LDz x(t)-\frac{16\Gamma(6)}{\Gamma(5.5)}t^{4.5}+
\frac{20\Gamma(4)}{\Gamma(3.5)}t^{2.5}-\frac{5}{\Gamma(1.5)}t^{0.5}\right)^4\,dt
$$
subject to $x(0)=0$ and $x(1)=1$. 
The minimum value of this functional is zero and the minimizer is
$$
x(t)=16t^{5}-20t^{3}+5t.
$$
\end{example}
Discretizing the first variation as discussed above, 
leads to a nonlinear system of algebraic equation.
Its solution, using different step sizes, 
is depicted in Figure~\ref{FigEx2}.
\begin{figure}[!h]
\center
\includegraphics[scale=0.5]{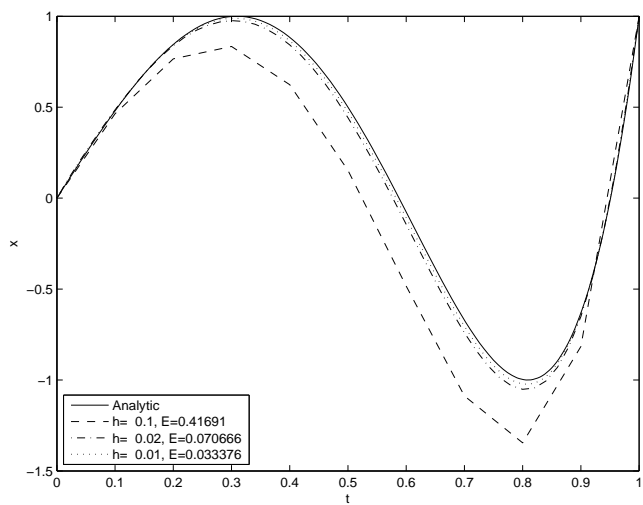}
\caption{Exact solution {\it versus} numerical approximations to Example~\ref{Ex2}.}
\label{FigEx2}
\end{figure}


\subsection{An isoperimetric fractional variational problem}

\begin{example}
\label{Ex3}
Let us search the minimizer of
$$
J(x)=\int_0^1\left(t^4 +\left(\LDz x(t)\right)^2\right)\,dt
$$
subject to the boundary conditions
$$
x(0)=0 \quad \mbox{and}\quad x(1)=\frac{16}{15\Gamma(0.5)}
$$
and the integral constraint
$$
\int_0^1t^2\,\LDz x(t)\,dt=\frac{1}{5}.
$$
In \cite{Almeida1} it is shown that the solution to this problem is the function
$$
x(t)=\frac{16t^{2.5}}{15\Gamma(0.5)}.
$$
\end{example}
Because $x$ does not satisfy the fractional Euler--Lagrange equation associated
to the integral constraint, one can take $\lambda_0=1$
and the auxiliary function \eqref{eq:aux:f} is
$F=t^4 +\left(\LDz x(t)\right)^2+\lambda \, t^2\,\LDz x(t)$.
Now we calculate the first variation of $\int_0^1 F dt$. An extra unknown,
$\lambda$, is present in the new setting, that is obtained by discretizing
the integral constraint, as explained in Section~\ref{sec:2}.
The solutions to the resulting algebraic system,
with different step sizes, are given in Figure~\ref{FigIso}.
\begin{figure}[!h]
\centering
\includegraphics[scale=0.5]{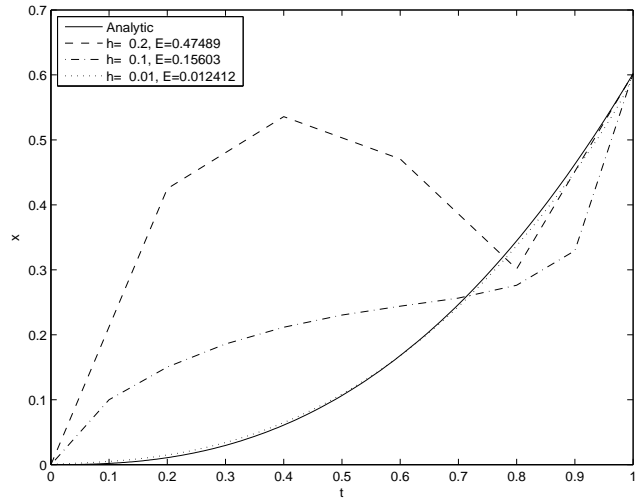}
\caption{Exact {\it versus} numerical approximations
to the isoperimetric problem of Example~\ref{Ex3}.}
\label{FigIso}
\end{figure}

\CP
\chapter{Indirect methods}
\label{Indirect}

As in the classical case, indirect methods\index{Indirect method} in fractional 
sense provide necessary conditions of optimality using the first variation. 
Fractional Euler--Lagrange equations are now a well-known and well-studied 
subject in fractional calculus. For a simple problem of the form \eqref{Functional}, 
following \cite{AgrawalForm}, a necessary condition implies that the solution 
must satisfy a fractional boundary value differential equation.

Let $x(\cdot)$ have a continuous left Riemann--Liouville 
derivative of order $\a$ and $J[x]$ be a functional of the form
\begin{equation}
\label{CopyMainProb}
J[x(\cdot)]=\int_a^b L(t, x(t), \LD x(t))dt
\end{equation}
subject to the boundary conditions $x(a)=x_a$ and $x(b)=x_b$.
Recall that
\begin{equation}
\label{ELeqnIndirect}
\left\{\begin{array}{l}
\frac{\partial L}{\partial x}+\RD\frac{\partial L}{\partial~\LDa x}=0\\
x(a)=x_a,\quad x(b)=x_b
\end{array}\right.
\end{equation}
is the fractional Euler--Lagrange equation\index{Euler--Lagrange equation! Fractional} 
and is a necessary optimality condition.

Many variants of \eqref{ELeqnIndirect} can be found in the literature. 
Different types of fractional terms have been embedded in the Lagrangian 
and appropriate versions of Euler--Lagrange equations have been derived 
using proper integration by parts formulas. 
See \cite{Agrawal,APTIndInt,Atan,Malinowska,MyID:207} for details.

For fractional optimal control problems, a so called Hamiltonian system 
is constructed using Lagrange multipliers. For example, cf. \cite{Ozlem}. 
Assume that we are required  to minimize a functional of the form
$$
J[x(\cdot),u(\cdot)]=\int_a^b L(t, x(t),u(t))dt
$$
such that $x(a)=x_a$, $x(b)=x_b$ and $\LD x(t)=f(t, x(t),u(t))$. Similar 
to the classical methods, one can introduce a Hamiltonian\index{Hamiltonian}
$$
H=L(t, x(t),u(t))+\lambda(t) f(t, x(t),u(t)),
$$
where $\lambda(t)$ is considered as a Lagrange multiplier.\index{Lagrange multipliers}
In this case we define the augmented functional as
$$
J[x(\cdot),u(\cdot)]=\int_a^b [H(t, x(t),u(t),\lambda(t))-\lambda(t)\LD x(t)]dt.
$$
Optimizing the latter functional results 
into the following necessary optimality conditions:
\begin{equation}
\label{HamIndirect}
\left\{\begin{array}{l}
\LD x(t)=\displaystyle\frac{\partial H}{\partial \lambda}\\[8pt]
\RD \lambda(t)=\displaystyle\frac{\partial H}{\partial x}
\end{array}\right., \qquad \frac{\partial H}{\partial u}=0.
\end{equation}
Together with the prescribed boundary conditions, 
this makes a two point fractional boundary value problem.

These arguments reveal that, like in the classical case, fractional variational 
problems end with fractional boundary value problems. To reach an optimal solution, 
one needs to deal with a fractional differential equation or a system 
of fractional differential equations. There are a few attempts in the literature
to present analytic solutions to fractional variational problems. Simple problems
have been treated in \cite{AlD}; some other examples are presented in \cite{AtanHam}.

Many solution methods, theoretical and numerical, furnish the classical theory 
of differential equations; nevertheless, solving a fractional differential equation 
is a rather tough task \cite{Kai}. To benefit from those methods, especially 
all solvers that are available to solve an integer-order differential equation numerically, 
we can either approximate a fractional variational problem by an equivalent integer-order 
one or approximate the necessary optimality conditions \eqref{ELeqnIndirect} and \eqref{HamIndirect}. 
The rest of this section discusses two types of approximations that are used to transform 
a fractional problem to one in which only integer-order derivatives are present, i.e., 
we approximate the original problem by substituting a fractional term by its corresponding 
expansion formulas. This is mainly done by case studies on certain examples. The examples 
are chosen so that either they have a trivial solution or it is possible to get an analytic 
solution using the fractional Euler--Lagrange equations \cite{PATFracDer}.

By substituting the approximations \eqref{expanInt} or \eqref{expanMom} for the fractional 
derivative in \eqref{CopyMainProb}, the problem is transformed to
\begin{eqnarray*}
J[x(\cdot)]&\approx&\int_a^b L\left(t, x(t), 
\sum_{k=0}^{N}\frac{(-1)^{k-1}\a x^{(k)}(t)}{k!(k-\a)\Gamma(1-\a)}(t-a)^{k-\a}\right)dt\\
&=&\int_a^b L'\left(t, x(t), \dot{x}(t), \ldots,x^{(N)}(t)\right)dt
\end{eqnarray*}
or
\begin{eqnarray*}
J[x(\cdot)]&\approx&\int_a^b L\left(t, x(t),\frac{Ax(t)}{(t-a)^{\a}}
+\frac{B\dot{x}(t)}{(t-a)^{\a-1}}-\sum_{p=2}^N \frac{C(\a,p)V_p(t)}{(t-a)^{p+\a-1}}\right)dt\\
&=&\int_a^b L'\left(t, x(t), \dot{x}(t), V_2(t), \ldots,V_N(t)\right)dt
\end{eqnarray*}
with
$$
\left\{
\begin{array}{l}
\dot{V}_p(t)=(1-p)(t-a)^{p-2}x(t)\\
V_p(a)=0, \qquad p=2,3,\ldots
\end{array}
\right.
$$
The former problem is a classical variational problem containing higher order derivatives. 
The latter is a multi-state problem, subject to an ordinary differential equation constraint. 
Together with the boundary conditions, both above problems belong 
to classes of well studied variational problems.

To accomplish a detailed study, as test problems, we consider here Example~\ref{Example2},
\begin{equation}\label{Exmp1Indirect}
\left\{
\begin{array}{l}
J[x(\cdot)]=\int_0^1 \left(\LDz x(t)-\dot{x}^2(t)\right) dt \rightarrow \min\\
x(0)=0,~x(1)=1,
\end{array}
\right.
\end{equation}
and the following example.
\begin{example}\label{Example4}
Given $\a\in (0,1)$, consider the functional
\begin{equation}\label{Exmp2Indirect}
J[x(\cdot)]=\int_0^1 (\LD x(t)-1)^2dt
\end{equation}
to be minimized subject to the boundary conditions $x(0)=0$ and $x(1)=\frac{1}{\Gamma(\a+1)}$. 
Since the integrand in \eqref{Exmp2Indirect} is non-negative, the functional attains its minimum 
when $\LD x(t)=1$, \textrm{i.e.}, for $x(t)=\frac{t^{\a}}{\Gamma(\a+1)}$.
\end{example}
We illustrate the use of the two different expansions separately.


\section{Expansion to integer orders}

Using approximation \eqref{expanInt} for the fractional derivative 
in \eqref{Exmp1Indirect}, we get the approximated problem
\begin{equation}\label{expanCOV}
\begin{aligned}
&\tilde{J}[x(\cdot)]=\int_0^1 \left[\sum_{n=0}^NC(n,\a)t^{n-\a}x^{(n)}(t)
-\dot{x}^2(t)\right]dt\longrightarrow\min\\
&x(0)=0,\quad x(1)=1,
\end{aligned}
\end{equation}
which is a classical higher-order problem of the calculus of variations that depends 
on derivatives up to order $N$. The corresponding necessary optimality condition
is a well-known result.
\begin{theorem}[\textrm{cf.}, \textrm{e.g.}, \cite{Lebedev}]
Suppose that $x(\cdot)\in C^{2N}[a,b]$ minimizes
$$
\int_a^b L(t,x(t),x^{(1)}(t),x^{(2)}(t),\ldots,x^{(N)}(t))dt
$$
with given boundary conditions
\begin{eqnarray*}
x(a)=a_0, & &x(b)=b_0,\\
x^{(1)}(a)=a_1, & &x^{(1)}(b)=b_1,\\
&\vdots&\\
x^{(N-1)}(a)=a_{N-1}, & & x^{(N-1)}(b)=b_{N-1}.
\end{eqnarray*}
Then $x(\cdot)$ satisfies the Euler--Lagrange equation\index{Euler--Lagrange equation!}
\begin{equation}
\label{ELN}
\frac{\partial L}{\partial x}-\frac{d}{dt}\left(\frac{\partial L}{\partial x^{(1)}}\right)
+\frac{d^2}{dt^2}\left(\frac{\partial L}{\partial x^{(2)}}\right)
-\cdots+(-1)^N\frac{d^N}{dt^N}\left(\frac{\partial L}{\partial x^{(N)}}\right)=0.
\end{equation}
\end{theorem}
In general \eqref{ELN} is an ODE of order $2N$, depending on the order $N$ 
of the approximation we choose, and the method leaves $2N-2$ parameters unknown. 
In our example, however, the Lagrangian in \eqref{expanCOV} is linear with respect 
to all derivatives of order higher than two. 
The resulting Euler--Lagrange equation is the second-order ODE
\begin{equation*}
\sum_{n=0}^N (-1)^nC(n,\a)\frac{d^n}{dt^n}(t^{n-\a})
-\frac{d}{dt}\left[-2\dot{x}(t)\right]=0,
\end{equation*}
that has the solution
\begin{equation*}
x(t)=M_1(\a,N)t^{2-\a}+M_2(\a,N)t,
\end{equation*}
where
\begin{eqnarray*}
M_1(\a,N)&=&-\frac{1}{2\Gamma(3-\a)}\left[\sum_{n=0}^N(-1)^n\Gamma(n+1-\a)C(n,\a)\right],\\
M_2(\a,N)&=&\left[1+\frac{1}{2\Gamma(3-\a)}\sum_{n=0}^N(-1)^n\Gamma(n+1-\a)C(n,\a)\right].
\end{eqnarray*}
Figure~\ref{expIntFig} shows the analytic solution together with several approximations. 
It reveals that by increasing $N$, approximate solutions do not converge to the analytic one. 
The reason is the fact that the solution \eqref{solEx51} to Example~\ref{Example2} 
is not an analytic function. We conclude that \eqref{expanInt} may not be a good choice 
to approximate fractional variational problems. In contrast, as we shall see, 
the approximation \eqref{expanMom} leads to good results.
\begin{figure}
\begin{center}
\includegraphics[scale=.7]{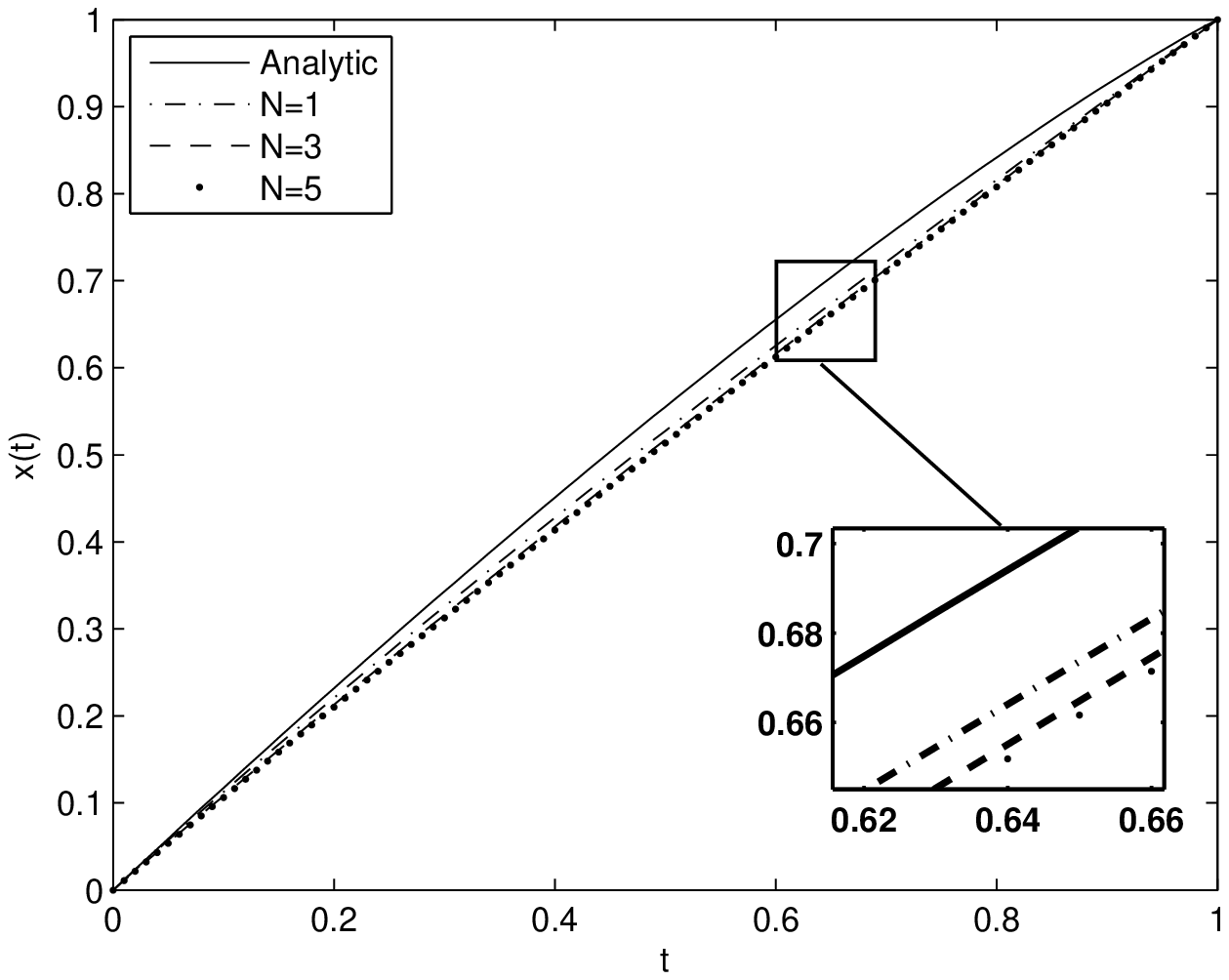}
\caption{Analytic {\it versus} approximate solutions to Example~\ref{Example2}
using approximation \eqref{expanInt} with $\a=0.5$.}
\label{expIntFig}
\end{center}
\end{figure}

To solve Example~\ref{Example2} using \eqref{expanInt} 
as an approximation for the fractional derivative, the problem becomes
\begin{equation*}
\begin{aligned}
&\tilde{J}[x(\cdot)]=\int_0^1 \left(\sum_{n=0}^N 
C(n,\a)t^{n-\a}x^{(n)}(t)-1\right)^2dt\longrightarrow\min\\
&x(0)=0,\quad x(1)=\frac{1}{\Gamma(\a+1)}.
\end{aligned}
\end{equation*}
The Euler--Lagrange equation \eqref{ELN} gives a $2N$ order ODE. 
For $N \ge 2$ this approach is inappropriate since the two given boundary 
conditions $x(0)=0$ and $x(1)=\frac{1}{\Gamma(\a+1)}$ are not enough 
to determine the $2N$ constants of integration.


\section{Expansion through the moments of a function}
\index{Moment of a function}

If we use \eqref{expanMom} to approximate the optimization problem \eqref{Exmp1Indirect}, 
with $A=A(\a,N)$, $B=B(\a,N)$ and $C_p=C(\a,p)$, we have
\begin{equation}
\label{momCOV}
\begin{split}
\tilde{J}[x(\cdot)]&=\int_0^1 \left[At^{-\a}x(t)
+Bt^{1-\a}\dot{x}(t)-\sum_{p=2}^N
C_pt^{1-p-\a}V_p(t)-\dot{x}^2(t)\right]dt\longrightarrow\min\\
\dot{V}_p(t)&=(1-p)t^{p-2}x(t),\quad p=2,3,\ldots,N,\\
V_p(0)&=0, \quad p=2,3,\ldots,N,\\
x(0)&=0,\quad x(1)=1.
\end{split}
\end{equation}
Problem \eqref{momCOV} is constrained with a set of ordinary differential equations 
and is natural to look to it as an optimal control problem \cite{Pontryagin}.
For that, we introduce the control variable $u(t) = \dot{x}(t)$. Then, using 
the Lagrange multipliers\index{Lagrange multipliers} $\lambda_1, \lambda_2,\ldots,\lambda_N$,
and the Hamiltonian system, one can reduce \eqref{momCOV} 
to the study of the two point boundary value problem
\begin{equation}
\label{tpbvp}
\left\{
\begin{array}{rl}
\dot{x}(t)&=\frac{1}{2}Bt^{1-\a}-\frac{1}{2}\lambda_1(t),\\
\dot{V}_p(t)&=(1-p)t^{p-2}x(t),\quad p=2,3,\ldots,N,\\
\dot{\lambda}_1(t)&=At^{-\a}-\sum_{p=2}^N(1-p)t^{p-2}\lambda_p(t),\\
\dot{\lambda}_p(t)&=-C_pt^{(1-p-\a)},\quad p=2,3,\ldots,N,\\
\end{array}\right.
\end{equation}
with boundary conditions
\begin{equation*}
\left\{
\begin{array}{l}
x(0)=0, \\
V_p(0)=0,\quad p=2,3,\ldots,N,
\end{array}
\right.\qquad\left\{
\begin{array}{l}
x(1)=1,\\
\lambda_p(1)=0,\quad p=2,3,\ldots,N,
\end{array}
\right.
\end{equation*}
where $x(0)=0$ and $x(1)=1$ are given. We have $V_p(0)=0$, $p=2,3,\ldots,N$, 
due to \eqref{sysVp} and $\lambda_p(1)=0$, $p=2,3,\ldots,N$, because $V_p$ 
is free at final time for $p=2,3,\ldots,N$ \cite{Pontryagin}. In general, 
the Hamiltonian\index{Hamiltonian} system is a nonlinear, 
hard to solve, two point boundary value problem that needs special numerical methods. 
In this case, however, \eqref{tpbvp} is a non-coupled system of ordinary 
differential equations and is easily solved to give
\begin{equation*}
x(t)=M(\a,N)t^{2-\a}-\sum_{p=2}^N \frac{C(\a,p)}{2p(2-p-\a)}t^p
+\left[ 1-M(\a,N)+\sum_{p=2}^N \frac{C(\a,p)}{2p(2-p-\a)}\right]t,
\end{equation*}
where
$$
M(\a,N)=\frac{1}{2(2-\a)}\left[ B(\a,N)-\frac{A(\a,N)}{1-\a}
-\sum_{p=2}^N \frac{C(\a,p)(1-p)}{(1-\a)(2-p-\a)} \right].
$$
Figure~\ref{expMomFig} shows the graph of $x(\cdot)$ for different values of $N$.
\begin{figure}
\begin{center}
\includegraphics[scale=.7]{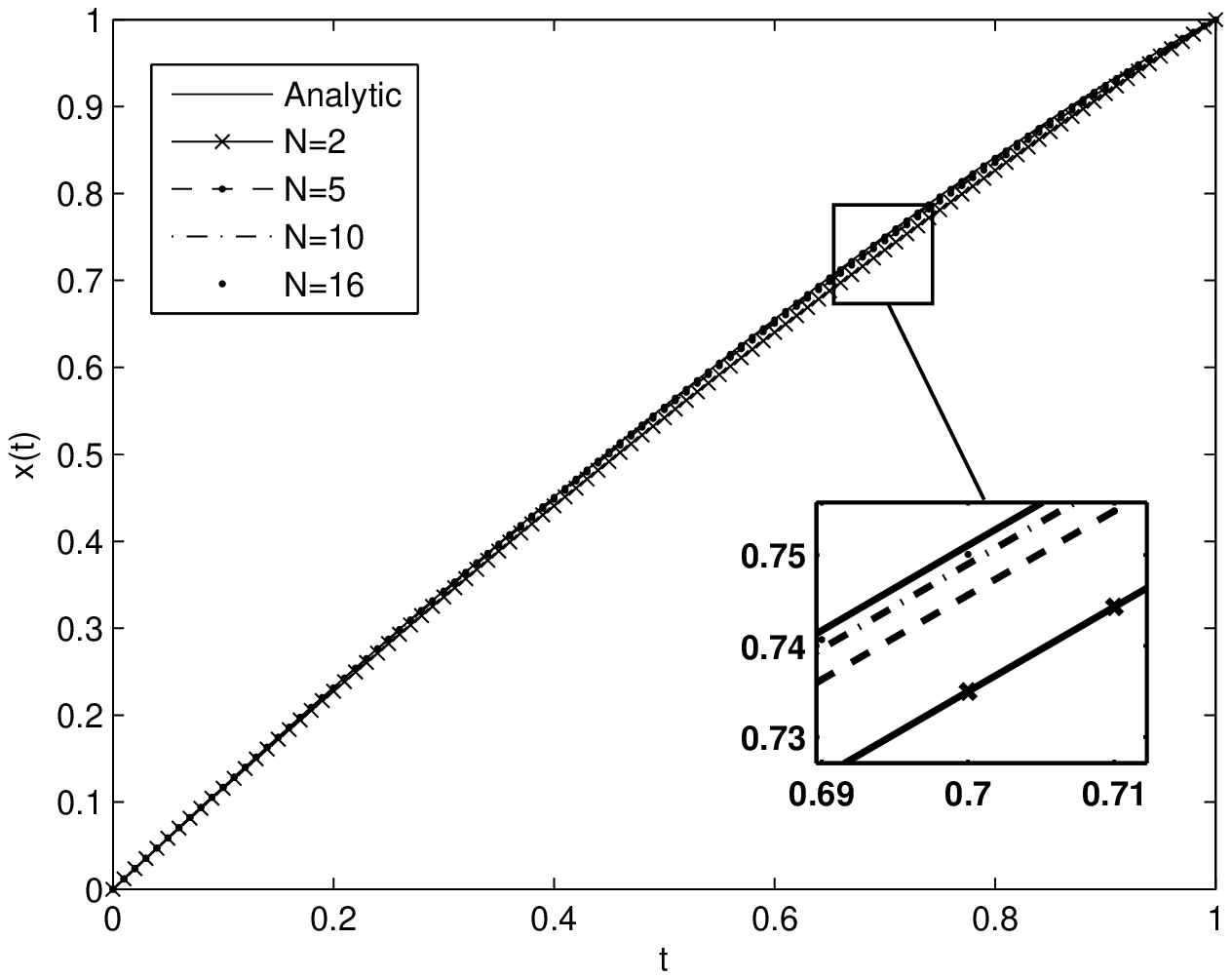}
\caption{Analytic {\it versus} approximate solutions to Example~\ref{Example2}
using approximation \eqref{expanMom} with $\a=0.5$.}
\label{expMomFig}
\end{center}
\end{figure}

Let us now approximate Example~\ref{Example4} using \eqref{expanMom}.
The resulting minimization problem has the following form:
\begin{equation}\label{ex52Mom}
\begin{aligned}
&\tilde{J}[x(\cdot)]=\int_0^1 \left[At^{-\a}x(t)
+Bt^{1-\a}\dot{x}(t)-\sum_{p=2}^N
C_pt^{1-p-\a}V_p(t)-1\right]^2dt\longrightarrow\min\\
& \dot{V}_p(t)=(1-p)t^{p-2}x(t),\quad p=2,3,\ldots,N,\\
& V_p(0)=0, \quad p=2,3,\ldots,N,\\
&x(0)=0,\quad x(1)=\frac{1}{\Gamma(\a+1)}.
\end{aligned}
\end{equation}
Following the classical optimal control approach 
of Pontryagin \cite{Pontryagin}, this time with
$$
u(t)=At^{-\a}x(t)+Bt^{1-\a}\dot{x}(t)-\sum_{p=2}^N C_pt^{1-p-\a}V_p(t),
$$
we conclude that the solution to \eqref{ex52Mom} 
satisfies the system of differential equations
\begin{equation}
\label{ex52tpbvp}
\left\{
\begin{array}{rl}
\dot{x}(t)&=-AB^{-1}t^{-1}x(t)+\sum_{p=2}^NB^{-1}C_pt^{-p}V_p(t)
+\frac{1}{2}B^{-2}t^{2\a-2}\lambda_1(t)+B^{-1}t^{\a-1},\\
\dot{V}_p(t)&=(1-p)t^{p-2}x(t),\quad p=2,3,\ldots,N,\\
\dot{\lambda}_1(t)&=AB^{-1}t^{-1}\lambda_1
-\sum_{p=2}^N(1-p)t^{p-2}\lambda_p(t),\\
\dot{\lambda}_p(t)&=-B^{-1}C_pt^{-p}\lambda_1,
\quad p=2,3,\ldots,N,\\
\end{array}\right.
\end{equation}
where $A=A(\a,N)$, $B=B(\a,N)$ and $C_p=C(\a,p)$ are defined according 
to Section~\ref{SecAppDer}, subject to the boundary conditions
\begin{equation}
\label{sysB52}
\left\{
\begin{array}{l}
x(0)=0,\\
V_p(0)=0,\quad p=2,3,\ldots,N,
\end{array}
\right.\qquad\left\{
\begin{array}{l}
x(1)=\displaystyle\frac{1}{\Gamma(\a+1)},\\
\lambda_p(1)=0,\quad p=2,3,\ldots,N.
\end{array}
\right.
\end{equation}
The solution to system \eqref{ex52tpbvp}--\eqref{sysB52}, 
with $N=2$, is shown in Figure~\ref{expMomFig52}.
\begin{figure}
\begin{center}
\includegraphics[scale=.7]{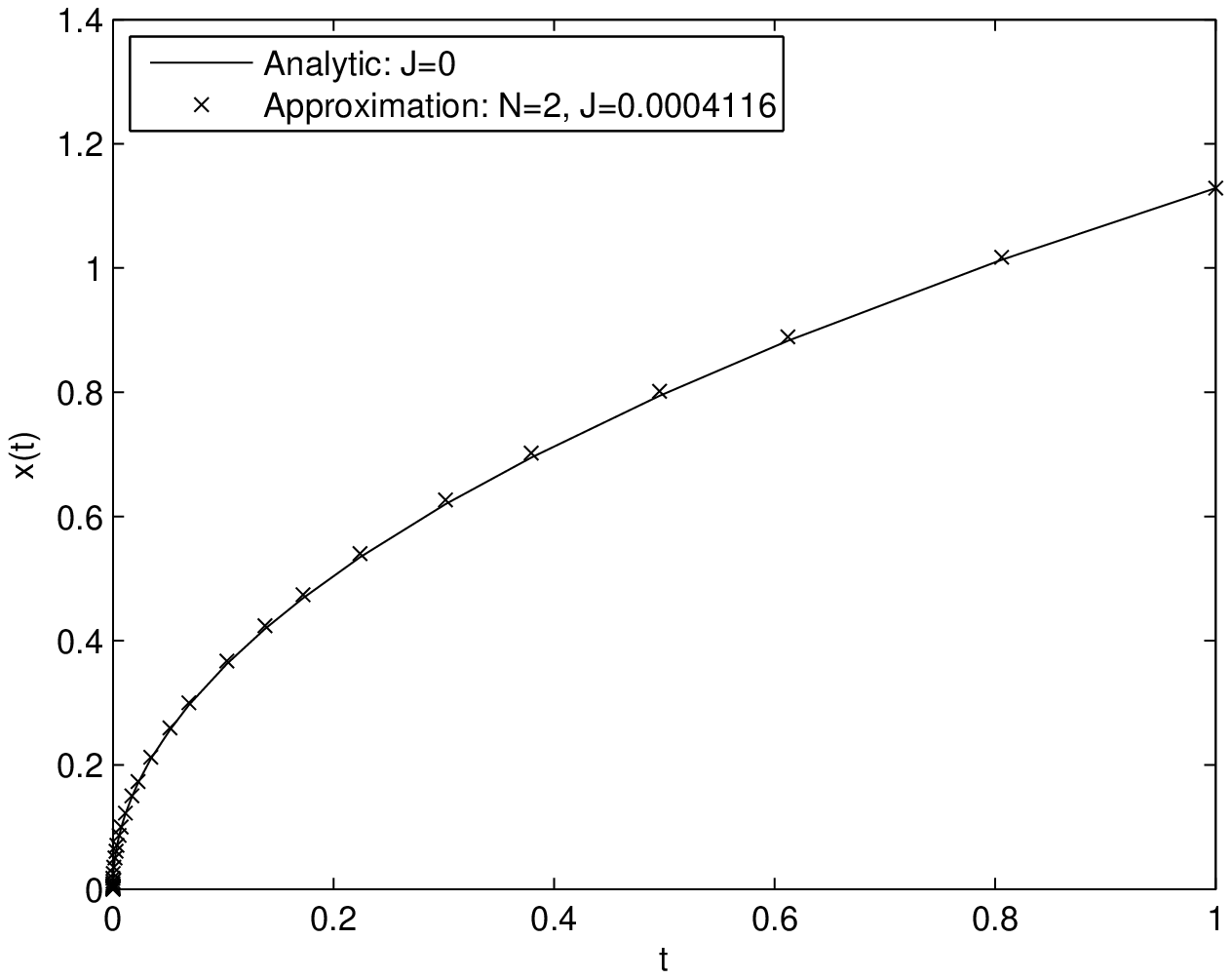}
\caption{Analytic {\it versus} approximate solution to Example~\ref{Example4}
using approximation \eqref{expanMom} with $\a=0.5$.}
\label{expMomFig52}
\end{center}
\end{figure}

\CP
\chapter{Fractional optimal control with free end-points}
\label{FreeTime}

This chapter is devoted to fractional order optimal control problems in which
the dynamic control system involves integer
and fractional order derivatives and the terminal time is free.
Necessary conditions for a state/control/terminal-time triplet
to be optimal are obtained. Situations with constraints present
at the end time are also considered.
Under appropriate assumptions, it is shown that the
obtained necessary optimality conditions become sufficient.
Numerical methods to solve the problems are presented,
and some computational simulations are discussed in detail \cite{PATFree}.


\section{Necessary optimality conditions}
\label{SecNecCond}

Let $\a\in(0,1)$, $a\in\mathbb{R}$, $L$ and $f$ be two differentiable functions
with domain $[a,+\infty)\times \mathbb{R}^2$, and
$\phi : [a,+\infty)\times \mathbb{R} \rightarrow \mathbb{R}$
be a differentiable function. The fundamental problem
is stated in the following way:
\begin{equation}
\label{Opt:func}
J[x,u,T]=\int_a^T L(t,x(t),u(t))\,dt+\phi(T,x(T))\longrightarrow\min
\end{equation}
subject to the control system
\begin{equation}
\label{Opt:dynamic}
M \dot{x}(t) + N\;\LCD x(t) = f(t,x(t),u(t)),
\end{equation}
and the initial boundary condition
\begin{equation}
\label{Opt:bound}
x(a)=x_a,
\end{equation}
with $(M,N)\not=(0,0)$ and $x_a$ a fixed real number.
Our goal is to generalize previous works
on fractional optimal control problems
by considering the end time, $T$, free and the dynamic control system \eqref{Opt:dynamic}
involving integer and fractional order derivatives. For convenience,
we consider the one-dimensional case. However, using similar techniques,
the results can be easily extended to problems with multiple states and multiple controls.
Later we consider the cases $T$ and/or $x(T)$ fixed.
Here, $T$ is a variable number with $a< T<\infty$.
Thus, we are interested not only on the optimal trajectory
$x$ and optimal control function $u$,
but also on the corresponding time $T$ for which
the functional $J$ attains its minimum value.
We assume that the state variable $x$ is differentiable and that
the control $u$ is piecewise continuous. When $N=0$ we obtain a classical
optimal control problem; the case $M=0$ with fixed $T$ has already been
studied for different types of fractional order derivatives
(see, e.g., \cite{AgrawalNum,Agrawal2,Agrawal3,Gastao0,gasta1,Tricaud,Tricaud2}).
In \cite{Jelicic} a special type of the proposed problem is also studied
for fixed $T$.

\begin{remark}
In this chapter the terminal time $T$ is a free decision variable
and, \textit{a priori}, no constraints are imposed.
For future research, one may wish to consider a class of
fractional optimal control problems in which the terminal time
is governed by a stopping condition. Such problems were recently investigated,
within the classical (integer-order) framework, in \cite{Lin1,Lin2}.
\end{remark}


\subsection{Fractional necessary conditions}

To deduce necessary optimality conditions that an optimal triplet
$(x,u,T)$ must satisfy, we use a Lagrange multiplier
to adjoin the dynamic constraint \eqref{Opt:dynamic}
to the performance functional \eqref{Opt:func}. To start,
we define the Hamiltonian\index{Hamiltonian} function $H$ by
\begin{equation}
\label{def:Hamiltonian}
H(t,x,u,\lambda)=L(t,x,u)+\lambda f(t,x,u),
\end{equation}
where $\lambda$ is a Lagrange multiplier,
so that we can rewrite the initial problem as minimizing
$$
\mathcal{J}[x,u,T,\lambda]=\int_a^T \left[H(t,x,u,\lambda)
-\lambda(t)[M \dot{x}(t)+N\;\LCD x(t)]\right]\,dt+\phi(T,x(T)).
$$
Next, we consider variations\index{Variation} of the form
$$
x+ \delta x, \quad u+\delta u, \quad T+\delta T, \quad \lambda+\delta\lambda,
$$
with $\delta x(a)=0$ by the imposed boundary condition \eqref{Opt:bound}.
Using the well-known fact that the first variation of $\mathcal J$
must vanish when evaluated along a minimizer, we get
\begin{multline*}
0 = \int_a^T\Biggl(\frac{\partial H}{\partial x}\delta x+\frac{\partial H}{\partial u}\delta u
+\frac{\partial H}{\partial \lambda}\delta \lambda-\delta \lambda\left(M \dot{x}(t)
+N\;\LCD x(t)\right)\\
-\lambda(t)\left(M \dot{\delta x}(t)+N\;\LCD\delta x(t)\right)\Biggr)dt
+\delta T\bigl[H(t,x,u,\lambda)\\
-\lambda(t)\left(M \dot{x}(t)+N\;\LCD x(t)\right)\bigr]_{t=T}
+\frac{\partial\phi}{\partial t}(T,x(T)) \delta T
+\frac{\partial \phi}{\partial x}(T,x(T))\left(\dot{x}(T)\delta T+\delta x(T)\right)
\end{multline*}
with the partial derivatives of $H$ evaluated at $(t,x(t),u(t),\lambda(t))$.
Integration by parts gives the relations
$$
\int_a^T\lambda(t) \dot{\delta x}(t)\, dt
=-\int_a^T \delta x(t)  \dot{\lambda}(t)\,dt+\delta x(T)\lambda(T)
$$
and
$$
\int_a^T\lambda(t) \LCD \delta x(t)\,dt
=\int_a^T \delta x(t) \, \RDT\lambda(t)\,dt
+\delta x(T)[\RIT\lambda(t)]_{t=T}.
$$
Thus, we deduce the following formula:
\begin{multline*}
\int_a^T\left[\delta x \left(\frac{\partial H}{\partial x}
+M \dot{\lambda} -N\;\RDT\lambda\right)+\delta u \frac{\partial H}{\partial u}
+\delta \lambda\left( \frac{\partial H}{\partial \lambda}-M \dot{x}-N\;\LCD x\right)\right]dt\\
-\delta x(T)\left[M \lambda+N\;\RIT\lambda-\frac{\partial \phi}{\partial x}(t,x)\right]_{t=T}\\
+\delta T\left[H(t,x,u,\lambda)-\lambda[M \dot{x}+N\;\LCD x]+\frac{\partial \phi}{\partial t}(t,x)
+\frac{\partial \phi}{\partial x}(t,x)\dot{x}\right]_{t=T}=0.
\end{multline*}
Now, define the new variable
$$
\delta x_T=[x+\delta x](T+\delta T)-x(T).
$$
Because $\delta\dot{x}(T)$ is arbitrary,
in particular one can consider variation functions
for which $\delta \dot{x}(T)=0$.
By Taylor's theorem,
$$
[x+\delta x](T+\delta T)-[x+\delta x](T)=\dot{x}(T)\delta T+O(\delta T^2),
$$
where $\displaystyle \lim_{\zeta\to0} \frac{O(\zeta)}{\zeta}$ is finite, and so
$\delta x(T)=\delta x_T-\dot{x}(T)\delta T+O(\delta T^2)$.
In conclusion, we arrive at the expression
\begin{multline*}
\delta T\left[H(t,x,u,\lambda)-N\lambda(t)\LCD x(t)+N\dot{x}(t)\RIT\lambda(t)
+\frac{\partial \phi}{\partial t}(t,x(t))\right]_{t=T}\\
+\int_a^T\left[\delta x \left(\frac{\partial H}{\partial x}
+M \dot{\lambda}(t)-N\;\RDT\lambda(t)\right)
+\delta \lambda\left(\frac{\partial H}{\partial \lambda}-M \dot{x}(t)-N\;\LCD x(t)\right)\right.\\
\left.+\delta u \frac{\partial H}{\partial u}\right]dt
-\delta x_T\left[M \lambda(t)+N\;\RIT\lambda(t)-\frac{\partial \phi}{\partial x}(t,x(t))\right]_{t=T}
+O(\delta T^2)=0.
\end{multline*}
Since the variation functions were chosen arbitrarily, the following theorem is proven.

\begin{theorem}
\label{Opt:MainTheo}
If $(x,u,T)$ is a minimizer of \eqref{Opt:func} under the dynamic constraint
\eqref{Opt:dynamic} and the boundary condition \eqref{Opt:bound},
then there exists a function $\lambda$ for which the triplet $(x,u,\lambda)$  satisfies:
\begin{itemize}
\item the \emph{Hamiltonian system}
\begin{equation}
\label{OPT:HamilSyst}
\begin{cases}
M \dot{\lambda}(t) - N\;\RDT\lambda(t) 
= - \displaystyle\frac{\partial H}{\partial x}(t,x(t),u(t),\lambda(t))\\[8pt]
M \dot{x}(t)+N\;\LCD x(t) 
= \displaystyle\frac{\partial H}{\partial \lambda}(t,x(t),u(t),\lambda(t))
\end{cases}
\end{equation}
for all $t\in[a,T]$;

\item the \emph{stationary condition}
\begin{equation}
\label{OPT:stacionary}
\frac{\partial H}{\partial u}(t,x(t),u(t),\lambda(t))=0
\end{equation}
for all $t\in[a,T]$;

\item and the \emph{transversality conditions}\index{Transversality conditions}
\begin{equation}
\label{OPT:transversality}
\begin{gathered}
\left[H(t,x(t),u(t),\lambda(t))-N\lambda(t)\LCD x(t)+N\dot{x}(t)\RIT\lambda(t)
+\frac{\partial \phi}{\partial t}(t,x(t))\right]_{t=T}=0,\\
\left[M \lambda(t) +N\;\RIT\lambda(t)-\frac{\partial \phi}{\partial x}(t,x(t))\right]_{t=T}=0;
\end{gathered}
\end{equation}
\end{itemize}
where the Hamiltonian $H$ is defined by \eqref{def:Hamiltonian}.
\end{theorem}

\begin{remark}
In standard optimal control, a free terminal time problem
can be converted into a fixed final time problem
by using the well-known transformation $s = t/T$ (see Example~\ref{exm42}).
This transformation does not work in the fractional setting.
Indeed, in standard optimal control, translating the problem from time $t$
to a new time variable $s$ is straightforward: the chain rule gives
$\frac{d x}{d s} = \frac{d x}{d t} \frac{d t}{d s}$.
For Caputo or Riemann--Liouville fractional derivatives,
the chain rule has no practical use and such conversion is not possible.
\end{remark}

Some interesting special cases are obtained when restrictions
are imposed on the end time $T$ or on $x(T)$.

\begin{corollary}
\label{OPT:MainCor}
Let $(x,u)$ be a minimizer of \eqref{Opt:func} under the dynamic constraint
\eqref{Opt:dynamic} and the boundary condition \eqref{Opt:bound}.
\begin{enumerate}
\item If $T$ is fixed and $x(T)$ is free,
then Theorem~\ref{Opt:MainTheo} holds with the transversality conditions
\eqref{OPT:transversality} replaced by
$$
\left[M \lambda(t)+N\;\RIT\lambda(t)-\frac{\partial \phi}{\partial x}(t,x(t))\right]_{t=T}=0.
$$

\item If $x(T)$ is fixed and $T$ is free,
then Theorem~\ref{Opt:MainTheo} holds with the transversality
conditions\index{Transversality conditions} \eqref{OPT:transversality} replaced by
$$
\left[H(t,x(t),u(t),\lambda(t))-N\lambda(t)\LCD x(t)+N\dot{x}(t)\RIT\lambda(t)
+\frac{\partial \phi}{\partial t}(t,x(t))\right]_{t=T}=0.
$$

\item If $T$ and $x(T)$ are both fixed,
then Theorem~\ref{Opt:MainTheo} holds
with no transversality conditions.

\item If the terminal point $x(T)$ belongs to a fixed curve, i.e.,
$x(T)=\gamma(T)$ for some differentiable curve $\gamma$, then
Theorem~\ref{Opt:MainTheo} holds with the transversality conditions
\eqref{OPT:transversality} replaced by
\begin{multline*}
\Biggl[H(t,x(t),u(t),\lambda(t))-N\lambda(t)\LCD x(t)
+N\dot{x}(t)\RIT\lambda(t)+\frac{\partial \phi}{\partial t}(t,x(t))\\
-\dot{\gamma}(t)\left(M \lambda(t)+N\;\RIT\lambda(t)
-\frac{\partial \phi}{\partial x}(t,x(t))\right)\Biggr]_{t=T}=0.
\end{multline*}

\item If $T$ is fixed and $x(T)\geq K$ for some fixed $K\in\mathbb{R}$,
then Theorem~\ref{Opt:MainTheo} holds with the transversality conditions
\eqref{OPT:transversality} replaced by
\begin{gather*}
\left[M \lambda(t)+N\;\RIT\lambda(t)
-\frac{\partial \phi}{\partial x}(t,x(t))\right]_{t=T}\leq 0,\\
(x(T)-K)\left[M \lambda(t)+N\;\RIT\lambda(t)
-\frac{\partial \phi}{\partial x}(t,x(t))\right]_{t=T}=0.
\end{gather*}

\item If $x(T)$ is fixed and $T\leq K$ for some fixed $K\in\mathbb{R}$,
then Theorem~\ref{Opt:MainTheo} holds with the transversality conditions
\eqref{OPT:transversality} replaced by
$$
\left[H(t,x(t),u(t),\lambda(t))-N\lambda(t)\LCD x(t)+N\dot{x}(t)\RIT\lambda(t)
+\frac{\partial \phi}{\partial t}(t,x(t))\right]_{t=T}\geq 0,
$$
\begin{multline*}
\left[H(t,x(t),u(t),\lambda(t))-N\lambda(t)\LCD x(t)+N\dot{x}(t)\RIT\lambda(t)
+\frac{\partial \phi}{\partial t}(t,x(t))\right]_{t=T}\\
\times (T-K) = 0.
\end{multline*}
\end{enumerate}
\end{corollary}

\begin{proof}
The first three conditions are obvious. The fourth follows from
$$
\delta x_T=\gamma(T+\delta T)-\gamma(T)=\dot{\gamma}(T)\delta T+O(\delta T^2).
$$
To prove \textit{5}, observe that we have two possible cases. If $x(T)>K$,
then $\delta x_T$ may take negative and positive values, and so we get
$$
\left[M \lambda(t)+N\;\RIT\lambda(t)-\frac{\partial \phi}{\partial x}(t,x(t))\right]_{t=T}=0.
$$
On the other hand, if $x(T)=K$, then $\delta x_T\geq 0$ and so by the KKT theorem
$$
\left[M \lambda(t)+N\;\RIT\lambda(t)-\frac{\partial \phi}{\partial x}(t,x(t))\right]_{t=T}\leq 0.
$$
The proof of the last condition is similar.
\end{proof}

Case~1 of Corollary~\ref{OPT:MainCor} was proven in \cite{Gastao0}
for $(M,N)=(0,1)$ and $\phi \equiv 0$. Moreover, if $\a=1$, then we obtain the classical
necessary optimality conditions for the standard optimal control problem (see, e.g., \cite{Chiang}):
\begin{itemize}
\item the Hamiltonian system
$$
\begin{cases}
\dot{x}(t) = \displaystyle\frac{\partial H}{\partial \lambda}(t,x(t),u(t),\lambda(t)),\\[8pt]
\dot{\lambda}(t) = - \displaystyle\frac{\partial H}{\partial x}(t,x(t),u(t),\lambda(t)),
\end{cases}
$$
\item the stationary condition $$\displaystyle \frac{\partial H}{\partial u}(t,x(t),u(t),\lambda(t))=0,$$
\item the transversality condition $\lambda(T)=0$.
\end{itemize}


\subsection{Approximated integer-order necessary optimality conditions}

Using approximation \eqref{expanMom}, and the relation between Caputo 
and Riemann--Liouville derivatives, up to order $K$, we can transform 
the original problem \eqref{Opt:func}--\eqref{Opt:bound}
into the following classical problem:
\begin{equation*}
\tilde{J}[x,u,T]=\int_a^TL(t,x(t),u(t))\,dt+\phi(T,x(T))\longrightarrow \min
\end{equation*}
subject to
\begin{equation*}
\begin{cases}
\dot{x}(t)=\displaystyle \frac{f(t,x(t),u(t))
-NA(t-a)^{-\a}x(t)+\sum_{p=2}^KNC_p(t-a)^{1-p-\a}V_p(t)
-\frac{x(a)(t-a)^{-\a}}{\Gamma(1-\a)}}{M+NB(t-a)^{1-\a}},\\[0.25cm]
\dot{V}_p(t)=(1-p)(t-a)^{p-2}x(t), \quad p=2,\ldots,K,
\end{cases}
\end{equation*}
and
\begin{equation}
\label{App:bound}
\begin{cases}
x(a)=x_a,\\
V_p(a)=0, \quad p=2,\ldots,K,
\end{cases}
\end{equation}
where $A=A(\a,K)$, $B=B(\a,K)$ and $C_p=C(\a,p)$
are the coefficients in the approximation \eqref{expanMom}.
Now that we are dealing with an integer-order problem,
so we can follow a classical procedure (see, e.g., \cite{Kirk}),
by defining the Hamiltonian $H$ by
\begin{equation*}
\begin{split}
H=L(t,x,u)&+\frac{\lambda_1 \left(f(t,x,u) -NA(t-a)^{-\a}x
+\sum_{p=2}^KNC_p(t-a)^{1-p-\a}V_p
-\frac{x(a)(t-a)^{-\a}}{\Gamma(1-\a)}\right)}{M+NB(t-a)^{1-\a}}\\
&+\sum_{p=2}^K \lambda_p (1-p)(t-a)^{p-2}x.
\end{split}
\end{equation*}
Let $\bm{\lambda}=(\lambda_1,\lambda_2,\ldots,\lambda_K)$ 
and $\mathbf{x}=(x,V_2,\ldots,V_K)$.
The necessary optimality conditions
\begin{equation*}
\frac{\partial H}{\partial u}=0, \quad
\begin{cases}
\dot{\mathbf{x}}=\displaystyle \frac{\partial H}{\partial \bm{\lambda}},\\[0.25cm]
\dot{\bm{\lambda}}=-\displaystyle \frac{\partial H}{\partial \mathbf{x}},
\end{cases}
\end{equation*}
result in a two point boundary value problem.
Assume that $(T^*,\mathbf{x}^*,\bm{u}^*)$ is the optimal triplet.
In addition to the boundary conditions \eqref{App:bound},
the transversality conditions imply
$$
\left[\frac{\partial \phi}{\partial \mathbf{x}}(T^*,
\mathbf{x}^*(T))\right]^{tr}\delta \mathbf{x}_T
+\left[H(T^*, \mathbf{x}^*(T),\bm{u}^*(T),\bm{\lambda}^*(T))
+\frac{\partial \phi}{\partial t}(T^*, \mathbf{x}^*(T))\right]\delta T=0,
$$
where $tr$ denotes the transpose.
Because $V_p$, $p=2,\ldots,K$, are auxiliary variables
whose values $V_p(T)$, at the final time $T$, are free, we have
$$
\lambda_p(T)=\frac{\partial\phi}{\partial V_p}\Big |_{t=T} = 0,
\quad p=2,\ldots,K.
$$
The value of $\lambda_1(T)$ is determined from the value of $x(T)$.
If $x(T)$ is free, then $\lambda_1(T)=\frac{\partial\phi}{\partial x}|_{t=T}$.
Whenever the final time is free, a transversality condition of the form
$$
\left[H\left(t,\mathbf{x}(t),\bm{u}(t),\bm{\lambda}(t)\right)
-\frac{\partial\phi}{\partial t}\left(t,\mathbf{x}(t)\right)\right]_{t=T}=0
$$
completes the required set of boundary conditions.


\section{A generalization}

The aim is now to consider a generalization of the
optimal control problem \eqref{Opt:func}--\eqref{Opt:bound}
studied in Section~\ref{SecNecCond}. Observe that the initial point $t=a$ is in fact
the initial point for two different operators: for the integral in \eqref{Opt:func}
and, secondly, for the left Caputo fractional derivative given by the dynamic constraint
\eqref{Opt:dynamic}. We now consider the case where the lower bound of the integral
of $J$ is greater than the lower bound of the fractional derivative.
The problem is stated as follows:
\begin{equation}
\label{eq:gJ}
J[x,u,T]=\int_A^TL(t,x(t),u(t))\,dt+\phi(T,x(T))\longrightarrow\min
\end{equation}
under the constraints
\begin{equation}
\label{eq:gCS}
M \dot{x}(t)+N\;\LCD x(t)=f(t,x(t),u(t)) \quad \mbox{and} \quad x(A)=x_A,
\end{equation}
where $(M,N)\not=(0,0)$, $x_A$ is a fixed real, and $a<A$.

\begin{remark}
We have chosen to consider the initial condition
on the initial time $A$ of the cost integral,
but the case of initial condition $x(a)$ instead of $x(A)$
can be studied using similar arguments.
Our choice seems the most natural: the interval of interest is $[A,T]$
but the fractional derivative is a non-local operator and has ``memory''
that goes to the past of the interval $[A,T]$ under consideration.
\end{remark}

\begin{remark}
In the theory of fractional differential equations,
the initial condition is given at $t=a$. To the best of our
knowledge there is no general theory about uniqueness of solutions
for problems like \eqref{eq:gCS}, where the fractional derivative
involves $x(t)$ for $a<t<A$ and the initial condition is given at $t=A$.
Uniqueness of solution is, however, possible. Consider, for example,
${_0^C D ^\alpha _t} x(t)=t^2$. Applying the fractional integral to both sides of equality
we get $x(t)=x(0)+ 2 t^{2+\alpha}/\Gamma(3+\alpha)$ so, knowing a value for $x(t)$,
not necessarily at $t=0$, one can determine $x(0)$ and by doing so $x(t)$.
A different approach than the one considered here
is to provide an initialization function for $t\in[a,A]$.
This initial memory approach was studied for fractional
continuous-time linear control systems in \cite{MyID:163}
and \cite{Dorota}, respectively for Caputo and Riemann--Liouville derivatives.
\end{remark}
The method to obtain
the required necessary optimality conditions follows the same procedure
as the one discussed before. The first variation gives
\begin{equation*}
\begin{split}
0 = \int_A^T&\Biggl[\frac{\partial H}{\partial x}\delta x+\frac{\partial H}{\partial u}\delta u
+\frac{\partial H}{\partial \lambda}\delta \lambda-\delta \lambda\left(M \dot{x}(t)+N\;\LCD x(t)\right)\\
&-\lambda(t)\left(M \dot{\delta x}(t)+N\;\LCD\delta x(t)\right)\Biggr]dt
+\frac{\partial \phi}{\partial x}(T,x(T))\left(\dot{x}(T)\delta T+\delta x(T)\right)\\
&+\frac{\partial\phi}{\partial t}(T,x(T)) \delta T
+\delta T\left[H(t,x,u,\lambda)-\lambda(t)\left(M \dot{x}(t)+N\;\LCD x(t)\right)\right]_{t=T},
\end{split}
\end{equation*}
where the Hamiltonian $H$ is as in \eqref{def:Hamiltonian}.
Now, if we integrate by parts, we get
$$
\int_A^T\lambda(t) \dot{\delta x}(t)\, dt
=-\int_A^T \delta x(t)  \dot{\lambda}(t)\,dt+\delta x(T)\lambda(T),
$$
and
\begin{equation*}
\begin{split}
\int_A^T & \lambda(t) \LCD \delta x(t)\,dt
=\int_a^T \lambda(t) \LCD \delta x(t)\,dt-\int_a^A\lambda(t) \LCD \delta x(t)\,dt\\
&=\int_a^T \delta x(t) \, \RDT\lambda(t)\,dt+[\delta x(t) \RIT\lambda(t)]_{t=a}^{t=T}
-\int_a^A \delta x(t) \, {_tD^\a_A}\lambda(t)\,dt\\
&\quad -[\delta x(t){_tI^{1-\a}_A}\lambda(t)]_{t=a}^{t=A}\\
&=\int_a^A \delta x(t) [\RDT\lambda(t)-{_tD^\a_A}\lambda(t)]\,dt
+\int_A^T \delta x(t) \, \RDT\lambda(t)\,dt\\
&\quad +\delta x(T)[\RIT\lambda(t)]_{t=T}
-\delta x(a)[{_aI^{1-\a}_T}\lambda(a)-{_aI^{1-\a}_A}\lambda(a)].
\end{split}
\end{equation*}
Substituting these relations into the first variation of $J$, we conclude that
\begin{equation*}
\begin{split}
\int_A^T&\left[\left(\frac{\partial H}{\partial x}
+M \dot{\lambda} -N\;\RDT\lambda\right)\delta x
+\frac{\partial H}{\partial u} \delta u
+\left( \frac{\partial H}{\partial \lambda}-M \dot{x}-N\;\LCD x\right)\delta \lambda \right]dt\\
&-N\int_a^A \delta x [\RDT\lambda-{_tD^\a_A}\lambda]\,dt
-\delta x[M \lambda+N\;\RIT\lambda-\frac{\partial \phi}{\partial x}(t,x)]_{t=T}\\
&+\delta T[H(t,x,u,\lambda)-\lambda[M \dot{x}+N\;\LCD x]
+\frac{\partial \phi}{\partial t}(t,x)
+\frac{\partial \phi}{\partial x}(t,x)\dot{x}]_{t=T}\\
&+N\delta x(a)[{_aI^{1-\a}_T}\lambda(a)-{_aI^{1-\a}_A}\lambda(a)]=0.
\end{split}
\end{equation*}
Repeating the calculations as before, we prove the following optimality conditions.

\begin{theorem}
If the triplet $(x,u,T)$ is an optimal solution to problem \eqref{eq:gJ}--\eqref{eq:gCS},
then there exists a function $\lambda$ for which the following conditions hold:
\begin{itemize}
\item the \emph{Hamiltonian system}
\begin{equation*}
\begin{cases}
M \dot{\lambda}(t)-N\;\RDT\lambda(t) 
= - \displaystyle \frac{\partial H}{\partial x}(t,x(t),u(t),\lambda(t))\\[0.25cm]
M \dot{x}(t)+N\;\LCD x(t) 
= \displaystyle \frac{\partial H}{\partial \lambda}(t,x(t),u(t),\lambda(t))
\end{cases}
\end{equation*}
for all $t\in[A,T]$, and
$\RDT\lambda(t)-{_tD^\a_A}\lambda(t)=0$ for all $t\in[a,A]$;

\item the \emph{stationary condition}
$$
\frac{\partial H}{\partial u}(t,x(t),u(t),\lambda(t))=0
$$
for all $t\in[A,T]$;

\item the \emph{transversality conditions}
\begin{gather*}
\left[H(t,x(t),u(t),\lambda(t))-N\lambda(t)\LCD x(t)+N\dot{x}(t)\RIT\lambda(t)
+\frac{\partial \phi}{\partial t}(t,x(t))\right]_{t=T}=0,\\
\left[M \lambda(t)+N\;\RIT\lambda(t)-\frac{\partial \phi}{\partial x}(t,x(t))\right]_{t=T}=0,\\
\left[{_tI^{1-\a}_T}\lambda(t)-{_tI^{1-\a}_A}\lambda(t)\right]_{t=a}=0;
\end{gather*}
\end{itemize}
with the Hamiltonian $H$ given by \eqref{def:Hamiltonian}.
\end{theorem}

\begin{remark}
If the admissible functions take fixed values at both $t=a$ and $t=A$,
then we only obtain the two transversality conditions evaluated at $t=T$.
\end{remark}


\section{Sufficient optimality conditions}

In this section we show that, under some extra hypotheses,
the obtained necessary optimality conditions
are also sufficient. To begin, let us recall the notions
of convexity and concavity for $C^1$ functions of several variables.

\begin{definition}\label{ConvDef}
Given $k\in\{1,\ldots,n\}$ and a function
$\Psi:D\subseteq \mathbb{R}^n\to \mathbb{R}$
such that $\partial \Psi / \partial t_i$ exist and
are continuous for all $i\in\{k,\ldots,n\}$,
we say that $\Psi$ is convex (concave) in $(t_k,\ldots,t_n)$ if
\begin{multline*}
\Psi(t_1+\theta_1,\ldots,t_{k-1}+\theta_{k-1},t_k+\theta_k,
\ldots,t_n+\theta_n)-\Psi(t_1,\ldots,t_{k-1},t_k,\ldots,t_n)\\
\geq \, (\leq) \, \frac{\partial \Psi}{\partial t_k}(t_1,\ldots,t_{k-1},t_k,\ldots,t_n)\theta_k
+ \cdots + \frac{\partial \Psi}{\partial t_n}(t_1,\ldots,t_{k-1},t_k,\ldots,t_n)\theta_n
\end{multline*}
for all $(t_1,\ldots,t_n),(t_1+\theta_1,\ldots,t_n+\theta_n)\in D$.
\end{definition}

\begin{theorem}
\label{thm:suff:cond}
Let $(\overline{x}, \overline{u}, \overline{\lambda})$ be a triplet satisfying conditions
\eqref{OPT:HamilSyst}--\eqref{OPT:transversality} of Theorem~\ref{Opt:MainTheo}.
Moreover, assume that
\begin{enumerate}
\item $L$ and $f$ are convex on $x$ and $u$, and $\phi$ is convex in $x$;
\item $T$ is fixed;
\item $\overline{\lambda}(t)\geq 0$ for all $t \in [a,T]$ or $f$ is linear in $x$ and $u$.
\end{enumerate}
Then $(\overline x,\overline u)$ is an optimal solution to problem
\eqref{Opt:func}--\eqref{Opt:bound}.
\end{theorem}

\begin{proof}
From \eqref{OPT:HamilSyst} we deduce that
$$
\frac{\partial L}{\partial x}(t,\overline x(t),\overline u(t))
=-M\dot{\overline \lambda}(t)+N\;\RDT\overline \lambda(t)
-\overline\lambda(t)\frac{\partial f}{\partial x}(t,\overline x(t),\overline u(t)).
$$
Using \eqref{OPT:stacionary},
$$
\frac{\partial L}{\partial u}(t,\overline x(t),\overline u(t))
=-\overline\lambda(t)\frac{\partial f}{\partial u}(t,\overline x(t),\overline u(t)),
$$
and \eqref{OPT:transversality} gives
$[M\overline \lambda(t)+N\;{_tI_T^{1-\alpha}}\overline \lambda(t)
-\frac{\partial \phi}{\partial x}(t,\overline x(t))]_{t=T}=0$.
Let $(x,u)$ be admissible, i.e., let \eqref{Opt:dynamic}
and\eqref{Opt:bound} be satisfied for $(x,u)$. In this case,
\begin{equation*}
\begin{split}
\triangle J&=J[x,u]-J[\overline x,\overline u]\\
&=\int_a^T \left[L(t,x(t),u(t))-L(t,\overline x(t),\overline u(t))\right] dt
+\phi(T,x(T))-\phi(T,\overline x(T))\\
&\geq \int_a^T \left[\frac{\partial L}{\partial x}(t,\overline x(t),
\overline u(t)) (x(t)-\overline x(t))
+\frac{\partial L}{\partial u}(t,\overline x(t),\overline u(t))
(u(t)-\overline u(t))\right] dt\\
&\quad +\frac{\partial \phi}{\partial x}(T,\overline x(T))(x(T)-\overline x(T))\\
&=\int_a^T \biggl[ -M\dot{\overline \lambda}(t)(x(t)
-\overline x(t))+N\;\RDT\overline \lambda(t)(x(t)-\overline x(t))\\
&\hspace{1.5cm} -\overline\lambda(t)
\frac{\partial f}{\partial x}(t,\overline x(t),\overline u(t))(x(t)-\overline x(t))
-\overline\lambda(t)\frac{\partial f}{\partial u}(t,\overline x(t),
\overline u(t))(u(t)-\overline u(t))\biggr] dt\\
&\quad +\frac{\partial \phi}{\partial x}(T,\overline x(T))(x(T)-\overline x(T)).
\end{split}
\end{equation*}
Integrating by parts, and noting that $x(a)=\overline x(a)$, we obtain
\begin{equation*}
\begin{split}
\triangle J&\geq\int_a^T \overline \lambda(t)\Biggl[ M\left(\dot{x}(t)-\dot{\overline x}(t)\right)
+N\left(\LCD x(t)-\LCD \overline x(t)\right)\\
&\hspace{2.5cm} -\frac{\partial f}{\partial x}\left(t,\overline x(t),
\overline u(t)\right)\left(x(t)-\overline x(t)\right)
-\frac{\partial f}{\partial u}\left(t,\overline x(t),\overline u(t)\right)
\left(u(t)-\overline u(t)\right)\Biggr] dt\\
&\quad+\left[\frac{\partial \phi}{\partial x}(t,\overline x(t))
-M\overline \lambda(t)-N\;{_tI_T^{1-\alpha}}\overline \lambda(t)\right]_{t=T}
\left(x(T)-\overline x(T)\right),
\end{split}
\end{equation*}
and finally
\begin{equation*}
\begin{split}
\triangle J&\geq \int_a^T \Biggl[\overline \lambda(t)\left[ 
f(t,x(t),u(t))-f\left(t,\overline x(t),\overline u(t)\right) \right]
-\overline\lambda(t)\frac{\partial f}{\partial x}(t,\overline x(t),
\overline u(t))\left(x(t)-\overline x(t)\right)\\
&\qquad -\overline\lambda(t)\frac{\partial f}{\partial u}(t,
\overline x(t),\overline u(t))\left(u(t)-\overline u(t)\right)\Biggr] dt\\
&\geq \int_a^T \overline \lambda(t) \Biggl[
\frac{\partial f}{\partial x}\left(t,\overline x(t),
\overline u(t)\right)\left(x(t)-\overline x(t)\right)
+\frac{\partial f}{\partial u}\left(t,\overline x(t),
\overline u(t)\right)\left(u(t)-\overline u(t)\right)\\
&\qquad -\frac{\partial f}{\partial x}\left(t,\overline x(t),
\overline u(t)\right)\left(x(t)-\overline x(t)\right)
-\frac{\partial f}{\partial u}\left(t,\overline x(t),
\overline u(t)\right)\left(u(t)-\overline u(t)\right)\Biggr] dt\\
&= 0.
\end{split}
\end{equation*}
\end{proof}

\begin{remark}
If the functions in Theorem~\ref{thm:suff:cond} are strictly convex
instead of convex, then the minimizer is unique.
\end{remark}


\section{Numerical treatment and examples}

Here we apply the necessary conditions of Section~\ref{SecNecCond}
to solve some test problems. Solving an optimal control problem,
analytically, is an optimistic goal and is impossible
except for simple cases. Therefore, we apply numerical and computational methods
to solve our problems. In each case we try to solve the problem either
by applying fractional necessary conditions or by approximating the problem
by a classical one and then solving the approximate problem.


\subsection{Fixed final time}
\label{sub:sec:fft}

We first solve a simple problem with fixed final time.
In this case the exact solution, i.e.,
the optimal control and the corresponding optimal trajectory,
is known, and hence we can compare it with the approximations
obtained by our numerical method.

\begin{example}
\label{exm41}
Consider the following optimal control problem:
$$
J[x,u]=\int_0^1 \left(t u(t)-(\a+2)x(t)\right)^2\,dt \longrightarrow \min
$$
subject to the control system
$$
\dot{x}(t)+{^C_0D^\a_t} x(t)=u(t)+t^2,
$$
and the boundary conditions
$$
x(0)=0, \quad x(1)=\frac{2}{\Gamma(3+\a)}.
$$
The solution is given by
$$
\left(\overline x(t),\overline u(t)\right)
=\left(\frac{2t^{\a+2}}{\Gamma(\a+3)},\frac{2t^{\a+1}}{\Gamma(\a+2)}\right),
$$
because $J(x,u) \geq 0$ for all pairs $(x,u)$ and
$\overline x(0)=0$, $\overline x(1)=\frac{2}{\Gamma(3+\a)}$,
$\dot{\overline{x}}(t) = \overline{u}(t)$ and
${^C_0D^\a_t} \overline x (t)=t^2$ with $J(\overline x,\overline u)=0$.
It is trivial to check that $(\overline x,\overline u)$ satisfies
the fractional necessary optimality conditions given by
Theorem~\ref{Opt:MainTheo}/Corollary~\ref{OPT:MainCor}.
\end{example}

Let us apply the fractional necessary conditions
to the above problem. The Hamiltonian is
$H=\left(tu-(\a+2)x\right)^2+\lambda u+\lambda t^2$.
The stationary condition \eqref{OPT:stacionary} implies that
for $t \ne 0$
$$
u(t)=\frac{\a+2}{t}x(t)-\frac{\lambda(t)}{2t^2},
$$
and hence
\begin{equation}
\label{emx1Ham}
H=-\frac{\lambda^2}{4t^2}+\frac{\a+2}{t}x\lambda+t^2\lambda,
\quad t \ne 0.
\end{equation}
Finally, \eqref{OPT:HamilSyst} gives
$$
\begin{cases}
\displaystyle\dot{x}(t)+{^C_0D^\a_t} x(t)
=-\frac{\lambda}{2t^2}+\frac{\a+2}{t}x(t)+t^2\\[8pt]
\displaystyle-\dot{\lambda}(t)+{_tD^\a_1}\lambda(t)
=\frac{\a+2}{t}\lambda(t)
\end{cases}
,\quad
\begin{cases}
x(0)=0\\
\displaystyle x(1)=\frac{2}{\Gamma(3+\a)}.
\end{cases}
$$
At this point, we encounter a fractional boundary value problem that needs
to be solved in order to reach the optimal solution. A handful of methods can be found
in the literature to solve this problem. Nevertheless, we use approximations
\eqref{expanMom} and \eqref{expanMomR}, up to order $N$, that have been introduced
in \cite{Atan2} and used in \cite{Jelicic,PATFracInt}. With our choice of approximation,
the fractional problem is transformed into a classical (integer-order) boundary value problem:
$$
\begin{cases}
\displaystyle\dot{x}(t)=\left[\left(\frac{\a+2}{t}-At^{-\a}\right)x(t)
+\sum_{p=2}^N C_pt^{1-p-\a}V_p(t)-\frac{\lambda(t)}{2t^2}
+t^2\right]\frac{1}{1+Bt^{1-\a}}\\[5pt]
\dot{V}_p(t)=(1-p)t^{p-2}x(t), \quad p=2,\ldots,N\\[5pt]
\displaystyle\dot{\lambda}(t)=\left[\left(A(1-t)^{-\a}-\frac{\a+2}{t}\right)\lambda(t)
-\sum_{p=2}^N C_p(1-t)^{1-p-\a}W_p(t)\right]\frac{1}{1+B(1-t)^{1-\a}}\\
\dot{W}_p(t)=-(1-p)(1-t)^{p-2}\lambda(t), \quad p=2,\ldots,N,
\end{cases}
$$
subject to the boundary conditions
$$
\begin{cases}
\displaystyle x(0)=0,\quad x(1)=\frac{2}{\Gamma(3+\a)},\\
V_p(0)=0, \quad p=2,\ldots,N,\\
W_p(1)=0, \quad p=2,\ldots,N.
\end{cases}
$$
The solutions are depicted in Figure~\ref{exm41FNCFig}
for $N=2$, $N=3$ and $\a=1/2$. Since the exact solution
for this problem is known, for each $N$ we compute
the approximation error by using the maximum norm.
Assume that $\overline x(t_i)$ are the approximated values
on the discrete time horizon $a=t_0,t_1,\ldots,t_n$.
Then the error is given by
$$
E = \max_{i}(|x(t_i)-\overline x(t_i)|).
$$
\begin{figure}
\begin{center}
\subfigure[$x(t), N=2$]{\includegraphics[scale=0.405]{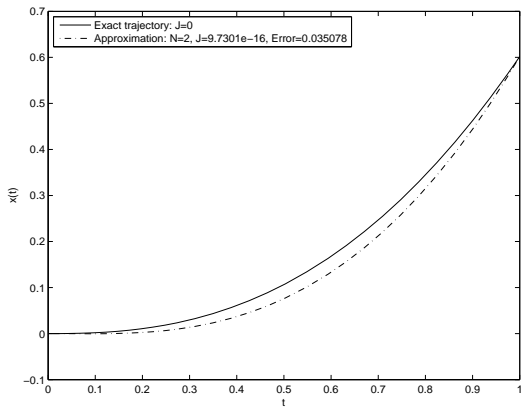}}
\subfigure[$u(t), N=2$]{\includegraphics[scale=0.405]{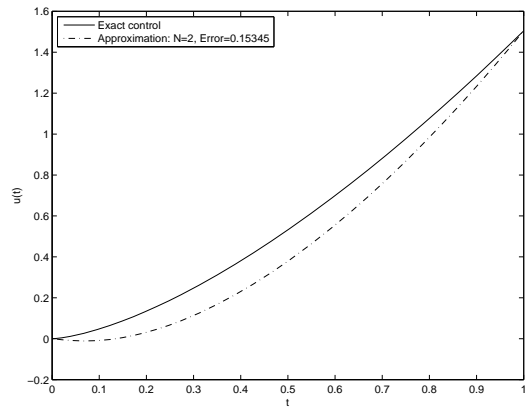}}
\subfigure[$x(t), N=3$]{\includegraphics[scale=0.405]{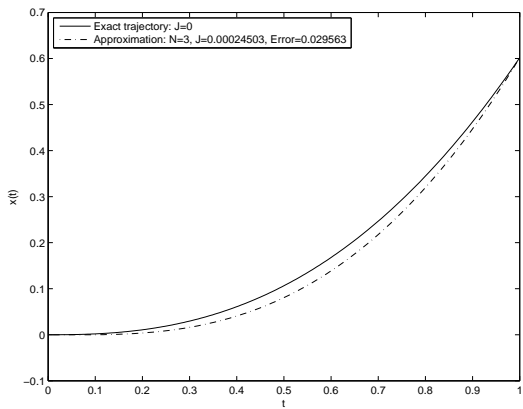}}
\subfigure[$u(t), N=3$]{\includegraphics[scale=0.405]{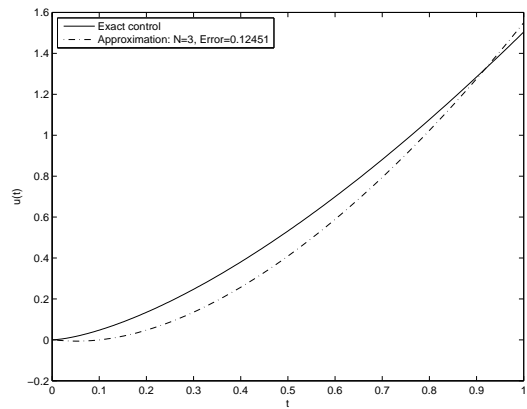}}
\end{center}
\caption{Exact solution (solid lines) for the problem in Example~\ref{exm41} with $\a = 1/2$
{\it versus} numerical solutions (dashed lines) obtained using approximations
\eqref{expanMom} and \eqref{expanMomR} up to order $N$
in the fractional necessary optimality conditions.}
\label{exm41FNCFig}
\end{figure}

Another approach is to approximate the original problem by using
\eqref{expanMom} for the fractional derivative. Following the procedure
discussed in Section~\ref{SecNecCond}, the problem of Example~\ref{exm41}
is approximated by
$$
\tilde{J}[x,u]=\int_0^1(tu-(\a+2)x)^2\,dt \longrightarrow \min
$$
subject to the control system
$$
\begin{cases}
\dot{x}(t)[1+B(\a,N)t^{1-\a}]+A(\a,N)t^{-\a}x(t)
-\sum_{p=2}^N C(\a,p)t^{1-p-\a}V_p(t)=u(t)+t^2\\
\dot{V}_p(t)=(1-p)t^{p-2}x(t),
\end{cases}
$$
and boundary conditions
$$
x(0)=0, \quad x(1)=\frac{2}{\Gamma(3+\a)},
\quad V_p(0)=0, \quad p=2,3,\ldots,N.
$$
The Hamiltonian system for this classical optimal control problem is
\begin{equation*}
\begin{split}
H=\left(tu-(\a+2)x\right)^2
&+\frac{\lambda_1(-A(\a,N)t^{-\a}x
+\sum_{p=2}^N C(\a,p)t^{1-p-\a}V_p+u+t^2)}{1+B(\a,N)t^{1-\a}}\\
&+\sum_{p=2}^N (1-p)t^{p-2}\lambda_p x.
\end{split}
\end{equation*}
Using the stationary condition $\frac{\partial H}{\partial u}=0$, we have
$$
u(t)=\frac{\a+2}{t}x(t)-\frac{\lambda_1(t)}{2t^2(1+B(\a,N)t^{1-\a})}
\quad \text{for } t \ne 0.
$$
Finally, the Hamiltonian becomes
\begin{equation}
\label{exmHamiltonian}
H=\phi_0\lambda_1^2+\phi_1x\lambda_1+\sum_{p=2}^N \phi_pV_p\lambda_1
+\phi_{N+1}\lambda_1+\sum_{p=2}^N(1-p)t^{p-2}x\lambda_p, \quad t \ne 0,
\end{equation}
where
\begin{equation}
\label{eq:phi0:phi1}
\phi_0(t)=\frac{-1}{4t^2(1+B(\a,N)t^{1-\a})^2},
\quad \phi_1(t)=\frac{\a+2-A(\a,N)t^{1-\a}}{t(1+B(\a,N)t^{1-\a})},
\end{equation}
and
\begin{equation}
\label{eq:phi2:phi3}
\phi_p(t)=\frac{C(\a,p)t^{1-p-\a}}{1+B(\a,N)t^{1-\a}},
\quad \phi_{N+1}(t)=\frac{t^2}{1+B(\a,N)t^{1-\a}}.
\end{equation}
The Hamiltonian system
$\dot{\mathbf{x}}=\frac{\partial H}{\partial \mathbf{\lambda}}$,
$\dot{\bm{\lambda}}=-\frac{\partial H}{\partial \mathbf{x}}$,
gives
$$
\begin{cases}
\dot{x}(t)=2\phi_0(t)\lambda_1(t)+\phi_1(t)x(t)+\sum_{p=2}^N \phi_p(t)V_p(t)+\phi_{N+1}(t)\\
\dot{V_p}=(1-p)t^{p-2}x(t), \quad p=2,\ldots,N\\
\dot{\lambda_1}=-\phi_1(t)\lambda_1(t)+\sum_{p=2}^N(p-1)t^{p-2}\lambda_p\\
\dot{\lambda_p}=-\phi_p(t)\lambda_1(t), \quad p=2,\ldots,N,
\end{cases}
$$
subject to the boundary conditions
$$
\begin{cases}
x(0)=0\\
V_p(0)=0, \quad p=2,\ldots,N
\end{cases}
,\qquad
\begin{cases}
\displaystyle x(1)=\frac{2}{\Gamma(3+\a)}\\
\lambda_p(1)=0, \quad p=2,\ldots,N.
\end{cases}
$$
This two-point boundary value problem 
was solved using MATLAB$^\circledR$\index{MATLAB} \textsf{bvp4c}
built-in function for $N=2$ and $N=3$. 
The results are depicted in Figure~\ref{exm41Fig}.
\begin{figure}
\begin{center}
\subfigure[$x(t), N=2$]{\includegraphics[scale=0.54]{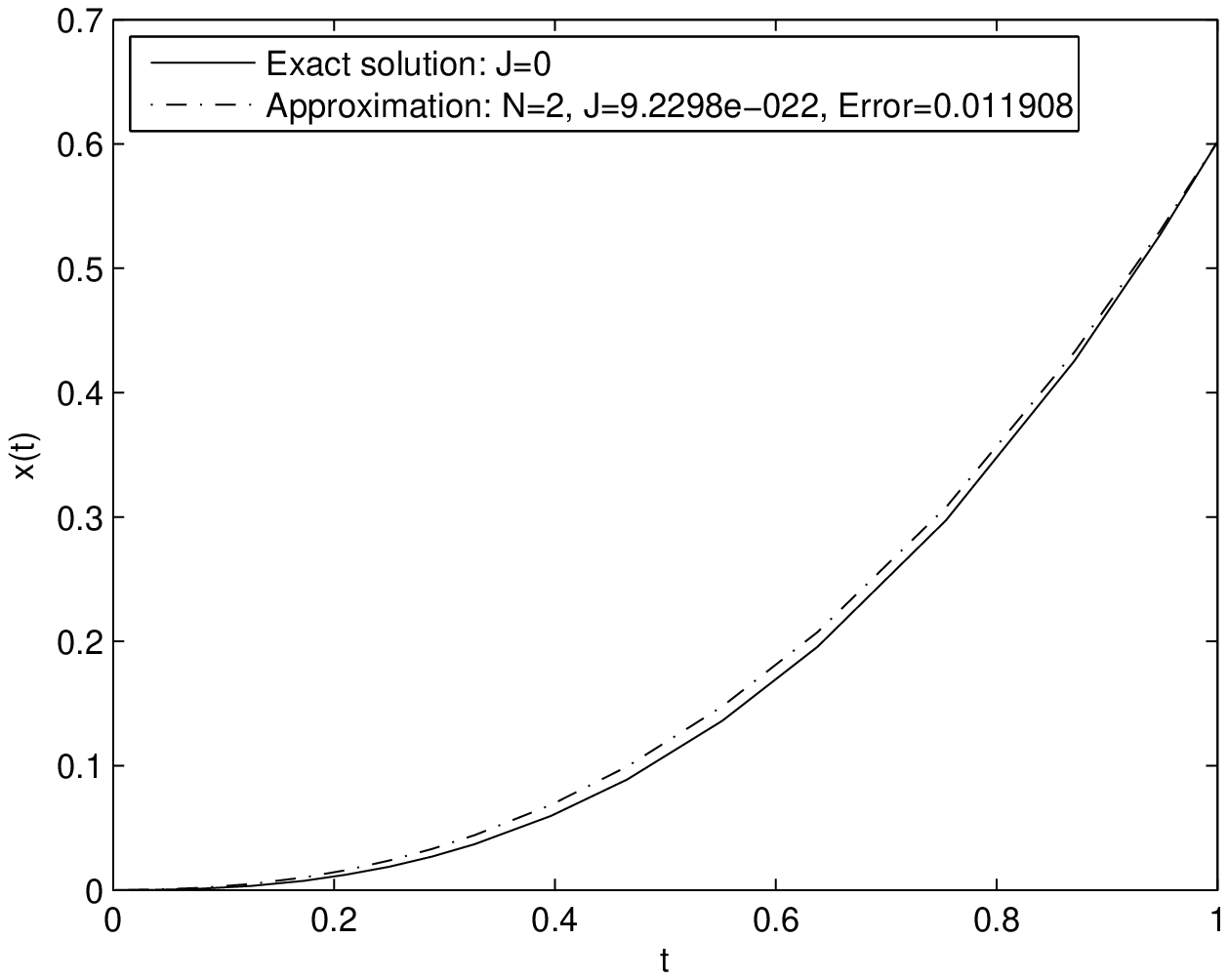}}
\subfigure[$u(t), N=2$]{\includegraphics[scale=0.54]{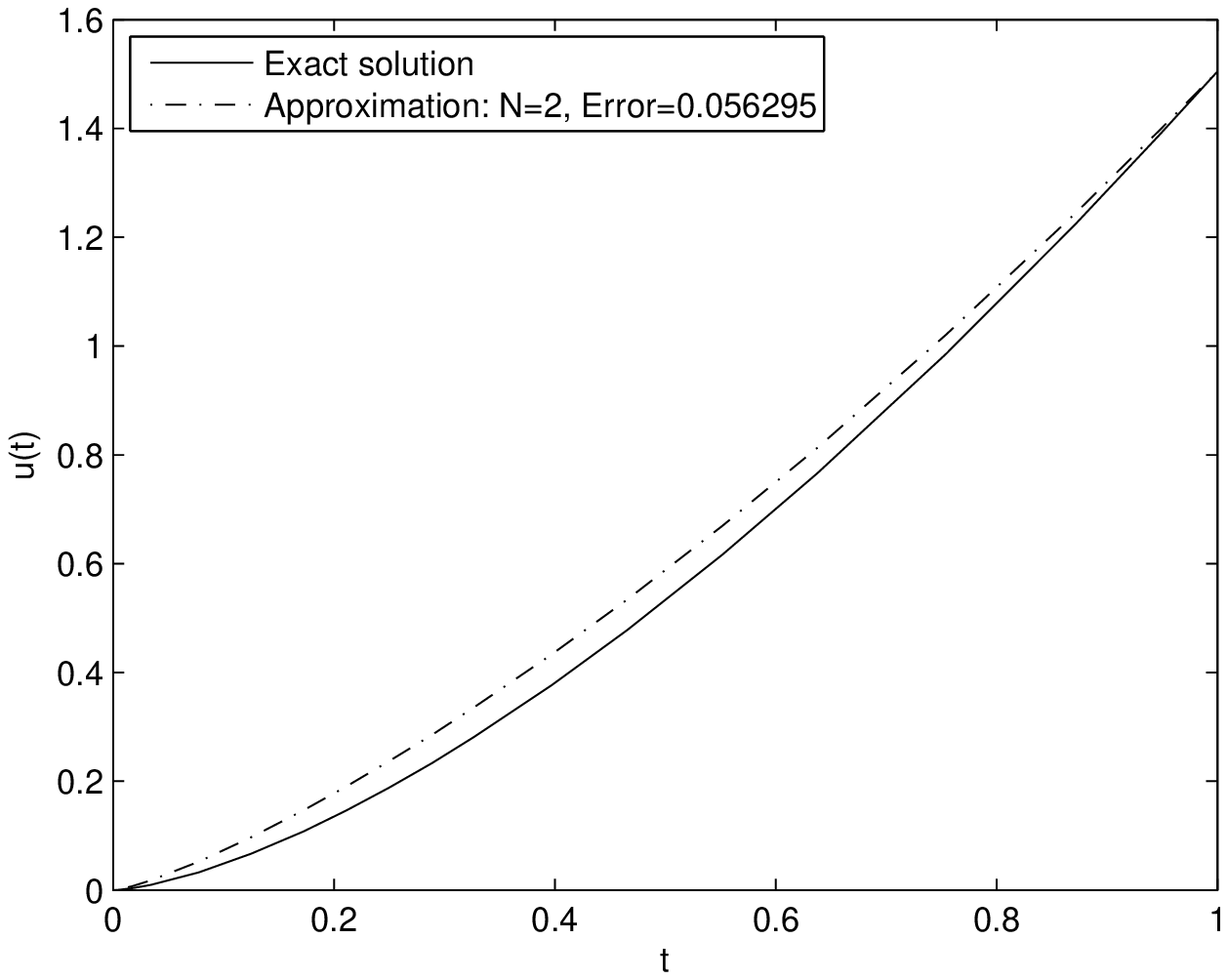}}
\subfigure[$x(t), N=3$]{\includegraphics[scale=0.54]{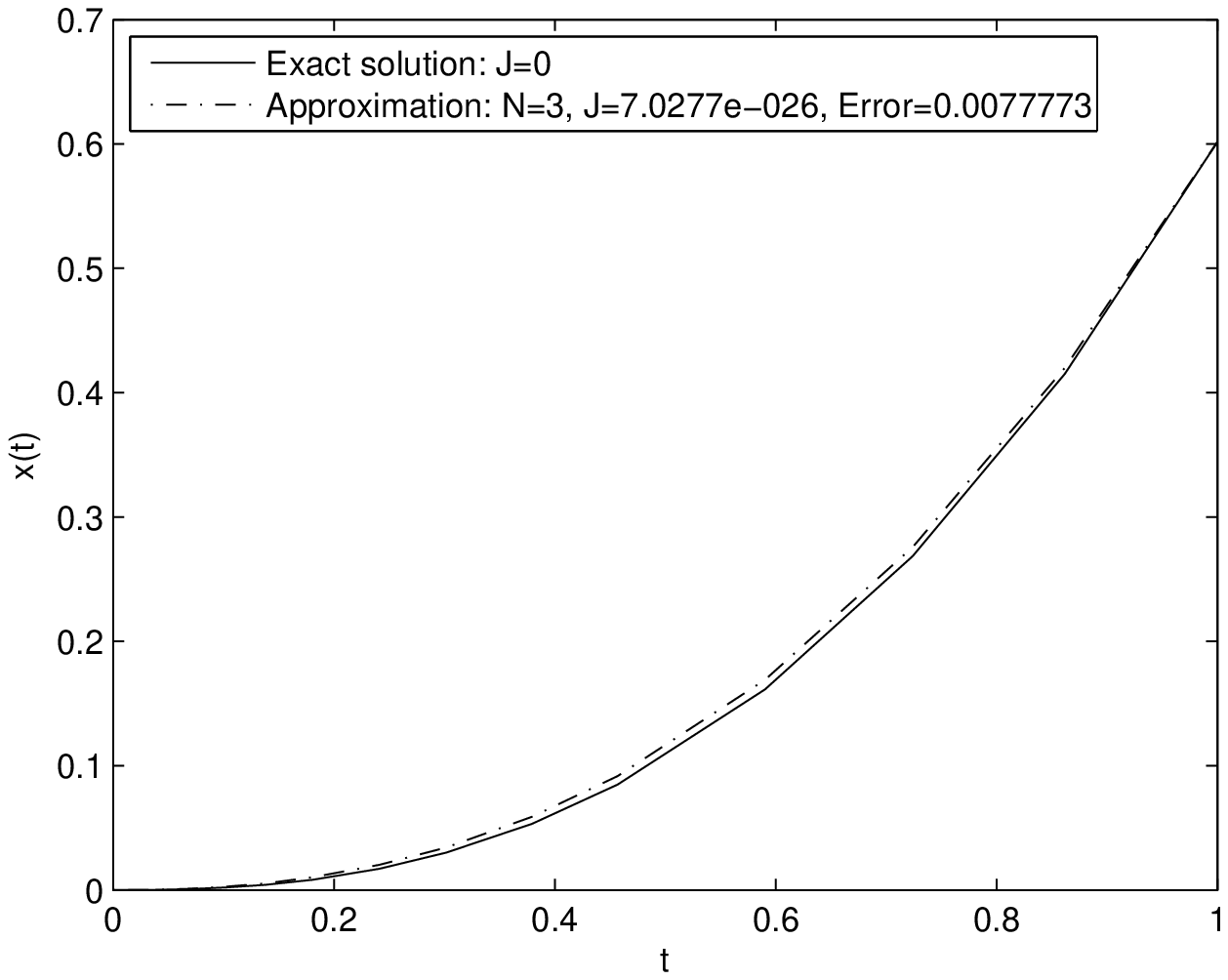}}
\subfigure[$u(t), N=3$]{\includegraphics[scale=0.54]{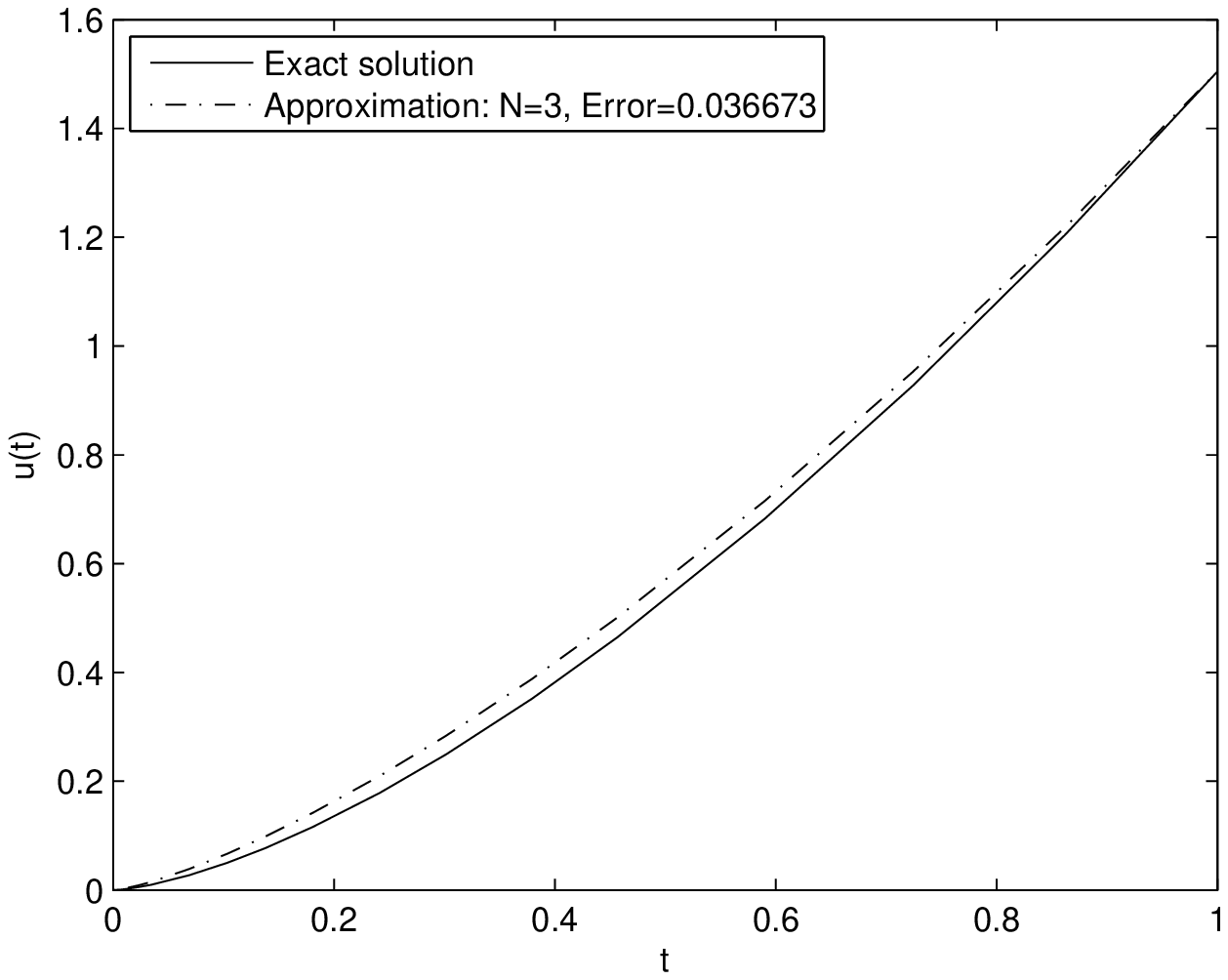}}
\end{center}
\caption{Exact solution (solid lines) for the problem in Example~\ref{exm41} with $\a = 1/2$
{\it versus} numerical solutions (dashed lines) obtained by approximating the fractional order
optimal control problem using \eqref{expanMom} up to order $N$ and then
solving the classical necessary optimality conditions with MATLAB$^\circledR$ \textsf{bvp4c}
built-in function.} \label{exm41Fig}
\end{figure}


\subsection{Free final time}

The two numerical methods discussed in Section~\ref{sub:sec:fft}
are now employed to solve a fractional order optimal control
problem with free final time $T$.

\begin{example}
\label{exm42}
Find an optimal triplet $(x(\cdot),u(\cdot),T)$ that minimizes
$$
J[x,u,T]=\int_0^T(tu-(\a+2)x)^2\,dt
$$
subject to the control system
$$
\dot{x}(t)+{^C_0D^\a_t} x(t)=u(t)+t^2
$$
and boundary conditions
$$
x(0)=0, \quad x(T)=1.
$$
An exact solution to this problem is not known
and we apply the two numerical procedures
already used with respect to the fixed final time
problem in Example~\ref{exm41}.
\end{example}

We begin by using the fractional necessary optimality conditions that,
after approximating the fractional terms, results in
$$
\begin{cases}
\displaystyle \dot{x}(t)=\left[\left(\frac{\a+2}{t}-At^{-\a}\right)x(t)
+\sum_{p=2}^N C_pt^{1-p-\a}V_p(t)
-\frac{\lambda(t)}{2t^2}+t^2\right]\frac{1}{1+Bt^{1-\a}}\\[5pt]
\dot{V}_p(t)=(1-p)t^{p-2}x(t), \quad p=2,\ldots,N\\[5pt]
\displaystyle\dot{\lambda}(t)=\left[\left(A(1-t)^{-\a}-\frac{\a+2}{t}\right)\lambda(t)
-\sum_{p=2}^N C_p(1-t)^{1-p-\a}W_p(t)\right]\frac{1}{1+B(1-t)^{1-\a}}\\
\dot{W}_p(t)=-(1-p)(1-t)^{p-2}\lambda(t), \quad p=2,\ldots,N,
\end{cases}
$$
subject to the boundary conditions
$$
\begin{cases}
x(0)=0,\quad x(T)=1,\\
V_p(0)=0, \quad p=2,\ldots,N,\\
W_p(T)=0, \quad p=2,\ldots,N.
\end{cases}
$$
The only difference here with respect to Example~\ref{exm41}
is that there is an extra unknown, the terminal time $T$.
The boundary condition for this new unknown is chosen appropriately from the
transversality conditions discussed in Corollary~\ref{OPT:MainCor}, i.e.,
$$
[H(t,x,u,\lambda)-\lambda(t)\LCD x(t)+\dot{x}(t)\RIT\lambda(t)]_{t=T}=0,
$$
where $H$ is given as in \eqref{emx1Ham}.
Since we require $\lambda$ to be continuous,
$\RIT\lambda(t)|_{t=T}=0$ (cf. \cite[pag.~46]{Miller}) and so  $\lambda(T)=0$.
One possible way to proceed consists in translating the problem into the interval
$[0,1]$ by the change of variable $t=Ts$ \cite{Avvakumov}. In this setting,
either we add $T$ to the problem as a new state variable with dynamics $\dot{T}(s)=0$,
or we treat it as a parameter. We use the latter, to get the following
parametric boundary value problem:
$$
\begin{cases}
\displaystyle\dot{x}(s)=\frac{\left[\left(\frac{\a+2}{Ts}-A(Ts)^{-\a}\right)x(s)
+\sum_{p=2}^N C_p(Ts)^{1-p-\a}V_p(s)
-\frac{\lambda(s)}{2(Ts)^2}+(Ts)^2\right] T}{1+B(Ts)^{1-\a}},\\[5pt]
\dot{V}_p(s)=T(1-p)(Ts)^{p-2}x(s), \quad p=2,\ldots,N,\\[5pt]
\displaystyle\dot{\lambda}(s)=\frac{\left[\left(A(1-Ts)^{-\a}-\frac{\a+2}{Ts}\right)\lambda(s)
-\sum_{p=2}^N C_p(1-Ts)^{1-p-\a}W_p(s)\right] T}{1+B(1-Ts)^{1-\a}},\\
\dot{W}_p(s)=-T(1-p)(1-Ts)^{p-2}\lambda(s), \quad p=2,\ldots,N,
\end{cases}
$$
subject to the boundary conditions
$$
\begin{cases}
x(0)=0\\
V_p(0)=0, \quad p=2,\ldots,N\\
W_p(1)=0, \quad p=2,\ldots,N
\end{cases}
,\qquad
\begin{cases}
x(1)=1\\
\lambda(1)=0.
\end{cases}
$$
This parametric boundary value problem is solved for $N=2$ and $\a=0.5$
with MATLAB$^\circledR$ \textsf{bvp4c} function. The result is shown
in Figure~\ref{exm42fig} (dashed lines).
\begin{figure}
\begin{center}
\subfigure[$x(t), N=2$]{\includegraphics[scale=0.405]{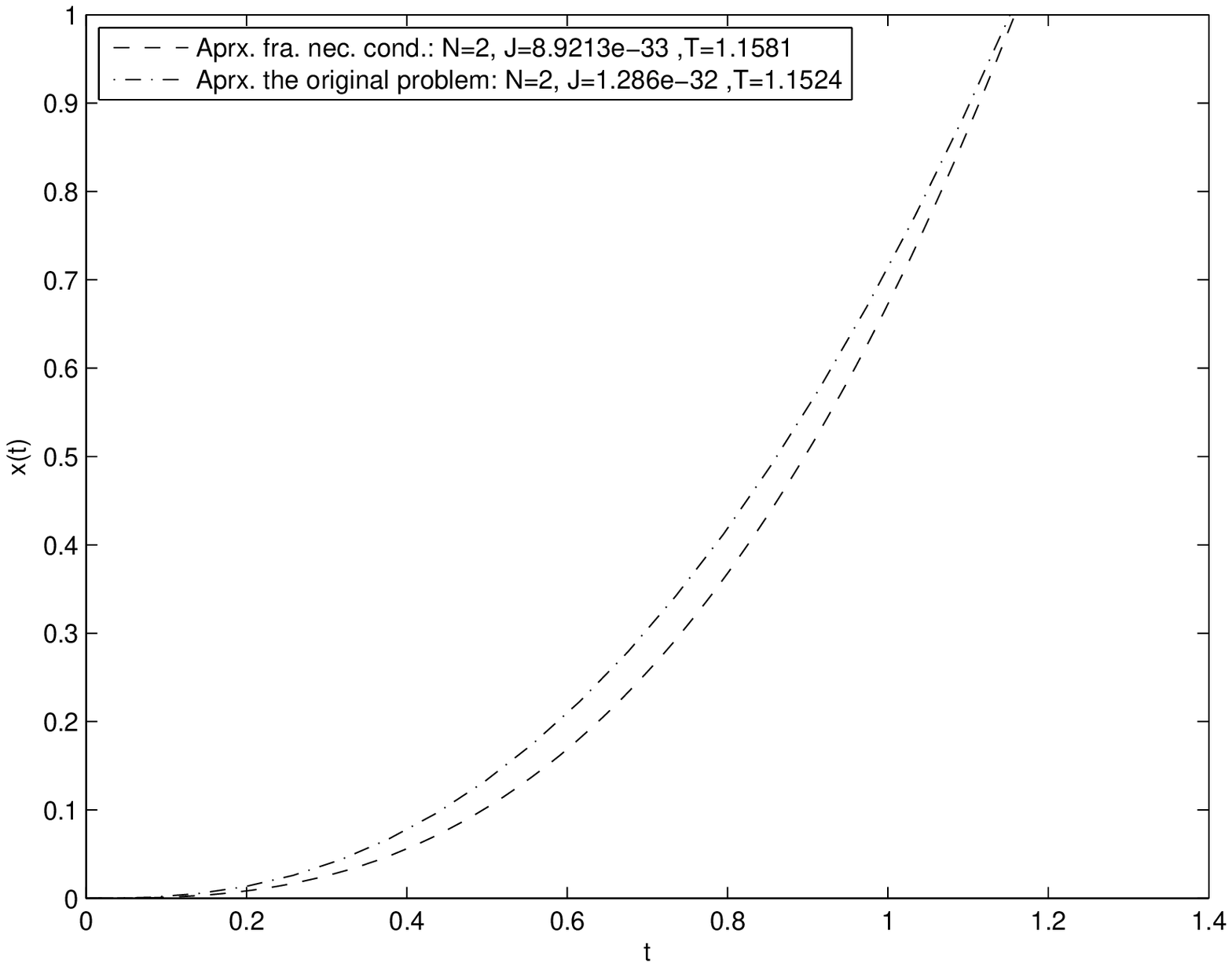}}
\subfigure[$u(t), N=2$]{\includegraphics[scale=0.405]{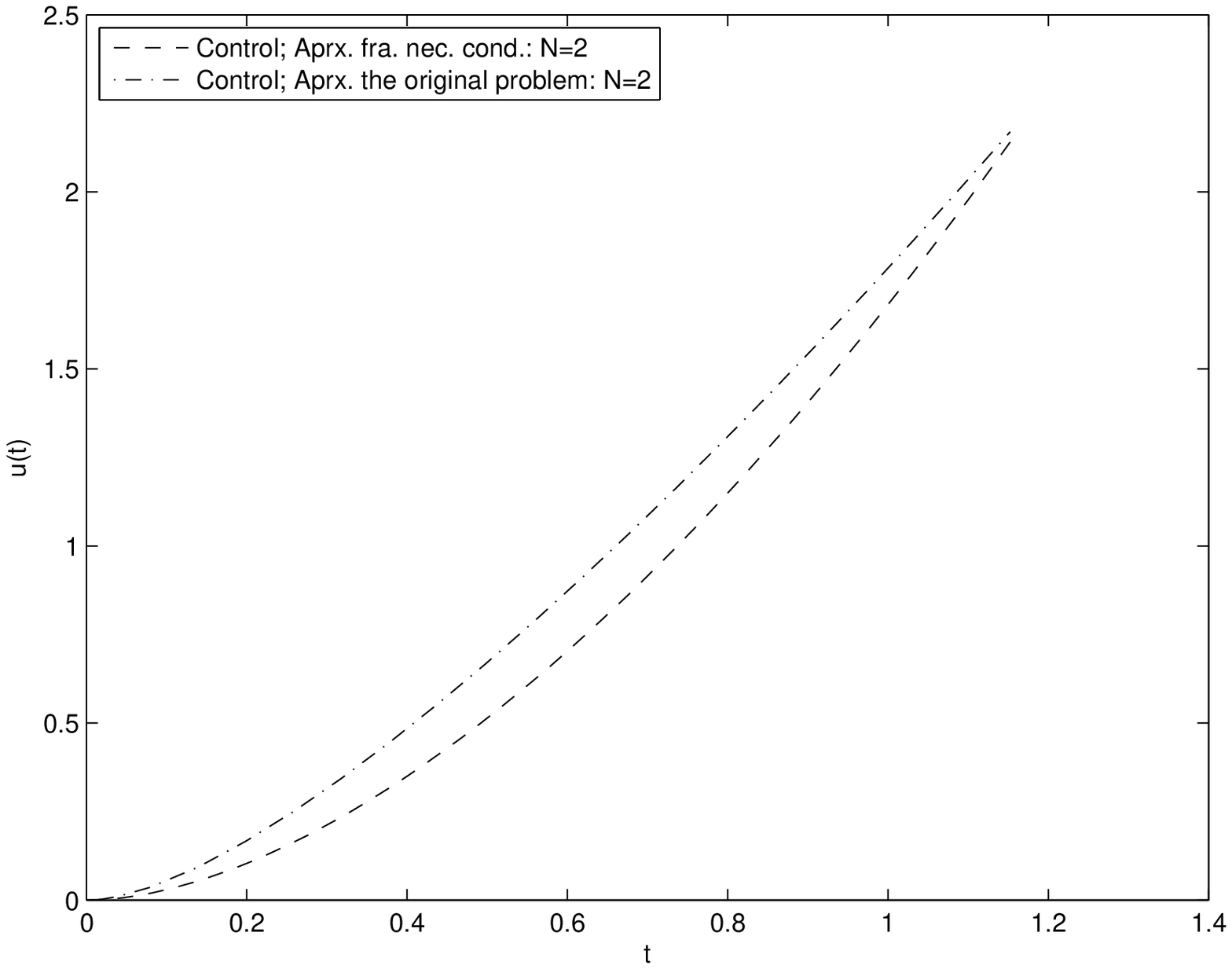}}
\end{center}
\caption{Numerical solutions to the free final time problem of Example~\ref{exm42}
with $\alpha = 1/2$, using fractional necessary optimality conditions (dashed lines) and
approximation of the problem to an integer-order optimal control problem (dash-dotted lines).}
\label{exm42fig}
\end{figure}

We also solve Example~\ref{exm42} with $\alpha = 1/2$ by directly transforming it
into an integer-order optimal control problem with free final time. As is well known
in the classical theory of optimal control, the Hamiltonian must vanish at the terminal point
when the final time is free, i.e., one has $H|_{t=T}=0$ with $H$ given
by \eqref{exmHamiltonian} \cite{Kirk}. For $N=2$, the necessary optimality conditions
give the following two point boundary value problem:
$$
\begin{cases}
\dot{x}(t)=2\phi_0(t)\lambda_1(t)+\phi_1(t)x(t)+\phi_2(t)V_2(t)+\phi_{3}(t)\\
\dot{V_2}=-x(t)\\
\dot{\lambda_1}=-\phi_1(t)\lambda_1(t)+x(t)\\
\dot{\lambda_2}=-\phi_2(t)\lambda_1(t),
\end{cases}
$$
where $\phi_0(t)$ and $\phi_1(t)$ are given by
\eqref{eq:phi0:phi1} and $\phi_2(t)$ and $\phi_{3}(t)$
by \eqref{eq:phi2:phi3} with $p = N = 2$.
The trajectory $x$ and corresponding $u$
are shown in Figure~\ref{exm42fig} (dash-dotted lines).

\CP
\chapter{Fractional variational problems depending on indefinite integrals}
\label{IndInt}

In this chapter we obtain necessary optimality conditions
for variational problems with a Lagrangian
depending on a Caputo fractional derivative, a fractional and
an indefinite integral. Main results give fractional
Euler--Lagrange type equations and natural boundary conditions,
which provide a generalization of previous results found in the literature.
Isoperimetric problems, problems with holonomic constraints and those
depending on higher-order Caputo derivatives,
as well as fractional Lagrange problems, are considered.
Our main contribution is an extension of the results presented
in \cite{AGRA1,Gregory} by considering Lagrangians containing an antiderivative,
that in turn depend on the unknown function, a fractional integral, and a Caputo
fractional derivative (Section~\ref{sec:Fundprob}).
Transversality conditions are studied in Section~\ref{sec:natbound},
where the variational functional $J$ depends also on the terminal time $T$,
$J[x,T]$, and where we obtain conditions for a pair $(x,T)$ to be an optimal
solution to the problem. In Section~\ref{sec:IsoProb}
we consider isoperimetric problems with integral constraints
of the same type as the cost functionals considered in Section~\ref{sec:Fundprob}.
Fractional problems with holonomic constraints are considered in Section~\ref{sec:Holonomic}.
The situation when the Lagrangian depends on higher
order Caputo derivatives, \textrm{i.e.}, it depends on
$^C_aD_t^{\alpha_k}x(t)$ for $\alpha_k\in(k-1,k)$, $k\in\{1,\ldots,n\}$,
is studied in Section~\ref{sec:Higher}, while Section~\ref{sec:FracOpt}
considers fractional Lagrange problems and the Hamiltonian
approach. In Section~\ref{sec:SufConditions}
we obtain sufficient conditions of optimization
under suitable convexity assumptions
on the Lagrangian \cite{APTIndInt}.


\section{The fundamental problem}
\label{sec:Fundprob}

Let $\alpha\in(0,1)$ and $\beta>0$.
The problem that we address is stated in the following way:
minimize the cost functional
\begin{equation}
\label{funct}
J[x]=\int_a^b L(t,x(t),{^C_aD_t^\alpha}x(t),{_aI_x^\beta}x(t),z(t))dt,
\end{equation}
where the variable $z$ is defined by
$$
z(t)=\int_a^t l(\t,x(\t),{^C_aD_\t^\alpha}x(\t),{_aI_\t^\beta}x(\t))d\t,
$$
subject to the boundary conditions
\begin{equation}
\label{bound}
x(a)=x_a \quad \mbox{and} \quad x(b)=x_b.
\end{equation}
We assume that the functions $(t,x,v,w,z)\to L(t,x,v,w,z)$
and  $(t,x,v,w)\to l(t,x,v,w)$ are of class $C^1$, and the trajectories
$x:[a,b]\to\mathbb{R}$ are absolute continuous functions, $x \in AC([a,b];\mathbb{R})$,
such that ${^C_aD_t^\alpha}x(t)$ and ${_aI_t^\beta}x(t)$ exist and are continuous on $[a,b]$.
We denote such class of functions by $\mathcal{F}([a,b];\mathbb{R})$.
Also, to simplify, by $[\cdot]$ and $\{\cdot\}$ we denote the operators
$$
[x](t)=(t,x(t),{^C_aD_t^\alpha}x(t),{_aI_t^\beta}x(t),z(t))
\quad \mbox{and}\quad \{x\}(t)=(t,x(t),{^C_aD_t^\alpha}x(t),{_aI_t^\beta}x(t)).
$$

\begin{theorem}
\label{ELTEo}
Let $x \in \mathcal{F}([a,b];\mathbb{R})$ be a minimizer
of $J$ as in \eqref{funct},
subject to the boundary conditions \eqref{bound}.
Then, for all $t\in[a,b]$, $x$ is a solution
of the fractional equation
\begin{multline}
\label{ELeq}
\frac{\partial L}{\partial x}[x](t)
+{_tD^\alpha_b}\left( \frac{\partial L}{\partial v}[x](t) \right)
+{_tI_b^\beta}\left(\frac{\partial L}{\partial w}[x](t)\right)
+\int_t^b \frac{\partial L}{\partial z}[x](\t)d\t\cdot \frac{\partial l}{\partial x}\{x\}(t)\\[5pt]
+{_tD^\alpha_b}\left( \int_t^b \frac{\partial L}{\partial z}[x](\t)d\t
\cdot \frac{\partial l}{\partial v}\{x\}(t)  \right)
+{_tI^\beta_b}\left( \int_t^b \frac{\partial L}{\partial z}[x](\t)d\t
\cdot \frac{\partial l}{\partial w}\{x\}(t)  \right)=0.
\end{multline}
\end{theorem}

\begin{proof}
Let $h\in \mathcal{F}([a,b];\mathbb{R})$
be such that $h(a)=0=h(b)$, and $\epsilon$
be a real number with $|\epsilon| \ll 1$. If we define $j$ as
$j(\epsilon)=J(x+\epsilon h)$, then $j'(0)=0$.
Differentiating $j$ at $\epsilon=0$, we get
\begin{multline*}
\int_a^b \left[ \frac{\partial L}{\partial x}[x](t)h(t)
+ \frac{\partial L}{\partial v}[x](t){^C_aD^\alpha_t}h(t)
+ \frac{\partial L}{\partial w}[x](t){_aI^\beta_t}h(t)\right.\\[5pt]
\left.\;\;+\frac{\partial L}{\partial z}[x](t)\int_a^t\left(
\frac{\partial l}{\partial x}\{x\}(\t)h(\t)
+\frac{\partial l}{\partial v}\{x\}(\t){^C_aD^\alpha_\t}h(\t)
+\frac{\partial l}{\partial w}\{x\}(\t){_aI^\beta_\t}h(\t)\right)d\t\right]dt=0.
\end{multline*}
The necessary condition \eqref{ELeq} follows
from the next relations and the fundamental lemma
of the calculus of variations
(\textrm{cf.}, \textrm{e.g.}, \cite[p.~32]{Brunt}):
$$
\int_a^b \frac{\partial L}{\partial v}[x](t){^C_aD^\alpha_t}h(t) dt
=\int_a^b {_t D_b^\alpha} \left(\frac{\partial L}{\partial v}[x](t) \right)h(t)dt
+ \left[{_tI_b^{1-\alpha}}\left(\frac{\partial L}{\partial v}[x](t) \right) h(t)\right]_a^b,
$$
$$
\int_a^b \frac{\partial L}{\partial w}[x](t){_aI^\beta_t}h(t) dt
=\int_a^b {_t I_b^\beta} \left(\frac{\partial L}{\partial w}[x](t) \right)h(t)dt,
$$
\begin{align*}
\int_a^b & \frac{\partial L}{\partial z}[x](t)\left(\int_a^t
\frac{\partial l}{\partial x}\{x\}(\t)h(\t) d\t \right) dt\\
&=\int_a^b \left( -\frac{d}{dt}\int_t^b\frac{\partial L}{\partial z}[x](\t)d\t \right)
\left( \int_a^t \frac{\partial l}{\partial x}\{x\}(\t)h(\t) d\t \right) dt\\
&=\int_a^b \left(\int_t^b\frac{\partial L}{\partial z}[x](\t)d\t \right)
\frac{\partial l}{\partial x}\{x\}(t)h(t) \, dt
-\left[\int_t^b\frac{\partial L}{\partial z}[x](\t)d\t \int_a^t
\frac{\partial l}{\partial x}\{x\}(\t)h(\t) d\t  \right]_a^b\\
&= \int_a^b \left(\int_t^b\frac{\partial L}{\partial z}[x](\t)d\t \right)
\frac{\partial l}{\partial x}\{x\}(t)h(t) \, dt,
\end{align*}
\begin{align*}
\int_a^b & \frac{\partial L}{\partial z}[x](t)\left(\int_a^t
\frac{\partial l}{\partial v}\{x\}(\t){^C_aD^\alpha_\t}h(\t) d\t \right) dt\\
&\hspace{2cm}= \int_a^b \left( -\frac{d}{dt}\int_t^b\frac{\partial L}{\partial z}[x](\t)d\t \right)
\left( \int_a^t \frac{\partial l}{\partial v}\{x\}(\t){^C_aD^\alpha_\t}h(\t) d\t \right)  dt\\
&\hspace{2cm}= \left[- \left(\int_t^b\frac{\partial L}{\partial z}[x](\t)d\t \right)\left(\int_a^t
\frac{\partial l}{\partial v}\{x\}(\t){^C_aD^\alpha_\t}h(\t) d\t\right) \right]_a^b\\
&\hspace{3cm}+\int_a^b \left(\int_t^b\frac{\partial L}{\partial z}[x](\t)d\t \right)
\frac{\partial l}{\partial v}\{x\}(t){^C_aD^\alpha_t}h(t) \,  dt\\
&\hspace{2cm}= \int_a^b {_tD^\alpha_b}\left(\int_t^b\frac{\partial L}{\partial z}[x](\t)d\t
\frac{\partial l}{\partial v}\{x\}(t)\right)h(t) \, dt\\
&\hspace{3cm}+\left[ {_tI^{1-\alpha}_b}\left(\int_t^b\frac{\partial L}{\partial z}[x](\t)d\t
\frac{\partial l}{\partial v}\{x\}(t) \right) h(t) \right]_a^b,
\end{align*}
and
$$
\int_a^b \frac{\partial L}{\partial z}[x](t)\left(\int_a^t
\frac{\partial l}{\partial w}\{x\}(t){_aI^\beta_\t}h(\t) d\t \right) dt
=\int_a^b {_tI^\beta_b}\left(\int_t^b\frac{\partial L}{\partial z}[x](\t)d\t
\frac{\partial l}{\partial w}\{x\}(t)\right)h(t) \, dt.
$$
\end{proof}

The fractional Euler--Lagrange equation \eqref{ELeq} involves
not only fractional integrals and fractional derivatives,
but also indefinite integrals\index{Indefinite integrals}. 
Theorem~\ref{ELTEo} gives a necessary condition 
to determine the possible choices for extremizers.

\begin{example}
\label{example:bra}
Consider the functional
\begin{equation}
\label{example}
J[x]=\int_0^1 \left[({^C_0D_t^\alpha}x(t)
-\Gamma(\alpha+2)t)^2+z(t)\right]dt,
\end{equation}
where $\alpha \in (0,1)$ and
$$
z(t)=\int_0^t (x(\t)-\t^{\alpha+1})^2 \, d\t,
$$
defined on the set
$$
\left\{ x\in \mathcal{F}([0,1];\mathbb{R})
:\, x(0)=0 \, \mbox{ and }\,  x(1)=1\right\}.
$$
Let
\begin{equation}
\label{eq:ext:ne}
x_\alpha(t)=t^{\alpha+1},\quad t\in[0,1].
\end{equation}
Then,
$$
{^C_0D_t^\alpha}x_\alpha(t)=\Gamma(\alpha+2)t.
$$
Since $J(x)\geq0$ for all admissible functions $x$,
and $J(x_\alpha)=0$, we have that $x_\alpha$ is a minimizer of $J$.
The Euler--Lagrange equation\index{Euler--Lagrange equation} 
applied to \eqref{example} gives
\begin{equation}
\label{eq:fELeq:ne}
{_tD_1^\alpha}({^C_0D_t^\alpha}x(t)-\Gamma(\alpha+2)t)
+\int_t^1 1dt \, (x(t)-t^{\alpha+1})=0.
\end{equation}
Obviously, $x_\alpha$ is a solution of the fractional
differential equation \eqref{eq:fELeq:ne}.
\end{example}

The extremizer \eqref{eq:ext:ne}
of Example~\ref{example:bra} is smooth on the
closed interval $[0,1]$. This is not always
the case. As next example shows, minimizers
of \eqref{funct}--\eqref{bound} are not necessarily
$C^1$ functions.

\begin{example}
\label{example1}
Consider the following fractional variational 
problem\index{Variational problem}: to minimize the functional
\begin{equation}
\label{funcExample}
J[x]=\int_0^1\left[ \left({^C_0D_t^\alpha}x(t)-1\right)^2+z(t)\right]dt,
\end{equation}
on
$$
\left\{ x \in \mathcal{F}([0,1];\mathbb{R})
\, : \, x(0)=0 \quad \mbox{and}
\quad x(1)=\frac{1}{\Gamma(\alpha+1)}\right\},
$$
where $z$ is given by
$$
z(t)=\int_0^t \left( x(\t)-\frac{\t^\alpha}{\Gamma(\alpha+1)}\right)^2d\t.
$$
Since ${^C_0D_t^\alpha}t^\alpha=\Gamma(\alpha+1)$,
we deduce easily that function
\begin{equation}
\label{eq:gs:ex}
\overline{x}(t)=\frac{t^\alpha}{\Gamma(\alpha+1)}
\end{equation}
is the global minimizer to the problem. Indeed, $J(x)\geq 0$
for all $x$, and $J(\overline x)=0$. Let us see that
$\overline{x}$ is an extremal for $J$.
The fractional Euler--Lagrange equation \eqref{ELeq} becomes
\begin{equation}
\label{ELeqExample}
2 \, {_tD_1^\alpha} ({^C_0D_t^\alpha}x(t)-1) + \int_t^1 1\, d\t
\cdot  2 \left( x(t)-\frac{t^\alpha}{\Gamma(\alpha+1)}\right)=0.
\end{equation}
Obviously, $\overline{x}$ is a solution of equation \eqref{ELeqExample}.
\end{example}

\begin{remark}
The minimizer \eqref{eq:gs:ex} of Example~\ref{example1}
is not differentiable at 0, as $0<\alpha<1$. However,
$\overline{x}(0)=0$ and  ${^C_0D_t^\alpha}\overline{x}(t)
={_0D_t^\alpha}\overline{x}(t)=\Gamma(\alpha+1)$ for any
$t \in [0,1]$.
\end{remark}

\begin{corollary}[\textrm{cf.} equation (9) of \cite{AGRA1}]
If $x$ is a minimizer of
\begin{equation}
\label{funct7}
J[x]=\int_a^b L(t,x(t),{^C_aD_t^\alpha}x(t))dt,
\end{equation}
subject to the boundary conditions \eqref{bound},
then $x$ is a solution of the fractional equation
$$
\frac{\partial L}{\partial x}[x](t)
+{_tD^\alpha_b}\left( \frac{\partial L}{\partial v}[x](t) \right)=0.
$$
\end{corollary}

\begin{proof}
Follows from Theorem~\ref{ELTEo} with an $L$ that
does not depend on ${_aI_t^\beta}x$ and $z$.
\end{proof}

We now derive the Euler--Lagrange equations for functionals
containing several dependent variables,
\textrm{i.e.}, for functionals of type
\begin{equation}
\label{functMul}
J[x_1,\ldots,x_n]
=\int_a^b L(t,x_1,\ldots, x_n,
{^C_aD_t^\alpha}x_1,\ldots,{^C_aD_t^\alpha}x_n,
{_aI_t^\beta}x_1,\ldots,{_aI_t^\beta}x_n,z(t))dt,
\end{equation}
where $n\in \mathbb{N}$ and $z$ is defined by
$$
z(t)=\int_a^t l(\t,x_1(\t),\ldots,x_n(\t),{^C_aD_\t^\alpha}x_1(\t),
\ldots,{^C_aD_\t^\alpha}x_n(\t),{_aI_\t^\beta}x_1(\t),\ldots,{_aI_t^\beta}x_n(\t))d\t,
$$
subject to the boundary conditions
\begin{equation}
\label{boundMul}
x_k(a)=x_{a,k} \quad \mbox{and}
\quad x_k(b)=x_{b,k},\quad k\in\{1,\ldots,n\}.
\end{equation}
To simplify, we consider $x$ as the vector
$x=(x_1,\ldots,x_n)$.
Consider a family of variations $x+\epsilon h$, where
$|\epsilon| \ll 1$ and $h=(h_1,\ldots,h_n)$.
The boundary conditions \eqref{boundMul} imply that
$h_k(a)=0=h_k(b)$, for $k\in\{1,\ldots,n\}$.
The following theorem can be easily proved.

\begin{theorem}
\label{ELeqMult}
Let $x$ be a minimizer of $J$ as in \eqref{functMul},
subject to the boundary conditions \eqref{boundMul}.
Then, for all $k\in\{1,\ldots,n\}$ and for all $t\in[a,b]$,
$x$ is a solution of the fractional Euler--Lagrange equation
\begin{multline*}
\frac{\partial L}{\partial x_k}[x](t)+{_tD^\alpha_b}\left( 
\frac{\partial L}{\partial v_k}[x](t) \right)+
{_tI^\beta_b}\left( \frac{\partial L}{\partial w_k}[x](t) \right)
+\int_t^b \frac{\partial L}{\partial z}[x](\t)d\t\cdot
\frac{\partial l}{\partial x_k}\{x\}(t)\\
+{_tD^\alpha_b}\left( \int_t^b \frac{\partial L}{\partial z}[x](\t)d\t
\cdot \frac{\partial l}{\partial v_k}\{x\}(t)  \right)
+{_tI^\beta_b}\left( \int_t^b \frac{\partial L}{\partial z}[x](\t)d\t
\cdot \frac{\partial l}{\partial w_k}\{x\}(t) \right)=0.
\end{multline*}
\end{theorem}


\section{Natural boundary conditions}
\label{sec:natbound}

In this section we consider a more general question.
Not only the unknown function $x$
is a variable in the problem,
but also the terminal time $T$ is an unknown.
For $T\in[a,b]$, consider the functional
\begin{equation}
\label{funct4}
J[x,T]=\int_a^T L[x](t)dt,
\end{equation}
where
$$
[x](t)=(t,x(t),{^C_aD_t^\alpha}x(t),{_aI_t^\beta}x(t),z(t)).
$$
The problem consists in finding a pair
$(x,T)\in \mathcal{F}([a,b];\mathbb{R})\times [a,b]$
for which the functional $J$ attains a minimum value.
First we recall a property that will be used later
in the proof of Theorem~\ref{TeoNatural}.

\begin{remark}
\label{remarkIntegral}
If $\phi$ is a continuous function, then (\textrm{cf.} \cite[p.~46]{Miller})
$$
\lim_{t\to T}{_tI_T^{1-\alpha}}\phi(t)=0
$$
for any $\alpha \in (0,1)$.
\end{remark}

\begin{theorem}
\label{TeoNatural}
Let $(x,T)$ be a minimizer of $J$ as in \eqref{funct4}.
Then, for all $t\in[a,T]$, $(x,T)$ is a solution
of the fractional equation
\begin{multline*}
\frac{\partial L}{\partial x}[x](t)+{_tD^\alpha_T}\left(
\frac{\partial L}{\partial v}[x](t) \right)+{_tI_T^\beta}
\left(\frac{\partial L}{\partial w}[x](t)\right)
+\int_t^T \frac{\partial L}{\partial z}[x](t)dt
\cdot \frac{\partial l}{\partial x}\{x\}(t)\\
+{_tD^\alpha_T}\left( \int_t^T \frac{\partial L}{\partial z}[x](\t)d\t
\cdot \frac{\partial l}{\partial v}\{x\}(t)  \right)
+{_tI^\beta_T}\left( \int_t^T \frac{\partial L}{\partial z}[x](\t)d\t
\cdot \frac{\partial l}{\partial w}\{x\}(t)  \right)=0
\end{multline*}
and satisfies the transversality conditions
$$
 \left[{_tI_T^{1-\alpha}}\left(  \frac{\partial L}{\partial v}[x](t)
+ \int_t^T  \frac{\partial L}{\partial z}[x](\t)\, d\t
\cdot \frac{\partial l}{\partial v}\{x\}(t) \right)\right]_{t=a}=0
$$
and
$$
L[x](T)=0.
$$
\end{theorem}

\begin{proof}
Let $h\in \mathcal{F}([a,b];\mathbb{R})$
be a variation, and let $\triangle T$
be a real number. Define the function
$$
j(\epsilon)=J[x+\epsilon h,T+\epsilon \triangle T]
$$
with $|\epsilon| \ll 1$.
Differentiating $j$ at $\epsilon=0$, and using the same procedure
as in Theorem~\ref{ELTEo}, we deduce that
\begin{align*}
0=&\triangle T \cdot L[x](T)+\int_a^T \left[\frac{\partial L}{\partial x}[x](t)
+{_tD^\alpha_T}\left( \frac{\partial L}{\partial v}[x](t) \right)
+{_tI_T^\beta}\left(\frac{\partial L}{\partial w}[x](t)\right)\right.\\
&\left.\hspace{35mm}+\int_t^T \frac{\partial L}{\partial z}[x](\t)d\t\cdot
\frac{\partial l}{\partial x}\{x\}(t)+{_tD^\alpha_T}\left( 
\int_t^T \frac{\partial L}{\partial z}[x](\t)d\t
\cdot \frac{\partial l}{\partial v}\{x\}(t)\right)\right.\\
&\left.\hspace{35mm}+{_tI^\beta_T}\left( \int_t^T \frac{\partial L}{\partial z}[x](\t)d\t
\cdot \frac{\partial l}{\partial w}\{x\}(t)  \right)\right]h(t)dt\\
&+ \left[ {_tI_T^{1-\alpha}}\left(\frac{\partial L}{\partial v}[x](t) \right)
h(t) \right]_a^T+\left[ {_tI^{1-\alpha}_T}
\left(\int_t^T\frac{\partial L}{\partial z}[x](\t)d\t
\cdot \frac{\partial l}{\partial v}\{x\}(t) \right) h(t)  \right]_a^T.
\end{align*}
The theorem follows from the arbitrariness of $h$ and $\triangle T$.
\end{proof}

\begin{remark}
If $T$ is fixed, say $T=b$, then $\triangle T=0$
and the transversality conditions reduce to
\begin{equation}
\label{NaturalBoundCond}
\left[{_tI_b^{1-\alpha}}\left( \frac{\partial L}{\partial v}[x](t)
+ \int_t^b \frac{\partial L}{\partial z}[x](\t)\, d\t \cdot
\frac{\partial l}{\partial v}\{x\}(t) \right)\right]_a=0.
\end{equation}
\end{remark}

\begin{example}
Consider the problem of minimizing the functional $J$ as in \eqref{funcExample},
but without given boundary conditions. Besides equation \eqref{ELeqExample},
extremals must also satisfy
\begin{equation}
\label{boundExample}
\left[{_tI_1^{1-\alpha}}\left( {^C_0D_t^\alpha}x(t)-1 \right)\right]_0=0.
\end{equation}
Again, $\overline x$ given by \eqref{eq:gs:ex} is a solution
of \eqref{ELeqExample} and \eqref{boundExample}.
\end{example}

As a particular case, the following result of
\cite{AGRA1} is deduced.

\begin{corollary}[\textrm{cf.} equations (9) and (12) of \cite{AGRA1}]
\label{Cor:AGRA1}
If $x$ is a minimizer of $J$ as in \eqref{funct7},
then $x$ is a solution of
$$
\frac{\partial L}{\partial x}[x](t)
+{_tD^\alpha_b}\left( \frac{\partial L}{\partial v}[x](t) \right)=0,
$$
and satisfies the transversality condition\index{Transversality conditions}
$$
\left[{_tI_b^{1-\alpha}}\left(\frac{\partial L}{\partial v}[x](t)\right)\right]_a=0.
$$
\end{corollary}

\begin{proof}
The Lagrangian $L$ in \eqref{funct7}
does not depend on ${_aI_t^\beta}x$ and $z$,
and the result follows from Theorem~\ref{TeoNatural}.
\end{proof}

\begin{remark}
\label{new:rem:6}
Observe that the condition
$$
\left[{_tI_b^{1-\alpha}}\left(\frac{\partial L}{\partial v}[x](t)\right)\right]_b=0
$$
is implicitly satisfied in Corollary~\ref{Cor:AGRA1}
(\textrm{cf.} Remark~\ref{remarkIntegral}).
\end{remark}


\section{Fractional isoperimetric problems}
\label{sec:IsoProb}

An isoperimetric\index{Isoperimetric problem} problem deals with the question
of optimizing a given functional
under the presence of an integral constraint.
This is a very old question, with its origins in the ancient Greece.
They where interested in determining the shape of a closed curve
with a fixed length and maximum area. This problem is known
as Dido's problem, and is an example of an isoperimetric
problem of the calculus of variations \cite{Brunt}.
For recent advancements on the subject we refer the reader to
\cite{Almeida2,Almeida3,iso:ts,MOMA09} and references therein.
In our case, within the fractional context, we state the isoperimetric
problem in the following way.
Determine the minimizers of a given functional
\begin{equation}
\label{funct2}
J[x]=\int_a^b L(t,x(t),{^C_aD_t^\alpha}x(t),{_aI_t^\beta}x(t),z(t))dt,
\end{equation}
subject to the boundary conditions
\begin{equation}
\label{bound2}
x(a)=x_a \quad \mbox{and} \quad x(b)=x_b,
\end{equation}
and the fractional integral constraint
\begin{equation}
\label{funct3}
\int_a^b G(t,x(t),{^C_aD_t^\alpha}x(t),{_aI_t^\beta}x(t),z(t))dt
=\gamma, \quad \gamma\in\mathbb{R},
\end{equation}
where $z$ is defined by
$$
z(t)=\int_a^t l(\t,x(\t),{^C_aD_\t^\alpha}x(\t),{_aI_\t^\beta}x(\t))d\t.
$$
As usual, we assume that all the functions $(t,x,v,w,z)\to L(t,x,v,w,z)$,
$(t,x,v,w)\to l(t,x,v,w)$, and $(t,x,v,w,z)\to G(t,x,v,w,z)$ are of class $C^1$.

\begin{theorem}
Let $x$ be a minimizer of $J$ as in \eqref{funct2},
under the boundary conditions \eqref{bound2} and isoperimetric constraint \eqref{funct3}.
Suppose that $x$ is not an extremal for $I$ in \eqref{funct3}. Then there exists
a constant $\lambda$ such that $x$ is a solution of the fractional equation
\begin{multline*}
\frac{\partial F}{\partial x}[x](t)
+{_tD^\alpha_b}\left( \frac{\partial F}{\partial v}[x](t) \right)
+{_tI^\beta_b}\left( \frac{\partial F}{\partial w}[x](t) \right)
+\int_t^b \frac{\partial F}{\partial z}[x](\t)d\t
\cdot \frac{\partial l}{\partial x}\{x\}(t)\\
+{_tD^\alpha_b}\left( \int_t^b \frac{\partial F}{\partial z}[x](\t)d\t
\cdot \frac{\partial l}{\partial v}\{x\}(t)\right)
+{_tI^\beta_b}\left( \int_t^b \frac{\partial F}{\partial z}[x](\t)d\t
\cdot \frac{\partial l}{\partial w}\{x\}(t)  \right)=0,
\end{multline*}
where $F=L-\lambda G$, for all $t\in[a,b]$.
\end{theorem}

\begin{proof}
Let $\epsilon_1,\epsilon_2\in \mathbb{R}$ be two real numbers such that
$|\epsilon_1|\ll1$ and $|\epsilon_2|\ll1$, with $\epsilon_1$ free
and $\epsilon_2$ to be determined later,
and let $h_1$ and $h_2$ be two functions satisfying
$$
h_1(a)=h_1(b)=h_2(a)=h_2(b)=0.
$$
Define functions $j$ and $i$ by
$$
j(\epsilon_1,\epsilon_2)=J[x+\epsilon_1h_1+\epsilon_2h_2]
$$
and
$$
i(\epsilon_1,\epsilon_2)=I(x+\epsilon_1h_1+\epsilon_2h_2)-\gamma.
$$
Doing analogous calculations as in the proof of Theorem~\ref{ELTEo}, one has
\begin{align*}
\left.\frac{\partial i}{\partial \epsilon_2} \right|_{(0,0)}
=& \int_a^b\left[\frac{\partial G}{\partial x}[x](t)
+\int_t^b \frac{\partial G}{\partial z}[x](\t)d\t
\cdot \frac{\partial l}{\partial x}\{x\}(t)\right.\\
&\hspace{1cm}\left.+{_tD^\alpha_b}\left( 
\frac{\partial G}{\partial v}[x](t) \right) + {_tD^\alpha_b}\left(
\int_t^b \frac{\partial G}{\partial z}[x](\t)d\t
\cdot \frac{\partial l}{\partial v}\{x\}(t)  \right)\right.\\
&\hspace{1cm}\left.+{_tI^\alpha_b}\left( \frac{\partial G}{\partial w}[x](t) \right)
+{_tI^\beta_b}\left( \int_t^b \frac{\partial G}{\partial z}[x](\t)d\t
\cdot \frac{\partial l}{\partial w}\{x\}(t)  \right)\right] h_2(t) \, dt.
\end{align*}
By hypothesis, $x$ is not an extremal for $I$
and therefore there must exist a function $h_2$ for which
$$
\left.\frac{\partial i}{\partial \epsilon_2} \right|_{(0,0)}\not=0.
$$
Since $i(0,0)=0$, by the implicit function theorem there exists
a function $\epsilon_2(\cdot)$, defined
in some neighborhood of zero, such that
\begin{equation}
\label{iso:const:pr}
i(\epsilon_1,\epsilon_2(\epsilon_1))=0.
\end{equation}
On the other hand, $j$ attains a minimum value at $(0,0)$ when subject
to the constraint \eqref{iso:const:pr}.
Because $\nabla i(0,0)\neq (0,0)$,
by the Lagrange multiplier\index{Lagrange multipliers} rule \cite[p.~77]{Brunt}
there exists a constant $\lambda$ such that
$$
\nabla(j(0,0)-\lambda i(0,0))=(0,0).
$$
So
$$
\left.\frac{\partial j}{\partial \epsilon_1} \right|_{(0,0)}
-\lambda \left.\frac{\partial i}{\partial \epsilon_1} \right|_{(0,0)}=0.
$$
Differentiating $j$ and $i$ at zero, and
doing the same calculations as before, we get the desired result.
\end{proof}

Using the abnormal Lagrange multiplier rule \cite[p.~82]{Brunt},
the previous result can be generalized to include
the case when the minimizer is an extremal of $I$.

\begin{theorem}
Let $x$ be a minimizer of $J$ as in \eqref{funct2},
subject to the constraints \eqref{bound2} and \eqref{funct3}.
Then there exist two constants $\lambda_0$ and $\lambda$,
not both zero, such that $x$ is a solution of equation
\begin{multline*}
\frac{\partial K}{\partial x}[x](t)+{_tD^\alpha_b}\left( 
\frac{\partial K}{\partial v}[x](t) \right)
+{_tI^\beta_b} \left( \frac{\partial K}{\partial w}[x](t) \right)
+\int_t^b \frac{\partial K}{\partial z}[x](t)dt
\cdot \frac{\partial l}{\partial x}\{x\}(t)\\
+{_tD^\alpha_b}\left( \int_t^b \frac{\partial K}{\partial z}[x](\t)d\t
\cdot \frac{\partial l}{\partial v}\{x\}(t)\right)
+{_tI^\beta_b}\left( \int_t^b \frac{\partial K}{\partial z}[x](\t)d\t
\cdot \frac{\partial l}{\partial w}\{x\}(t)  \right)=0
\end{multline*}
for all $t\in[a,b]$, where $K=\lambda_0 L-\lambda G$.
\end{theorem}

\begin{corollary}[\textrm{cf.} Theorem~3.4 of \cite{Almeida}]
Let $x$ be a minimizer of
$$
J[x]=\int_a^b L(t,x(t),{^C_aD_t^\alpha}x(t))dt,
$$
subject to the boundary conditions
$$
x(a)=x_a \quad \mbox{and} \quad x(b)=x_b,
$$
and the isoperimetric constraint
$$
\int_a^b G(t,x(t),{^C_aD_t^\alpha}x(t))dt
=\gamma, \quad \gamma\in\mathbb{R}.
$$
Then, there exist two constants $\lambda_0$ and $\lambda$,
not both zero, such that $x$ is a solution of equation
$$
\frac{\partial K}{\partial x}\left(t,x(t),{^C_aD_t^\alpha}x(t)\right)
+{_tD^\alpha_b}\left( \frac{\partial K}{\partial v}
\left(t,x(t),{^C_aD_t^\alpha}x(t)\right) \right) =0
$$
for all $t\in[a,b]$, where $K=\lambda_0 L-\lambda G$. Moreover, if $x$
is not an extremal for $I$, then we may take $\lambda_0=1$.
\end{corollary}


\section{Holonomic constraints}
\label{sec:Holonomic}

In this section we consider the following problem. Minimize the functional
\begin{equation}
\label{funct8}
J[x_1,x_2]=\int_a^b L(t,x_1(t),x_2(t),{^C_aD_t^\alpha}x_1(t),
{^C_aD_t^\alpha}x_2(t),{_aI_t^\beta}x_1(t),{_aI_t^\beta}x_2(t),z(t))dt,
\end{equation}
where $z$ is defined by
$$
z(t)=\int_a^t l(t,x_1(\t),x_2(\t),{^C_aD_\t^\alpha}x_1(\t),
{^C_aD_\t^\alpha}x_2(\t),{_aI_\t^\beta}x_1(\t),{_aI_\t^\beta}x_2(\t))d\t,
$$
when restricted to the boundary conditions\index{Boundary conditions}
\begin{equation}
\label{boundconst8}
(x_1(a),x_2(a))=(x_1^a,x_2^a)
\mbox{ and } (x_1(b),x_2(b))=(x_1^b,x_2^b),
\quad x_1^a,x_2^a,x_1^b,x_2^b\in\mathbb{R},
\end{equation}
and the holonomic constraint
\begin{equation}
\label{subsconst}
g(t,x_1(t),x_2(t))=0.
\end{equation}
As usual, here
$$
(t,x_1,x_2,v_1,v_2,w_1,w_2,z)\to L(t,x_1,x_2,v_1,v_2,w_1,w_2,z),
$$
$$
(t,x_1,x_2,v_1,v_2,w_1,w_2)\to l(t,x_1,x_2,v_1,v_2,w_1,w_2)
$$
and
$$
(t,x_1,x_2)\to g(t,x_1,x_2)
$$
are all smooth. In what follows we make use
of the operator $[\cdot,\cdot]$ given by
$$
[x_1,x_2](t)
= (t,x_1(t),x_2(t),{^C_aD_t^\alpha}x_1(t),{^C_aD_t^\alpha}x_2(t),
{_aI_t^\beta}x_1(t),{_aI_t^\beta}x_2(t),z(t))\, ,
$$
we denote
$(t,x_1(t),x_2(t))$ by $(t,\mathbf{x}(t))$,
and the Euler--Lagrange equation obtained in \eqref{ELeq}
with respect to $x_i$ by $(ELE_i)$, $i=1,2$.

\begin{remark}
For simplicity, we are considering functionals depending
only on two functions $x_1$ and $x_2$. Theorem~\ref{thm:hol}
is, however, easily generalized for $n$ variables $x_1,\ldots,x_n$.
\end{remark}

\begin{theorem}
\label{thm:hol}
Let the pair $(x_1,x_2)$ be a minimizer of $J$ as in \eqref{funct8},
subject to the constraints \eqref{boundconst8}--\eqref{subsconst}.
If $\frac{\partial g}{\partial x_2}\not=0$,
then there exists a continuous function $\lambda:[a,b]\to\mathbb{R}$
such that $(x_1,x_2)$ is a solution of
\begin{multline}\label{ELequation2}
\frac{\partial F}{\partial x_i}[x_1,x_2](t)
+\int_t^b \frac{\partial F}{\partial z}[x_1,x_2](\t)d\t\cdot
\frac{\partial l}{\partial x_i}\{x_1,x_2\}(t)\\
+{_tD^\alpha_b}\left( \frac{\partial F}{\partial v_i}[x_1,x_2](t) \right)
+{_tD^\alpha_b}\left( \int_t^b \frac{\partial F}{\partial z}[x_1,x_2](\t)d\t
\cdot \frac{\partial l}{\partial v_i}\{x_1,x_2\}(t)  \right)\\
~~+{_tI_b^\beta}\left(\frac{\partial F}{\partial w_i}[x_1,x_2](t)\right)
+{_tI^\beta_b}\left( \int_t^b \frac{\partial F}{\partial z}[x_1,x_2](\t)d\t
\cdot \frac{\partial l}{\partial w_i}\{x_1,x_2\}(t)\right)=0
\end{multline}
for all $t\in[a,b]$ and $i=1,2$,
where $F[x_1,x_2](t)=L[x_1,x_2](t)-\lambda (t) g(t,\mathbf{x}(t))$.
\end{theorem}

\begin{proof}
Consider a variation\index{Variation} of the optimal solution of type
$$
(\overline x_1,\overline x_2)= (x_1+\epsilon h_1,x_2+\epsilon h_2),
$$
where $h_1,h_2$ are two functions defined on $[a,b]$ satisfying
$$
h_1(a)=h_1(b)=h_2(a)=h_2(b)=0,
$$
and $\epsilon$ is a sufficiently small real parameter. Since
$\frac{\partial g}{\partial x_2}(t,\overline{x}_1(t),\overline{x}_2(t))\not=0$
for all $t\in[a,b]$, we can solve equation
$g(t,\overline x_1(t),\overline x_2(t))=0$
with respect to $h_2$, $h_2=h_2(\epsilon,h_1)$. Differentiating
$J(\overline x_1,\overline x_2)$ at $\epsilon=0$, and proceeding
similarly as done in the proof of Theorem~\ref{ELTEo}, we deduce that
\begin{equation}
\label{subsequation}
\int_a^b (ELE_1)h_1(t)+(ELE_2)h_2(t)\, dt=0.
\end{equation}
Besides, since $g(t,\overline x_1(t),\overline x_2(t))=0$,
differentiating at $\epsilon=0$ we get
\begin{equation}
\label{defh2}
h_2(t)=-\frac{\frac{\partial g}{\partial
x_1}(t,\mathbf{x}(t))}{\frac{\partial g}{\partial
x_2}(t,\mathbf{x}(t))}h_1(t).
\end{equation}
Define the function $\lambda$ on $[a,b]$ as
\begin{equation}
\label{subslambda}
\lambda(t)=\frac{(ELE_2)}{\frac{\partial g}{\partial x_2}(t,\mathbf{x}(t))}.
\end{equation}
Combining \eqref{defh2} and \eqref{subslambda},
equation \eqref{subsequation} can be written as
$$
\int_a^b \left[(ELE_1)-\lambda(t) \frac{\partial g}{\partial
x_1}(t,\mathbf{x}(t))\right] h_1(t)\, dt=0.
$$
By the arbitrariness of $h_1$, if follows that
$$
(ELE_1)-\lambda(t) \frac{\partial g}{\partial x_1}(t,\mathbf{x}(t))=0.
$$
Define $F$ as
$$
F[x_1,x_2](t)=L[x_1,x_2](t)-\lambda (t) g(t,\mathbf{x}(t)).
$$
Then, equations \eqref{ELequation2} follow.
\end{proof}


\section{Higher order Caputo derivatives}
\label{sec:Higher}

In this section we consider fractional variational 
problems in presence of higher order Caputo derivatives.
We will restrict ourselves to the case where the orders
are non-integer, since the integer case is already well studied
in the literature (for a modern account see \cite{MyID:194,ferreira,natorres}).

Let $n\in\mathbb{N}$, $\beta > 0$,
and $\alpha_k\in\mathbb{R}$ be such that
$\alpha_k\in(k-1,k)$ for $k\in\{1,\ldots,n\}$.
Admissible functions $x$ belong to $AC^n([a,b];\mathbb{R})$
and are such that ${^C_aD_t^{\alpha_k}}x$, $k = 1, \ldots, n$,
and ${_aI_t^\beta}x$ exist and are continuous on $[a,b]$.
We denote such class of functions by $\mathcal{F}^n([a,b];\mathbb{R})$.
For $\alpha=(\alpha_1,\ldots,\alpha_n)$, define the vector
\begin{equation}
\label{eq:y:ho:l}
{_a^C D _t^\alpha}x(t)
=({_a^C D _t^{\alpha_1}}x(t),\ldots,{_a^C D _t^{\alpha_n}}x(t)).
\end{equation}
The optimization problem is the following: to minimize or maximize
the functional
\begin{equation}
\label{funct5}
J[x]=\int_a^b L(t,x(t),{^C_aD_t^\alpha}x(t),{_aI_t^\beta}x(t),z(t))dt,
\end{equation}
$x \in \mathcal{F}^n([a,b];\mathbb{R})$,
subject to the boundary conditions
\begin{equation}
\label{bound5}
x^{(k)}(a)=x_{a,k} \quad \mbox{and}
\quad x^{(k)}(b)=x_{b,k}, \quad k\in\{0,\ldots,n-1\},
\end{equation}
where $z: [a,b] \to \mathbb{R}$ is defined by
$$
z(t)=\int_a^t l(\t,x(\t),{^C_aD_\t^\alpha}x(\t),{_aI_\t^\beta}x(\t))d\t.
$$

\begin{theorem}
\label{thm:16}
If $x \in \mathcal{F}^n([a,b];\mathbb{R})$
is a minimizer of $J$ as in \eqref{funct5},
subject to the boundary conditions \eqref{bound5},
then $x$ is a solution of the fractional equation
\begin{multline*}
\frac{\partial L}{\partial x}[x](t)
+\sum_{k=1}^n{_tD^{\alpha_k}_b}\left( \frac{\partial L}{\partial v_k}[x](t) \right)
+{_tI^\beta_b}\left( \frac{\partial L}{\partial w}[x](t) \right)
+\int_t^b \frac{\partial L}{\partial z}[x](t)dt
\cdot \frac{\partial l}{\partial x}\{x\}(t)\\
+\sum_{k=1}^n{_tD^{\alpha_k}_b}\left( \int_t^b \frac{\partial L}{\partial z}[x](\t)d\t
\cdot \frac{\partial l}{\partial v_k}\{x\}(t)\right)
+{_tI^\beta_b}\left( \int_t^b \frac{\partial L}{\partial z}[x](\t)d\t
\cdot \frac{\partial l}{\partial w}\{x\}(t)  \right)=0
\end{multline*}
for all $t\in[a,b]$, where
$[x](t) = \left(t,x(t),{^C_aD_t^\alpha}x(t),{_aI_t^\beta}x(t),z(t)\right)$
with ${^C_aD_t^\alpha}x(t)$ as in \eqref{eq:y:ho:l}.
\end{theorem}

\begin{proof}
Let $h \in \mathcal{F}^n([a,b];\mathbb{R})$
be such that $h^{(k)}(a)=h^{(k)}(b)=0$,
for $k\in\{0,\ldots,n-1\}$. Define the new function
$j$ as $j(\epsilon)=J(x+\epsilon h)$. Then
\begin{multline}
\label{eq1}
0=\int_a^b \left[ \frac{\partial L}{\partial x}[x](t)h(t)
+ \sum_{k=1}^n \frac{\partial L}{\partial v_k}[x](t){^C_aD^{\alpha_k}_t}h(t)
+\frac{\partial L}{\partial w}[x](t){_aI^\beta_t}h(t)\right.\\
\left.+\frac{\partial L}{\partial z}[x](t)\int_a^t
\left( \frac{\partial l}{\partial x}\{x\}(\t)h(\t)
+\sum_{k=1}^n\frac{\partial l}{\partial v_k}\{x\}(\t){^C_aD^{\alpha_k}_\t}h(\t)
+\frac{\partial l}{\partial w}\{x\}(\t){_aI^\beta_\t}h(\t) \right)d\t\right]dt.
\end{multline}
Integrating by parts, we get that
\begin{multline*}
\int_a^b \frac{\partial L}{\partial v_k}[x](t){^C_aD^{\alpha_k}_t}h(t) dt
=\int_a^b {_t D_b^{\alpha_k}}\left(\frac{\partial L}{\partial v_k}[x](t)\right)h(t)dt\\
+\sum_{m=0}^{k-1}\left[{_tD_b^{\alpha_k+m-k}}\left(\frac{\partial L}{\partial v_k}[x](t)\right)
h^{(k-1-m)}(t)\right]_a^b =\int_a^b {_t D_b^{\alpha_k}}\left(
\frac{\partial L}{\partial v_k}[x](t) \right)h(t)dt
\end{multline*}
for all $k\in\{1,\ldots,n\}$. Moreover, one has
$$
\int_a^b \frac{\partial L}{\partial w}[x](t){_aI^\beta_t}h(t) dt
=\int_a^b {_t I_b^\beta} \left(\frac{\partial L}{\partial w}[x](t) \right)h(t)dt,
$$
$$
\int_a^b \frac{\partial L}{\partial z}[x](t)\left(\int_a^t
\frac{\partial l}{\partial x}\{x\}(\t)h(\t) d\t \right)  dt
=\int_a^b \left(\int_t^b\frac{\partial L}{\partial z}[x](\t)d\t
\right) \frac{\partial l}{\partial x}\{x\}(t)h(t) \,  dt,
$$
\begin{align*}
\int_a^b & \frac{\partial L}{\partial z}[x]\left(\int_a^t
\frac{\partial l}{\partial v_k}\{x\}{^C_aD^{\alpha_k}_\t}h d\t \right)dt
=\int_a^b \left(\int_t^b\frac{\partial L}{\partial z}[x](\t)d\t \right)
\frac{\partial l}{\partial v_k}\{x\}(t){^C_aD^{\alpha_k}_t}h \,  dt\\
&\hspace{5mm}= \int_a^b {_tD^{\alpha_k}_b}
\left(\int_t^b\frac{\partial L}{\partial z}[x](\t)d\t
\frac{\partial l}{\partial v_k}\{x\}(t) \right)h(t) \, dt\\
&\hspace{2cm} +\sum_{m=0}^{k-1}\left[{_tD_b^{\alpha_k+m-k}}\left(
\int_t^b\frac{\partial L}{\partial z}[x](\t)d\t
\frac{\partial l}{\partial v_k}\{x\}(t)\right)
h^{(k-1-m)}(t)\right]_a^b\\
&\hspace{5cm}= \int_a^b {_tD^{\alpha_k}_b}\left(
\int_t^b\frac{\partial L}{\partial z}[x](\t)d\t
\frac{\partial l}{\partial v_k}\{x\}(t) \right)h(t) \, dt,
\end{align*}
and
$$
\int_a^b \frac{\partial L}{\partial z}[x](t)\left(\int_a^t
\frac{\partial l}{\partial w}\{x\}(t){_aI^\beta_\t}h(\t) d\t
\right)dt=\int_a^b {_tI^\beta_b}
\left(\int_t^b\frac{\partial L}{\partial z}[x](\t)d\t
\frac{\partial l}{\partial w}\{x\}(t) \right)h(t) \, dt.
$$
Replacing these last relations into equation \eqref{eq1},
and applying the fundamental lemma of the calculus of variations,
we obtain the intended necessary condition.
\end{proof}

We now consider the higher-order problem without
the presence of boundary conditions \eqref{bound5}.

\begin{theorem}
If $x \in \mathcal{F}^n([a,b];\mathbb{R})$
is a minimizer of $J$ as in \eqref{funct5},
then $x$ is a solution of the fractional equation
\begin{multline*}
\frac{\partial L}{\partial x}[x](t)
+\sum_{k=1}^n{_tD^{\alpha_k}_b}\left(
\frac{\partial L}{\partial v_k}[x](t) \right)+{_tI^\beta_b}
\left( \frac{\partial L}{\partial w}[x](t) \right)
+\int_t^b \frac{\partial L}{\partial z}[x](\t)d\t
\cdot \frac{\partial l}{\partial x}\{x\}(t)\\
+\sum_{k=1}^n{_tD^{\alpha_k}_b}\left( \int_t^b
\frac{\partial L}{\partial z}[x](\t)d\t
\cdot \frac{\partial l}{\partial v_k}\{x\}(t)  \right)
+{_tI^\beta_b}\left( \int_t^b \frac{\partial L}{\partial z}[x](\t)d\t
\cdot \frac{\partial l}{\partial w}\{x\}(t)  \right)=0
\end{multline*}
for all $t\in[a,b]$, and satisfies the natural boundary conditions
\begin{equation}
\label{eq:nbc:ho}
\sum_{m=k}^n \left[ {_tD_b^{\alpha_m-k}}\left(
\frac{\partial L}{\partial v_k}[x](t)
+ \int_t^b\frac{\partial L}{\partial z}[x](t)dt
\frac{\partial l}{\partial v_k}\{x\}(t)\right) \right]_a^b=0,
\quad \mbox{for all} \quad k\in\{1,\ldots,n\}.
\end{equation}
\end{theorem}

\begin{proof}
The proof follows the same pattern as the proof of Theorem~\ref{thm:16}.
Since admissible functions $x$ are not required to satisfy given boundary conditions,
the variation functions $h$ may take any value at the boundaries as well,
and thus the condition
\begin{equation}
\label{HigBoundCons}
h^{(k)}(a)=h^{(k)}(b)=0, \quad \mbox{for } k\in\{0,\ldots,n-1\},
\end{equation}
is no longer imposed \textit{a priori}. If we consider the first variation
of $J$ for variations $h$ satisfying condition \eqref{HigBoundCons},
we obtain the Euler--Lagrange equation. Replacing it on the expression
of the first variation, we conclude that
$$
\sum_{k=1}^n \sum_{m=0}^{k-1}\left[{_tD_b^{\alpha_k
+m-k}}\left(\frac{\partial L}{\partial v_k}[x](t)
+\int_t^b\frac{\partial L}{\partial z}[x](\t)d\t
\frac{\partial l}{\partial v_k}\{x\}(t)  \right)h^{(k-1-m)}(t)\right]_a^b=0.
$$
To obtain the transversality condition with respect to $k$,
for $k\in\{1,\ldots,n\}$, we consider variations satisfying the condition
$$
h^{(k-1)}(a)\not=0\not=h^{(k-1)}(b) \quad \mbox{and }
h^{(j-1)}(a)=0=h^{(j-1)}(b), \quad \mbox{for all }
j\in\{0,\ldots,n\}\setminus\{k\}.
$$
\end{proof}

\begin{remark}
Some of the terms that appear in the natural boundary
conditions \eqref{eq:nbc:ho} are equal to zero
(\textrm{cf.} Remark~\ref{remarkIntegral} and
Remark~\ref{new:rem:6}).
\end{remark}


\section{Fractional optimal control problems}
\label{sec:FracOpt}

We now prove a necessary optimality condition
for a fractional Lagrange problem, when the Lagrangian
depends again on an indefinite integral.
Consider the cost functional defined by
\begin{equation}
\label{funct6}
J[x,u]=\int_a^b L\left(t,x(t),u(t),{_aI_t^\beta}x(t),z(t)\right)dt,
\end{equation}
to be minimized or maximized
subject to the fractional dynamical system
\begin{equation}
\label{dynamic6}
{^C_aD_t^\alpha}x(t)=f(t,x(t),u(t),{_aI_t^\beta}x(t),z(t)),
\end{equation}
and the boundary conditions
\begin{equation}
\label{bound6}
x(a)=x_a \quad \mbox{and} \quad x(b)=x_b,
\end{equation}
where
$$
z(t)=\int_a^t l\left(\t,x(\t),
{^C_aD_\t^\alpha}x(\t),{_aI_\t^\beta}x(\t)\right)d\t.
$$
We assume the functions $(t,x,v,w,z)\to f(t,x,v,w,z)$,
$(t,x,v,w,z)\to L(t,x,v,w,z)$, and $(t,x,v,w)\to l(t,x,v,w)$,
to be of class $C^1$ with respect to all their arguments.

\begin{remark}
If $f(t,x(t),u(t),{_aI_t^\beta}x(t),z(t))=u(t)$,
the Lagrange problem \eqref{funct6}--\eqref{bound6}
reduces to the fractional variational problem
\eqref{funct}--\eqref{bound} studied
in Section~\ref{sec:Fundprob}.
\end{remark}

An optimal solution is a pair of functions $(x,u)$ that minimizes $J$
as in \eqref{funct6}, subject to the fractional
dynamic equation \eqref{dynamic6} and the boundary conditions \eqref{bound6}.

\begin{theorem}
If $(x,u)$ is an optimal solution to the fractional
Lagrange problem \eqref{funct6}--\eqref{bound6}, then there exists
a function $p$ for which the triplet $(x,u,p)$ satisfies the Hamiltonian system
$$
\left\{
\begin{array}{ll}
{^C_aD_t^\alpha}x&=\frac{\partial H}{\partial p}\lceil x,u,p \rceil,\\
{_tD_b^\alpha}p&=\frac{\partial H}{\partial x}\lceil x,u,p \rceil+{_tI_b^\beta}
\left(\frac{\partial H}{\partial w}
\lceil x,u,p \rceil\right)
+\int_t^b \frac{\partial H}{\partial z}\lceil x,u,p \rceil(\t)d\t\cdot
\frac{\partial l}{\partial x}\{x\}\\
&\quad +{_tD^{\alpha}_b}\left(
\int_t^b \frac{\partial H}{\partial z}\lceil x,u,p \rceil(\t)d\t\cdot
\frac{\partial l}{\partial v}\{x\}\right)+{_tI^{\beta}_b}\left(
\int_t^b \frac{\partial H}{\partial z}\lceil x,u,p \rceil(\t)d\t\cdot
\frac{\partial l}{\partial w}\{x\}\right)
\end{array}
\right.$$
and the stationary condition
$$
\frac{\partial H}{\partial u}\lceil x,u,p \rceil(t)=0,
$$
where the Hamiltonian $H$ is defined by
$$
H\lceil x,u,p \rceil(t)=L(t,x(t),u(t),{_aI_t^\beta}x(t),z(t))
+p(t)f(t,x(t),u(t),{_aI_t^\beta}x(t),z(t))
$$
and
$$
\lceil x,u,p \rceil(t)= (t,x(t),u(t),{_aI_t^\beta}x(t),z(t),p(t))\, ,
\quad \{x\}(t)=(t,x(t),{^C_aD_t^\alpha}x(t),{_aI_t^\beta}x(t)).
$$
\end{theorem}

\begin{proof}
The result follows applying Theorem~\ref{ELeqMult} to
$$
{J^*}[x,u,p]=\int_a^b  H\lceil x,u,p \rceil(t)
-p(t) {^C_aD_t^\alpha}x(t) dt
$$
with respect to $x$, $u$ and $p$.
\end{proof}

In the particular case when $L$ does not depend
on ${_aI_t^\beta}x$ and $z$,
we obtain \cite[Theorem~3.5]{Gastao0}.

\begin{corollary}[Theorem~3.5 of \cite{Gastao0}]
Let $(x(t),u(t))$ be a solution of
$$
J[x,u]=\int_a^b L(t,x(t),u(t))dt \longrightarrow \min
$$
subject to the fractional control system
${^C_aD_t^\alpha}x(t)=f(t,x(t),u(t))$
and the boundary conditions $x(a)=x_a$ and $x(b)=x_b$.
Define the Hamiltonian by
$H(t,x,u,p)=L(t,x,u) + p f(t,x,u)$.
Then there exists a function $p$ for which
the triplet $(x,u,p)$ fulfill the Hamiltonian system
$$
\begin{cases}
\displaystyle{^C_aD_t^\alpha}x(t)
=\frac{\partial H}{\partial p}(t,x(t),u(t),p(t)),\\[8pt]
\displaystyle{_tD_b^\alpha}p(t)
=\frac{\partial H}{\partial x}(t,x(t),u(t),p(t)),
\end{cases}
$$
and the stationary condition
$\frac{\partial H}{\partial u}(t,x(t),u(t),p(t))=0$.
\end{corollary}


\section{Sufficient conditions of optimality}
\label{sec:SufConditions}

Recall Definition \ref{ConvDef}, the notions
of convexity and concavity for $C^1$ functions of several variables.
\begin{theorem}
Consider the functional $J$ as in \eqref{funct},
and let $x \in \mathcal{F}([a,b];\mathbb{R})$
be a solution of the fractional
Euler--Lagrange equation \eqref{ELeq}
satisfying the boundary conditions \eqref{bound}.
Assume that $L$ is convex in $(x,v,w,z)$.
If one of the two following conditions is satisfied,
\begin{enumerate}
\item $l$ is convex in $(x,v,w)$ and
$\frac{\partial L}{\partial z}[x](t) \geq 0$ for all $t \in [a,b]$;
\item $l$ is concave in $(x,v,w)$ and
$\frac{\partial L}{\partial z}[x](t) \leq 0$ for all $t \in [a,b]$;
\end{enumerate}
then $x$ is a (global) minimizer of problem \eqref{funct}--\eqref{bound}.
\end{theorem}

\begin{proof}
Consider $h$ of class $\mathcal{F}([a,b];\mathbb{R})$
such that $h(a)=h(b)=0$. Then,
\begin{align*}
&J[x+h]- J[x] = \int_a^b L\Biggl(t,x(t)+h(t),{^C_aD_t^\alpha}x(t)
+{^C_aD_t^\alpha}h(t),{_aI_t^\beta}x(t)+{_aI_t^\beta}h(t),\\
&\hspace{4cm}\int_a^t l(\t,x(\t)+h(\t),{^C_aD_\t^\alpha}x(\t)
+{^C_aD_t^\alpha}h(t),{_aI_t^\beta}x(t)+{_aI_t^\beta}h(t))dt \Biggr)dt\\
&\hspace{30mm} -\int_a^b L\left(t,x(t),{^C_aD_t^\alpha}x(t),
{_aI_t^\beta}x(t),\int_a^t l(\t,x(\t),
{^C_aD_\t^\alpha}x(\t),{_aI_\t^\beta}x(\t))d\t\right)dt\\
&\hspace{1cm}\geq \int_a^b \left[ \frac{\partial L}{\partial x}[x](t)h(t)
+\frac{\partial L}{\partial v}[x](t){^C_aD^\alpha_t}h(t)
+\frac{\partial L}{\partial w}[x](t){_aI^\beta_t}h(t)\right.\\
&\hspace{15mm} \left. +\frac{\partial L}{\partial z}[x](t)
\int_a^t\left( \frac{\partial l}{\partial x}\{x\}(\t)h(\t)
+\frac{\partial l}{\partial v}\{x\}(\t){^C_aD^\alpha_\t}h(\t)
+\frac{\partial l}{\partial w}\{x\}(\t){_aI^\beta_\t}h(\t)\right)d\t\right]dt\\
&\hspace{10mm}= \int_a^b \left[ \frac{\partial L}{\partial x}[x](t)
+{_tD^\alpha_b}\left( \frac{\partial L}{\partial v}[x](t) \right)
+{_tI_b^\beta}\left(\frac{\partial L}{\partial w}[x](t)\right)
+\int_t^b \frac{\partial L}{\partial z}[x](\t)d\t
\cdot \frac{\partial l}{\partial x}\{x\}(t)\right.\\
&\hspace{15mm}\left. +{_tD^\alpha_b}\left(
\int_t^b \frac{\partial L}{\partial z}[x](\t)d\t
\cdot \frac{\partial l}{\partial v}\{x\}(t)  \right)
+{_tI^\beta_b}\left( \int_t^b \frac{\partial L}{\partial z}[x](\t)d\t
\cdot \frac{\partial l}{\partial w}\{x\}(t)  \right)\right] h(t)dt\\
&\hspace{10mm}= 0.
\end{align*}
\end{proof}

One can easily include the case when the
boundary conditions \eqref{bound} are not given.

\begin{theorem}
Consider functional $J$ as in \eqref{funct}
and let $x \in \mathcal{F}([a,b];\mathbb{R})$
be a solution of the fractional Euler--Lagrange
equation \eqref{ELeq} and the fractional natural boundary condition
\eqref{NaturalBoundCond}. Assume
that $L$ is convex in $(x,v,w,z)$. If one
of the two next conditions is satisfied,
\begin{enumerate}
\item $l$ is convex in $(x,v,w)$ and
$\frac{\partial L}{\partial z}[x](t) \geq 0$ for all $t \in [a,b]$;
\item $l$ is concave in $(x,v,w)$ and
$\frac{\partial L}{\partial z}[x](t) \leq 0$ for all $t \in [a,b]$;
\end{enumerate}
then $x$ is a (global) minimizer of \eqref{funct}.
\end{theorem}


\section{Examples}

We illustrate with Examples~\ref{example:bra} and \ref{example1}
how the approximation \eqref{expanMom} provides an accurate
and efficient numerical method to solve fractional variational problems
in the presence of special constraints.

\begin{example}
\label{ex:new:sec6}
We obtain an approximated solution to the problem
considered in Example~\ref{example:bra}.
Since $x(0)=0$, the Caputo derivative coincides
with the Riemann--Liouville derivative
and we can approximate the fractional
problem using \eqref{expanMom}.
We reformulate the problem using the Hamiltonian formalism by
letting $^C_0D_t^{\alpha}x(t)=u(t)$. Then,
\begin{equation}
\label{eq:h:1}
A(\alpha,N)t^{-\alpha}x(t)+B(\alpha,N)t^{1-\alpha}\dot{x}(t)
-\sum_{k=2}^{N}C(k,\alpha)t^{1-k-\alpha}v_k(t)=u(t).
\end{equation}
We also include the variable $z(t)$ with
$$
\dot{z}(t)=\left( x(t)-t^{\alpha+1}\right)^2.
$$
In summary, one has the following Lagrange problem:
\begin{equation}
\label{AppExample2}
\begin{gathered}
\tilde{J}[x] = \int_0^1 [(u(t)
-\Gamma(\alpha+2)t)^2+z(t)]dt \longrightarrow \min \\
\begin{cases}
\dot{x}(t) = -AB^{-1}t^{-1}x(t)
+\sum_{k=2}^{N}B^{-1}C_kt^{-k}v_k(t)+B^{-1}t^{\alpha-1}u(t)\\
\dot{v}_k(t) = (1-k)t^{k-2}x(t),\qquad k=1,2,\ldots\\
\dot{z}(t) = \left( x(t)-t^{\alpha+1}\right)^2,
\end{cases}
\end{gathered}
\end{equation}
subject to the boundary conditions $x(0)=0$, $z(0)=0$ and $v_k(0)=0$, $k=1,2,\ldots$
Setting $N=2$, the Hamiltonian is given by
\begin{multline*}
H=-[(u(t)-\Gamma(\alpha+2)t)^2+z(t)]+p_1(t)\left(-AB^{-1}t^{-1}x(t)\right.\\
\left.+B^{-1}C_2t^{-2}v_2(t)+B^{-1}t^{\alpha-1}u(t)\right)
-p_2(t)x(t)+p_3(t)\left( x(t)-t^{\alpha+1}\right)^2.
\end{multline*}
Using the classical necessary optimality condition for problem \eqref{AppExample2},
we end up with the following two point boundary value problem:
\begin{equation}
\label{sys2}
\left\{
\begin{array}{ll}
\dot{x}(t)   & =-AB^{-1}t^{-1}x(t)+ B^{-1}C_2t^{-2}v_2(t)
+\frac{1}{2}B^{-2}t^{2\alpha-2}p_1(t)+\Gamma(\alpha+2)B^{-1}t^{\alpha}\\
\dot{v}_2(t) & =-x(t)\\
\dot{z}(t)   & =(x(t)-t^{\alpha+1})^2\\
\dot{p}_1(t) & =AB^{-1}t^{-1}p_1(t)+p_2(t)-2p_3(t)(x(t)-t^{\alpha+1})\\
\dot{p}_2(t) & =-B^{-1}C_2t^{-2}p_1(t)\\
\dot{p}_3(t) & = 1,
\end{array}
\right.
\end{equation}
subject to the boundary conditions
\begin{equation}
\label{sysB2}
\begin{cases}
x(0)=0 \\
v_2(0)=0\\
z(0)=0
\end{cases}
,\qquad
\begin{cases}
x(1)=1\\
p_2(1)=0\\
p_3(1)=0.
\end{cases}
\end{equation}
We solved system \eqref{sys2} subject to \eqref{sysB2}
using the MATLAB$^\circledR$ built-in function \textsf{bvp4c}.
The resulting graph for $x(t)$, together with the corresponding
value of $J$, is given in Figure~\ref{Fig2}.
\begin{figure}[!ht]
\begin{center}
\includegraphics[scale=.8]{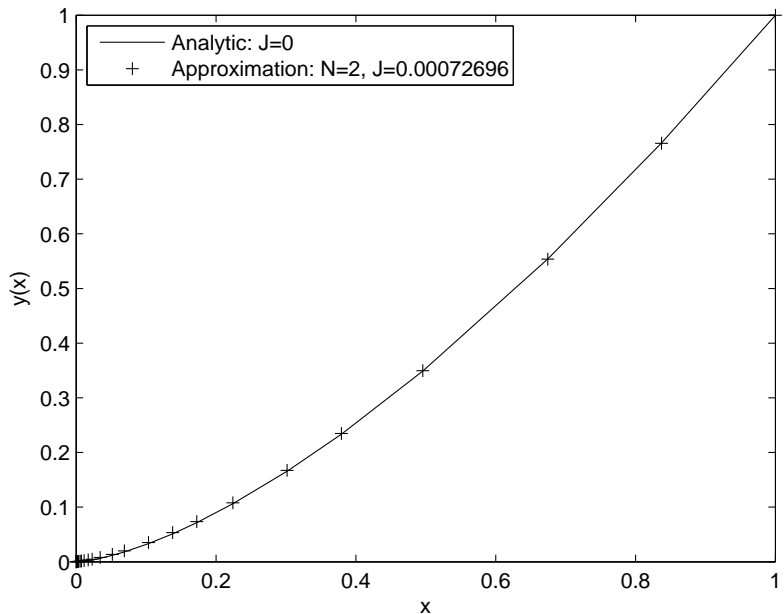}
\caption{Analytic {\it versus} numerical solution
to problem of Example~\ref{example:bra}.}\label{Fig2}
\end{center}
\end{figure}
\end{example}

This numerical method works well, even in the case
the minimizer is not a Lipschitz function.

\begin{example}
An approximated solution to the problem
considered in Example~\ref{example1} can be obtained
following exactly the same steps as in Example~\ref{ex:new:sec6}.
Recall that the minimizer \eqref{eq:gs:ex} to that
problem is not a Lipschitz function.
As before, one has $x(0)=0$ and the Caputo derivative coincides
with the Riemann--Liouville derivative. We approximate the fractional
problem using \eqref{expanMom}.
Let $^C_0D_t^{\alpha}x(t)=u(t)$. Then \eqref{eq:h:1} holds.
In this case the variable $z(t)$ satisfies
$$
\dot{z}(t)=\left( x(t)-\frac{t^\alpha}{\Gamma(\alpha+1)}\right)^2,
$$
and we approximate the fractional variational problem
with the following classical one:
\begin{equation*}
\begin{gathered}
\tilde{J}[x] = \int_0^1 \left[(u(t)-1)^2+z(t)\right]dt \longrightarrow \min \\
\begin{cases}
\dot{x}(t) = -AB^{-1}t^{-1}x(t)
+\sum_{k=2}^{N}B^{-1}C_kt^{-k}v_k(t)+B^{-1}t^{\alpha-1}u(t)\\
\dot{v}_k(t) = (1-k)t^{k-2}x(t),\qquad k=1,2,\ldots\\
\displaystyle\dot{z}(t) 
= \left( x(t)-\frac{t^\alpha}{\Gamma(\alpha+1)}\right)^2,
\end{cases}
\end{gathered}
\end{equation*}
subject to the boundary conditions $x(0)=0$, $z(0)=0$ and $v_k(0)=0$, $k=1,2,\ldots$
Setting $N=2$, the Hamiltonian is given by
\begin{multline*}
H=-[(u(t)-1)^2+z(t)]+p_1(t)\left(-AB^{-1}t^{-1}x(t)+B^{-1}C_2t^{-2}v_2(t)
+B^{-1}t^{\alpha-1}u(t)\right)\\
-p_2(t)x(t)+p_3(t)\left( x(t)-\frac{t^\alpha}{\Gamma(\alpha+1)}\right)^2.
\end{multline*}
The classical theory \cite{Pontryagin} tells us to solve the system
\begin{equation}
\label{sys}
\left\{
\begin{array}{ll}
\dot{x}(t)   & =-AB^{-1}t^{-1}x(t)+ B^{-1}C_2t^{-2}v_2(t)
+\displaystyle\frac{1}{2}B^{-2}t^{2\alpha-2}p_1(t)+B^{-1}t^{\alpha-1}\\[5pt]
\dot{v}_2(t) & =-x(t)\\[5pt]
\dot{z}(t)   & =\left(x(t)-\displaystyle\frac{t^\alpha}{\Gamma(\alpha+1)}\right)^2\\[5pt]
\dot{p}_1(t) & =AB^{-1}t^{-1}p_1(t)+p_2(t)-2p_3(t)\left(x(t)
-\displaystyle\frac{t^\alpha}{\Gamma(\alpha+1)}\right)\\
\dot{p}_2(t) & =-B^{-1}C_2t^{-2}p_1(t)\\
\dot{p}_3(t) & = 1,
\end{array}
\right.
\end{equation}
subject to boundary conditions
\begin{equation}
\label{sysB}
\begin{cases}
x(0)=0 \\
v_2(0)=0\\
z(0)=0
\end{cases}
,\qquad
\begin{cases}
\displaystyle x(1)=\frac{1}{\Gamma(\alpha+1)}\\
p_2(1)=0\\
p_3(1)=0.
\end{cases}
\end{equation}
As done in Example~\ref{ex:new:sec6}, we solved \eqref{sys}--\eqref{sysB}
using the MATLAB$^\circledR$ built-in function \textsf{bvp4c}.
The resulting graph for $x(t)$, together with the corresponding
value of $J$, is given in Figure~\ref{Fig} in contrast with
the exact minimizer \eqref{eq:gs:ex}.
\begin{figure}[!ht]
\begin{center}
\includegraphics[scale=.8]{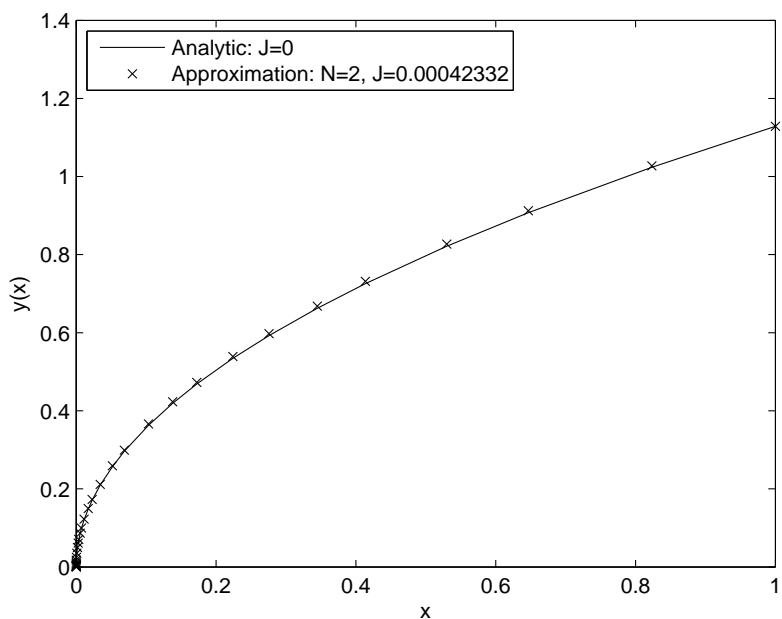}
\caption{Analytic {\it versus} numerical solution
to problem of Example~\ref{example1}.}\label{Fig}
\end{center}
\end{figure}
\end{example}

\chapter*{Conclusion and future work}
\markboth{\MakeUppercase{Conclusion and future work}}{}
\addcontentsline{toc}{part}{Conclusion and future work}

The realm of numerical methods in scientific fields is vastly growing due 
to the very fast progresses in computational sciences and technologies. 
Nevertheless, the intrinsic complexity of fractional calculus, caused 
partially by non-local properties of fractional derivatives and integrals, 
makes it rather difficult to find efficient numerical methods in this field. 
It seems enough to mention here that, up to the time of this thesis  
and to the best of our knowledge, there is no routine available for solving 
a fractional differential equation as Runge--Kutta for ordinary ones. Despite this fact, 
however, the literature exhibits a growing interest and  improving achievements 
in numerical methods for fractional calculus in general 
and fractional variational problems specifically.

This thesis is devoted to discussing some aspects of the very well-known methods 
for solving variational problems. Namely, we studied the notions of direct 
and indirect methods in the classical calculus of variation and also we mentioned 
some connections to optimal control. Consequently, we introduced the generalizations 
of these notions to the field of fractional calculus of variations 
and fractional optimal control.

The method of finite differences, as discussed here, seems to be a potential first candidate 
to solve fractional variational problems. Although a first order approximation 
was used for all examples, the results are satisfactory and even though 
it is more complicated than in the classical case, it still inherits 
some sort of simplicity and an ease of implementation.

The outcomes of our works related to direct methods are as follows:
\begin{itemize}
\item S. Pooseh, R. Almeida\ and\ D. F. M. Torres, Discrete Direct Methods 
in the Fractional Calculus of Variations, Proceedings of FDA'2012, May 14-17, 2012, 
Hohai University, Nanjing, China. Paper \#042, Winner of a best oral presentation award \cite{fda12};
\item S. Pooseh, R. Almeida and D.F.M. Torres, Discrete direct methods 
in the fractional calculus of variations, 
Comput. Math. Appl.,{\bf 66} (2013), no. 5, 668--676 \cite{PATDisDir};
\item S. Pooseh, R. Almeida\ and\ D. F. M. Torres, 
A discrete time method to the first variation of fractional order variational functionals, 
Cent. Eur. J. Phys, in press \cite{PATDisVar}.
\end{itemize}

Roughly speaking, an Euler-like direct method reduces a variational problem
to the solution of a system of algebraic equations. When the system is linear,
we can freely increase the number of mesh points, $n$, and obtain better solutions
as long as the resulted matrix of coefficients is invertible. The method
is very fast, in this case, and the execution time is of order $10^{-4}$
for Examples~\ref{Example1} and \ref{Example2}. It is worth, however, to keep in mind
that the  Gr\"{u}nwald--Letnikov approximation is of first order, $\mathcal{O}(h)$,
and even a large $n$ cannot result in a high precision. Actually, by increasing $n$,
the solution slowly converges and in Example~\ref{Example2}, a grid of $30$ points
has the same order of error, $10^{-3}$, as a $5$ points grid.
The situation is completely different when the problem ends with a nonlinear system.
In Example~\ref{Example3}, a small number of mesh points, $n=5$, results in a poor
solution with the error $E=1.4787$. The MATLAB$^\circledR$ built in function \textsf{fsolve}
takes $0.0126$ seconds to solve the problem. As one increases the number of mesh points,
the solution gets closer to the analytic solution and the required time increases drastically.
Finally, by $n=90$ we have $E=0.0618$ and the time is $T=26.355$ seconds.


In practice, we have no idea about the solution in advance and the worst case should
be taken into account. Comparing the results of the three examples considered,
reveals that for a typical fractional variational problem,
the Euler-like direct method needs a large number of mesh points
and most likely a long running time.

The lack of efficient numerical methods for fractional variational problems, 
is overcome partially by the indirect methods of this thesis. 
Once we transformed the fractional variational problem to an approximated classical one, 
the majority of classical methods can be applied to get an approximate solution. 
Nevertheless, the procedure is not completely straightforward. The singularity 
of fractional operators is still present in the approximating formulas 
and it makes the solution procedure more complicated.

During the last three decades, several numerical methods have been developed 
in the field of fractional calculus. Some of their advantages, disadvantages, 
and improvements, are given in \cite{sug:r3}. Based on two continuous expansion 
formulas \eqref{ExpIntInf} and \eqref{expanMomInf} for the left Riemann--Liouville 
fractional derivative, we studied two approximations \eqref{expanInt} and \eqref{expanMom} 
and their applications in the computation of fractional derivatives. Despite the fact 
that the approximation \eqref{expanInt} encounters some difficulties from the presence 
of higher-order derivatives, it exhibits better results at least for the evaluation 
of fractional derivatives. The same studies were carried out for fractional integrals 
as well as some other fractional operators, namely Hadamard derivatives and integrals, 
and Caputo derivatives.

The full details regarding these approximations and their advantages, 
disadvantages and applications can be found in the following papers:
\begin{itemize}
\item S. Pooseh, R. Almeida\ and\ D. F. M. Torres, 
Numerical approximations of fractional derivatives with applications, 
Asian Journal of Control {\bf 15} (2013), no.~3, 698--712 \cite{PATFracDer};
\item S. Pooseh, R. Almeida\ and\ D. F. M. Torres, 
Approximation of fractional integrals by means of derivatives, 
Comput. Math. Appl. {\bf 64} (2012), no.~10, 3090--3100 \cite{PATFracInt};
\item S. Pooseh, R. Almeida and D.F.M. Torres, 
Expansion formulas in terms of integer-order derivatives 
for the Hadamard fractional integral and derivative,  
Numerical Functional Analysis and Optimization 33 (2012) No 3, 301--319 \cite{PATHad}.
\end{itemize}

Approximation \eqref{expanMom} can also be generalized to include higher-order derivatives 
in the form of \eqref{Gen}. The possibility of using \eqref{expanMom} to compute fractional 
derivatives for a set of tabular data was discussed. Fractional differential equations 
are also treated successfully. In this case the lack of initial conditions makes \eqref{expanInt} 
less useful. In contrast, one can freely increase $N$, the order of approximation \eqref{expanMom}, 
and find better approximations. Comparing with \eqref{expanAtan}, our modification provides better results.

For fractional variational problems, the proposed expansions may be used at two different 
stages during the solution procedure. The first approach, the one considered 
in Chapter~\ref{Indirect}, consists in a direct approximation of the problem, 
and then treating it as a classical problem, using standard methods to solve it.
The second approach, Section~\ref{sub:sec:fft}, is to apply the fractional 
Euler--Lagrange equation and then to use the approximations 
in order to obtain a classical differential equation.

The results concerning the application 
of the approximations proposed in this work have been published as follows:
\begin{itemize}
\item R. Almeida, S. Pooseh\ and\ D. F. M. Torres, 
Fractional variational problems depending on indefinite integrals, 
Nonlinear Anal. {\bf 75} (2012), no.~3, 1009--1025 \cite{APTIndInt};
\item S. Pooseh, R. Almeida\ and\ D. F. M. Torres, 
Fractional order optimal control problems with free terminal time, 
J. Ind. Manag. Optim. 10 (2014), no. 2, 363--381 \cite{PATFree};
\item S. Pooseh, R. Almeida\ and\ D. F. M. Torres, 
Free fractional optimal control problems, 
2013 European Control Conference (ECC) July 17-19, 2013, 
Zurich, Switzerland \cite{PATZurich}.
\item S. Pooseh, R. Almeida\ and\ D. F. M. Torres, 
A numerical scheme to solve fractional optimal control problems, 
Conference Papers in Mathematics, vol. 2013, Article ID 165298, 10 pages, 2013 \cite{PATScheme}.
\end{itemize}

The direct methods for fractional variational problems presented in this thesis,
can be improved in some stages. One can try different approximations
for the fractional derivative that exhibit higher order precisions, 
e.g. Diethelm's backward finite differences \cite{Kai2}.
Better quadrature rules can be applied to discretize the functional and, finally,
we can apply more sophisticated algorithms for solving the resulting system of algebraic equations.
Further works are needed to cover different types of fractional variational problems.

Regarding indirect methods, the idea of transforming a fractional problem 
to a classic one seems a useful way of extending the available classic methods 
to the field of fractional variational problems. Nevertheless, improvements 
are needed to avoid the singularities of the approximations \eqref{expanMom} 
and \eqref{expanInt}. A more practical goal is to implement some software packages 
or tools to solve certain classes of fractional variational problems. 
Following this research direction may also end in some solvers 
for fractional differential equations.

In the course of this thesis we have also studied the use of fractional calculus
in epidemiology, that is not included in this thesis \cite{PoosehRod}. 
The proposed approach is illustrated
with an outbreak of dengue disease, which is motivated by
the first dengue epidemic ever recorded in the Cape Verde islands
off the coast of west Africa, in 2009.
Describing the reality through a mathematical model,
usually a system of differential equations,
is a hard task that has an inherent compromise
between simplicity and accuracy. In our work, we consider
a very basic model to dengue epidemics. It turns out that, in general,
this basic/classical model does not provide enough good results.
In order to have better results, that fit the reality,
more specific and complicated set of differential equations
have been investigated in the literature, see
\cite{Rodrigues2010a,Rodrigues2010b,Rodrigues2011} and references therein.
We have proposed a completely new approach to the subject.
We keep the simple model and substitute the usual (local)
derivatives by (non-local) fractional differentiation. The use
of fractional derivatives allow us to model memory effects,
and result in a more powerful approach to epidemiological models:
one can then design the order $\alpha$ of fractional differentiation
that best corresponds to reality. The classical case is recovered
by taking the limit when $\alpha$ goes to one. Our investigations
show that even a simple fractional model 
may give surprisingly good results \cite{PoosehRod}.
However, the transformation of a classical model into
a fractional one makes it very sensitive to the order of
differentiation $\a$: a small change in $\a$
may result in a big change in the final result.
 This work was presented at ICNAAM 2011,
Numerical Optimization and Applications Symposium:
\begin{itemize}
\item S. Pooseh, H. S. Rodrigues\ and\ D. F. M. Torres, 
Fractional Derivatives in Dengue Epidemics, 
Numerical Analysis and Applied Mathematics ICNAAM 2011, 
AIP Conf. Proc. 1389, 739--742 (2011) \cite{PoosehRod},
\end{itemize}
and a more sophisticated study has been reported in \cite{KaiDeng}.

Our work can be extended in several ways:
by fractionalizing more sophisticated models;
by considering different orders of fractional derivatives
for each one of the state variables, \textrm{i.e.},
models of non-commensurate order. Finally, we can combine
the results of this PhD thesis in the framework of fractional
optimal control of epidemic models.

\CP

\CP
\addcontentsline{toc}{part}{Index}
\printindex
\CP

\end{document}